\documentclass[11pt]{article}

\usepackage{etex}
\usepackage{xcolor}
\usepackage{url}
\usepackage{ulem}
\usepackage{framed}
\usepackage{threeparttable}
\usepackage[algo2e,linesnumbered,vlined,ruled]{algorithm2e}
\usepackage{graphics,graphicx,epsf,subfigure,epstopdf}
\usepackage{comment}
\usepackage{times}
\usepackage{amsmath,amsxtra,amsfonts,amscd,amssymb,bm}
\usepackage{supertabular,multirow}
\usepackage{booktabs}
\usepackage{enumerate}
\usepackage{mathtools}
\usepackage{tikz,pgfplots,pgfplotstable}
\usepackage{xcolor}
\usepackage{makecell}
\usepackage{tabularx}
\usepackage{algorithm}
\usepackage{algorithmic}
\definecolor{skyblue}{rgb}{0.53, 0.81, 0.92}
\definecolor{limegreen}{rgb}{0.2, 0.8, 0.2}
\usepackage[colorlinks=true, linkcolor=black, citecolor=black, urlcolor=black]{hyperref}

\newtheorem{theorem}{Theorem}[section]
\newtheorem{remark}[theorem]{Remark}
\newtheorem{lemma}[theorem]{Lemma}
\newtheorem{example}[theorem]{Example}
 \newtheorem{proposition}[theorem]{Proposition}
 \newtheorem{definition}[theorem]{Definition}
\newtheorem{assumption}[theorem]{Assumption}
\newtheorem{proof}[theorem]{Proof}

\newtheorem{corollary}[theorem]{Corollary}

\newcommand{\be}{\begin{equation}}
\newcommand{\ee}{\end{equation}}
\newcommand{\bee}{\begin{equation*}}
\newcommand{\eee}{\end{equation*}}
\newcommand{\bea}{\begin{eqnarray}}
\newcommand{\eea}{\end{eqnarray}}
\newcommand{\beaa}{\begin{eqnarray*}}
\newcommand{\eeaa}{\end{eqnarray*}}

\newcommand{\R}{\mathbb{R}}

\newcommand{\cX}{\mathcal{X}}

\newcommand{\cL}{\mathcal{L}}

\newcommand{\dist}{\mathrm{dist}}

\newcommand{\st}{\mathrm{s.\,t.}\,\,}

\newcommand{\ACal}{\mathcal{A}}

\newcommand{\ICal}{\mathcal{I}}

\newcommand{\sign}{\mathrm{sign}}

\newcommand{\cI}{\mathcal{I}}

\newcommand{\prox}{\mathrm{prox}}

\newcommand{\X}{\mathcal{X}}
\newcommand{\Ecal}{{\mathcal{E}}}

\newcommand{\diag}{\mathrm{diag}}

\newcommand{\cK}{\mathcal{K}}

\newcommand{\iprod}[2]{\left\langle{#1},{#2}\right\rangle}

\newcommand{\Ical}{\mathcal{I}}

\newcommand{\Bb}{\mathbf{B}}

\newcommand{\bA}{\mathbf{A}}

\newcommand{\supgrad}{\bar\partial}
\numberwithin{equation}{section}

\definecolor{Gray}{rgb}{0.5,0.5,0.5}
\DeclareMathOperator*{\argmin}{argmin}

\definecolor{ochd}{RGB}{178,129,0} 	%

\newcommand{\revise}[1]{{#1}}

\title{  The Augmented Lagrangian Methods: Overview and Recent Advances }  

\author{Kangkang Deng\thanks{School of Science,  National University of Defense Technology, Changsha, 410073,
CHINA. E-mail: \texttt{\href{mailto:freedeng1208@gmail.com}{freedeng1208@gmail.com}}} \and 
Rui Wang\thanks{School of Mathematics, Southwestern University of Finance and Economics, Chengdu 611130, CHINA. E-mail: \texttt{\href{mailto:ruiwang@bicmr.pku.edu.cn}{ruiwang@bicmr.pku.edu.cn}}}\and
Zhenyuan Zhu\thanks{Beijing International Center for Mathematical Research, Peking University, Beijing 100871, CHINA. E-mail: \texttt{\href{mailto:zhenyuanzhu@pku.edu.cn}{zhenyuanzhu@pku.edu.cn}}}\and 
Junyu Zhang\thanks{Department of Industrial Systems Engineering and Management, National University of Singapore, Singapore. E-mail: \texttt{\href{mailto:junyuz@nus.edu.sg}{junyuz@nus.edu.sg}}} \and
Zaiwen Wen\thanks{corresponding author. Beijing International Center for Mathematical Research, Center for Machine Learning Research and Changsha Institute for Computing and Digital Economy, Peking University, Beijing 100871, CHINA. E-mail: \texttt{\href{mailto:wenzw@pku.edu.cn}{wenzw@pku.edu.cn}}}
}

\begin{document}

\maketitle
\vspace{-0.5cm}
\begin{abstract}

Large-scale constrained optimization is pivotal in modern scientific, engineering, and industrial computation, often involving complex systems with numerous variables and constraints. This paper provides a unified and comprehensive perspective on constructing augmented Lagrangian functions (based on Hestenes-Powell-Rockafellar augmented Lagrangian) for various optimization problems, including nonlinear programming and convex and nonconvex composite programming. We present the augmented Lagrangian method (ALM), covering its theoretical foundations in both convex and nonconvex cases, and discuss several successful examples and applications. Recent advancements have extended ALM's capabilities to handle nonconvex constraints and ensure global convergence to first and second-order stationary points. For nonsmooth convex problems, ALM utilizes proximal operations, preserving desirable properties such as locally linear  convergence rates. Furthermore, recent progress has refined the complexity analysis for ALM and tackled challenging integer programming instances. This review aims to offer a thorough understanding of ALM's benefits and limitations, exploring different ALM variants designed to enhance convergence and computational performance. We also illustrate effective algorithms for ALM subproblems across different types of optimization problems and highlight practical implementations in several fields.

\end{abstract}

\newpage
\tableofcontents
 \newpage
\section{Introduction}
    Large-scale constrained optimization plays a crucial role in modern scientific, engineering, and industrial computation. These problems often involve complex systems with {a large number of variables and constraints}. In particular,  
    many of them share the common general form of
\begin{equation}\label{general_opt_pro}
\min_{x\in \mathbb{R}^n}\ f(x)+h(x)\quad \text { s.t. }\quad\ c(x) \in \mathcal{Q}, x\in \mathcal{K},
\end{equation}
where $f:\mathbb{R}^n \rightarrow \mathbb{R}$ and $h:\mathbb{R}^n \rightarrow \mathbb{R}$ respectively represent the smooth and nonsmooth components of the objective function, 
$c(x):\mathbb{R}^n \rightarrow \mathbb{R}^m$ is a nonlinear mapping, and  
$\mathcal{Q}$ and $\mathcal{K}$ are two closed sets. Typical examples of formulation \eqref{general_opt_pro} include linear and nonlinear programming, sparse optimization, low-rank matrix recovery, image reconstruction, combinatorial optimization, etc. \cite{bertsekas1997nonlinear,rockafellar1997convex,donoho2006compressed,cortes1995support,beck2009fast,chan2013constrained,papadimitriou2013combinatorial}. Based on different problem backgrounds, the convexity, linearity, and even continuity conditions of the above mappings and sets may vary from case to case. In general, when either the objective function or the constraints are nonlinear, problem \eqref{general_opt_pro} is called nonlinear programming (NLP), for which a general discussion can be found in  \cite{bertsekas1997nonlinear,NocedalWright06,sun2006optimization}.  
Further developments in this area include the dual approach to solving nonlinear programming problems by Rockafellar \cite{rockafellar1973dual}.   When $f,h$ and $c$ are convex functions, 
$\mathcal{Q}$, $\mathcal{K}$ are convex sets, problem \eqref{general_opt_pro} becomes a convex optimization problem. It is significant due to its desirable properties such as the guarantee of finding a global minimum efficiently and appears frequently in machine learning, signal processing, and finance \cite{beck2017first,boyd2004convex,bertsekas2015convex}.  When the objective function or constraints are nonconvex, the problem can be regarded as a nonconvex optimization problem. It is more challenging due to the presence of multiple local minima but is highly relevant in fields such as deep learning, control theory, and engineering design \cite{bolte2014proximal,bolte2018nonconvex,li2021rate}.

In recent years, significant progress is being made in the field of optimization. Advanced algorithms and computational techniques are developed to address the challenges posed by large-scale and nonsmooth problems. Among these advancements, the augmented Lagrangian method (ALM) emerges as a particularly powerful tool. ALM combines the strengths of penalty methods and Lagrange multipliers, offering robust and efficient solutions for complex optimization problems. {The conceptual foundation of the ALM can be traced back to the seminal work of Arrow and Solow \cite{arrow1958gradient}, where they propose a framework based on differential equations to locate saddle points of the augmented Lagrangian function by simultaneously updating both the primal variables and the multipliers. The formal ALM is first established in the works of Hestenes \cite{hestenes1969multiplier} and Powell \cite{powell1969method} in the 1960s, where it is originally referred to as the “method of multipliers”, which is designed to solve equality constrained problems. The extension to inequality constraints is described by Rockafellar \cite{rockafellar1971duality,rockafellar1973multiplier}.}  {We briefly summarize several influential surveys and books on the ALM. Bertsekas \cite{bertsekas2014constrained} offers a systematic treatment of classical Lagrange multiplier methods and augmented Lagrangian frameworks for constrained optimization. Ito and Kunisch \cite{ito2008lagrange} focus on nonlinear variational problems in infinite-dimensional spaces, employing Lagrange multiplier theory to develop efficient convex constrained optimization algorithms based on augmented Lagrangian techniques, with applications in PDE-constrained optimal control, image analysis, mechanical contact problems, and related fields. Fortin and Glowinski \cite{fortin2000augmented} emphasize the theoretical foundations of ALM and demonstrate its applicability to a broad class of boundary value problems for partial differential equations, including Stokes and Navier--Stokes systems as well as elastoplasticity models, with applications in mathematical physics, continuum mechanics, and engineering sciences. Birgin and Martínez \cite{birgin2014practical} address the practical implementation of ALM, offering detailed algorithmic guidance, illustrative examples, and software such as the Algencan solver, along with MATLAB code examples. While these contributions are foundational, most existing surveys focus on specific problem classes or classical constrained optimization settings.}

Recent breakthroughs have extended the capabilities of ALM to handle nonconvex constraints \cite{gasimov2002augmented,birgin2005numerical, houska2016augmented,chen2017augmented,li2021augmented,luo2019solving,rockafellar2023convergence} and demonstrate the global convergence to first and second-order stationary points \cite{birgin2018augmented,sahin2019inexact}. For nonsmooth problems,  ALM serves as an efficient approach by utilizing proximal operations on the nonsmooth terms within its iterations \cite{bolte2014proximal,dhingra2018proximal,de2023constrained,zhu2017nonconvex}, which preserves desirable properties like superlinear local convergence rates \cite{fernandez2012local,rockafellar2023convergence,li2018highly}. Moreover, recent work has provided refined complexity analysis for ALM \cite{birgin2020complexity, lan2016iteration,xie2021complexity,li2021rate,kong2023accelerated} and tackled challenging integer programming instances \cite{feizollahi2017exact,gu2020exact,kannan1994augmented}. Therefore, despite being a classic method, these new theoretical and practical developments in the ALM field motivate us to provide a comprehensive review covering the latest advances and applications.
 In the implementation and analysis of the ALM algorithm, the two key steps are 
 \begin{itemize}
 \item[(1)] constructing an augmented Lagrangian (AL) function; 
 \item[(2)] designing an ALM algorithm by solving primal subproblem and making dual update.
 \end{itemize}

{ For the first step, this paper primarily focuses on the Hestenes-Powell-Rockafellar {(HPR)} AL function \cite{hestenes1969multiplier,powell1969method} based on the quadratic penalty, which remains the most prevalent in the literature. 
  When tackling convex composite problems—where the objective decomposes into multiple components—one common strategy is to introduce auxiliary variables that decouple these terms, then to build an AL function for the resulting reformulation. We offer a unified interpretation of all such constructions through the lens of the Moreau envelope. It provides a smooth approximation of any convex function and, when applied to an indicator function of a constraint, recovers precisely the quadratic penalty. By approximating both difficult constraint sets and nonsmooth terms in the objective via their Moreau envelopes, one arrives naturally at the AL function. This viewpoint applies to virtually any problem structure and clarifies why the resulting AL function is well suited to efficient numerical solution. Further details are provided in Section \ref{sec2:AL:sd}.  
  To provide rigorous theoretical support for the ALM algorithm, we establish a key property of the augmented Lagrangian function—its strong duality—which we hereafter denote as AL strong duality. In the convex setting, AL strong duality coincides with the classical Lagrangian strong duality. } However, it has been shown that the strong duality property holds even without convexity by regarding the penalty coefficient of the AL function as a dual variable \cite{rockafellar1976augmented,bertsekas2014constrained,huang2003unified}. Recently, Feizollahi et al. \cite{feizollahi2017exact}  generalized such result to integer programming problems, where continuity also fails.   { To give readers a comprehensive view, we expand beyond HPR to review several important alternative constructions—including the shift-penalty function, the generalized augmenting function via dualizing parameterization, and a variety of other variants motivated by smoothness considerations (e.g. modified barrier functions \cite{polyak1992modified}, exponential AL, Bregman-distance-based AL, etc.).}

{ For the second step, the ALM generates iterates  by alternating between a (possibly inexact) minimization in primal variable and a multiplier update. Recently literature starts to interpret the ALM algorithm as a dual ascent method \cite{bertsekas2014constrained,zhang2022global}. The key supporting reason for such a perspective initially comes from the AL strong duality.    Therefore, we introduce the ALM algorithm from the perspective of the AL dual problem. We show that the ALM algorithm can be regarded as a supergradient method applied to the dual problem \eqref{eqn:AL-dual}, owing to the strong duality established in Section \ref{sec2:AL:sd}.}
 The classic ALM typically solves the primal subproblems exactly, which can be computationally expensive, especially for large-scale or ill-conditioned problems \cite{bertsekas2014constrained,rockafellar1973dual}. Inexact strategies such as linearized and proximal ALM \cite{yang2013linearized,xu2021first,dhingra2018proximal,hermans2022qpalm}, reduce computational complexity by solving approximations of the subproblems.
Kanzow's spectral gradient approach \cite{jia2023augmented} enables handling nonconvex nonsmooth subproblems via gradient sampling.  Bolte's alternating linearization method \cite{bolte2014proximal} combines active-set prediction with ALM.
For problems satisfying certain block structures, the block coordinate descent (BCD) method can be employed to solve the subproblem by updating individual blocks iteratively \cite{de2023constrained,tosserams2008augmented,chen2019extended}, alleviating the need for full minimization. For problems where the AL function exhibits favorable properties like smoothness, second-order methods \cite{zhao2010newton,birgin2008structured,he2023newton} have been applied to solve ALM subproblems efficiently.
In addition, adaptive and inexact primal updates \cite{curtis2015adaptive,sujanani2023adaptive,fernandez2012local,xu2021iteration} have been introduced to handle large-scale and complex problems more effectively. The flexibility in reformulating the ALM subproblems while preserving theoretical guarantees has been an important research area, enabling the development of practical ALM-based solvers for challenging structured optimization problems.

{In addition to providing a unified perspective on the construction of the AL function and the design of ALM, this survey is also motivated by the need to present the most recent advances in ALM and to offer comprehensive guidance on how to tailor ALM to practical applications. Overall, this paper addresses the following two questions:

\begin{itemize}
    \item What are the most recent theoretical and algorithmic developments in ALM?
    \item How can one construct the AL function and design ALM for specific applications?
\end{itemize}
To answer the first question, we consider both convex and nonconvex settings and detail three key theoretical properties of ALM—global convergence, local convergence, and iteration complexity, which are given in Section \ref{sec:convex case} and Section \ref{sec:nonconvex_case}. We further review several ALM variants in Section \ref{sec:variants_ALM}, including the linearized ALM, the proximal ALM, the accelerated ALM, emphasizing the settings in which each variant demonstrates particular advantages. We also review two closely related algorithms, including the alternating direction method of multipliers (ADMM) and primal–dual methods.  With respect to the second question, we illustrate the application of ALM to an assortment of prototypical problem classes. In particular, we explore efficient strategies for solving the primal subproblem—often the dominant computational cost in ALM—in these contexts (refer to Section \ref{sec:examples}). Finally, we present numerical experiments that validate the practical performance and effectiveness of the proposed ALM designs in Section \ref{sec-numerical-experiments}. }

\subsection{Notation}
We use the following notations for different forms of the AL functions. Let $\mathcal{L}_\rho$ denote the classical AL function; 
$\mathbb{L}_\rho$ represents the AL function after eliminating slack variables, simplifying the expression;
 $L_\rho$ denotes a modified Lagrangian function used for specific problem types such as integer programming. Let $\mathbb{S}^n$ be the collection of all $n$-by-$n$ symmetric matrices. For any
 matrix $X \in \R^{n\times n}$, $\diag(X)$ denotes a column vector consisting of all diagonal entries of $X$.
 For any vector $x \in \R^n$, $\text{Diag}(x)$ is an $n$-by-$n$ diagonal matrix whose $i$-th diagonal entry is
 $x_i$. The signum function of a real number 
$a$ is a piecewise function which is defined as $\sign(a)=\left\{\begin{array}{ll}
   1, & a >0,\\
    0, &a =0,\\
 -1, & a <0.    
\end{array}\right.$
Given a set $\Omega$, we let $\Pi_\Omega(\cdot)$ and $\delta_\Omega(\cdot)$ denote the projection and the indicator function of the set $\Omega$, i.e.,
$\delta_\Omega(x)=\left\{\begin{array}{ll}
    +\infty, & x \notin \Omega,\\
    0, &x \in \Omega.
\end{array}\right.$ Define $\hat\Pi_\Omega(x)=x-\Pi_\Omega(x)$. Given $x \in \R^n$, we denote $J_c(x)$ as the Jacobi matrix of the function $c$ at $x$. The symbol $\preceq_{\cK}$ denotes the partial order induced by $\cK$, that is, 
$
  y \preceq_{\cK} z\quad \mathrm{if\ and \ only\ if}\quad z-y \in \cK$. 
Let $\mathrm{Tr}(A)$ denote the trace of a square matrix $A$, which is defined to be the sum of elements on the main diagonal of $A$. The notation $\|\cdot\|$ denotes the norm, calculated as the square root of the sum of the absolute squares of its elements, applicable to both matrices and vectors. The notation $\dist(x, \cK)$ represents the distance from the point $x$ to the set $\cK$, defined as: $\dist(x, \cK) = \min_{y \in \cK} \|x - y\|
$.  $a \circ b$ denotes the Hadamard product of $a \in \R^n$ and $b\in \R^n$, i.e., $a \circ b=(a_1b_1,a_2b_2,\cdots,a_nb_n)^{\top}$.

   \subsection{Organization} The rest of this paper is organized as follows. In Section \ref{sec2:AL:sd}, we discuss the AL function and strong duality. This includes the general form of constrained optimization problems, classic AL functions, and applications in convex and nonconvex optimization. Section \ref{section3} provides an overview of the ALM. We cover the augmented dual problem and a general ALM framework. Section \ref{sec:convex case} focuses on convex optimization problems, discussing global and local convergence, and iteration complexity. Section \ref{sec:nonconvex_case} provides a parallel discussion for nonconvex optimization problems. Section \ref{sec:variants_ALM} explores various ALM variants such as linearized, proximal,  accelerated ALMs, and ADMM.  Section \ref{sec:examples} presents applications of ALM in areas such as  nonconvex optimization, nonlinear programming, sparse optimization, semidefinite programming, manifold optimization, integer programming, and reinforcement learning. Numerical results on a few selected benchmark problems are shown in Section \ref{sec-numerical-experiments}. Finally, Section \ref{sec:summary} summarizes the main contributions and suggests directions for future research.

\section{Augmented Lagrangian function and strong duality}\label{sec2:AL:sd}

In this section, we present the augmented Lagrangian (AL) function as a critical component of the ALM for several classical problems, including general nonlinear programming, convex and nonconvex composite optimization, and discrete optimization. We will discuss constraint qualifications and optimality conditions within the framework of constrained optimization, establishing strong duality based on the AL function.  { Crucially, these strong duality results rely on the AL function rather than the standard Lagrangian function which fails to yield strong duality in nonconvex and discrete settings. By treating the penalty parameter in the AL function as an optimization decision variable, it is shown that strong duality persists even for nonconvex and discrete optimization problems. Moreover, since ALM can be interpreted as a supergradient‐ascent method applied to the AL dual problem (refer to Section \ref{sec:dualascent}); accordingly, the strong duality of the AL function provides the theoretical guarantees underpinning the convergence and performance of ALM.} %

\subsection{General form constrained optimization problem}\label{sec:general-opt}

To start the discussion of this section, let us first consider the following nonlinear optimization problem, which is the most fundamental formulation of the constrained optimization problems: 
\begin{equation}\label{prob}
    \begin{split}
    \min_{x\in\mathbb{R}^n}\ & f(x),\\
    \st\ & c_i(x)=0, i\in\mathcal{E},\\
    &c_i(x)\leq0, i\in\mathcal{I}.
    \end{split}
\end{equation}
In this formulation, $f(x):\mathbb{R}^n\rightarrow \mathbb{R}$ and $c_i(x):\mathbb{R}^n\rightarrow \mathbb{R}, i\in\mathcal{E}\cup\mathcal{I}$ are proper and lower semicontinuous (l.s.c.)  functions which can potentially be nonconvex and nonsmooth, with $\mathcal{E}$ and $\mathcal{I}$ denoting the index sets of equality and inequality constraints, respectively. 
Note that a function is l.s.c. if and only if it is closed \cite[Theorem 6]{rockafellar2009variational}, hence the two conditions are often assumed interchangeably in the literature. Next, to facilitate the form discussion, let us provide a few crucial definitions.
\begin{definition}[active set]\label{definition_Ax} 
Let $x$ be an arbitrary feasible solution to problem \eqref{prob}. A constraint is said to be active at $x$ if $c_i(x)=0$ for any $i\in \ICal\cup\Ecal$, and a constraint is said to be inactive at $x$ if the strict inequality $c_i(x)<0$ holds for any $i\in \ICal$. The {index} set of all active constraints at $x$ is called the active set at $x$, and we denote it {by} $\ACal(x):=\Ecal \cup \{i \in \ICal | c_i(x)=0\}$.  
\end{definition}

To establish the Karush-Kuhn-Tucker (KKT) conditions that provide guidance to algorithm design and solution sanity check, some regularity conditions, often referred to as constraint qualifications, are required to prevent degenerate behaviors at points of interest. 
Below we introduce a few {the} most popular constraint qualifications providing the continuous differentiability of $c$.

\begin{definition}[LICQ \cite{bertsekas1997nonlinear}]\label{LICQ}
Suppose $x$ is feasible to problem \eqref{prob} 
 and  $c_i$ is continuously differentiable for each $i\in\mathcal{E}\cup\mathcal{I}$. We say that the linear independence constraint qualification (LICQ) holds at $x$ if the gradients of all active constraints $\{\nabla c_i (x), i \in \ACal(x)\}$ are linearly independent.
\end{definition}

While the LICQ offers a straightforward and convenient sufficient condition for deriving the KKT system, it can be somewhat restrictive, as many relevant problems may not satisfy the LICQ even though the KKT conditions remain well-defined at local minima. Therefore, we introduce the Mangasarian-Fromovitz Constraint Qualification (MFCQ) as a popular alternative.

\begin{definition}[MFCQ \cite{mangasarian1967fritz}]\label{MFCQ}
Suppose $x$ is feasible to problem \eqref{prob} 
and  $c_i$ is continuously differentiable for each $i\in\mathcal{E}\cup\mathcal{I}$. We say that the Mangasarian-Fromovitz constraint qualification (MFCQ) holds at $x$ if there
exists a vector $w \in \R^n$ such that
\begin{equation*}
 \begin{split}
\nabla c_i(x)^\top w>0, & \quad \mathrm{for~all}~ i \in \ACal(x) \cap \ICal,\\
\nabla c_i(x)^\top w=0, & \quad \mathrm{for~all}~ i \in \Ecal,
\end{split}
\end{equation*}
and the gradients of equality constraints $\{\nabla c_i (x), i \in \Ecal\}$ are linearly independent.
\end{definition}

If problem \eqref{prob} is convex, that is, both the objective function $f(x)$ and the inequality constraint functions $c_i, i\in\Ical$ are convex and the equality constraints are affine, a widely adopted constraint qualification is Slater's condition. Denote the domain of \eqref{prob} as $\mathcal{D}:=\cap_{i\in\Ecal\cup\Ical}\mathrm{dom}(c_i)\cap\mathrm{dom}(f)$ and denote the relative interior of $\mathcal{D}$ as $\mathrm{relint}(\mathcal{D})$, then we state Slater's condition as below.

\begin{definition}[Slater's condition \cite{boyd2004convex}]
\label{defn:slater}
For problem \eqref{prob}, suppose $f$ and $c_j$ are convex for each $j\in\Ical$, and $c_i$ is affine for each $i\in\Ecal\cup\mathcal{J}$ with $\mathcal{J}\subset\Ical$ being a subset of $\Ical$ that can potentially be empty. Then we say problem \eqref{prob} satisfies Slater's condition if there exists a feasible solution $x\in\textrm{relint}(\mathcal{D})$  such that 
$$c_i(x)=0 \,\mbox{ for }\, i\in\mathcal{E}, \quad c_j(x)\leq0\,\mbox{ for }\, j\in\mathcal{J},\ \mbox{and}\ c_k(x)<0\,\,\mbox{ for }\,\,  k\in\mathcal{I}\backslash\mathcal{J}.$$
\end{definition}

As modern applications of problem \eqref{prob} often incorporate nonsmooth objective functions, we introduce the following concept of the limiting Fr\'{e}chet subdifferential to facilitate the discussion of such nonsmooth problems.

\begin{definition}[Limiting Fr\'{e}chet subdifferential \cite{Kruger2003}]  We say a vector $g$ is a Fr\'{e}chet subgradient of a l.s.c. function $h$ at point $x \in \operatorname{dom} h$ if
$$h(y) \geq h(x)+\langle g,y-x  \rangle   + o\left(\|y- x\|\right),\quad \forall y \in \operatorname{dom} h,$$
where  $\langle \cdot, \cdot\rangle$ represents the standard vector inner product.
The set of Fr\'{e}chet subgradient of $h$ at $x$ is called Fr\'{e}chet subdifferential of $h$ at $x$ and is denoted as $\hat{\partial}h(x)$. For all $x \notin \operatorname{dom} h$, we default 
$\hat{\partial}h(x) = \emptyset.$ The limiting Fr\'{e}chet subdifferential, noted as $\partial h(x)$, is defined by 
$$\partial h(x)=\{v:\mathrm{there~is}~ x^k \rightarrow x ~\mathrm{and}~ v_k \in \hat{\partial}h(x^k)~ \mathrm{such~that}~ v_k \rightarrow v\}.$$
\end{definition}
It is known that when $h$ is a convex function, the limiting Fr\'{e}chet subdifferential coincides with the set of convex subgradients, that is
$$
\partial h(x)=\{v: h(y)-h(x) \geq  \langle v, y-x\rangle \text { for all } y \in \operatorname{dom} h\}.
$$
Similarly, one can define the supergradient for the concave function. Let $h: \mathbb{R}^n \rightarrow \mathbb{R}$ be a concave function.
A vector $v$ is a supergradient of $h$ at the point $x$ if for every $y$ it satisfies
$$h(y)-h(x) \leq  \langle v, y-x\rangle.$$
The set of supergradients of $h$ at $x$ is called the superdifferential of $h$ at $x$, denoted as
$$
\bar{\partial} h(x)=\{v: h(y)-h(x) \leq  \langle v, y-x\rangle \text { for all } y \in \operatorname{dom} h\}.
$$
With the above concepts and notations, we are now ready to present the existence of {the} KKT condition and the Lagrangian duality result.

\subsubsection{KKT condition and Lagrangian duality} 
The KKT condition for problem \eqref{prob} provides a necessary condition for a point to be considered an optimal solution, playing a central role in the design and analysis of constrained optimization algorithms. To discuss the KKT conditions, we begin by introducing the standard Lagrangian function for problem \eqref{prob}. Let $\lambda_i (i \in \mathcal{E})$ represent the Lagrange multiplier associated with each equality constraint, and let $\mu_i \geq 0 (i \in \mathcal{I})$ denote the Lagrange multiplier for each inequality constraint. The Lagrangian function for \eqref{prob} is then defined as
\[
\mathcal{L}(x,\lambda,\mu) = f(x) + \sum_{i \in \mathcal{E}} \lambda_i c_i(x) + \sum_{i \in \mathcal{I}} \mu_i c_i(x).
\]
Based on this notation, the following theorem establishes the KKT conditions as the first-order necessary conditions for optimal solutions.

\begin{theorem}
   Suppose $f(x)$ and $c_i(x), i\in\mathcal{E}\cup\mathcal{I}$ are continuously differentiable and let $x^*$ be %
   a local minimum of problem \eqref{prob}. If %
   LICQ or MFCQ is satisfied at $x^*$, then there exist Lagrange multipliers $\lambda_i^*, i \in \mathcal{E}$ and $\mu_i^*\geq0,i\in\mathcal{I}$ such that \quad
\begin{equation}\label{kkt}
\begin{cases}
    \nabla_x\mathcal{L}(x^*,\lambda^*,\mu^*)=0, & \qquad\mbox{\textnormal{(Lagrangian  stationarity)}}\\
    c_i(x^*)=0, i\in\mathcal{E},\,\,\, c_i(x^*)\leq 0, i\in\mathcal{I}, & \qquad\mbox{\textnormal{(Primal  feasibility)}}\\
    \mu_i^*\geq 0, i\in\mathcal{I}, & \qquad\mbox{\textnormal{(Dual  feasibility)}}\\
    \mu_i^* c_i(x^*)=0, i\in\mathcal{I}. & \qquad\mbox{\textnormal{(Complementary  slackness)}}
\end{cases}
\end{equation} 
\end{theorem}

{If} problem \eqref{prob} is convex,  the LICQ and MFCQ can be replaced {by} Slater's condition. In this case, the KKT condition is well-defined even without the smoothness of $f$ and $c_i$'s, and a strong duality result holds when viewing problem \eqref{prob} as a minimax saddle point problem with respect to the Lagrangian function $\mathcal{L}$ (see \cite[Theorem 28.4]{rockafellar1997convex}). 

\begin{theorem} 
\label{sec2:general:sd}
Suppose problem \eqref{prob} is convex and satisfies Slater's condition. Let $x^*$ be an optimal solution. If $f$ and $c_i$'s are continuously differentiable, then the KKT condition \eqref{kkt} holds for $x^*$. If any of $f$ and $c_i$'s are nonsmooth, then the KKT condition \eqref{kkt} still holds while replacing the Lagrangian stationarity condition with 
$0\in\partial \mathcal{L}(x^*,\lambda^*,\mu^*).$
In both cases, strong duality holds for the saddle point problem with respect to $\mathcal{L}$, that is,   
\begin{equation}
\label{eqn:Lag-saddle}
\min_x\max_{\mu\geq0,\lambda}\ \mathcal{L}(x,\lambda,\mu)=\max_{\mu\geq0,\lambda}\min_x\ \mathcal{L}(x,\lambda,\mu).
\end{equation}
\end{theorem}
 
It is important to recognize that the Lagrangian function $\mathcal{L}$ does not always satisfy strong duality in the absence of convexity and appropriate constraint qualifications. Aside from convex problems, the Lagrangian function $\mathcal{L}$ is seldom employed directly to design algorithms for nonconvex constrained optimization problems, as the resulting updates often lack convergence guarantees.

\subsubsection{Classic augmented Lagrangian function}\label{sec:classical-AL}
To facilitate the method of multipliers for general convex and nonconvex constrained problems, which will be discussed in later sections, we will introduce the augmented Lagrangian (AL) function based on {the Powell–Hestenes–\revise{Rockafellar} \cite{powell1969method,hestenes1969multiplier,rockafellar1974augmented}} for problem \eqref{prob}.  We begin with a special case of problem \eqref{prob} that has no inequality constraints. With $\rho > 0$ denoting a positive penalty parameter, the AL function \cite{hestenes1969multiplier,powell1969method} for this problem is given by
\begin{equation}
\label{AL-eq}
\mathcal{L}_\rho(x,\lambda)=f(x)+\sum_{i\in\mathcal{E}}\lambda_i c_i(x)+\frac\rho 2 \sum_{i\in\mathcal{E}}c_i^2(x).
\end{equation}

When problem \eqref{prob} includes inequality constraints, we can construct the AL function by first introducing slack variables to obtain an equivalent formulation \cite{rockafellar1971duality,rockafellar1973multiplier}:
\begin{equation}
\label{prob2}
    \begin{split}
    \min_{x,s}\ &\, f(x),\\
    \mathrm{s.t. }\, & \,c_i(x)=0, \hspace{8mm} i\in\mathcal{E},\\
    &\,c_i(x)+s_i=0, \ i\in\mathcal{I},\\
    &\,s_i\geq 0, \hspace{1.3cm} i\in\mathcal{I}.
    \end{split}
\end{equation}
By keeping the non-negative constraint on \(s\) explicit, we can define the AL function for the above problem following the structure used in the equality‐constrained case:
\begin{equation}
\label{AL0}
\mathcal{L}_\rho(x,s,\lambda,\mu) = f(x) + \sum_{i\in\mathcal{E}} \lambda_i c_i(x) + \sum_{i\in\mathcal{I}} \mu_i (c_i(x) + s_i) + \frac{\rho}{2} \sum_{i\in\mathcal{E}} c_i^2(x) + \frac{\rho}{2} \sum_{i\in\mathcal{I}} \left(c_i(x) + s_i\right)^2.
\end{equation}

Note that there is an additional slack variable \(s\) compared to the original problem \eqref{prob}. A standard approach is to eliminate this variable by partially minimizing over \(s \geq 0\), ensuring that the AL function depends solely on the original variable \(x\):
\[
\min_s\ \mathcal{L}_\rho(x,s,\lambda,\mu) \quad \text{s.t. } \ s_i \geq 0, \ i \in \mathcal{I}.
\]
Defining \([\alpha]_+ := \max\{\alpha, 0\}\), the unique optimal solution of the above problem is given by

\begin{equation}\label{sopt}
    s_i^* = \left[-\frac{\mu_i}{\rho} - c_i(x)\right]_+, \ i\in\mathcal{I}.
\end{equation}
Substituting \(s^*\) into \eqref{AL0} yields
\begin{equation}\label{AL}
\mathbb{L}_\rho(x,\lambda,\mu) = f(x) + \sum_{i\in\mathcal{E}} \lambda_i c_i(x) + \frac{\rho}{2} \sum_{i\in\mathcal{E}} c_i^2(x) + \frac{\rho}{2} \sum_{i\in\mathcal{I}} \left(\left[\frac{\mu_i}{\rho} + c_i(x)\right]_+^2 - \frac{\mu_i^2}{\rho^2}\right),
\end{equation}
which represents the classic AL function for inequality-constrained optimization problems. %

{We note that the above definition is based on the Hestenes-Powell-Rockafellar AL function, in which the penalty term is defined using the squared Euclidean norm. This formulation represents the most common and widely adopted form of the AL function in the literature. Nevertheless, alternative constructions of AL functions have also been proposed. In the following, we introduce two such variants: the shift-penalty AL function and a generalized AL function based on dualizing parameterization.}

\subsubsection{Alternative AL construction {via shift penalty}}
Similar to the quadratic penalty for equality constraints  \cite{Courant1943VariationalMF}, a quadratic penalty for problem \eqref{prob} can be written as 
\begin{equation}\label{penalty_inequality}
\Phi_\rho(x) := f(x) + \frac{\rho}{2} \sum_{i\in\mathcal{E}} c_i^2(x) + \frac{\rho}{2} \sum_{i\in\mathcal{I}} [c_i(x)]_+^2,
\end{equation}
for some penalty factor \(\rho > 0\). Minimizing the above function for successive values of \(\rho\) constitutes the classical external penalty method. By increasing the value of \(\rho\), this penalty method aims to achieve satisfactory feasibility. Alternatively, rather than simply increasing \(\rho\), one can modify the origin with respect to which infeasibility is penalized, leading to the concept of shifted penalties \cite{powell1969method, hestenes1969multiplier}. Specifically, a shifted penalty function \cite{fletcher1975ideal} for the inequality constraint problem \eqref{prob} is
 \begin{equation}
 \phi(x,\xi,\theta,\rho):= f(x)+\frac\rho2\sum_{i\in\mathcal{E}}(c_i(x)+\xi_i)^2+\frac{\rho}{2}\sum_{i\in\mathcal{I}}[c_i(x)+\theta_i]_+^2,\nonumber
 \end{equation}
where $\xi_i,\theta_i$ are the origin shifts and $\rho > 0$ controls the  penalty level. Introducing 
$\lambda_i=\rho\xi_i$ for $ i\in\mathcal{E}$ and $\mu_i=\rho\theta_i$ for $ i\in\mathcal{I}$, then omitting the terms independent of $x$ yields
\begin{equation}%
\begin{split}
\phi(x,\lambda,\mu,\rho)= &f(x)+\sum_{i\in\mathcal{E}}\lambda_i c_i(x)+\frac\rho 2 \sum_{i\in\mathcal{E}}c_i^2(x)\\
&+\sum_{i\in\mathcal{I}}
\left\{\begin{array}{ll}
\mu_i c_i(x)+\rho c^2_i(x)/2, & \mathrm{if}~ c_i(x) \geq -\mu_i/ \rho,\\
\mu_i c_i(x), & \mathrm{otherwise},
\end{array}\right.
\end{split}\nonumber
\end{equation}
which coincides with the definition given in \eqref{AL}. Therefore,
a local minimizer can be made to minimize $\Phi$ for finite $\rho$ by changing the origin of the penalty term.

\subsubsection{{Alternative AL construction via general augmenting function}} \label{subsubsec:Gen-AL}{Although this survey primarily focuses on the AL function based on quadratic penalty, it is still worth mentioning the general form of AL function constructed by general dualizing parameterization and augmenting function, which contains the AL function \eqref{AL} as a special case, see \cite{rockafellar1973dual,rockafellar2009variational}.}

{Define $\psi(x):=f(x) + \delta_D(c(x))$, where $D$ is a closed convex set and $c$ is a vector mapping. Then minimizing $\psi(x)$ over $\R^n$ is a constrained minimization of $f(x)$ subject to $c(x)\in D$. In \cite{rockafellar1976augmented}, Rockafellar and Wets introduce a so-called \textit{dualizing parameterization} of $\psi$ as a representation $\psi(x) = \phi(x,0)$ in terms of a proper function $\phi(x,u)$ that is  l.s.c. and convex in $u$. Let $\sigma$ be an arbitrary proper, l.s.c., and convex \textit{augmenting function}, satisfying $\min_y \sigma(y) = 0$ and $\arg\min_y \sigma(y) = 0$. Then they introduce the general Lagrangian $\cL(x,\nu)$ and augmented Lagrangian functions $\mathbb{L}_\rho(x,\nu)$ as 
\begin{align*}
\cL(x,\nu) &= \min_u \phi(x,u) - u^\top\nu,\\
\mathbb{L}_\rho(x,\nu) &= \min_u \phi(x,u)+\rho\sigma(u) - u^\top\nu.
\end{align*}
One particular choice of the dualizing parameterization and augmenting function can be 
$$\phi(x,u):= f(x) + \delta_D(c(x)+u)\qquad\mbox{and}\qquad \sigma(y) = \|y\|^2/2.$$
A direct computation gives 
\begin{align*}
\cL(x,\nu) &= f(x) + \nu^\top c(x) - \delta_D^*(\nu),\\
\mathbb{L}_\rho(x,\nu) &= f(x) +  \frac{\rho}{2}\Big\|\hat{\Pi}_D\Big(\frac{\nu}{\rho} + c(x)\Big)\Big\|^2 - \frac{\|\nu\|^2}{2\rho},
\end{align*} 
see also \cite[Example 11.46 and Example 11.57]{rockafellar2009variational}. Then setting $D = -\R_+^{|\mathcal{I}|}\times\{0\}^{|\mathcal{E}|}$ and separating $\nu$ into $\lambda$ and $\mu$ correspondingly recovers the AL function defined in \eqref{AL}. We should note that different selection of $\sigma$ can indeed lead to different AL function. However, as this survey primarily focuses on the AL function based on quadratic penalty, we will mostly consider the quadratic augmenting function in latter discussion.    
}

{
\subsubsection{Other AL function}
In addition to the classical Hestenes–Powell–Rockafellar AL function in Section \ref{sec:classical-AL} (which is based on quadratic function, and also is the main focus in this survey), as well as the previously mentioned shift-penalty and generalized augmenting functions, there exist several other classes of augmented Lagrangian functions. For inequality-constrained problems, the classical HPR AL function is only once differentiable, even when both the objective and constraint functions are of higher smoothness. To address this limitation, a class of  modified barrier functions (MBFs) has been proposed  \cite{polyak1992modified}:
$$
F(x, u, k)= \begin{cases}f(x)-k^{-1} \sum_{i\in \mathcal{I}} u_i \ln \left(k c_i(x)+1\right), & \text { if } x \in \operatorname{int} \Omega_k, \\ \infty, & \text { if } x \notin \operatorname{int} \Omega_k,\end{cases}
$$
where $\Omega_k=\left\{x: k f_i(x)+1 \geqslant 0, i=1, \ldots, m\right\}$. 
These functions combine the advantages of classical Lagrangian functions and traditional barrier functions \cite{carroll1961created,frisch1955logarithmic}, while circumventing their inherent shortcomings. When the objective function is twice continuously differentiable, MBFs also retain twice differentiability, making them well-suited for the application of Newton-type methods in solving subproblems efficiently. Several further developments and generalizations based on MBFs have also been proposed in the literature; see, for example, \cite{ben1997penalty,goldfarb1999modified} for details. Other variants of AL functions have also been studied, including the exponential augmented Lagrangian function \cite{tseng1993convergence}, and the Bregman augmented Lagrangian function \cite{yan2020bregman} that leverages the geometry induced by Bregman divergences \cite{bregman1967relaxation}. 

}

\subsubsection{Saddle point formulation}
\label{sec:saddle-basic}
Recent developments in the optimization theory of saddle point problems have resulted in a significant body of literature addressing convex minimization problems \eqref{prob} as minimax problems based on the Lagrangian function \(\mathcal{L}\) \cite{zhu2008efficient,chambolle2011first,chambolle2016ergodic} and the AL function \(\mathbb{L}_\rho\) \cite{zhu2022unified}. {For convex problems, strong duality for the standard Lagrangian is established by Theorem \ref{sec2:general:sd}. In this section, we concentrate on establishing and exploiting strong duality for the saddle‐point formulation based on the augmented Lagrangian:} %
\begin{equation}
\label{minmax}
\min_x \max_{\lambda,\mu} \ \mathbb{L}_\rho(x,\lambda,\mu).
\end{equation}
{
Compared to the classical saddle‐point problem built upon $\mathcal{L}$, formulation \eqref{minmax} accommodates both convex and certain nonconvex problems. In fact, by treating the penalty parameter $\rho$ as an auxiliary variable, one can extend strong duality of the AL function to a broad class of nonconvex optimization problems \cite{rockafellar2009variational,gasimov2002augmented}. Moreover, since our primary interest lies in the ALM, it is important to observe (cf.\ Section \ref{sec:dualascent}) that ALM can be viewed as a dual‐ascent procedure applied to the augmented‐Lagrangian dual problem:
\[
\max_{\lambda,\mu}\,\min_{x}\;\mathbb{L}_\rho(x,\lambda,\mu).
\]
Accordingly, establishing strong duality for $\mathbb{L}_\rho$ plays a central role both in analyzing convergence of ALM iterations and in clarifying the precise relationship between primal feasibility and dual‐variable updates.}  First, we will initiate our discussion with the AL function associated with a convex optimization problem. The proof is referred to Appendix \ref{appen:proof-sec2}.

\begin{theorem}\label{theorem:convex-al-duality}
 Suppose problem \eqref{prob} is convex and Slater's condition is satisfied, then strong duality holds for the saddle point problem \eqref{minmax}. That is, 
\begin{equation}\label{strong-duality}
    \max_{\lambda,\mu}\min_x\ \mathbb{L}_\rho(x,\lambda,\mu)\,=\,\min_x\max_{\lambda,\mu}\ \mathbb{L}_\rho(x,\lambda,\mu).
\end{equation}
\end{theorem}
  
Thus far, we have restricted our discussion to convex problems under Slater's condition. However, for any fixed \(\rho\), strong duality does not necessarily hold for general problems, particularly in nonconvex cases. To address this issue, it is essential to treat \(\rho\) as a decision variable and examine strong duality from the perspective of the penalty method. Let us first give the definition of level-bounded, which is crucial for the subsequent theorem. 
\begin{definition} 
 A function $f:\mathbb{R}^n \to \mathbb{R}$ is (lower) level-bounded if for every $\alpha \in \mathbb{R}$  the set $\{x \in \R^n:  f(x) \leq \alpha \}$  is bounded (possibly empty).
\end{definition} 

\begin{theorem}\label{sec1:nonconvex:strong_duality}
For problem \eqref{prob}, if the objective function $f$ is bounded from below, and { the quadratically penalized objective function $\Phi_\rho(x)=\mathbb{L}_\rho(x,0,0)$ is l.s.c. and level-bounded for sufficiently large $\rho$, see \eqref{penalty_inequality},}
then the strong duality holds in the sense that  
\begin{equation}\label{npnconvex_strong_duality}
    \mathop{\max}_{ \rho>0,\lambda,\mu}\min_x\ \mathbb{L}_\rho(x,\lambda,\mu)=\min_x\max_{ \rho>0,\lambda,\mu}\ \mathbb{L}_\rho(x,\lambda,\mu).
\end{equation}
\end{theorem}

{It is worth mentioning that for a general AL function with any qualified dualizing parameterization and augmenting function described in Section \ref{subsubsec:Gen-AL}, \cite[Theorem 11.59]{rockafellar2009variational} suggests that the above strong duality result still holds if the dualizing parameterization satisfies proper conditions. In detail, if one adopts the specific parameterization 
$$\phi(x,u) = f(x) + \sum_{i\in\mathcal{E}} \delta_{\{0\}}(c_i(x)+u_i) + \sum_{i\in\cI}\delta_{-\R_+}(c_i(x)+u_i),$$
then \cite[Theorem 11.59]{rockafellar2009variational} requires $\phi$ to be l.s.c. and level-bounded in $x$ locally and uniformly in $u$. In Theorem \ref{sec1:nonconvex:strong_duality}, we adopt a slightly different assumption that the quadratically penalized function $\mathbb{L}_\rho(x,0,0)$ to be l.s.c. and level bounded when $\rho$ is large enough. We adopt this change because it is more natural for the main focus of this paper, the quadratic penalty based AL function. In addition, the main tone of the survey is to avoid the technical sophistication for the sake of a broader audience, we choose not present the results from the dualizing parameterization and general AL perspective. A simple and self-contained proof of the theorem is provided in Appendix~\ref{appen:proof-sec2}.}

The distinction between the strong duality in convex and nonconvex cases lies in whether the penalty parameter \(\rho\) is treated as a decision variable in the saddle point formulation based on the AL function. {For convex problems, strong duality holds regardless of the choice of $\rho\geq0$, while for nonconvex problems, as demonstrated in the proof of Theorem  \ref{sec1:nonconvex:strong_duality}, the penalty parameter $\rho$ must be driven to infinity to ensure this property. This  directly relates why most augmented Lagrangian methods, surveyed in latter sections, require the penalty parameter $\rho\to+\infty$, leading to an infinite sequence of subproblems with increasing numerical difficulties. In this situation, the study of exact penalty properties becomes crucial for constrained nonconvex optimization from both theoretical and practical perspectives.}

\subsubsection{{Exact penalty representation}} 
{Denote $F^*_{\text{NCOP}}$ and $\mathcal{X}^*_{\text{NCOP}}$ the optimal value and optimal solution set for the nonconvex constrained optimization problem \eqref{prob}.  Then, following the definition in \cite{rockafellar2009variational}, we introduce the concept of \textit{exact penalty representation}. 
\begin{definition}
We say the vector $(\bar{\lambda},\bar{\mu})$ supports an exact penalty representation if 
$$F^*_{\text{NCOP}}=\min_x \mathbb{L}_\rho(x, \bar{\lambda},\bar{\mu})  \quad\mbox{and}\quad \mathcal{X}^*_{\text{NCOP}}=\argmin_x \mathbb{L}_\rho(x, \bar{\lambda},\bar{\mu})$$
for all sufficiently large $\rho >0$. In particular, a value $\bar{\rho} > 0$ is called an adequate penalty threshold if this property is satisfied for all $\rho \geq \bar{\rho}$.  
\end{definition}}

This property is crucial in designing algorithms for nonconvex problems. It allows the use of unconstrained optimization techniques by ensuring that the penalized problem accurately represents the original constrained problem. Several studies have provided insights into the conditions required for the exact penalty property to hold in nonconvex settings. The seminal work of \cite{rockafellar1976augmented} laid the foundation for comprehending these conditions in convex optimization, which has been further extended to nonconvex scenarios by subsequent research, see \cite{dolgopolik2018augmented,dolgopolik2018unified}.  {In addition, this property actually has a deep connection to the strong duality of nonconvex problems. By \cite[Theorem 11.61]{rockafellar2009variational}, a sufficient and necessary condition for the existence of such $\bar{\lambda},\bar{\mu}$, and $\bar{\rho}$ is stated below from the saddle point perspective. 
\begin{theorem}  
Under the same condition of Theorem \ref{sec1:nonconvex:strong_duality}, $(\bar{\lambda},\bar{\mu})$ supports an exact penalty representation for problem \eqref{prob} if and only if there exists a finite $0<\bar{\rho}<+\infty$ such that 
$$(\bar{\lambda}, \bar{\mu}, \bar{\rho}) \in \arg\max_{\lambda,\mu,\rho} \min_x \mathbb{L}_\rho(x, \lambda, \mu).$$ 
\end{theorem}   
Combined with the analysis of Theorem \ref{sec1:nonconvex:strong_duality} in Appendix \ref{appen:proof-sec2}, a direct consequence is that the parameter $\rho$ no longer needs to be driven to infinity to secure the strong duality property, and it will hold for any finite $\rho\geq\bar{\rho}$. 
\begin{corollary}
    Suppose there exists $(\bar{\lambda},\bar{\mu})$ that supports an exact penalty representation, and $\bar{\rho}>0$ is an adequate penalty threshold, then the strong duality holds \begin{equation}\label{npnconvex_strong_duality_exact}
    \mathop{\max}_{\lambda,\mu}\min_x\ \mathbb{L}_\rho(x,\lambda,\mu)=\min_x\max_{\lambda,\mu}\ \mathbb{L}_\rho(x,\lambda,\mu),
\end{equation}
for all the sufficiently large parameter $\rho\geq\bar{\rho}$.
\end{corollary}
}

We should note that, unfortunately, the quadratic augmenting function often does not support a finite adequate penalty threshold. On the contrary, {sharp augmenting functions such as $\sigma(x) = \|x\|$ or $\sigma(x) = \|x\|_1$, often called exact penalty functions, allows exact penalty representation with $(\bar{\lambda},\bar{\mu})=0$ and a finite threshold $\bar{\rho}<+\infty$ under mild conditions, see \cite[Example 11.62]{rockafellar2009variational}. This difference ends up a trade-off between the smoothness of subproblem and the exactness of constraint enforcement.}  

{Instead of adopting sharp augmenting functions, an alternative approach to ensure exactness with finite penalty parameter $\rho$ is to adopt a so-called \textit{exact augmented Lagrangian function}, which modifies the AL function by incorporating smooth penalty terms on the first order necessary condition. We should note that though sharing a similar name with exact penalty representation, the exact AL function is a completely different approach. It was introduced by Di Pillo and Grippo for nonlinear programming (NLP) problems with equality constraints \cite{di1979new} and inequality constraints \cite{di1982new}, and were later  studied in \cite{lucidi1988new,bertsekas2014constrained,di2002augmented} for NLP problems with box constraints, and nonlinear semidefinite programming \cite{fukuda2018exact}. \revise{Several exact augmented Lagrangian functions have been proposed for cone constrained optimization problems; see, for example, \cite{dolgopolik2018augmented}. 
}  As this AL function significantly deviates from the Hestenes-Powell-Rockafellar AL function that this survey focuses on, we will not excessively expand the discussion on this topic.    }

\subsection{Convex composite optimization problem}\label{sec:convex-comp}

{The AL function can also be extended to structured optimization problems, such as nonsmooth problems, conic-constrained problems, and others. In this section, we consider a class of convex composite optimization problems, which involve minimizing the sum of a smooth convex function and a possibly nonsmooth but proximally tractable convex function. In practice, many instances of problem \eqref{prob} arise from reformulations of the convex composite optimization problem. These problems naturally arise in a wide range of applications, including signal processing \cite{combettes2011proximal}, statistical learning \cite{gaines2018algorithms}, and control \cite{de2020constrained}. In particular, we consider the formulation:} 
\begin{equation}\label{sec2:pro:composite1}
\min_{x} \psi(x): = f(x) + h(x)\quad\mathrm{s.t.} \quad \mathcal{A}(x) \in \mathcal{Q}, x\in \mathcal{K}. 
\end{equation}
Let $\mathcal{X}$ and $\mathcal{Y}$ be real finite-dimensional Euclidean spaces and let $\mathcal{K}\subset \mathcal{X}, \mathcal{Q}\subset\mathcal{Y}$ be closed convex sets that are typically two cones but not necessarily so. 
Then we require $\mathcal{A}:\mathcal{X} \rightarrow \mathcal{Y}$  to be  a linear map, $f:\mathcal{X}\rightarrow \R$ to be  a smooth convex function, and $h:\mathcal{X}\rightarrow (-\infty,+\infty]$ to be a proper and l.s.c. convex function whose proximal operator can be efficiently evaluated. 

{
For the convex problem \eqref{sec2:pro:composite}, under Slater's condition, one can establish the KKT condition as the first-order sufficient and necessary condition for optimal solutions.

\begin{theorem}
  Suppose that Slater's condition holds for problem \eqref{sec2:pro:composite1},  $x^*$ is an optimal solution of  \eqref{prob} if and only if there are Lagrange multipliers $\lambda^*\in N_{\mathcal{Q}}(\mathcal{A}x^*)$ and $\mu^*\in N_{\mathcal{K}}(x^*)$ such that $(x^*,\lambda^*,\mu^*)$ is a solution of the following KKT system:
\begin{equation}\label{convex-kkt}
\begin{cases}
  0\in \nabla f(x^*) + \partial h(x^*) +\mathcal{A}^*(\lambda^*) + \mu^*, \\
  \mathcal{A}x^* \in \mathcal{Q} ,~~x^*\in \mathcal{K}.
\end{cases}
\end{equation} 
\end{theorem}

The KKT system is naturally associated with the Lagrangian function defined in \eqref{sec2:pro:lf}. Moreover, by introducing an auxiliary Lagrange multiplier 
$\nu^* \in \partial h(x^*)$, we can equivalently express the KKT conditions in the following form: 
\begin{equation}\label{convex-kkt-1}
\begin{cases}
  0 = \nabla f(x^*) + \nu^* +\mathcal{A}^*(\lambda^*) + \mu^*, \\
 \nu^*\in \partial h(x^*),~~\, \mathcal{A}x^* \in \mathcal{Q} ,~~\, x^*\in \mathcal{K}.
\end{cases}
\end{equation} 
This is associated with the Lagrangian function \eqref{sec2:pro:composite:auxilary-new}. 
}   
\subsubsection{Preliminaries}

To facilitate later discussion, let us first review the concepts and properties of the proximal operator and Moreau envelope.  
For a convex function $h:\mathcal{X}\rightarrow (-\infty,+\infty]$, at any $x\in\mathcal{X}$ and $t>0$, the Moreau envelope $e_{t}h(x)$ is defined as: 
\begin{equation}\label{moreau1}
    e_{t}h(x): = \min_{u\in\mathcal{X}}\,\,h(u)+ \frac{t}{2}\|u-x\|^2.
\end{equation}
According to \cite[Proposition 12.15]{bauschke2011convex}, the Moreau envelope $e_{t}h(\cdot)$ is real-valued, convex, continuous, and the minimum in \eqref{moreau1} is uniquely attained
for every $x\in\mathcal{X}$. The proximal operator of $h$ is defined as the unique minimizer of \eqref{moreau1}: 
\begin{equation}\label{def:prox}
    \prox_{h/t}(x) = \argmin_{u\in\mathcal{X}}\,\,h(u) + \frac{t}{2}\|  u - x\|^2.
\end{equation}
 It is shown that $e_{t}h(x)$ is differentiable in \cite[Proposition
12.29]{bauschke2011convex}, and its gradient 
\begin{equation}\label{me:gradient}
    \nabla e_{t}h(x) = t(x -  \prox_{h/t}(x))
\end{equation}
is $t$-Lipschitz continuous. Define $h^*(x): = \sup_{y}\{\left<x,y\right> - h(y)\}$ as the conjugate function of $h$. For any $x\in\mathcal{X}$, $t>0$ and any proper convex function  $h$, the Moreau decomposition theorem \cite{beck2017first} indicates that
\begin{equation}\label{moreau}
    \begin{aligned}
    {x}  = \prox_{h/t}({x})  +  \frac{1}{t} \prox_{t h^*}(t{x}).
    \end{aligned}
\end{equation}
When $h(\cdot) = \delta_{\mathcal{Q}}(\cdot)$ is the indicator function of a convex set $\mathcal{Q}$, the proximal operator is reduced to the projection operator $\Pi_{\mathcal{Q}}$:
\begin{equation*}
    \Pi_{\mathcal{Q}}({x}): = \argmin_{{y}\in\mathcal{Q}}\|x-y\|^2,
\end{equation*}
and the Moreau envelope $e_t \delta_{\mathcal{Q}}(x)$ reduces to the squared distance to $\mathcal{Q}$ scaled by $t/2$:
\begin{equation*}
    e_t \delta_{\mathcal{Q}}(x) = \frac{t}{2}\|x - \Pi_{\mathcal{Q}}({x})\|^2.
\end{equation*}
It follows from \eqref{me:gradient} that
$
    \nabla e_t \delta_{\mathcal{Q}}(x) = t(x - \Pi_{\mathcal{Q}}(x)). 
$
According to \eqref{moreau}, we denote the proximal operator of the conjugate function as
\begin{equation}\label{sec3:moreau:set}
    \hat{\Pi}_{\mathcal{Q}}(tx) :=\prox_{th^\star}(tx)= t(x-\Pi_{\mathcal{Q}}(x)),
\end{equation}
which is the residual of the projection operator scaled by $t$.

\subsubsection{AL functions for composite optimization}\label{sec:AL-composite}

Based on the above notations, we introduce the reformulation of problem \eqref{sec2:pro:composite1} as an instance of \eqref{prob} and hence derive its AL function. First, by incorporating the indicator functions of the convex sets $\mathcal{Q}$ and $\mathcal{K}$, problem \eqref{sec2:pro:composite1} can be equivalently written as follows:
\begin{equation}\label{sec2:pro:composite}
\min_{x} f(x) + h(x) + \delta_{\mathcal{Q}}(\mathcal{A}(x)) + \delta_{\mathcal{K}}(x). 
\end{equation}
Based on whether to decouple the smooth ($f$) and nonsmooth ($h$) terms, there are two possible ways to obtain the AL functions.  

\vspace{0.2cm}
\textbf{1) Classic AL function.} 
To decouple the nonsmooth terms in the objective function, a typical technique is to introduce auxiliary variables $w, v\in\mathcal{X}$ and rewrite problem \eqref{sec2:pro:composite} as

\begin{equation}\label{sec2:pro:composite:auxilary}
\begin{aligned}
    \min_{x,v,w} &~ f(x) + h(x)   + \delta_{\mathcal{Q}}(v) + \delta_{\mathcal{K}}(w) \\
    \mathrm{s.t. }\;&~ \mathcal{A}(x) = v, x = w.
\end{aligned}
\end{equation}
Then the Lagrangian function  of \eqref{sec2:pro:composite:auxilary} is 
\begin{equation}\label{sec2:pro:lf}
\begin{aligned}
     \ell(x,v,w;\lambda,\mu): =& f(x) + h(x) + \delta_{\mathcal{Q}}(v) + \delta_{\mathcal{K}}(w)    + \left<\lambda,\mathcal{A}(x) - v\right> + \left<\mu,x-w\right>,\!\!\!\!\!
\end{aligned}
\end{equation}
where $\lambda$ and $\mu$ are Lagrange multipliers associated with the equality constraints. 
Adding quadratic penalties of the linear equality constraints $\mathcal{A}(x) = v$ and $x = w$  with  a penalty parameter $\rho>0$, the standard AL function of \eqref{sec2:pro:composite:auxilary} is given by 
\begin{equation}\label{sec2:pro:composite:auxilary:al}
\begin{aligned}
    \mathcal{L}_{\rho}(x,v,w;\lambda,\mu):= \ell(x,v,w;\lambda,\mu) + \frac{\rho}{2} \|\mathcal{A}(x) - v \|^2 + \frac{\rho}{2}\| x - w\|^2.
\end{aligned}
\end{equation}
Note that the above AL function is associated with the reformulation \eqref{sec2:pro:composite:auxilary}, which involves the variable $x$ and the auxiliary variables $v, w$. However, for the original problem \eqref{sec2:pro:composite}, which only features the variable $x$, the corresponding AL function should solely include $x$ as well as the associated multiplies.
To introduce the AL function for the original problem \eqref{sec2:pro:composite}, a common approach is to eliminate the auxiliary variables in \eqref{sec2:pro:composite:auxilary:al}:
\begin{equation}\label{sec2:composite:al}
\begin{aligned}
    \mathbb{L}_{\rho}(x,\lambda,\mu): & = \min_{v,w} \mathcal{L}_{\rho}(x,v,w;\lambda,\mu).
\end{aligned}
\end{equation}
By incorporating the proximal operators, problem \eqref{sec2:composite:al} admits a closed-form solution:
\begin{equation}\label{eq:y-expression}
    v^* = \Pi_{\mathcal{Q}}(\mathcal{A}(x) + \lambda/\rho),\qquad  w^* = \Pi_{\mathcal{K}}(x + \mu/\rho) .
\end{equation} 
Substituting the solutions into \eqref{sec2:pro:composite:auxilary:al} yields 
\begin{equation}\label{eqn:AL-h}
\begin{split}
      \mathbb{L}_{\rho} (x,\lambda,\mu)
     =&   f(x)+h(x)  
       + \frac{\rho}{2} \left\|\mathcal{A}(x) +\frac\lambda \rho - \Pi_{\mathcal{Q}}\left(\mathcal{A}(x) 
     + \frac \lambda\rho\right)\right\|^2 \\&  + \frac{\rho}{2} \left\|x +  \frac \mu \rho - \Pi_{\mathcal{K}}\left(x +\frac \mu \rho\right)\right\|^2 - \frac{1}{2\rho}\left(\|\lambda\|^2  + \|\mu\|^2\right).
\end{split}
\end{equation}
In the above AL function, $\mathbb{L}_{\rho}$ only retains $x$ as the primal variable, while the Lagrange multipliers $\lambda$ and $\mu$ correspond to the constraints $\mathcal{A}(x)\in \mathcal{Q}$ and $x\in \mathcal{K}$, respectively. 
However, it can be observed that the function is nonsmooth due to the existence of $h(x)$. As a result, when we need to minimize $\mathbb{L}_{\rho} (x,\lambda,\mu)$ with respect to $x$, we have to stick to nonsmooth optimization methods such as the proximal gradient method, while efficient higher-order methods such as semi-smooth Newton method will not be applicable.

\vspace{0.2cm}
\textbf{2) An alternative smooth AL function.}
To obtain an alternative smooth AL function, the key point is to smooth $h(x)$ by the Moreau envelope. Specifically, we can introduce an extra variable $u$ based on \eqref{sec2:pro:composite:auxilary} to obtain
\begin{equation}
\label{sec2:pro:composite:auxilary-new}
\begin{aligned}
\min_{x,u,v,w} &~ f(x) + h(u) + \delta_{\mathcal{Q}}(v) + \delta_{\mathcal{K}}(w) \\
\mathrm{s.t. }\;&~ x=u,  \mathcal{A}(x) = v, x = w.
\end{aligned}
\end{equation}
Then the Lagrangian function  of \eqref{sec2:pro:composite:auxilary-new} is 
\begin{equation}\label{sec2:pro:lf-new}
\begin{aligned}
     \ell(x,u,v,w;\nu,\lambda,\mu): =& f(x) + h(u) + \delta_{\mathcal{Q}}(v) + \delta_{\mathcal{K}}(w)  + \left<\nu, x - u\right> \\
& + \left<\lambda,\mathcal{A}(x) - v\right> + \left<\mu,x-w\right>,
\end{aligned}
\end{equation}
where $\nu,\lambda$ and $\mu$ are Lagrange multipliers associated with the nonsmooth term $h$ and the equality constraints. 
Then the AL function defined by \eqref{sec2:pro:composite:auxilary:al} reduces to
\begin{equation*}
 \mathcal{L}_{\rho}(x,u,v,w,\nu,\lambda,\mu)\!:= \! \ell(x,u,v,w;\nu,\lambda,\mu) + \frac{\rho}{2} \|x - u\|^2  + \frac{\rho}{2} \|\mathcal{A}(x) - v \|^2 + \frac{\rho}{2}\| x - w\|^2.
\end{equation*}
Similarly, one can obtain the AL function by partially minimizing $\mathbb{L}_{\rho}$ over $u,v,w$:
\begin{equation}
     \mathbb{L}_{\rho} (x,\nu,\lambda,\mu): = \min_{u,v,w} \mathcal{L}_{\rho}(x,u,v,w,\nu,\lambda,\mu).
\end{equation}
The closed-form solutions of the above problem are given by
\begin{equation}\label{eq:y-expression-new}
    u^* = \prox_{h/\rho}\left(x + \nu/\rho\right),\quad  v^* = \Pi_{\mathcal{Q}}(\mathcal{A}(x) + \lambda/\rho),\quad w^* = \Pi_{\mathcal{K}}(x + \mu/\rho).
\end{equation} 
Hence, eliminating $u,v,w$ gives the following AL function
\begin{framed}
\begin{align}
     \mathbb{L}_{\rho} (x,\nu,\lambda,\mu)
     = &~ f(x) - \underbrace{\frac{1}{2\rho}(\|\nu\|^2 + \|\lambda\|^2  + \|\mu\|^2)}_{\text{dual variables}} \nonumber\\
     &+ \underbrace{\frac{\rho}{2} \left\|x + \frac \nu \rho-  \prox_{h/\rho}\left(x + \frac \nu \rho\right)  \right\|^2 + h\left(\prox_{h/\rho}\left(x + \frac{\nu}{\rho}\right)\right)}_{h(x): \quad e_{\rho}h\left(x + \frac{\nu}{\rho}\right)} \nonumber\\
     &+ \underbrace{\frac{\rho}{2} \left\|\mathcal{A}(x) + \frac{\lambda}{\rho} - \Pi_{\mathcal{Q}}\left(\mathcal{A}(x) + \frac{\lambda}{\rho}\right)\right\|^2}_{\mathcal{A}(x)\in\mathcal{Q}: \quad e_{\rho} \delta_{\mathcal{Q}}\left(\mathcal{A}(x) + \frac{\lambda}{\rho} \right)} \label{sec2:al:eliminate} \\
     & + \underbrace{\frac{\rho}{2}    \left\|x + \frac{\mu}{\rho} - \Pi_{\mathcal{K}}\left(x + \frac{\mu}{\rho}\right)\right\|^2}_{x\in\mathcal{K}:\quad e_{\rho} \delta_{\mathcal{K}} \left(x + \frac{\mu}{\rho}\right)}.\nonumber
\end{align}
\end{framed}
\noindent Based on the  Moreau envelope, \eqref{sec2:al:eliminate} can be rewritten in a more compact form:
\begin{equation}\label{eqn:AL-Moreau}
    \begin{split}
     \mathbb{L}_{\rho}(x,\nu,\lambda,\mu) 
   =& \,\, f(x) - \frac{1}{2\rho}(\|\nu\|^2 + \|\lambda\|^2  + \|\mu\|^2) \\
   & + e_{\rho}h\left(x + \frac{\nu}{\rho}\right) + e_{\rho} \delta_{\mathcal{Q}}\left(\mathcal{A}(x) + \frac{\lambda}{\rho} \right) + e_{\rho} \delta_{\mathcal{K}} \left(x + \frac{\mu}{\rho}\right).
   \end{split}
\end{equation}
{This formulation provides an alternative interpretation of the AL function. Intuitively, the AL function serves to render the inherently difficult components of the problem tractable via the Moreau envelope—specifically, the nonsmooth term $h$ and the constraints $\mathcal{A}(x)\in\mathcal{Q}$ and $x\in\mathcal{K}$ in \eqref{eqn:AL-Moreau}. Since the Moreau envelope of an indicator function is equivalent to the quadratic penalty over constraint, the AL function in \eqref{AL-eq} and \eqref{AL} can also be interpreted through the Moreau envelope. In principle, one seeks to construct an AL function that is readily solvable with respect to the primal variables—e.g., by ensuring that it is continuously differentiable. The smoothness of the AL function plays a central role in ALM algorithm design, since the primal subproblems are typically solved via iterative methods. Because the augmented Lagrangian is differentiable, one can readily apply first-order gradient schemes and, when appropriate, even second-order methods. We will examine these implications in detail in Section \ref{sec:examples}. Unless explicitly noted otherwise, all subsequent references to the AL function for convex composite problems pertain to the smooth form introduced in \eqref{eqn:AL-Moreau}.}

{ Of course, one may selectively retain certain terms: for instance, one can preserve the nonsmooth term $h$ (as in the AL function defined in \eqref{eqn:AL-h}), or maintain some constraints explicitly when they admit simple projection operators. Further details are provided in Section \ref{sec:examples}.}

\begin{remark}
{As shown in Section \ref{subsubsec:Gen-AL}, we can derive the AL function \eqref{eqn:AL-Moreau} by introducing a parameterization function \cite{rockafellar1973dual}. Let us choose the augmenting function  $\sigma(y) = \frac{1}{2}\|y\|^2$ with $y = (u,v,w)$ and the parameterization function $\phi$
\begin{equation*}
    \phi(x,u,v,w) = f(x) + h(x + u) + \delta_{\mathcal{Q}}(\mathcal{A}(x) + v) + \delta_{\mathcal{K}}(x + w). 
\end{equation*}
Then the AL function \eqref{eqn:AL-Moreau} is coincide with the following general AL function
$$
L_\rho(x,\nu,\lambda,\mu): = \inf_{u,v,w} \Big\{\phi(x,u,v,w) + \rho \sigma(y)     - \left<\nu,u\right>  - \left<\lambda,v\right> - \left<\mu,w\right>\Big\}. 
$$
Similarly, the AL function in \eqref{eqn:AL-h} can be derived in this manner by employing an alternative parameterization.
}
Compared to the parameterization function approach discussed here, the construction of the AL function $\mathbb{L}_{\rho}$ in  \eqref{sec2:composite:al} is more intuitive and easy to understand.
 
\end{remark}
\vspace{0.2cm}

\textbf{3) Extension to general multi-block composite problems.}
Note that the construction of the AL function \eqref{eqn:AL-Moreau} essentially preserves the smooth terms while transforming the nonsmooth and constraint indicator terms into the corresponding Moreau envelopes that couple with the respective Lagrange multipliers. Consequently, we can naturally formulate the AL function for the general multi-block composite optimization problem:
\begin{equation}\label{sec2:pro:composite-multiblock}
\min_{x}  f(x) + \sum_{i=1}^N h_i(\mathcal{B}_i(x_i)),\quad\mathrm{s.t. } \quad \mathcal{A}(x) \in \mathcal{Q},\ x\in \mathcal{K}, 
\end{equation}
where $f,\mathcal{A},\mathcal{Q},\mathcal{K}$ are the same as those in \eqref{sec2:pro:composite1} and $x = (x_1,\cdots,x_N)$ consists of $N$ block variables.  Let $\cX_i$ and $\mathcal{Y}_i$ be real finite-dimensional Euclidean spaces for each $i \in \{1,2,\cdots,N\}$, then we require $x_i\in \mathcal{X}_i$, $\mathcal{B}_i:\mathcal{X}_i\rightarrow \mathcal{Y}_i$ to be a linear map, and $h_i:\mathcal{Y}_i\rightarrow (-\infty,+\infty]$ to be a closed proper convex function. Then the AL function of \eqref{sec2:pro:composite-multiblock} can be written as
\begin{equation}\label{eqn:AL-Moreau:multi}
    \begin{split}
      \mathbb{L}_{\rho}(x,\nu,\lambda,\mu) 
   =&   f(x) - \frac{1}{2\rho}(\|\nu\|^2 + \|\lambda\|^2  + \|\mu\|^2)  + \sum_{i=1}^Ne_{\rho}h_i\left(\mathcal{B}_i(x_i) + \frac{\nu_i}{\rho}\right) \\
   & + e_{\rho} \delta_{\mathcal{Q}}\left(\mathcal{A}(x) + \frac{\lambda}{\rho} \right) + e_{\rho} \delta_{\mathcal{K}} \left(x + \frac{\mu}{\rho}\right),
   \end{split}
\end{equation}
where $\nu: = (\nu_1,\cdots,\nu_N),\lambda$, and $\mu$ are the Lagrange multipliers.

\subsubsection{Saddle point formulation}\label{convex_saddle_point}
Similar to Section \ref{sec:saddle-basic}, we will discuss the strong duality of the saddle point problem induced by the smooth AL function \eqref{eqn:AL-Moreau} for the composite optimization problem \eqref{sec2:pro:composite}:
\begin{equation}\label{sec2:composite:saddle}
\min_x\max_{\nu,\lambda,\mu}\ \mathbb{L}_{\rho}(x,\nu,\lambda,\mu).
\end{equation}
To see the equivalence of problem \eqref{sec2:composite:saddle} to the original problem \eqref{sec2:pro:composite}, it is sufficient to check that given an $x$, the optimal solution $(\nu,\lambda,\mu)$ of maximum problem in \eqref{sec2:composite:saddle}  satisfies 
$$x  = \prox_{h/\rho}\left( x + \nu/\rho \right) = \Pi_{\mathcal{K}}(x + \mu/\rho),\quad\mathcal{A}(x) = \Pi_{\mathcal{Q}}(\mathcal{A}(x) + \lambda/\rho).$$
Substituting it into  \eqref{sec2:composite:saddle} yields that 
\begin{equation}
\begin{aligned}
\min_x\max_{\nu,\lambda,\mu}\,\,\mathbb{L}_{\rho}(x,\nu,\lambda,\mu) & = \min_x f(x) + h(x) + \delta_{\mathcal{K}}(x) + \delta_{\mathcal{Q}}(\mathcal{A}(x)).\nonumber
\end{aligned}
\end{equation}
It is obvious that problem \eqref{sec2:composite:saddle} is equivalent to the original problem \eqref{sec2:pro:composite}, which explains the validity of considering saddle point formulation instead of the original problem.   Next let us establish the strong duality of the saddle point problem \eqref{sec2:composite:saddle}.  {The proof is referred to Appendix \ref{appen:proof-sec2}. } %

\begin{proposition}\label{sec2:composite:propo:duality}
Suppose the objective function of problem \eqref{sec2:pro:composite} has non-empty relative interior and solution set. Then the strong duality holds for $\mathbb{L}_{\rho}(x,\nu,\lambda,\mu)$, that is 
\begin{equation}\label{sec2:prop:strong-duality}
\max_{\nu,\lambda,\mu}\min_{x}\mathbb{L}_{\rho}(x,\nu,\lambda,\mu) = \min_{x}\max_{\nu,\lambda,\mu}\mathbb{L}_{\rho}(x,\nu,\lambda,\mu).
\end{equation}
\end{proposition}

\subsection{Nonconvex composite optimization}\label{sec:NCO}
In this section, we consider a nonconvex generalization of the composite problem \eqref{sec2:pro:composite1}:
\begin{equation}\label{nonconvex-composite}
\min_x\ f(x)+h(x)\quad \text { s.t. }\quad\ c(x) \in \mathcal{Q}, x\in \mathcal{K},
\end{equation}
where $f$ and $c$ are smooth functions, $h$ is a closed proper and lower bounded function, $\mathcal{Q}$ and $\mathcal{K}$ are both nonempty closed sets whose projection operators are easy to evaluate. The key differences from \eqref{sec2:pro:composite1} are that the  smooth function $f$, the nonsmooth function $h$ can be nonconvex and discontinuous, the mapping $c$ can be nonlinear, and the sets $\mathcal{Q},\mathcal{K}$ can be nonconvex and disconnected.

Due to the nonconvexity of $h$, the proximal mapping may not be unique. Therefore, we need to generalize the proximal operator to the nonconvex function. Specifically, when $h$ is closed, proper and lower bounded, it can be proved that the minimizers of \eqref{moreau1} consist a nonempty compact set. Then we define the proximal operator of nonconvex $h$ to be one of the minimizers. Without confusion, we utilize the same notation as in \eqref{def:prox}:
\[
    \prox_{h/t}(x) \in \argmin_{u}\left\{h(u) + \frac{t}{2}\|  u - x\|^2 \right\},
\]
and the projection operator can be defined similarly. Furthermore, the Moreau envelope $e_th(\cdot)$ of a nonconvex function is defined as
\[
    e_t h(x)=\inf_u \left\{h(u) + \frac{t}{2}\|  u - x\|^2 \right\}.
\]
\subsubsection{AL functions for nonconvex composite optimization}\label{sec:AL-nonconvex-composite}
In the convex case, we often deal with nonsmooth terms and constraints by using the Moreau envelope function, which is well-defined due to the original problem's convex nature. Due to the possible nonconvexity of the objective functions and the constraints, the construction of the AL function is different from the convex case. 

\vspace{0.2cm}
\textbf{1) Generalizations of AL functions in formal aspects.}
Similar to \eqref{sec2:pro:composite:auxilary-new}, we introduce extra variables $u,v,w$ and obtain
\begin{equation}
\label{prob:nc-slack}
\begin{aligned}
    \min_{x,u,v,w} &~ f(x) + h(u) + \delta_{\mathcal{Q}}(v) + \delta_{\mathcal{K}}(w) \\
    \mathrm{s.t. }\;&~ x=u,\, c(x) = v,\, x = w.
\end{aligned}\nonumber
\end{equation}
Then the corresponding AL function %
is given by
\begin{equation}\label{eqn:NC-AL}
\begin{aligned}
 \mathcal{L}_{\rho}(x,u,v,w,\nu,\lambda,\mu):=&  f(x) + h(u) + \delta_{\mathcal{Q}}(v) + \delta_{\mathcal{K}}(w)\\
&  + \left<\nu, x - u\right>  + \left<\lambda,c(x) - v\right> + \left<\mu,x-w\right> \\
& + \frac{\rho}{2} \|x - u\|^2  + \frac{\rho}{2} \|c(x) - v \|^2 + \frac{\rho}{2}\| x - w\|^2.
\end{aligned}
\end{equation}
When we follow from the last subsection, the AL function of \eqref{nonconvex-composite} can still be written as 
\begin{equation}\label{eqn:NC-comp-AL}
\begin{aligned}
     \mathbb{L}_{\rho}(x,\nu,\lambda,\mu) 
  = & \min_{u,v,w} \mathcal{L}_{\rho}(x,u,v,w,\nu,\lambda,\mu) \\
  = & f(x) - \frac{1}{2\rho}\left(\|\nu\|^2 + \|\lambda\|^2  + \|\mu\|^2\right) \\
   &+ e_{\rho}h\left(x + \frac{\nu}{\rho}\right) + e_{\rho} \delta_{\mathcal{Q}}\left(c(x) + \frac{\lambda}{\rho} \right) + e_{\rho} \delta_{\mathcal{K}} \left(x + \frac{\mu}{\rho}\right).
\end{aligned}
\end{equation}

While the AL functions in nonconvex scenarios formally resemble those in convex cases, their practical implications diverge significantly. Nonconvexity introduces challenges in algorithm design and {convergence} analysis. In particular, the Moreau envelope function is commonly employed to characterize convex functions. However, in the context of nonconvex functions, the gradient of their Moreau envelope is not well-defined. Consequently, this ambiguity poses substantial obstacles in solving sub-problems during the design of ALM algorithms. For example, if $\mathcal{Q}$ is nonconvex, then the Moreau envelope $e_{\rho} \delta_{\mathcal{Q}}\left(c(x) + \lambda/\rho \right)$ may be non-differentiable, even that the subgradient is not well-defined. Therefore, for the sake of algorithm design, the above AL function is only considered in the case that $h$ is a convex function and $\mathcal{Q},\mathcal{K}$ are two convex sets. Next, we will give another AL function when those three terms may be nonconvex.  

\vspace{0.3cm}
\textbf{2) Variants for different problems.} 
When tackling nonconvexity challenges, our goal shifts towards maintaining the differentiability of the Moreau envelope terms in the AL function. To achieve this target, we should leave the nonconvex nonsmooth functions or nonconvex constraints unchanged or introduce the auxiliary variable. For example, let us consider the case that $h$ is a convex function, $\mathcal{K}$ is a convex cone, but $\mathcal{Q}$ may be nonconvex. 
Due to the nonconvexity of $\mathcal{Q}$, we keep the constraint $v\in \mathcal{Q}$. For the convex constraint $x\in \mathcal{K}$ and the convex nonsmooth term $h$, we tackle them with Moreau envelopes. As a result, we end up with the following AL function where the auxiliary variable $v$ remains:
\begin{equation}\label{eqn:nonconvex-AL2}
\begin{split}
     \mathbb{L}_{\rho}(x,v,\nu,\lambda,\mu)   = &\min_{u,w} \mathcal{L}_{\rho}(x,u,v,w,\nu,\lambda,\mu) \\
   = & f(x) - \frac{1}{2\rho}(\|\nu\|^2 + \|\lambda\|^2  + \|\mu\|^2) \\
   &+ e_{\rho}h\left(x + \frac{\nu}{\rho}\right) + \frac{\rho}{2}\left\|c(x) - v + \frac{\lambda}{\rho} \right\|^2 + e_{\rho} \delta_{\mathcal{K}} \left(x + \frac{\mu}{\rho}\right).
\end{split}
\end{equation}
One can know that the AL function $ \mathbb{L}_{\rho}(x,v,\nu,\lambda,\mu)$ is continuously differentiable. Similar formulation can refer to \cite{de2023constrained}.

\subsubsection{Saddle point formulation}
As nonconvexity exists in the objective function as well as the constraints, the strong duality of the saddle point problem induced by the AL function \eqref{eqn:NC-comp-AL} relies on viewing the penalty factor $\rho$ as a decision variable: 
\begin{equation}\label{minmax-nonconvex-composite}
\min_x\max_{\nu,\lambda,\mu,\atop \rho>0}\,\mathbb{L}_{\rho}(x,\nu,\lambda,\mu).
\end{equation}
For nonconvex problems with conic $\mathcal{K}$ and $\mathcal{Q}$, the strong duality has been proved by Shapiro and Sun \cite{shapiro2004some}. In this part, we will illustrate the strong duality of problem \eqref{minmax-nonconvex-composite} under arbitrary closed and convex $\mathcal{K}$ and $\mathcal{Q}$. %
Denote $\phi(x) = f(x)+h(x)$ and consider the following parameterized problem associated with \eqref{nonconvex-composite}:
\begin{equation}\label{sec2:pro:composite1:par}
\min_{x} \,\, \phi(x)\quad\mathrm{s.t.} \quad c(x)+p \in \mathcal{Q}, x+q\in \mathcal{K}. 
\end{equation}
Denote by $v(p,q)$ the optimal value of problem \eqref{sec2:pro:composite1:par}, then it is obvious to see that $v(0,0)$ stands for the optimal value of problem \eqref{nonconvex-composite}.

\begin{proposition}\label{nonconvex_strong_dual}
{For problem \eqref{nonconvex-composite}, if the objective function $f+h$ is bounded from below, and  the quadratically penalized objective function $\Phi_\rho(x):=\mathbb{L}_\rho(x,0,0,0)$ in \eqref{eqn:NC-comp-AL} is l.s.c. and level-bounded for sufficiently large $\rho$, see \eqref{penalty_inequality},}
then the strong duality holds in the sense that  
\begin{equation}\label{eqn:nonconvex-strong}
\max_{\nu,\lambda,\mu,\atop \rho>0}\min_{x}\,\mathbb{L}_{\rho}(x,\nu,\lambda,\mu) \,=\, \min_{x}\max_{\nu,\lambda,\mu,\atop \rho>0}\mathbb{L}_{\rho}(x,\nu,\lambda,\mu).
\end{equation}
\end{proposition}
Note that the l.s.c. property of $v(p,q)$ at $(0,0)$ can be induced by the l.s.c. of $\phi$, the continuity of $c$, and the so-called inf-compactness condition \cite{bonnans2013perturbation}, and is hence a reasonable assumption.   
The proof of the proposition is omitted here as the analysis is similar to that of \cite{shapiro2004some,rockafellar1974augmented}.

\subsubsection{{The existence of the Lagrange multiplier}} 
{We next investigate the existence of  Lagrange multipliers in nonconvex settings, which is a fundamental aspect in establishing strong duality and analyzing saddle-point-based methods.
In convex optimization, the existence of Lagrange multipliers is straightforward, benefiting from the convex, closed, and bounded feasible regions.  However, in nonconvex optimization,
more detailed conditions and constraint qualifications are often required to ensure their existence. }

For the nonlinear programming \eqref{prob}, Rockafellar \cite{rockafellar1974augmented} thoroughly investigated the theoretical properties of the augmented Lagrangian duality method and introduced the quadratic growth condition for the existence of optimal values in the augmented dual problem.
{To further investigate the properties of augmented Lagrangian duality for a broader class of nonconvex optimization problems, we consider the general nonconvex composite optimization problem \eqref{nonconvex-composite}. The Lagrangian function of \eqref{nonconvex-composite} is defined by
$$
L(x,\lambda,\mu) := f(x)+h(x) + \langle \lambda, c(x) \rangle + \langle \mu, x \rangle,
$$
where $\lambda \in \R^m$ and $\mu \in \R^n$ are the Lagrange multipliers.
The first-order optimality condition at the point $\bar{x}$ is given by}
\begin{equation*}
\partial_x L(\bar{x},\lambda,\mu)=0,\quad \lambda \in \mathcal{N}^{\lim}_\mathcal{Q}(c(\bar{x})),\quad \mu \in \mathcal{N}_{\mathcal{K}}^{\lim}(\bar{x}),
\end{equation*}
where the limiting normal cone to $\mathcal{K}$ at $x$ is defined as 
\begin{equation}
\mathcal{N}_{\mathcal{K}}^{\lim}(x):={\lim\sup}_{v \rightarrow x} \text{cone}(v-\Pi_{\mathcal{K}}(v)).
\end{equation}
Recall that if the set $\mathcal{Q}$ is a convex cone, then 
the condition 
$\lambda \in \mathcal{N}_\mathcal{Q}^{\lim}(c(\bar{x}))$ is equivalent to the conditions:
$$c(\bar{x})\in \mathcal{Q}, \quad \lambda \in \mathcal{Q}^\circ, \quad \text{and} \quad \langle \lambda,c(\bar{x})\rangle = 0,$$ 
where $\mathcal{Q}^\circ:=\{\alpha:\langle \alpha,x\rangle \leq 0,\ \forall x \in \mathcal{Q}\} $ is the polar cone of the $\mathcal{Q}$.

Let $\bar{x}$ be a feasible solution to \eqref{nonconvex-composite}. If $\mathcal{Q}$ and $\mathcal{K}$ are closed nonempty convex sets, then Robinson's constraint qualification (CQ) is as follows:
$$
0 \in \text{int}\{c(\bar{x}) + \text{Im}(J(c(\bar{x})) - \mathcal{Q}\}, \quad 0 \in \text{int}\{\bar{x} + \R^n- \mathcal{Q}\},
$$
where $J(c(\bar{x}))$ represents the {Jacobian} of the function $c$ at $\bar{x}$, {and $\text{Im}(\cdot)$ denotes the image of the input matrix}. For nonlinear programming \eqref{prob}, Robinson's CQ reduces to the MFCQ. Under this constraint qualification,  we can refer to \cite[Theorem 3.9]{bonnans2013perturbation} for the existence of Lagrange multipliers. 

For the nonconvex composite optimization problem \eqref{nonconvex-composite}, the standard KKT conditions may not be satisfied at a local minimum \cite[Page 28]{birgin2014practical}. {Therefore, classical optimality conditions may fail, and the generalized stationarity concepts such as $M$-stationarity and $AM$-stationarity \cite{de2023constrained} are often adopted.}
 Since the objective function $q(x):=f(x)+h(x)$ may not be continuous, then we define a $q$-attentive convergence of a sequence $\{x^k\}$:
$$ x^k\overset{q}{\to}\bar{x}\quad \Leftrightarrow\quad x^k\to\bar{x}\quad\mathrm{with}\quad q(x^k)\to q(\bar{x}).$$

\begin{definition}($M$-stationarity) Let $x\in \R^n$ be a feasible point for \eqref{nonconvex-composite}. Then, $x^*$ is called a Mordukhovich-stationary point of \eqref{nonconvex-composite} if there exists a multiplier $\lambda^*$ and $\mu^*$ such that
\begin{align*}
    -(\nabla c(x^*)^\top \lambda^{*}+\mu^*) &\in \partial (f(x^*)+h(x^*)),\\
    \lambda^* &\in \mathcal{N}_{\mathcal{Q}}^{\lim}(c(x^*)), \\
    \mu^* &\in \mathcal{N}_{\mathcal{K}}^{\lim}(x^*).
\end{align*}
Furthermore, $x^*$is called an asymptotically $M$-stationary point ($AM$-stationarity) of \eqref{nonconvex-composite} if there exist sequences $\{x^k\}, \{\eta^k\}, \{\zeta^k\} \subset \R^n$ and $\{\theta^k\} \subset \R^m$ such that $x^k \stackrel{q}{\rightarrow} x^*, \eta^k \rightarrow 0,\theta^k \rightarrow 0, \zeta^k \rightarrow 0$ and
\begin{align*}
    -(\nabla {c(x^k)}^\top \lambda^{k}+\mu^k) +\eta^k &\in \partial (f(x^k)+h(x^k)),\\
    \lambda^k &\in \mathcal{N}_{\mathcal{Q}}^{\lim}(c(x^k)+\theta^k), \\
    \mu^k &\in \mathcal{N}_{\mathcal{K}}^{\lim}(x^k+\zeta^k),
\end{align*}
for all $k \in \mathbb{N}$.
\end{definition}

Note that if  $x^*\in \R^n$ is an $AM$-stationary point for \eqref{nonconvex-composite}, then $x^*$ is an $M$-stationary point for \eqref{nonconvex-composite}, refer to \cite[Corollary 2.7]{de2023constrained}.
Following the same procedure in \cite{de2023constrained}, we can derive the existence of the multipliers of the ALM for problem \eqref{nonconvex-composite} under certain conditions.

Furthermore, Shapiro and Sun in \cite{shapiro2004some} discussed the second-order optimality conditions ensuring the existence of augmented Lagrange multipliers. They also studied the sensitivity of the augmented Lagrangian minimizers to perturbations of the Lagrange multipliers.  R{\"u}ckmann and Shapiro \cite{ruckmann2009augmented} also derived some necessary and sufficient conditions for the existence of augmented Lagrange multipliers for solving semi-infinite optimization problems. Note that the augmenting functions used in the above papers are special. In \cite{zhou2014existence}, the authors further developed the existence theory of augmented Lagrange multipliers for cone programming using an augmenting function with the valley-at-zero property (which may be neither differentiable nor convex) rather than quadratic functions. A   comprehensive theory of AL functions for cone constrained optimization problems was presented in \cite{dolgopolik2018augmented}. It focuses on the existence of global saddle points for these functions and establishes a general method called the localization principle to prove their existence.

\subsection{Discrete Optimization}\label{IP-strong-duality}
In this subsection, we explore the application of the ALM to discrete optimization problems, with a particular focus on integer programming. We show how ALM can be modified to tackle challenges presented by mixed (linear) integer programming (MIP) and pure integer programming (IP), emphasizing both theoretical foundations and practical implementations.

\subsubsection{{Mixed} integer programming}\label{sec:ip}
Beyond continuous optimization, the augmented Lagrangian approach can also be applied to MIP problems.
To start with, we consider a basic MIP problem with only equality constraints:
\begin{equation}\label{mip}
\min_{x\in \mathbb{R}^n}\ c^{\top} x  \quad \mathrm{s.t.} \quad A x =b,~ x \in \mathcal{K},
\end{equation}
where $\mathcal{K}$ is a mixed integer linear set given by $\mathcal{K} := \{x \in \mathbb{Z}^{n_1} \times \R^{n_2}: Bx \leq d\}$ for some $B \in \R^{q \times n}$, $d \in  \R^q$, $A \in \R^{m \times n}, b \in \R^{m}$, and $n = n_1+n_2$. Here, $Ax=b$ stands for the complicating constraint, and relaxing them makes the problem easier. 
Denote by $z^{\mathrm{IP}}$ the optimal value of problem \eqref{mip} and $z^{\mathrm{LP}}$ the optimal value of its linear programming (LP) relaxation as follows:
\begin{equation}\label{mip-lp-relax}
\min\ c^{\top} x  \quad \mathrm{s.t.} \quad A x =b,~ Bx \leq d.
\end{equation}

{To eliminate the duality gap in integer programming, it is necessary to introduce an augmented Lagrangian function. In particular, for problem \eqref{mip}, Rockafellar and Wets \cite{rockafellar2009variational} proposed a generalized augmented Lagrangian of the form:
\begin{equation}\label{def:inter:alfunction}
\cL_{\rho}(x,\lambda)= c^{\top} x+\lambda^{\top}(Ax-b)+\rho \sigma(Ax-b),
\end{equation}
where $\rho>0$ is a penalty parameter, and $\sigma(\cdot)$ is an augmenting function described in Section \ref{subsubsec:Gen-AL}.}
Then naturally, problem \eqref{mip} can be equivalently formulated as:
$$\min_{x\in\cK} \max_{\lambda\in\R^m} \cL_{\rho}(x,\lambda).$$
Exchanging the min and max, we can define the augmented Lagrangian  relaxation $z_\rho^{\mathrm{LR}+}\!(\lambda)$ and augmented Lagrangian dual $z_\rho^{\mathrm{LD}+}$ of \eqref{mip} as
\begin{equation}\label{MIP_ALR}
z_\rho^{\mathrm{LR}+}(\lambda):=\inf _{x \in \mathcal{K}} \cL_{\rho}(x,\lambda),\quad z_\rho^{\mathrm{LD}+} :=\sup _{\lambda \in \mathbb{R}^m} z_\rho^{\mathrm{LR}+}(\lambda).
\end{equation}
If $z_\rho^{\mathrm{LD}+}=z^{\mathrm{IP}}$ for some $\rho>0$, then the augmented Lagrangian relaxation \eqref{MIP_ALR} can recover a primal optimal solution for the MIP problem \eqref{mip}.

  \subsubsection{Strong duality}
For MIP problem \eqref{mip}, Boland and Eberhard \cite{boland2015augmented} showed that, under appropriate assumptions, the duality gap $z_\rho^{\mathrm{LD}+}-z^{\mathrm{IP}}$  converges to zero as the penalty parameter $\rho$ goes to infinity. If $\mathcal{K}$ is a finite set of discrete elements, then the penalty parameter $\rho$ does not need to go to infinity to close the duality gap.
Feizollahi et al. \cite{feizollahi2017exact}  proved that the augmented Lagrangian dual has an exact penalty representation for general MIPs when the augmenting function $\sigma$ is an arbitrary norm (not squared) under mild conditions. That is, the solution of the augmented Lagrangian subproblem coincides with the solution of the original constrained optimization problem under certain conditions. {Gu et al.\ \cite{gu2020exact} and Bhardwaj et al.\ \cite{bhardwaj2024exact} extended the exactness result of Feizollahi et al.\ \cite{feizollahi2017exact} to mixed‐integer quadratic programming problems. Cordova et al.\ \cite{cordova2022revisiting} recast the augmented Lagrangian function \eqref{def:inter:alfunction} within the framework of a classical Lagrangian by means of a tailored reformulation of the original problem. They establish a zero duality gap and propose an implementable primal–dual bundle method.}

When the problem is a pure integer programming, the zero duality gap and exact penalty representation for pure integer programming (IP) problems can be established under weaker assumptions on the augmenting functions and $\mathcal{K}$ can be infinite \cite{feizollahi2017exact}. We illustrate this result by considering the following pure IP problem with only inequality constraints: 
\begin{equation}\label{ip_inequality}
\min \, c^{\top} x \quad \mathrm{s.t.} \quad A x  \leq b,~ x \in \bar{\mathcal{K}},
\end{equation}
where $\bar{\mathcal{K}}$ is a pure integer linear set  $\bar{\mathcal{K}} = \{x \in \mathbb{Z}^n: Bx \leq d\}$ for some $B \in \R^{q \times n} $, $d \in  \R^q$, $A \in \R^{m \times n}$, and $b \in \R^{m}$. 
Then one may define the AL function \cite{rui2024} as
\begin{equation}\label{IP_alf_inequality}
L_{\rho}(x,\mu)=c^{\top} x +\mu^{\top}(A x  - b)+\frac{\rho}{2}\left\|(A x -b)_+\right\|^2.
\end{equation}
Define the AL relaxation $z_\rho^{\mathrm{LR}+}(\mu)$ and the AL dual $z_\rho^{\mathrm{LD}+}$ as  
\begin{equation*} %
z_\rho^{\mathrm{LR}+}(\mu) := \min_{x \in \bar{\mathcal{K}}} ~L_{\rho}(x,\mu),\quad z_\rho^{\mathrm{LD}+}:=\max_{\mu\in \R_{+}^m} z_\rho^{\mathrm{LR}+}(\mu),\nonumber
\end{equation*} 
and denote $a_i$ the $i$-th row of the matrix $A$ and $b_i$ by the $i$-th element of the vector $b$. We present the strong duality of problem \eqref{ip_inequality}, which aligns with the Lemma 2.1 established in \cite{rui2024}. We also provide a complete proof {in Appendix \ref{appen:proof-sec2}} for clarity and completeness.

\begin{lemma}[Strong Duality]\label{IP_strongd}
 Suppose problem \eqref{ip_inequality} is feasible and an optimal solution exists.  If there exists some $\delta>0$ such that for any $ i \in \{1,2,\cdots,m\} $,
 \begin{equation}\label{delta}
\min_{x \in \bar{\mathcal{K}}}\big\{(a_i x-b_i)^2:  a_i x> b_i\big\}\geq \delta,
\end{equation}
 then there is a finite constant $\rho^* \in  (0, +\infty)$ such that
$$z_{\rho^*}^{\mathrm{LD}+} = f^{\mathrm{IP}}, $$
where $f^{\mathrm{IP}}:=\mathop{\min}\left\{c^{\top}x: x \in \bar{\mathcal{K}},Ax \leq  b\right\}$ denotes the optimal value of primal problem.
\end{lemma}

\section{Overview of Augmented Lagrangian Methods}\label{section3}

In the previous section, we introduced several variants of AL functions under different problem settings, establishing strong duality and equivalence to the saddle {point} problems. Building on these conclusions, we present the augmented Lagrangian method (ALM) based on the AL function. This section first unifies the AL functions across various problem settings and subsequently introduces the ALM from a dual gradient ascent perspective. We will discuss practical algorithms and subproblem solvers, followed by an overview of other methods that leverage the AL function.

\subsection{Augmented dual problem}

\subsubsection{Saddle point problem}
 To illustrate the ALM in a unified way, we present these AL functions in the form $\mathbb{L}_\rho(\textbf{x},\Lambda)$, where $\textbf{x}$ and $\Lambda$ are the collection of all primal variables and dual variables (Lagrange multipliers), respectively. If the original problem is convex, then it is equivalent to solving the following saddle point problem:
\begin{equation}\label{eqn:convexsaddle}
\min_\textbf{x}\max_\Lambda\ \mathbb{L}_\rho(\textbf{x},\Lambda).
\end{equation}
If the original problem is nonconvex, then we also need to  treat $\rho$ as a variable and maximize $\mathbb{L}_\rho(\textbf{x},\Lambda)$ with respect to $\rho$, resulting in the following  equivalent saddle point problem
\begin{equation}\label{eqn:nonconvexsaddle}
\min_\textbf{x}\max_{\rho>0,\Lambda}\ \mathbb{L}_\rho(\textbf{x},\Lambda).
\end{equation}

Note that the dual variable \(\Lambda\) in \eqref{eqn:convexsaddle} and \eqref{eqn:nonconvexsaddle} is free, meaning there are no additional domain constraints imposed on it. In contrast, extra constraints may need to be applied to \(\Lambda\) for the saddle point problem based on the Lagrangian function; for instance, a non-negativity constraint \(\mu \geq 0\) exists in \eqref{eqn:Lag-saddle}. In fact, the optimal dual variable \(\Lambda^\star\) in equations \eqref{eqn:convexsaddle} and \eqref{eqn:nonconvexsaddle} inherently satisfies these extra constraints without necessitating their explicit incorporation in the saddle point formulation. This characteristic offers several advantages, including enhanced flexibility in algorithm design and simplified implementation of the algorithm.

We also uniformly denote the AL function without eliminating the auxiliary variables as \(\mathcal{L}_\rho(\mathbf{x}, \mathbf{y}, \Lambda)\), where \(\mathbf{y}\) represents all the introduced auxiliary variables. For example, \(\mathbf{y} = s\) for the general optimization problem \eqref{AL0}, and \(\mathbf{y} = (u, v, w)\) for the convex composite problem \eqref{sec2:pro:composite:auxilary:al} and the nonconvex composite problem \eqref{nonconvex-composite}. It is straightforward to verify that \(\mathcal{L}_\rho(\mathbf{x}, \mathbf{y}, \Lambda)\) is linear with respect to \(\Lambda\), which implies that \(\mathcal{L}_\rho(\mathbf{x}, \mathbf{y}, \Lambda)\) is concave with respect to \(\Lambda\) for any given \(\mathbf{x}\) and \(\mathbf{y}\). Consequently, \(\mathbb{L}_\rho(\mathbf{x}, \Lambda) = \min_{\mathbf{y}} \mathcal{L}_\rho(\mathbf{x}, \mathbf{y}, \Lambda)\) is also concave in \(\Lambda\) as it is the infimum of concave functions.

\subsubsection{Supergradient of dual problem}\label{sec:supergrad}
According to the discussion of strong duality for both convex and nonconvex problems in Section \ref{sec2:AL:sd}, it holds that
\[
\min_\textbf{x}\max_{\rho>0,\Lambda}\ \mathbb{L}_\rho(\textbf{x},\Lambda)=\max_{\rho>0,\Lambda}\min_\textbf{x}\ \mathbb{L}_\rho(\textbf{x},\Lambda).
\]
Note that the variable $\rho$ can be ignored for  convex problems.
We define the augmented dual function as $\Phi(\Lambda,\rho):= \min_\textbf{x} \mathbb{L}_\rho(\textbf{x},\Lambda)$. Then the original minimization problem is equivalent to the following dual problem:
\begin{equation}\label{eqn:AL-dual}
\max_{\rho>0,\Lambda}\ \Phi(\Lambda,\rho).
\end{equation}
Similar to the description in the previous section, it holds that $\Phi(\Lambda,\rho)$ is a concave function since $\mathbb{L}_\rho(\textbf{x},\Lambda)$ is concave with respect to $(\Lambda,\rho)$ and $\Phi(\Lambda,\rho)$ is the infimum of $\mathbb{L}_\rho(\textbf{x},\Lambda)$ with respect to $\textbf{x}$. The following theorem provides the supergradient of $\Phi(\Lambda, \rho)$. 
\begin{theorem}\label{danskin-Phi}
Suppose that $\mathbf{x}_\Lambda$ attains the infimum of the problem $\min_\mathbf{x} \mathbb{L}_\rho(\mathbf{x},\Lambda)$. Let $g_\Lambda$ and $g_\rho$ be supergradients of $\mathbb{L}_\rho(\mathbf{x}_\Lambda,\cdot)$ in $\Lambda$ and $\rho$, respectively. Then we have $g_\Lambda\in\supgrad_\Lambda\Phi(\Lambda,\rho)$ and $g_\rho\in\supgrad_\rho\Phi(\Lambda,\rho)$.
\end{theorem}

\revise{In fact, under mild conditions, this theorem is essentially Danskin-Demyanov’s theorem \cite{danskin1966theory,dem1968differentiability,dem1968solution}, and its generalized variant \cite{bertsekas1971control}, yet for the AL function, it allows a simple one line argument.} Namely, for any $\Gamma$, it holds that 
$$
\Phi(\Gamma,\rho)=\min_\textbf{x}\mathbb{L}_\rho(\textbf{x},\Gamma)\leq \mathbb{L}_\rho(\textbf{x}_\Lambda,\Gamma)\leq \mathbb{L}_\rho(\textbf{x}_\Lambda,\Lambda)+\langle g_\Lambda, \Gamma-\Lambda\rangle=\Phi(\Lambda,\rho)+\langle g_\Lambda, \Gamma-\Lambda\rangle,
$$
indicating $g_\Lambda\in\supgrad_\Lambda\Phi(\Lambda,\rho)$. Following the same procedure, one obtains $g_\rho\in\supgrad_\rho\Phi(\Lambda,\rho)$.

It should be noted that for integer programming problems where the primal feasible region is discrete while the dual variable $\Lambda$ and the penalty factor $\rho$ are continuous, the supergradients with respect to $\Lambda$ and $\rho$ are still well-defined and Theorem \ref{danskin-Phi} remains valid.

\subsection{Prototype algorithm of ALM}
{The ALM was first formally established in the pioneering works of Hestenes \cite{hestenes1969multiplier} and Powell \cite{powell1969method}, which generate iterates $(\mathbf{x}^k,\mathbf{\Lambda}^k)$ by alternating between a (possibly inexact) minimization in $\mathbf{x}$ and a multiplier update in $\mathbf{\Lambda}$. In this section, we introduce the ALM algorithm from the perspective of the AL dual problem, which has been given in \cite{bertsekas2014constrained,zhang2022global}. We show that the ALM algorithm can be regarded as a supergradient method applied to the dual problem \eqref{eqn:AL-dual}, owing to the strong duality established in Section \ref{sec2:AL:sd} and the representation of the supergradient for the dual problem provided in Theorem \ref{danskin-Phi}. We begin by presenting a prototypical algorithm and then, for different problem classes, discuss how to construct ALM algorithms from this dual viewpoint.}
\subsubsection{ALM as dual ascent}\label{sec:dualascent}
The supergradient method is the most straightforward approach for solving problem \eqref{eqn:AL-dual}. Since \eqref{eqn:AL-dual} is the dual problem, it is aptly referred to as the dual ascent method. If the original problem is convex, the maximization with respect to $\rho$ is unnecessary and the update scheme is
\[
    \Lambda^{k+1}\in\Lambda^k+\sigma \supgrad_\Lambda \Phi(\Lambda^k,\rho),
\]
where $\sigma>0$ is a stepsize. According to Theorem \ref{danskin-Phi}, the supergradient of $\Phi$ can be calculated by the supergradient of $\mathbb{L}_\rho$. Specifically, we need to calculate the global minimizer of $\mathbb{L}_\rho$ with respect to $\textbf{x}$ first, and then calculate the supergradient of $\Phi$  with respect to the dual variable at the minimizer. Thus, the scheme can be written as
\begin{align}
\label{ALM:primal} \textbf{x}^{k+1}&\in\argmin_\textbf{x} \mathbb{L}_\rho(\textbf{x},\Lambda^k),\\
\label{ALM:dual} \Lambda^{k+1}&\in\Lambda^k+\sigma \supgrad_\Lambda \mathbb{L}_\rho(\textbf{x}^{k+1},\Lambda^k).
\end{align}
This is exactly the update scheme of the ALM.
Therefore, the ALM can be interpreted as performing the supergradient ascent to maximize the dual problem.  

For nonconvex problems, we also need to maximize $\Phi$ with respect to $\rho$.
In addition to the update of $\textbf{x}$ and $\Lambda$ in \eqref{ALM:primal} and \eqref{ALM:dual}, the penalty factor $\rho_k$ needs to be updated as
\begin{equation}\label{ALM:rho}
\rho_{k+1}\in\rho_k+\sigma\supgrad_\rho \Phi(\Lambda^{k+1},\rho_k).
\end{equation}
Similarly, by Theorem \ref{danskin-Phi}, the complete scheme can be written as 
\begin{align}
\label{prototype:NC-ALM-x} \textbf{x}^{k+1}&\in\argmin_\textbf{x} \mathbb{L}_{\rho_k}(\textbf{x},\Lambda^k),\\
\label{prototype:NC-ALM-Lambda} \Lambda^{k+1}&\in\Lambda^k+\sigma_k\supgrad_\Lambda \mathbb{L}_{\rho_k}(\textbf{x}^{k+1},\Lambda^k),\\
\label{prototype:NC-ALM-rho} \rho_{k+1}&\in\rho_k+\sigma_k\supgrad_\rho \mathbb{L}_{\rho_k}(\textbf{x}^{k+1},\Lambda^{k+1}).
\end{align}

The update rule of $\rho_k$ looks different from the popular rule, where $\rho_k$ increases by a constant multiple greater than one. In fact, both of the update rules make $\rho_k$ tend to $+\infty$, which will be proved later in specific cases. Although the two schemes lead to different growth rates of $\rho_k$, the fundamental essence is the same.

{In addition to the supergradient interpretation, the ALM can also be viewed through the lens of the proximal point algorithm (PPA) \cite{rockafellar1976augmented}. Both perspectives arise from dual formulations: the supergradient approach is based on the AL dual problem \eqref{eqn:AL-dual}, whereas the PPA viewpoint relies on the classical Lagrangian dual problem (such as the right-hand in \eqref{eqn:Lag-saddle}). However, classical Lagrangian strong duality holds only in convex settings, as discussed in Theorem \ref{sec2:general:sd},   which restricts the applicability of the PPA-based interpretation. Moreover, gradient-based methods tend to be more transparent and easier to implement than proximal-point schemes. Consequently, we henceforth adopt the supergradient perspective on the AL dual problem as a unifying framework for constructing ALM algorithms across a broad spectrum of optimization problems.}

\subsubsection{ALM for specific problems}
To explicitly present the ALM, we need to compute the supergradient of $\mathbb{L}_{\rho}$ with respect to $\Lambda$ and $\rho$. Additionally, to address the minimization problem, we must also determine the subgradient of $\mathbb{L}_{\rho}$ with respect to $\textbf{x}$. In this section, we derive these gradients and present explicit forms of ALM for various scenarios.

\vspace{0.3cm}
\textbf{1) General optimization problem.} 
We first consider the basic formulation \eqref{prob}. The corresponding AL function $\mathbb{L}_{\rho}(x,\lambda,\mu)$ is given in \eqref{AL}. Then the explicit gradient formulation of the AL function with respect to primal variable, Lagrange multipliers and penalty parameter  can be given as
{\begin{equation}\label{gradient_cx}
    \begin{split}
    \nabla_{\lambda_i}\mathbb{L}_{\rho}(x,\lambda,\mu)&=c_i(x), i\in\mathcal{E},\\
    \nabla_{\mu_i}\mathbb{L}_{\rho}(x,\lambda,\mu)&=\max\left\{-\frac{\mu_i}{\rho},c_i(x)\right\}, i\in\mathcal{I},\\
    \nabla_\rho\mathbb{L}_{\rho}(x,\lambda,\mu)&=\frac12 \sum_{i\in\mathcal{E}} c^2_i(x)+\frac12\sum_{i\in\mathcal{I}} \left(\max\left\{-\frac{\mu_i}{\rho},c_i(x)\right\}\right)^2.
    \end{split}
\end{equation}
Furthermore, if functions $f(x), c_i(x), i\in\mathcal{E}\cup\mathcal{I}$ are differentiable, it holds that
\begin{equation*}
    \begin{split}
    \nabla_x\mathbb{L}_{\rho}(x,\lambda,\mu)=&\nabla f(x)+\sum_{i\in\mathcal{E}}\lambda_i \nabla c_i(x)+\sum_{i\in\mathcal{I}}\mu_i \nabla c_i(x)\\
    &+\rho \sum_{i\in\mathcal{E}}c_i(x) \nabla c_i(x)+\rho\sum_{i\in\mathcal{I}}\left[\frac{\mu_i}{\rho}+ c_i(x)\right]_+ \nabla c_i(x).
    \end{split}
\end{equation*}}

The above conclusion can be obtained through a direct derivation. Having this {expression} at hand, we can substitute it into \eqref{prototype:NC-ALM-x}-\eqref{prototype:NC-ALM-rho} to give the explicit formulation. Let $\sigma_k=\rho_k$, then the update scheme can be written as
$$
\left\{
\begin{aligned}
    x^{k+1}&\in\argmin_x \mathbb{L}_{\rho_k}(x,\lambda^k,\mu^k),\nonumber\\
    \lambda_i^{k+1}&=\lambda_i^k+\rho_k c_i(x^{k+1}),i\in\mathcal{E},\nonumber\\
    \mu_i^{k+1}&=\max\left\{\mu_i^k+\rho_k c_i(x^{k+1}),0\right\},i\in\mathcal{I},\nonumber\\
    \rho_{k+1}&=\rho_k+\frac{\rho_k}2\left(\sum_{i\in\mathcal{E}} c^2_i(x^{k+1})+\sum_{i\in\mathcal{I}}\left( \max\left\{-\frac{\mu_i^{k+1}}{\rho_k}, c_i(x^{k+1})\right\}\right)^2\right). \nonumber
\end{aligned}
\right.
$$
We can  observed that the gradient with respect to $\rho$ is always non-negative, that is, $$\nabla_\rho\mathbb{L}_{\rho}(x,\lambda,\mu)\geq 0.$$ Thus, the sequence $\{\rho_k\}$ is monotonically increasing.

\vspace{0.3cm}
\textbf{2) Convex composite problem.} Similarly, we can also give the results on the convex composite problem discussed in Section \ref{sec:convex-comp}.  %
   { For the convex composite problem \eqref{sec2:pro:composite1}, the AL function $\mathbb{L}_{\rho}$ given in \eqref{sec2:al:eliminate} is differentiable and its gradient can be computed as;
    \begin{equation}\label{eq:comp-grad}
    \begin{aligned}
     \nabla_\lambda \mathbb{L}_{\rho}(x,\nu,\lambda,\mu)  &= \mathcal{A}(x) - \Pi_{\mathcal{Q}}(\mathcal{A}(x) + \lambda/\rho), \\
      \nabla_\mu  \mathbb{L}_{\rho}(x,\nu,\lambda,\mu)  &= x - \Pi_{\mathcal{K}}(x + \mu/\rho), \\
      \nabla_\nu \mathbb{L}_{\rho}(x,\nu,\lambda,\mu) &= x - \prox_{h/\rho}\left(x + \nu/\rho\right),\\
      \nabla_x \mathbb{L}_{\rho}(x,\nu,\lambda,\mu) &= \nabla f(x) + \prox_{\rho h^*}\left(\rho x + \nu\right)  \\
       &\quad + \rho \mathcal{A}^*(\hat{\Pi}_{\mathcal{Q}}(\mathcal{A}(x) + \lambda/\rho) ) + \rho \hat{\Pi}_{\mathcal{K}}(x + \mu/\rho).
    \end{aligned}
\end{equation}}
In fact, since the Moreau envelope function is continuously differentiable for the convex case, we can easily derive the following expressions for the gradients with respect to $x$ by combining with equation \eqref{me:gradient} and \eqref{moreau}:
\begin{align*}
    \nabla_x e_{\rho}h(x + \nu/\rho) &= \rho(  x + \nu/ \rho - \prox_{h/\rho}(x + \nu/\rho)) = \prox_{\rho h^*}\left(\rho x + \nu\right), \\
     \nabla_x e_{\rho} \delta_{\mathcal{K}} \left(x + \mu/\rho\right) & = \rho(x + \mu/\rho- \Pi_{\mathcal{K}}(x + \mu/\rho)) = \rho \hat{\Pi}_{\mathcal{K}}(x + \mu/\rho), \\
    \nabla_x e_\rho\delta_{\mathcal{Q}}\left(\mathcal{A}(x) + \lambda/\rho \right) & =  \rho \mathcal{A}^*(\mathcal{A}(x)  \!+\! \lambda/\rho - \Pi_{\mathcal{Q}}(\mathcal{A}(x) \!+\! \lambda/\rho) ) = \rho \mathcal{A}^*\hat{\Pi}_{\mathcal{Q}}(\mathcal{A}(x) \!+\! \lambda/\rho). 
\end{align*}
Given the differentiability of the Moreau envelope, the gradient of $\mathbb{L}_\rho(x,\nu,\lambda,\mu)$ with respect to $x$ can be easily derived. The gradients with respect to $\nu,\lambda,\mu$ can be obtained similarly.

It is important to mention that for the convex composite optimization problem \eqref{sec2:pro:composite1}, strong duality holds without needing to update 
$\rho$ as a variable, as discussed in Section \ref{convex_saddle_point}. Therefore, by setting the dual stepsize $\sigma = \rho$, the update scheme \eqref{ALM:primal}$\sim$\eqref{ALM:dual} can be rewritten as:
$$
\left\{
\begin{aligned}
\label{composite-x-subprob}x^{k+1} &= \argmin_x\mathbb{L}_{\rho}(x,\nu^k,\lambda^k,\mu^k),\\
\nonumber \nu^{k+1} &= \nu^k+\rho\left(x^{k+1} - \prox_{h/\rho}\left(x^{k+1} + \nu^k/\rho\right)\right),
\\
\nonumber \lambda^{k+1} &= \lambda^k+\rho\left(\mathcal{A}(x^{k+1})- \Pi_{\mathcal{Q}}(\mathcal{A}(x^{k+1}) + \lambda^k/\rho)\right), \\
\nonumber \mu^{k+1} &= \mu^k+\rho\left(x^{k+1} - \Pi_{\mathcal{K}}(x^{k+1} + \mu^k/\rho)\right).
\end{aligned}
\right.
$$
Since $\mathbb{L}_{\rho}(x,\nu,\lambda,\mu)$ is differentiable with respect to $x$, the subproblem of $x$ can be solved by the gradient-type methods.

Note that we define  an alternative formulation of the AL function $\mathbb{L}_{\rho}(x,\lambda,\mu)$ in \eqref{eqn:AL-h}, where the nonsmooth term is retained. The following update scheme can also be provided similarly.
$$
\left\{
\begin{aligned}
x^{k+1} &\in \argmin_x \mathbb{L}_{\rho}(x,\lambda^k,\mu^k),\\
\lambda^{k+1} &= \lambda^k+\rho\left(\mathcal{A}(x^{k+1})- \Pi_{\mathcal{Q}}(\mathcal{A}(x^{k+1}) + \lambda^k/\rho)\right), \\
\mu^{k+1} &= \mu^k+\rho\left(x^{k+1} - \Pi_{\mathcal{K}}(x^{k+1} + \mu^k/\rho)\right).
\end{aligned}
\right.
$$
In this case, $\mathbb{L}_{\rho}(x,\lambda,\mu)$ is nondifferentiable with respect to $x$ and the gradient descent method cannot be applied. However, since the proximal operator of $h$ is efficient to evaluate, we can utilize the proximal gradient method instead. 

\vspace{0.3cm}
\textbf{3) Nonconvex composite problem.}
For the nonconvex composite problem \eqref{nonconvex-composite}, the AL function $\mathbb{L}_{\rho}(\cdot)$ given in \eqref{eqn:NC-comp-AL} is concave with respect to the dual variables $\nu,\lambda,\mu$. Since the proximal mappings of nonconvex functions may not be unique for \eqref{nonconvex-composite}, $\mathbb{L}_{\rho}(x,\nu,\lambda,\mu)$ is nonsmooth with respect to both primal and dual variables. 
Similar to Theorem \ref{danskin-Phi}, we assume that 
\[
(u^*,v^*,w^*)\in\argmin_{u,v,w} \mathcal{L}_\rho(x,u,v,w,\nu,\lambda,\mu)
\]
is some optimal point such that  $u^* = \prox_{h/\rho}\left(x + \nu/\rho\right), v^* = \Pi_{\mathcal{Q}}(c(x) + \lambda/\rho)$ and $ w^* = \Pi_{\mathcal{K}}(x + \mu/\rho)$. Then  $\mathbb{L}_\rho(x,\nu,\lambda,\mu)$ is concave with respect to $\nu,\lambda,\mu, \rho$ and $$\supgrad_{\nu,\lambda,\mu,\rho} \mathcal{L}_\rho(x,u^*,v^*,w^*,\nu,\lambda,\mu)\in\supgrad_{\nu,\lambda,\mu,\rho} \mathbb{L}_\rho(x,\nu,\lambda,\mu).$$ 
 Hence, the supergradient of $\mathbb{L}_{\rho}(x,\nu,\lambda,\mu)$ of the dual variables can still be calculated as %
{\begin{equation}\label{eq:nonconvex-diff}
    \begin{aligned}
    \supgrad_\nu \mathbb{L}_{\rho}(x,\nu,\lambda,\mu)& \ni x - \prox_{h/\rho}\left(x + \nu/\rho\right),\\
    \supgrad_\lambda \mathbb{L}_{\rho}(x,\nu,\lambda,\mu) & \ni c(x) - \Pi_{\mathcal{Q}}(c(x) + \lambda/\rho), \\
    \supgrad_\mu  \mathbb{L}_{\rho}(x,\nu,\lambda,\mu) & \ni x - \Pi_{\mathcal{K}}(x + \mu/\rho), \\
    \supgrad_{\rho} \mathbb{L}_{\rho}(x,\nu,\lambda,\mu) & \ni\frac12\|x-\prox_{h/\rho}\left(x + \nu/\rho\right)\|^2+\frac12\|c(x)-\Pi_{\mathcal{Q}}(c(x) + \lambda/\rho)\|^2\\
    &\quad +\frac12\|x-\Pi_{\mathcal{K}}(x + \mu/\rho)\|^2.
    \end{aligned}
\end{equation}}

For the nonconvex composite optimization problem \eqref{nonconvex-composite}, let the dual stepsize be $\sigma_k = \rho_k$, then the update scheme \eqref{prototype:NC-ALM-x}$\sim$\eqref{prototype:NC-ALM-rho} can be rewritten as
$$
\left\{
\begin{aligned}
x^{k+1}&\in\argmin_{x} \mathbb{L}_{\rho_k}(x,\nu^k,\lambda^k,\mu^k),\\
\nu^{k+1} &= \nu^k+\rho_k\left(x^{k+1} - \prox_{h/\rho}\left(x^{k+1} + \nu^k/\rho_k\right)\right),
\\
\lambda^{k+1} &= \lambda^k+\rho_k\left(c(x^{k+1})- \Pi_{\mathcal{Q}}(c(x^{k+1}) + \lambda^k/\rho_k)\right), \\
\mu^{k+1} &= \mu^k+\rho_k\left(x^{k+1} - \Pi_{\mathcal{K}}(x^{k+1} + \mu^k/\rho_k)\right),\\
\rho_{k+1}&\in\rho_k+\partial_\rho \mathbb{L}_{\rho_k}(x^{k+1},\nu^{k+1},\lambda^{k+1},\mu^{k+1}).
\end{aligned}
\right.
$$
Note that the sequence $\{\rho_k\}$ is monotonically increasing. Thus, the update rule of $\rho_k$ can also be seen as multiplication by a factor. 

We also discuss the different variant of AL functions in Section \ref{sec:AL-nonconvex-composite}, which derives different formulation of ALM. Taking the AL function defined in \eqref{eqn:nonconvex-AL2} as an example, the corresponding ALM can be written as
$$
\left\{
\begin{aligned}
(x^{k+1}, v^{k+1})&\in\argmin_{v\in\mathcal{Q}, x} \mathbb{L}_{\rho_k}(x,v, \nu^k,\lambda^k,\mu^k),\\
\nu^{k+1} &= \nu^k+\rho_k\left(x^{k+1} - \prox_{h/\rho}\left(x^{k+1} + \nu^k/\rho_k\right)\right),
\\
\lambda^{k+1} &= \lambda^k+\rho_k\left(c(x^{k+1})- v^{k+1}\right), \\
\mu^{k+1} &= \mu^k+\rho_k\left(x^{k+1} - \Pi_{\mathcal{K}}(x^{k+1} + \mu^k/\rho_k)\right),\\
\rho_{k+1}&\in\rho_k+\partial_\rho \mathbb{L}_{\rho_k}(x^{k+1},v^{k+1}, z^{k+1},\lambda^{k+1},\mu^{k+1}).
\end{aligned}
\right.
$$
In this case, the subproblem can be alternately minimized with respect to $x$ and $v$, which is also known as BCD. Each block is easy to solve since the AL function is differentiable with respect to $x$ and the block of $v$ has closed-form solution.

\vspace{0.3cm}
\textbf{4) Integer programming.} %
 {   For the integer programming \eqref{ip_inequality}, the AL function defined in \eqref{IP_alf_inequality} is differentiable and its gradient can be given as follows:
    \begin{align*}
     \nabla_x \mathrm{L}_{\rho}(x,\mu)  &=c +A^{\top}\mu+\rho A^{\top}(A x -b)_+, \\
      \nabla_\mu  \mathrm{L}_{\rho}(x,\mu)  &= A x -b, \\
      \nabla_\rho \mathrm{L}_{\rho}(x,\mu) &= \frac{1}{2}\|(Ax-b)_+\|^2.
    \end{align*}}
Using the projected subgradient method to update parameters $(\mu, \rho)$  with the dual stepsize $\sigma_k = \rho_k$, the update scheme \eqref{prototype:NC-ALM-x}$\sim$\eqref{prototype:NC-ALM-rho} can be rewritten as
$$
\left\{
\begin{aligned}
     x^{k+1}&\in\argmin_{x \in \bar{\mathcal{K}}} \mathrm{L}_{\rho_k}(x,\mu^k),\\
 \mu^{k+1}&=\max\left\{\mu^k+\rho_k (Ax^{k+1}-b),0\right\},\\
 \rho_{k+1}&=\rho_k+ \frac{\rho_k}{2}\|(Ax^{k+1}-b)_+\|^2.
\end{aligned}
\right.
$$

\subsection{A general ALM framework}
In the ALM scheme, the $\textbf{x}$-subproblem \eqref{prototype:NC-ALM-x} often cannot be exactly solved in practice. Let $\Lambda^k,\rho_k$ be the latest variables in iterate $k$, then we can solve the $k$-th $\textbf{x}$-subproblem \eqref{prototype:NC-ALM-x} inexactly and the scheme {\eqref{prototype:NC-ALM-x}$\sim$\eqref{prototype:NC-ALM-rho}} becomes 
$$
\left\{
\begin{aligned}
\label{ALM_scheme_app_x} \textbf{x}^{k+1}&\approx\argmin_\textbf{x} \mathbb{L}_{\rho_k}(\textbf{x},\Lambda^k),\\
\label{ALM_scheme_app_l} \Lambda^{k+1}&\in\Lambda^k+\sigma_k\supgrad_\Lambda \mathbb{L}_{\rho_k}(\textbf{x}^{k+1},\Lambda^k),\\
\label{ALM_scheme_app_r} \rho_{k+1}&\in\rho_k+\sigma_k\supgrad_\rho \mathbb{L}_{\rho_k}(\textbf{x}^{k+1},\Lambda^{k+1}).
\end{aligned}
\right.
$$
Therefore, we need to formally define the optimality measure to quantify the level of inexactness in solving the subproblem. Generally, if the optimality measure is small enough, and the violation of the constraints is also controlled, then the iterate can be verified as an approximate solution to the original problem.

\subsubsection{Optimality measure}
According to our previous description, we establish two criteria for the approximate solution, targeting the various problems
mentioned above. For any variables $\textbf{x},\Lambda,\rho$, we introduce a criterion   $\varsigma(\textbf{x},\Lambda,\rho)$ to  measure the accuracy of solving  subproblems, and a second criterion $\vartheta(\textbf{x},\Lambda,\rho)$ to measure the constraint violation. Next, we present specific definitions of $\varsigma$ and $\vartheta$ in various problem formulations.

\vspace{0.3cm}
\textbf{1) General optimization problem.}
For the general problem \eqref{prob}, the $(\eta,\epsilon)$-stationary point is defined as
\begin{align}
 \label{sec3:general_stop_u}   \varsigma(\textbf{x},\Lambda,\rho) &:=\|\nabla_\textbf{x} \mathbb{L}_{\rho}(\textbf{x},\Lambda)\|\leq\eta, \\
 \label{sec3:general_stop_v}  \vartheta(\textbf{x},\Lambda, \rho) &:=\|\nabla_\Lambda \mathbb{L}_{\rho}(\textbf{x},\Lambda)\|\leq\epsilon,
\end{align}
where the definitions of the gradients have been given in \eqref{gradient_cx}.

For the inequality in \eqref{sec3:general_stop_u}, we can always make it hold approximately for given $\Lambda$ and $\rho$. Specifically, we can obtain an approximate solution to the $k$-th subproblem $\min_\textbf{x} \mathbb{L}_{\rho_k}(\textbf{x},\Lambda^k)$ that satisfies
\[
    \varsigma(\textbf{x}^{k+1},\Lambda^k,\rho_k)=\|\nabla_\textbf{x} \mathbb{L}_{\rho_k}(\textbf{x}^{k+1},\Lambda^k)\|\leq\eta_k.
\]
Therefore, it can serve as a stopping criterion of the subproblem.

The inequality in \eqref{sec3:general_stop_v} essentially bounds the constraint violation. We know that problem \eqref{prob} is equivalent to the equality constrained problem \eqref{prob2}, with the constraint violation defined by  
\[
\vartheta(\textbf{x},\Lambda,\rho)=\left(\sum_{i\in\mathcal{E}}c_i^2(x)+\sum_{i\in\mathcal{I}}(c_i(x)+s_i)^2\right)^{\frac12},
\]
where the auxiliary variable $\{s_i\}$ depends on $\Lambda$ and $\rho$ according to \eqref{sopt}.
Substituting \eqref{sopt} into it, we obtain the formal definition 
\[
\vartheta(\textbf{x},\Lambda,\rho)=\left(\sum_{i\in\mathcal{E}}c_i^2(x)+\sum_{i\in\mathcal{I}}\left(c_i(x)+\max\left\{-\frac{\mu_i}{\rho}-c_i(x),0\right\}\right)^2\right)^{\frac12}.
\]

\vspace{0.3cm}
\textbf{2) Convex composite problem.}
For the convex composite problem \eqref{sec2:pro:composite1}, the $(\eta,\epsilon)$-stationary point for the convex composite problem is defined the same as \eqref{sec3:general_stop_u}-\eqref{sec3:general_stop_v} and the corresponding gradients are defined in  \ref{eq:comp-grad}.

We can also use $\varsigma(\textbf{x}_{k+1},\Lambda_k,\rho_k)\leq\eta_k$ as the stopping criterion of the subproblem. The second inequality $\vartheta(\textbf{x},\Lambda,\rho)\leq\epsilon$ bounds the constraint violation. Specifically, problem \eqref{sec2:pro:composite1} is equivalent to the equality constrained problem \eqref{sec2:pro:composite:auxilary-new} whose constraint violation is 
\[
\vartheta(\textbf{x},\Lambda,\rho) = \left(\|x-u\|^2+\|\mathcal{A}(x)-v\|^2+\|x-w\|^2\right)^{\frac12}, 
\]
where $u,v,w$ depend on $\Lambda$ and $\rho$ according to \eqref{eq:y-expression-new}. Substituting \eqref{eq:y-expression-new} into it yields the formal definition
\begin{align*}
\vartheta(\textbf{x},\Lambda,\rho)=&\left(\left\|x - \prox_{h/\rho}\left(x + \frac{\nu}{\rho}\right)\right\|^2\right.\\
&\quad \left.+\left\|\mathcal{A}(x) - \Pi_{\mathcal{Q}}\left(\mathcal{A}(x) + \frac{\lambda}{\rho}\right)\right\|^2+\left\|x - \Pi_{\mathcal{K}}\left(x + \frac{\mu}{\rho}\right)\right\|^2\right)^{\frac12}.
\end{align*}

\vspace{0.3cm}
\textbf{3) Nonconvex composite problem.}
For the nonconvex composite problem \eqref{nonconvex-composite}, the AL function is not differentiable but it is concave with respect to $\Lambda$. Hence the supergradient $\bar \partial_\Lambda \mathbb{L}_{\rho_k}(\textbf{x},\Lambda)$ is well-defined. However, since $\mathbb{L}_{\rho_k}(\textbf{x},\Lambda)$ may be neither convex nor concave with respect to $\textbf{x}$, we need to use the Clarke subdifferential. The corresponding $(\eta,\epsilon)$-stationary point is defined as
\begin{align}
    &\label{nc-subopt1} \varsigma(\textbf{x},\Lambda,\rho) := \dist\left(0, \hat \partial_\textbf{x} \mathbb{L}_{\rho}(\textbf{x},\Lambda)\right)\leq\eta,\\
    &\label{nc-subopt2} \vartheta(\textbf{x},\Lambda,\rho):=\dist\left(0, \bar \partial_\Lambda \mathbb{L}_{\rho}(\textbf{x},\Lambda)\right)\leq\epsilon.
\end{align}

The inequality \eqref{nc-subopt1} can serve as the stopping criterion of the subproblem. However, the distance to the set may be difficult to evaluate. We have to combine the specific problem formulation to give the estimation of the distance. Alternatively, we can use the relative change of the iterates or the AL function value at two consecutive iteration points as the stopping criterion. We have to mention that alternative is practical but not rigorous. Let $\textbf{x}^{k,0}:=\textbf{x}^k$ and $\textbf{x}^{k,j}$ represent the $j$-th inner iterate when solving the $k$-th subproblem, then the stopping criterion can be
\begin{equation}\label{sec3:consecutive_stop}
\begin{aligned}
\varsigma(\textbf{x}^{k,j+1},\Lambda^k,\rho_k)&=\frac{\|\textbf{x}^{k,j+1}-\textbf{x}^{k,j}\|}{\max\{\|\textbf{x}^{k,j}\|,1\}} \leq\eta_k,\quad \text{or}\\
\varsigma(\textbf{x}^{k,j+1},\Lambda^k,\rho_k)&=\frac{|\mathbb{L}_{\rho_{k}}(\textbf{x}^{k,j+1},\Lambda^{k})-\mathbb{L}_{\rho_{k}}(\textbf{x}^{k,j},\Lambda^{k})|}{\max\{\mathbb{L}_{\rho_{k}}(\textbf{x}^{k,j},\Lambda^{k}),1\}}\leq\eta_k.
\end{aligned}
\end{equation}
That is, when either of the inequalities holds, we exit the inner loop and set the last inner iterate as the approximate solution of the subproblem.

Similarly, the criterion \eqref{nc-subopt2} also bounds the constraint violation. We can utilize the supergradient expression \eqref{eq:nonconvex-diff} for dual variables to obtain an element in the set and hence give an upper bound of the distance. Specifically, we can define $\bar \vartheta(\textbf{x},\Lambda,\rho)\geq \vartheta(\textbf{x},\Lambda,\rho)$ as a surrogate function where
\begin{align*}
\bar \vartheta(\textbf{x},\Lambda,\rho) =& \left\{\dist^2\left(0, x-\prox_{h/\rho}\left(x + \frac{\nu}{\rho}\right)\right)+ \dist^2\left(0, \mathcal{A}(x) - \Pi_{\mathcal{Q}}\left(\mathcal{A}(x) + \frac{\lambda}{\rho}\right)\right)\right.\\
&\left.+\dist^2\left(0, x - \Pi_{\mathcal{K}}\left(x + \frac{\mu}{\rho}\right)\right)\right\}^{\frac12}.
\end{align*}

\vspace{0.3cm}
\textbf{4) Integer programming.} For integer programming problem in the form of \eqref{ip_inequality}, the feasible set is discrete. Hence, employing a stopping criterion based on achieving a sufficiently small gradient norm of the AL function, analogous to \eqref{sec3:general_stop_u}, is generally impractical in the minimization of the AL function.  Then the stopping criterion for the subproblem generated by the ALM algorithm can be implemented as follows:
\begin{equation}\label{sec3:ux0}
  \varsigma(\textbf{x}^{k,j+1},\Lambda^k,\rho_k)=\|\textbf{x}^{k,j+1}-\textbf{x}^{k,j}\|=0.
 \end{equation}
If this condition is difficult to satisfy, we can use the stopping criterion shown in \eqref{sec3:consecutive_stop} instead. Note that the point that satisfies \eqref{sec3:ux0} is only a stationary point of the problem, which may not necessarily be a locally optimal solution.

For the outer iterations of the ALM,  the dual multiplier $\mu$ and the penalty parameter $\rho$ are updated via the subgradient method. One limitation of the subgradient method is the lack of a well-defined stopping criterion, which ideally would use the duality gap between the primal problem and the augmented dual problem \cite{burer2006solving}. However, calculating this duality gap is not straightforward.
Alternatively, we can implement the stopping criteria as that the primal iterates satisfy the constraint violation tolerance at a sufficiently small threshold $\epsilon$, i.e.,
\[
\vartheta(\textbf{x}^{k+1},\Lambda^k,\rho_k)=\left\|[A\textbf{x}^{k+1}-b]_+\right\|^2 \leq \epsilon.
\]
Furthermore, we note that for many integer programming problems in practice, if ALM obtains a feasible solution, the objective value of that solution typically serves as an upper bound (UB) on the optimal value. Once a lower bound (LB) on the optimal value is determined through some approach, we can also utilize the optimality gap, which is based on the relative difference between the upper and lower bounds, i.e., 
\[
\vartheta(\textbf{x}^{k+1},\Lambda^k,\rho_k)=\text{gap}=|\text{UB}-\text{LB}|/|\text{UB}|\leq \epsilon,
\]
as a stopping criterion. 
The benefit of utilizing the optimality gap as a stopping criterion is twofold: it enables assessing the quality of the obtained solution, and it provides the potential to find the global optimal solution.

\subsubsection{A practical algorithm}

Based on the preceding analysis, we introduce a practical algorithm (Algorithm \ref{alg: practical-ALM}) generalized from the traditional ALM \cite{NocedalWright06}. Different update schemes are adopted based on whether the constraint violation satisfies $\vartheta(\textbf{x}^{k+1}, \Lambda^k, \rho_k) \leq \epsilon_k$. If this condition is satisfied, the penalty parameter remains unchanged in the subsequent iteration since the current value of $\rho_k$ effectively controls the constraint violations. The Lagrange multipliers are updated via \eqref{ALM_scheme_app_l}, and the tolerances $\eta_{k+1}$ and $\epsilon_{k+1}$ are preemptively tightened in preparation for the next iteration. Conversely, if $\vartheta(\textbf{x}^{k+1}, \Lambda^k, \rho_k) \leq \epsilon_k$ is not satisfied, the penalty parameter is increased to ensure that the upcoming subproblem prioritizes the reduction of constraint violations. In this scenario, the Lagrange multiplier estimates remain unchanged, as the primary focus is on enhancing feasibility.

\begin{algorithm2e}[h]\label{alg: practical-ALM}
	\caption{Practical ALM}
	\textbf{Input}: Initial point $\textbf{x}^0$, initial multipliers $\Lambda^0$, initial penalty factor $\rho_0$, convergence tolerances $\epsilon$ and $\eta$, constant $0<\alpha\leq\beta\leq1$ and $\kappa>1$. Let $\eta_0=\frac{1}{\rho_0}, \epsilon_0=\frac1{\rho_0^\alpha}$.\\
	\For{$k=0,1,2,\dots$}
	{

	Find an $\eta_k$-stationary point $\textbf{x}^{k+1}$ of the subproblem \eqref{ALM_scheme_app_x} such that 
    $$ \varsigma(\textbf{x}^{k+1},\Lambda^k,\rho_k)\leq\eta_k.\vspace{-0.4cm}
    $$ \label{algo:subprob-line}\\
    \uIf{$\vartheta(\textbf{x}^{k+1},\Lambda^k,\rho_k)\leq\epsilon_k$}
    {\uIf{$\vartheta(\textbf{x}^{k+1},\Lambda^k,\rho_k)\leq\epsilon$ and $\varsigma(\textbf{x}^{k+1},\Lambda^k,\rho_k)\leq\eta$}
         {\textbf{stop} with approximate solution $\textbf{x}^{k+1}$.}
    Update the dual variables
    $$
        \Lambda^{k+1}\in \Lambda^k+\rho_k \partial_\Lambda \mathbb{L}_{\rho_k}(\textbf{x}^{k+1},\Lambda^k),
    $$\\
     and tighten the tolerances $\rho_{k+1}=\rho_k, \eta_{k+1}=\frac{\eta_k}{\rho_{k+1}}, \epsilon_{k+1}=\frac{\epsilon_k}{\rho_{k+1}^\beta}$.}
    \uElse{Increase the penalty parameter and tighten the tolerances
    $\Lambda^{k+1}=\Lambda^k, \rho_{k+1}=\kappa\rho_k, \eta_{k+1}=\frac1{\rho_{k+1}}, \epsilon_{k+1}=\frac1{\rho_{k+1}^\alpha}$.}
    }
\end{algorithm2e}

\subsubsection{Solving the ALM subproblem efficiently}
As depicted in Line \ref{algo:subprob-line} of Algorithm \ref{alg: practical-ALM}, when solving the subproblem \eqref{ALM_scheme_app_x}, the various properties of the AL function under different problem settings necessitate the application of distinct algorithms. 
Next, we will provide a few possible subproblem solvers that work under different smoothness and convexity conditions of the problems.

For the convex composite optimization problem \eqref{sec2:pro:composite1}, we present two different forms of AL functions in Section \ref{sec:convex-comp}. The first form defined in \eqref{eqn:AL-h} retains the nonsmooth convex function $h$, resulting in a nonsmooth AL function overall. However, since the sets $\mathcal{Q}, \mathcal{K}$ are closed and convex, this implies that apart from $h$, the other parts of AL are smooth. The proximal gradient method is the most common algorithm for solving such subproblems. For the second form defined in \eqref{sec2:al:eliminate}, we smooth the nonsmooth function $h$ using the Moreau envelope, resulting in the AL function being already smooth. Hence, we can directly apply simple gradient-based methods such as gradient descent, Nesterov's accelerated gradient method, etc. Moreover, by carefully exploring the smoothness of this function, we find it to be semi-smooth, allowing us to even use second-order algorithms, such as the semi-smooth Newton method, greatly enhancing the efficiency of solving subproblems. This is a significant advantage introduced by the AL function form \eqref{sec2:al:eliminate}. Refer to Section \ref{sec:nonsmooth} for further discussion on the semi-smooth method \cite{wang2023decomposition, zhou2022semismooth, hermans2022qpalm}.

For nonconvex composite optimization problem \eqref{nonconvex-composite}, the AL function defined in \eqref{eqn:nonconvex-AL2} is not necessarily differentiable with respect to $\textbf{x}$. Therefore, directly applying first-order or second-order methods may be impractical. One feasible approach is the block coordinate descent (BCD) method. Specifically, we can initially consider the AL function before elimination, i.e., $\mathcal{L}_\rho(\textbf{x},\textbf{y},\Lambda)$, and solve the subproblem \eqref{ALM_scheme_app_x} by minimizing $\min_{\textbf{x},\textbf{y}}\mathcal{L}_\rho(\textbf{x},\textbf{y},\Lambda)$. Then, the BCD method minimizes $\textbf{x}$ and $\textbf{y}$ alternatively until a global solution of the subproblem is found. This approach treats the primal variable as two separate blocks and solves them individually. Its advantage lies in preserving the original differentiability, thus facilitating a more efficient solution structure when solving BCD subproblems by effectively leveraging the inherent structure \cite{de2023constrained}.

\section{Convex Optimization Problems}\label{sec:convex case}
In this section, we will focus on the theoretical aspect of ALM for convex optimization problems, including the global convergence, local convergence and the iteration complexity. To avoid confusion, we suppose the AL functions considered in this section are all constructed by eliminating all the auxiliary variables, namely, the collection of all primal variables $\textbf{x}$ defined in the previous section only consists of the original decision variable $x$. Therefore, we will write $x$ instead of $\textbf{x}$ throughout the following discussion. In particular, our primary focus in this section is on problem \eqref{sec2:pro:composite1} and its corresponding  AL function defined in \eqref{eqn:AL-Moreau}.

\subsection{Global convergence}
The global convergence of the ALM for convex optimization problems has been extensively studied and well-established in \cite{bertsekas1997nonlinear,bertsekas2014constrained}. Under mild conditions, the ALM guarantees to find an optimal solution and multiplier, regardless of initialization. For convex inequality-constrained 
 nonlinear programming problems, Rockafellar \cite{rockafellar1976augmented} demonstrated that the convergence properties of the ALM can be determined by its relationship with the dual proximal point algorithm (PPA).

Inspired by \cite{rockafellar1976augmented}, we extend the global convergence results of ALM to a broader class of convex composite optimization problem \eqref{sec2:pro:composite1}. Since obtaining an explicit solution for solving the AL subproblem is not straightforward, we often use iterative algorithms to approximate it. Therefore, we begin by introducing several inexact conditions considered by Rockafellar \cite{rockafellar1976augmented} for employing the ALM to solve \eqref{sec2:pro:composite1}: 
\begin{align}
& \mathbb{L}_{\rho_k} (x^{k+1},\Lambda^k)  - \inf_x \mathbb{L}_{\rho_k} (x,\Lambda^k) \leq \frac{\epsilon_k^2}{2\rho_k},\quad  \epsilon_k \geq 0, \ \sum_{k=1}^\infty \epsilon_k < +\infty,\label{sec:inexact:condition1}\\
&\mathbb{L}_{\rho_k}(x^{k+1},\Lambda^k) - \inf_x \mathbb{L}_{\rho_k}(x,\Lambda^k) \leq \frac{\delta_k^2}{2\rho_k}\|\Lambda^{k+1} - \Lambda^k\|_2^2, \quad  \delta_k \geq 0, \ \sum_{k=1}^\infty \delta_k < +\infty, \label{sec:inexact:condition2}\\
&\dist(0, \partial_x \mathbb{L}_{\rho_k}(x^{k+1},\Lambda^k)) \leq \frac{\delta_k^\prime}{\rho_k}\|\Lambda^{k+1} - \Lambda^k\|_2, \quad   0 \leq  \delta_k^\prime \rightarrow 0, \label{sec:inexact:condition3} 
\end{align}
where the AL function is defined in \eqref{eqn:AL-Moreau} with $\Lambda^k=(\nu^k,\lambda^k,\mu^k)$, and  $\epsilon_k,\delta_k$ and $\delta_k^\prime$ are pre-determined parameters.

Since $\inf_x \mathbb{L}_{\rho_k}(x,\Lambda^k)$ is unknown, it is infeasible to verify the inexact conditions \eqref{sec:inexact:condition1} and \eqref{sec:inexact:condition2}. However, if $\mathbb{L}_{\rho_k}(x,\Lambda^k)$ is $\alpha$-strongly convex, then \cite{kort1976combined} established that: 
$$\mathbb{L}_{\rho_k}(x,\Lambda^k) - \inf_x \mathbb{L}_{\rho_k}(x,\Lambda^k) \leq \frac{1}{2\alpha} \mathrm{dist}^2(0, \partial_x \mathbb{L}_{\rho_k}(x,\Lambda^k)).$$
Therefore, we can further construct the following  inexact conditions:
\begin{align}
&\dist(0, \partial_x \mathbb{L}_{\rho_k}(x^{k+1},\Lambda^k)) \leq \sqrt{\frac{\alpha}{\rho_k}}\epsilon_k,\quad  \epsilon_k \geq 0, \ \sum_{k=1}^\infty \epsilon_k < +\infty,\label{sec:inexact:condition4}\\
&\dist(0, \partial_x \mathbb{L}_{\rho_k}(x^{k+1},\Lambda^k)) \leq \sqrt{\frac{\alpha}{\rho_k}}\delta_k\|\Lambda^{k+1} - \Lambda^k\|_2^2, \quad  \delta_k \geq 0, \ \sum_{k=1}^\infty \delta_k < +\infty, \label{sec:inexact:condition5}\\
&\dist(0, \partial_x \mathbb{L}_{\rho_k}(x^{k+1},\Lambda^k)) \leq \frac{\delta_k^\prime}{\rho_k}\| \Lambda^{k+1} - \Lambda^k\|_2, \quad   0 \geq \delta_k^\prime \rightarrow 0. \label{sec:inexact:condition6} 
\end{align}
By setting the inner stopping criterion $ \varsigma(x^{k+1},\Lambda^k,\rho_k)\leq\eta_k$ to either \eqref{sec:inexact:condition4}, \eqref{sec:inexact:condition5}, or \eqref{sec:inexact:condition6} in Algorithm \ref{alg: practical-ALM}, we obtain the framework of the inexact ALM.

\begin{remark}
We briefly summarize the role of the three stopping criteria \eqref{sec:inexact:condition1}-\eqref{sec:inexact:condition3} in the convergence analysis.
\begin{itemize}
\item[(i)] Criterion \eqref{sec:inexact:condition1} with $\epsilon_k \rightarrow 0$ can be used to show that $\{x^k\}$ asymptotically converges to an optimal solution of the primal problem \eqref{sec2:pro:composite1} if the sequence $\{\Lambda^k\}$ is bounded and converges to an optimal dual solution dual problem \cite{rockafellar1973dual}. 
\item[(ii)] Under certain appropriate conditions, the satisfaction of criterion \eqref{sec:inexact:condition2} will imply a linear convergence rate for $\{\Lambda^k\}$ to $\Lambda^*$.
\item[(iii)] Criterion \eqref{sec:inexact:condition3} with $\sum_{k=1}^\infty \delta_k <\infty$ implies \eqref{sec:inexact:condition2}  if the objective function $f$ is strongly convex. Moreover, using the criterion \eqref{sec:inexact:condition3} can help establish the relationship between $\|x^k-x^*\|$ and $\|\Lambda^k-\Lambda^*\|$.
\end{itemize}
\end{remark}

We now present the global convergence result under these stopping criteria as follows. 

\begin{proposition}\label{sec4:global-convergence-convex}
Suppose that the Lagrange multipliers exist. Let $\{(x^k,\Lambda^k)\}$ be the sequence
generated by the ALM for \eqref{sec2:pro:composite1}  under criterion \eqref{sec:inexact:condition1}. Then, the whole sequence $\{\Lambda^k\}$ converges to some $\Lambda^*$, where $\Lambda^*$ is the optimal solution to the dual problem of \eqref{sec2:pro:composite1}, and the sequence $\{x^k\}$ satisfies for all $k\geq 0$,
\begin{eqnarray}
&& \left\|x^{k+1} -   \prox_{h/\rho_k}\left(x^{k+1} + \frac{\nu^k}{\rho_k}\right) \right\| \leq \frac{1}{\rho_k} \|\nu^{k+1} - \nu^k\| \rightarrow 0, \label{sec4:convex:globalc0}\\
   &&\left\|\mathcal{A}(x^{k+1})-\Pi_\mathcal{Q}\left(\mathcal{A}(x^{k+1})+ \frac{\lambda^k}{\rho_k}\right)\right\| \leq \frac{1}{\rho_k}\|\lambda^{k+1}- \lambda^k\| \rightarrow 0, \label{sec4:convex:globalc1}\\
&& \left\|x^{k+1}- \Pi_\mathcal{K}\left(x^{k+1}+ \frac{\mu^k}{\rho_k}\right)\right\| \leq \frac{1}{\rho_k}\|\mu^{k+1}- \mu^k\| \rightarrow 0,\label{sec4:convex:globalc2}\\
&&\psi(x^{k+1})-\psi^* \leq \frac{1}{2\rho_k}(\epsilon_k^2+\|\Lambda^k\|^2-\|\Lambda^{k+1}\|^2), \label{sec4:convex:globalc3}
\end{eqnarray}
where $\psi^*$ is the optimal value of  \eqref{sec2:pro:composite1}.
 Moreover, if \eqref{sec2:pro:composite1} admits a nonempty and bounded solution set, then the sequence $\{x^k\}$ is also bounded, and all of its accumulation points are optimal solutions to \eqref{sec2:pro:composite1}.
\end{proposition}

The proposition mentioned above draws heavily from \cite[Theorem 4]{rockafellar1976augmented}. We can observe that the expressions \eqref{sec4:convex:globalc0}-\eqref{sec4:convex:globalc3} are slightly different from those presented in \cite[Theorem 4 (4.13) \& (4.14)]{rockafellar1976augmented} since we consider the convex composite problem. 
\begin{remark}
(a) Rockafellar \cite{rockafellar1976augmented} demonstrated that the iterates can be guaranteed to converge if the approximate solutions of the subproblems satisfy at least one of the three inexact conditions.
(b) If we consider employing an alternative AL function in \eqref{eqn:AL-h}, where the nonsmooth term $h(x)$ is retained within the AL function, then the above results including the stopping criteria and proposition still hold.
\end{remark}

Finally, we review several pieces of literature on the global convergence of convex optimization problems. If the augmented Lagrangian subproblem is solved exactly, the convergence result for the general optimization problem \eqref{prob} with inequality constraints can be found in \cite[Proposition 8]{rockafellar1976monotone}. For problems involving only equality constraints, global convergence is outlined in Proposition 4.2.1 of \cite{bertsekas1997nonlinear}. Specifically, if exact minimization is performed at each step for the augmented Lagrangian subproblem, then any limit point of the sequence \(\{x^k\}\) will be a global minimum of the original problem \eqref{prob} with equality constraints. 
In cases where the objective function in \eqref{prob} is non-Lipschitz, the convergence of the Augmented Lagrangian Method (ALM) for \eqref{prob} was demonstrated by \cite{chen2017augmented} under a relaxed constant positive linear dependence condition and a suitable basic qualification (BQ) condition, which addresses the non-Lipschitz nature of the objective function. For linearly constrained convex optimization problems, the convergence of the sequence of iterates generated by the ALM, without additional assumptions such as strong convexity, was established in \cite{boct2023fast}. For the convex composite problem, Cui et al. \cite{cui2019r} demonstrated $R$-superlinear convergence of the KKT residuals generated by the ALM under a quadratic growth condition imposed on the dual problem. Additionally, they introduced practical and straightforward stopping criteria, replacing conditions \eqref{sec:inexact:condition1}-\eqref{sec:inexact:condition3} for the augmented Lagrangian subproblems.

\subsection{Local convergence}\label{sec:convex:local}
For problem \eqref{sec2:pro:composite1}, we give the local convergence rate of ALM.   Denote $\psi(\nu,\lambda,\mu): = \inf_{x,u,v,w} \ell(x,u,v,w;\nu,\lambda,\mu)$, where the Lagrangian function $\ell$ is defined in \eqref{sec2:pro:lf-new}. Recall that the dual problem of \eqref{sec2:pro:composite1} is defined as
\begin{equation}\label{sec4:pro:composite:dual0}
      \max_{\nu,\lambda,\mu} \ \psi(\nu,\lambda,\mu).
\end{equation}
We define maps $T_{\psi}: = \partial \psi$ and $T_{\ell}:= (\partial_x \ell,\partial_u \ell,\partial_v \ell,\partial_w \ell,-\partial_\nu \ell,-\partial_\lambda \ell,-\partial_\mu  \ell)$. If $\psi \neq -\infty$, then $\psi$ is proper and $T_{\psi}$ is a maximal monotone operator such that the solutions to $0\in T_{\psi}(\nu,\lambda,\mu)$ are the optimal solution to  \eqref{sec4:pro:composite:dual0}.

The next proposition shows that one can obtain the Q-linear convergence rate of the dual sequence generated by the ALM applied to \eqref{sec2:pro:composite1}. The proof follows from Theorem 5 in \cite{rockafellar1976augmented}. It should be noted that the local convergence is established in the case that $k$ is sufficiently large. The penalty parameter $\rho_k$ satisfies $0<\rho_k \uparrow \rho_\infty \leq 0$.

\begin{proposition} 
Suppose the dual problem is feasible and let the ALM in \eqref{composite-x-subprob} be executed with stopping criterion \eqref{sec:inexact:condition2} applied to $\phi_k$. If $T_{\psi}^{-1}$ is Lipschitz continuous at the origin with modulus $a_g$ and $\{\Lambda^k\}$ is bounded, where  with  $\Lambda^k = (\nu^k,\lambda^k,\mu^k)$, then $\Lambda^k \rightarrow \bar{\Lambda}$, where $\bar{\Lambda}$ is the unique optimal solution to \eqref{sec4:pro:composite:dual0}, and 
\begin{equation}
    \begin{aligned}
  \|\Lambda^{k+1}-\bar{\Lambda}\| \leq \theta_k \|\Lambda^k-\bar{\Lambda}\|
    \\
    \end{aligned}
\end{equation} for all $k$ sufficiently large,
where
$
\theta_k=\left[a_g\left(a_g^2+\rho_k^2\right)^{-1 / 2}+\delta_k\right]\left(1-\delta_k\right)^{-1},
$
and it holds that $\lim_{k\rightarrow \infty} \theta_k=a_g\left(a_g^2+\rho_{\infty}^2\right)^{-1 / 2}<1$. 
Moreover, the conclusions of Proposition \ref{sec4:global-convergence-convex} about $\left\{x^k\right\}$ are valid in \eqref{sec4:convex:globalc3} with  $\epsilon_k=\delta_k\|\Lambda^{k+1}-\Lambda^k\|.$

If in addition to \eqref{sec:inexact:condition2} and the condition on $T_{\psi}^{-1}$ one has \eqref{sec:inexact:condition3} and the stronger condition that $T_\ell^{-1}$  is Lipschitz continuous at the origin with modulus $a_\ell \left(\geqslant a_g\right)$, then $x^k \rightarrow \bar{x}$, where $\bar{x}$ is the unique optimal solution to \eqref{sec2:pro:composite1}, and one has $$\|x^{k+1}-\bar{x}\| \leqslant \theta_k^{\prime}  \|\Lambda^{k+1}-\Lambda^k\|.$$ 
for all $k$ sufficiently large, 
where $\theta_k^{\prime}=a_\ell\left(1+\delta_k^{\prime}\right) / \rho_k \rightarrow \theta_{\infty}^{\prime}=a_\ell / \rho_{\infty}$.

\end{proposition}

Finally, we review some related works for the local convergence of ALM for solving convex optimization problems. 
Apart from the condition proposed in \cite{rockafellar1976augmented},   more relaxed conditions were studied in \cite{luque1984asymptotic}.  In addition, an ergodic convergence rate result of the inexact ALM that uses constant penalty parameters or geometrically increasing penalty parameters was first established in \cite{xu2021iteration}.  
Aside from nonlinear programming, numerous studies have focused on convex composite programming. Liu et al. \cite{liu2019nonergodic} proposed an inexact ALM and analyzed its non-ergodic convergence rate for the linearly constrained composite convex optimization problem. Sabach and Teboulle \cite{sabach2022faster} developed a unified algorithmic framework for various Lagrangian-based methods and established a non-ergodic convergence rate of \(\mathcal{O}(1/k)\) as \(k \rightarrow \infty\) for both the feasibility measure and the objective function value. Notably, these convergence rates can be improved to \(\mathcal{O}(1/k^2)\) under strong convexity. Similar results can be found in \cite{lan2016iteration,liu2019nonergodic}. To enhance the convergence rate, Bo{\c{t}} et al. \cite{boct2023fast} achieved convergence rates of \(\mathcal{O}(1/k^2)\) for the primal-dual gap, feasibility measure, and objective function value, obtaining a sequence of iterates generated by a fast algorithm without requiring additional assumptions such as strong convexity. Li et al. \cite{li2018qsdpnal} established the convergence rate of a two-phase ALM under mild conditions for a convex quadratic semidefinite programming problem and demonstrated R-superlinear convergence of the KKT residuals. For a convex composite conic problem, R-superlinear convergence of the KKT residuals under only a mild quadratic growth condition on the dual problem was established in \cite{cui2022augmented,cui2019r}.

\subsection{Iteration complexity analysis}
To obtain a comprehensive understanding of the ALM's performance, we show its iterative complexity analysis. While convergence is concerned with how an algorithm reaches a solution, complexity analysis quantifies its performance by assessing the amount of computational resources it consumes, providing a more quantitative perspective on its efficiency and practicality.

When assessing the complexity of the ALM, the analysis typically entails a two-step approach. First, it involves defining the inexactness criteria for the subproblems and subsequently establishing the complexity of the outer loop, referred to as inexact ALM. 
Second, it necessitates the development of an efficient subproblem solver and the formulation of the solver's complexity in achieving an approximate solution that complies with the inexactness criteria. The amalgamation of these complexities, pertaining to both the outer and inner loops, yields a comprehensive assessment of the overall complexity.
In this section, we illustrate the complexity analysis approach via a composite convex programming problem with linear equality and nonlinear inequality constraints:
\begin{equation}\label{complexity:prob}
    \min_{x\in\mathcal{X}}\,\,f_0(x)= f(x)+h(x) \quad \st \ Ax=b,\ g(x)\leq 0.
\end{equation}
where $\mathcal{X}$ is a closed convex set, $h$ is a closed convex function, $f$ and $g$ are differentiable convex functions with Lipschitz continuous gradient.
Then the AL function is
\begin{equation}\label{AL:xyy}
\begin{split}
     \mathbb{L}_{\rho} (x,\nu,\lambda,\mu)
     = &~ f(x) - \frac{1}{2\rho}(\|\nu\|^2 + \|\lambda\|^2  + \|\mu\|^2)\\
     &+ \frac{\rho}{2} \left\|x + \frac \nu \rho-  \prox_{h/\rho}\left(x + \frac \nu \rho\right)  \right\|^2 + h\left(\prox_{h/\rho}\left(x + \frac{\nu}{\rho}\right)\right)\\
     &+ \frac{\rho}{2} \left\|\left(g(x) + \frac{\mu}{\rho}\right)_+\right\|^2 + \frac{\rho}{2} \left\|{A}x-b + \frac{\lambda}{\rho}\right\|^2.
\end{split}
\end{equation}
Next, let us provide the definition of the approximate solution.
\begin{definition}[primal $\epsilon$-solution]
    Let $f_0^\star$ be the optimal value of \eqref{complexity:prob}. Given any $\epsilon\geq0$, a point $x\in\mathcal{X}$ is called a primal $\epsilon$-solution to \eqref{complexity:prob} if
    \[
        |f_0(x)-f_0^\star|\leq\epsilon,\qquad \|Ax-b\|+\|(g(x))_+\|\leq\epsilon.
    \]
\end{definition}

Now we can utilize the inexact ALM to obtain a primal $\epsilon$-solution. Since the AL function \eqref{AL:xyy} is gradient Lipschitz continuous, the Nesterov's accelerated gradient method (NAG) \cite{nesterov2013introductory} is a natural choice for solving the subproblem. The remaining task is to determine how to set the precision of solving the subproblem. The following theorem presents the different complexities under different subproblem precisions.
\begin{theorem}\textnormal{(\cite[Theorem 1]{xu2021iteration}).\,\,}
For a given $\epsilon>0$,  choose a positive integer $K$ and constants $C_1, C_2 > 0$. In each iteration, we apply NAG to solve the AL subproblem such that the following condition holds $$\mathbb{L}_{\rho_k}(x^{k+1},\nu^{k},\lambda^k,\mu^k)\leq \min_{x\in\mathcal{X}}\mathbb{L}_{\rho_k}(x,\nu^{k},\lambda^k, \mu^k)+\epsilon_k.$$
Let ${(x^k, \nu^k, \lambda^k, \mu^k)}_{k=0}^K$ be the iterates with parameters set as one of the following:
\begin{enumerate}
    \item $\rho_k = \frac{C_1}{K\epsilon}, \epsilon_k=\frac\epsilon2\frac{C_2}{C_1}, \forall k$.
    \item $\rho_k=\rho_0\sigma^k, \forall k$ for certain $\rho_0>0$ and $\sigma>1$ such that $\sum_{k=0}^{K-1}\rho_k=\frac{C_1}{\epsilon}$ and $\epsilon_k=\frac\epsilon2\frac{C_2}{C_1}, \forall k$.
    \item $\rho_k=\rho_0\sigma^k, \forall k$ for certain $\rho_0>0$ and $\sigma>1$ such that $\sum_{k=0}^{K-1}\rho_k=\frac{C_1}{\epsilon}$. If $f_0$ is convex, let $\epsilon_k = \frac{C_2}{2\rho_k^{\frac13}}\frac{1}{\sum_{t=0}^{K-1}\rho_t^{\frac23}}, \forall k$, and if $f_0$ is strongly convex, let $\epsilon_k=\frac{C_2}{2\rho_k^{\frac12}}\frac1{\sum_{t=0}^{K-1}\rho_t^{\frac12}}, \forall k$.
\end{enumerate}
Then we have the following results:
\begin{enumerate}
    \item For each setting, the average iterate $\bar{{x}}^K := \sum_{k=0}^{K-1} \frac{\rho_k{x}^{k+1}}{\sum_{t=0}^{K-1}\rho_t}$ is a primal 
 $\epsilon$-solution, %
    where the hidden constant relies on $C_1, C_2$ and dual solution $(\nu^\star, \lambda^\star, \mu^\star)$.
    \item For the second and third settings, the actual point $x^K$ is also a primal $\epsilon$-solution.
    \item For each setting, the total number of evaluations on $\nabla f$ and $\nabla g$ is $O(\sqrt{K}\epsilon^{-1}+K\epsilon^{-\frac12})$ if $f_0$ is convex and $O(K+\sqrt{K}\epsilon^{-\frac12}|\log\epsilon|)$ if $f_0$ is strongly convex.
\end{enumerate}
\end{theorem}
For more details, please refer to the ergodic and nonergodic convergence rates established in \cite{xu2021iteration}.

\subsubsection{Strongly convex case}

To illustrate how to use the ALM framework to design efficient algorithms under strong convexity, we take \cite{zhu2023optimal} as an example, where the following strongly convex problem with linear inequality constraint is considered:
\[
\min_x\ f(x) \quad \st \ Ax\leq b,
\]
where $f(x)$ is $\mu_f$-strongly convex and $L_f$-smooth. We assume that the matrix $A$ is full row rank and satisfies $\|A\|\leq L_A$ and the minimum singular value of $A$ is no smaller than $\mu_A$.
The AL function can be written as
\[
\mathbb{L}_\rho(x,\lambda) = f(x)+\frac\rho2\left\|\left[Ax-b+\frac\lambda\rho\right]_+\right\|^2-\frac{\|\lambda\|^2}{2\rho}.
\]
The ALM is
\begin{align}
\label{SC:ALM1} x_{k+1}&=\argmin_x \mathbb{L}_\rho(x,\lambda_k),\\
\label{SC:ALM2} \lambda_{k+1}&=\left[\lambda_{k}+\rho(Ax_{k+1}-b)\right]_+.
\end{align}
It can be derived that $\mathbb{L}_\rho(x,\lambda)$ is $\mu_f$-strongly convex and $\left(L_f+\rho L_A^2\right)$-smooth with respect to $x$. Hence when solving the subproblem \eqref{SC:ALM1}, it is efficient to apply NAG on it, which leads to the linear convergence rate. The algorithm framework is present in Algorithm \ref{algo:SC-APPA}.

\begin{algorithm2e}[H]\label{algo:SC-APPA}
    \caption{Inexact ALM for strongly convex problems}
    \textbf{Input:} initial point $x_0,\lambda_0$, smoothness parameter $L_f$, strong convexity parameter $\mu_f$, minimal singular value $\mu_A$, penalty parameter $\rho$, radius parameter $D$.\\
 \textbf{Initialize:} $\sigma=\frac{\mu_A^2\rho}{12L_f}, \delta_k=(1-\sigma)^{\frac k2}D$. \\
	\For{$k=1,\dots, T$}
	{
        Starting from the last iterate $x_k$, an approximate solution $x_{t+1}$ of \eqref{SC:ALM1} is computed by NAG, satisfying $\|x_k-x_k^*\|\leq \delta_k$ where $x_k^*$ is the exact solution.\\
        Update the dual variable by \eqref{SC:ALM2}.
	}
    \textbf{Output:} $x_T$.
\end{algorithm2e}

The design of Algorithm \ref{algo:SC-APPA} includes a precise error tolerance for subproblems, ensuring exponential error reduction. The complexity results of the algorithm are as follows.
\begin{theorem}\textnormal{(\cite[Theorem 1, Corollary 1]{zhu2023optimal}).\,\,}\label{thm:SC-upper}
Let $\rho\leq\frac{L_f}{\mu_A^2}$ and $\|x_0-x^*\|\leq D$, in each iteration of NAG, the number of inner iteration $T_k$ can be upper bound by
\begin{align*}
    T_k\leq 8\sqrt{\frac{L_f+\rho L_A^2}{\mu_f}}\log\left(\frac{10\kappa_f\kappa_A D_*}{D}\right),
\end{align*}
where $D_*:= \|x_0-x^*\|+\frac{L_A}{L_f}\|\lambda_0-\lambda^*\|$, $\kappa_A=\frac{\|A\|}{\mu_A}$, and $\kappa_f=\frac{L_f}{\mu_f}$. Furthermore, for Algorithm \ref{algo:SC-APPA} to find an approximate solution $x_T$ satisfying $\|x_T-x^*\|\leq \epsilon$, the number of outer iterations is
\begin{align*}
    T\leq \frac{12L_f}{\mu_A^2\rho}\log\left(100\kappa_f\kappa_A\frac{D}{\epsilon}\right).
\end{align*}
Suppose that $D_\star\leq D$. In order to find an approximate solution $x_T$ satisfying $\|x_T-x^*\|\leq\epsilon$, the total number of gradient evaluations for Algorithm \ref{algo:SC-APPA} is bounded by
    \[
    \tilde O\left(\kappa_A\sqrt{\kappa_f}\log(D/\epsilon)\right).
    \]
\end{theorem}
The theorem demonstrates that both the outer and inner loop of ALM converge in the linear rate, making the entire algorithm converge linearly.

\subsubsection{Convex case}
For any given $x_0$ and $\epsilon>0$, we construct the following auxiliary problem:
\begin{equation}\label{eqn:C-obj}
\min_x\ f(x)+h(Ax-b)+\frac{\epsilon}{2D^2}\|x-x_0\|^2.
\end{equation}
The smooth part $f(x)+\frac{\epsilon}{2D^2}\|x-x_0\|^2$ is strongly convex and hence we can apply Algorithm \ref{algo:SC-APPA} to solve the problem. The following corollary illustrates that the approximate solution of \eqref{eqn:C-obj} is also an approximate solution of the original convex problem and the overall complexity is also optimal.

\begin{corollary}\textnormal{(\cite[Corollary 3]{zhu2023optimal}).\,\,}\label{coro:C-upper}
    Suppose that $f(x)$ is convex and $\|x_0-x^*\|\leq D$. For any given $0<\sigma<+\infty$, Algorithm \ref{algo:SC-APPA} can be applied on problem \eqref{eqn:C-obj} and output an approximate solution $x_T$ satisfying $|f(x_T)-f(x^*)|\leq\epsilon$ and $\|[Ax_T-b]_+\|\leq\frac\epsilon\sigma$, such that the total number of gradient evaluations is bounded by
    \[
    \tilde O\left(\frac{\kappa_A\sqrt{L_f}D}{\sqrt{\epsilon}}\right),
    \]
    where $\tilde O$ also hides the logarithmic dependence on $\sigma$.
\end{corollary}

There has been a lot of research on the complexity of ALM, and we summarize some of them in Table \ref{convex:comparison:complexity}.
The basis pursuit problem is considered in \cite{aybat2012first}. The derived ALM involves a sequence of $\ell_1$-regularized least squares subproblems which are solved using an accelerated proximal gradient algorithm. It is proved that the outer loop converges linearly and the entire algorithm converges with the complexity of $O(\epsilon^{-1})$.
Lan et al. \cite{lan2016iteration} expand the problem to scenarios where the objective function is gradient Lipschitz continuous and the domain is a compact convex set. The inexact condition for subproblems is defined by the gap of the AL function values, expressed as:
\begin{equation}\label{eqn:inexact-cond1}
\mathbb{L}_{\rho_k}(x^{k+1},\mu^k) - \inf_x \mathbb{L}_{\rho_k}(x,\mu^k) \leq \eta_k.
\end{equation}
Here, the tolerance $\eta_k$ is a constant that depends on the scale of the optimal Lagrange multipliers which is approximated by the ``guess-and-check'' procedure. The algorithm consists of two stages: a primary stage and a post-processing stage where the post-processing stage includes a single inexact augmented Lagrangian step with stricter termination conditions. This two-stage design effectively reduces computational complexity.
Patrascu et al. \cite{patrascu2017adaptive} relax the tightness of the domain by setting the penalty factor $\rho$ sufficiently large, thereby controlling the number of outer iterations to $O(1)$, while the complexity of the inner iterations is $O(\epsilon^{-1})$.
Liu et al. \cite{liu2019nonergodic} consider the composite objective function and design a termination condition that is weaker and (potentially) easier to check than \eqref{eqn:inexact-cond1}, that is,
\[
\max _{x \in \mathbb{R}^n}\left\{\left\langle\nabla \hat{f}_\rho\left(x^{k+1},\lambda^k\right), x^{k+1}-x\right\rangle+h\left(x^{k+1}\right)-h(x)\right\} \leq \eta_k,
\]
where $\hat{f}_\rho\left(x,\lambda\right)=f(x)+\langle \lambda, Ax-b\rangle+\frac\rho2\|Ax-b\|^2$ is the smooth part of the AL function. 
Xu et al. \cite{xu2017accelerated} consider the same problem with \cite{liu2019nonergodic} and the proposed algorithm allows linearization to the differentiable function as well as the augmented term. Hence the subproblem can be solved exactly since it has a closed-form solution. The authors also combine the acceleration technique to obtain the complexity of $O(1/\epsilon)$.

\begin{table}[ht]
\centering
\footnotesize
\caption{Comparison of the complexity results of several methods to produce an $\epsilon$- solution. $\mathbb{L}_\rho$ gives the formulations of the AL functions and we omit the terms that only depend on the dual variable for the sake of brevity. For the subproblem solvers, Nesterov's accelerated gradient method is shortened as NAG, Nesterov's accelerated proximal gradient method is shortened as NAPG, the fast iterative shrinkage-thresholding algorithm is shortened as FISTA. The complexities are based on the number of gradient evaluations.}
  \setlength{\tabcolsep}{0.8mm}{
  \begin{tabularx}{\textwidth}{ >{\centering\arraybackslash}m{3cm}| >{\centering\arraybackslash}X| > {\centering\arraybackslash}X| >{\centering\arraybackslash}X| >{\centering\arraybackslash}X| >{\centering\arraybackslash}X}
\hline
Problem & $\mathbb{L}_\rho$ & Bounded domain & Subproblem solver & Complexity of outer loop & Overall complexity \\\hline
$\begin{array}{ll}
\min & \|x\|_1,\\
\st\ & Ax=b
\end{array}$ & $\|x\|_1+e_\rho\delta_{\{b\}}(Ax+\frac y\rho)$ &
 &  NAPG  & ${O}(\log(\epsilon^{-1}))$ &${O}(\epsilon^{-1})$ \cite{aybat2012first} \\\hline
 
$\begin{array}{ll}
&\min\limits_{x \in \X}\ f(x), \\
&\st\ \ Ax=0
\end{array}$ & $f(x)+e_\rho\delta_{\{0\}}(Ax+\frac y\rho)$ & \checkmark
 & NAG  & ${O}(\log(\epsilon^{-1}))$ & ${O}(\epsilon^{-\frac{7}{4}})$ \cite{lan2016iteration} \\\hline
 
$\begin{array}{ll}
\min & f(x), \\
\st\ & Ax=b
\end{array}$ & $f(x)+e_\rho\delta_{\{b\}}(Ax+\frac y\rho)$ & 
 & NAG  & $O(1)$ & ${O}(\epsilon^{-1})$ \cite{patrascu2017adaptive} \\\hline

$\begin{array}{ll}
\min & f(x)+h(x), \\
\st\ & Ax=b
\end{array}$ & $f(x)+h(x)+e_\rho\delta_{\{b\}}(Ax+\frac y\rho)$ & \checkmark
 & FISTA  & $O(\epsilon^{-\frac12})$ & ${O}(\epsilon^{-2})$ \cite{liu2019nonergodic}\\\hline

$\begin{array}{ll}
\min & f(x)+h(x), \\
\st\ & Ax=b
\end{array}$ & $f(x)+h(x)+e_\rho\delta_{\{b\}}(Ax+\frac y\rho)$ & 
& explicit & ${O}(\epsilon^{-1})$ & ${O}(\epsilon^{-1})$ \cite{xu2017accelerated} \\\hline

$\min f(x)+h(Ax-b)$ & $f(x)+h(y)+e_\rho\delta_{\{b\}}(Ax-y+\frac z\rho)$ &
& NAG  & ${O}\left(\log(\epsilon^{-1})\right)$ & ${O}(\epsilon^{-\frac12})$ \cite{zhu2023optimal} \\\hline
\end{tabularx}}
\label{convex:comparison:complexity}
\end{table}

\section{Nonconvex Optimization problems}\label{sec:nonconvex_case}
The application of the ALM to nonconvex optimization problems presents unique challenges and opportunities that {are} different from the convex {cases}. %
{In this section, we 
present a few representative global and local convergence results}, showing the theoretical guarantees that can be achieved despite the lack of convexity. We also explore recent advancements in iteration complexity analysis for certain classes of nonconvex problems. Throughout this section, we emphasize the theoretical foundations and practical considerations that make ALM a powerful tool for tackling nonconvex optimization challenges.

\subsection{Global convergence} The ALM has been widely studied and applied to solve nonconvex optimization problems due to its good practical performance. However,  establishing global convergence guarantees for ALM when applied to nonconvex problems is quite challenging compared to the convex setting. The global convergence for the nonconvex optimization typically refers to the convergence of the sequence of iterates generated by the algorithm to a stationary point of the problem, which is a point satisfying the first-order optimality conditions. For the nonconvex composite optimization problem, some convergence results of the ALM were derived in \cite{de2023constrained} recently. To analyze the convergence property of the ALM for the general nonconvex problem \eqref{nonconvex-composite}, we first give some assumptions for \eqref{nonconvex-composite}. 

\begin{assumption}\label{nonconvex_condition}
 Let $f$ and $c$ be continuously differentiable with locally Lipschitz continuous derivatives,  $h$ be proper, lower semicontinuous and prox-bounded, $\mathcal{Q}$ and $\mathcal{K}$ be nonempty and closed sets.  
\end{assumption}
 
When each AL subproblem \eqref{ALM_scheme_app_x} admits an approximate stationary point, any feasible accumulation point generated by the ALM is an approximate stationary point under subsequential attentive convergence. Before proceeding with the global convergence, we show the boundness of the sequence $\{\Lambda^k/\rho^k\}$. {The proof is referred to Appendix \ref{appen:sec5}.}

\begin{proposition} \label{prop:to0}
    If $\rho_k\to\infty$ as $k \to \infty$ in Algorithm \ref{alg: practical-ALM}, then we have $\frac{\Lambda^k}{\rho_k}\to 0$.
\end{proposition}

We then generalize the feasibility of limit points and global convergence property in \cite[Proposition 3.2, Theorem 3.3]{de2023constrained} to our Algorithm \ref{alg: practical-ALM} in the following proposition. {The proof is referred to Appendix \ref{appen:sec5}.}

\begin{theorem}\label{nonconvex_convergence}
Suppose Assumption \ref{nonconvex_condition} holds. A sequence $\{x_k\}$ of iterates is generated by Algorithm \ref{alg: practical-ALM} with $\epsilon\rightarrow 0$. Then each accumulation point $x^*$ of $\{x_k\}$ is feasible for \eqref{nonconvex-composite} if one of the following conditions holds:
\begin{itemize}
\item[(i).]  $\{\rho_k \}$ is bounded;
\item[(ii).] there exists some $C \in \R$ such that $\mathbb{L}_{\rho_k}(x^k;\Lambda^k) \leq C$ for all $k$.
\end{itemize}
Furthermore, we let $\{x^k\}_K$ be a subsequence such that $x^k \stackrel{q}{\rightarrow}_K x^*$, where the notation $\rightarrow_K$ denotes convergence along a subsequence indexed by 
$K$. Then, $x^*$ is an AM-stationary point for \eqref{nonconvex-composite}.
 \end{theorem} 

The aforementioned theorem shows the conditions for the feasibility of the accumulation point. In cases where these conditions are not met, we can demonstrate that the accumulation point obtained by the algorithm is an $M$-stationary point of an infeasibility measure. {The proof is referred to Appendix \ref{appen:sec5}.}

\begin{proposition} \label{prop:infeas}
Suppose Assumption \ref{nonconvex_condition} holds. A sequence $\{x_k\}$ of iterates is generated by Algorithm \ref{alg: practical-ALM} with $\epsilon\rightarrow 0$. Then each accumulation point $x^*$ of $\{x_k\}$ is an M-stationary point of the feasibility problem
\begin{equation} \label{feas}
	\min_x \ \|\hat{\Pi}_{\mathcal{Q}}(c(x))\|^2+\|\hat{\Pi}_{\mathcal{K}}(x)\|^2.
\end{equation}
\end{proposition}

 More recent work has focused on providing stronger guarantees by making additional assumptions on the problem structure or the algorithmic parameters. We now review global convergence analysis for some specific nonconvex optimization problems.  For the nonlinear programming \eqref{prob}, Bertsekas \cite{bertsekas2014constrained} established the convergence property of the ALM for both the primal sequence and the corresponding dual sequence. 
For the nonlinear programming \eqref{prob} with box constraints, Birgin and Mart{\'\i}nez \cite{birgin2014practical,birgin2020complexity} showed the global convergence of the ALM under very weak constraint qualifications. { Furthermore, if each AL subproblem is solved to an $\epsilon_k$‐global minimum, with $\epsilon_k \to \epsilon$, then it has been shown that the ALM converges to an $\epsilon$‐global minimizer of the original problem \cite{birgin2010global,birgin2014augmented,birgin2015optimality}. }

 For a nonconvex quadratic programming:
$$ \min \ \frac{1}{2}x^\top Q x + q^\top x, \quad \st\quad \ell \le Ax \le u,
$$
where $Q \in \R^{n\times n}$ is symmetric, $A \in \R^{m\times n}$, and  $\ell, u \in \R^m$, Hermans et al. \cite{hermans2022qpalm} proposed a proximal augmented Lagrangian method and showed its convergence to a stationary point. For a nonlinearly constrained nonconvex and nonsmooth problem:
$$  \min \ G(u, v) + H(v), \quad \st\quad \Phi(u) + Bv = 0,$$
where $G$ and $H$ are differentiable functions, possibly nonconvex, $\Phi(u)$ is a nonsmooth nonlinear mapping and $B$ is a full-column-rank matrix, the global convergence of a first-order primal-dual method based on the augmented Lagrangian is established in  \cite{zhu2023first} under an error bound condition. For the nonconvex constrained optimization problem: 
$$\min\ f(x),\quad \st\quad c(x) = 0, x \in X,$$ where  $X \subseteq \R^n$ is a compact set and functions $f$ and $c$ are continuous, Bagirov et al. \cite{bagirov2019sharp} developed a sharp AL-based method and established its global convergence.
For a nonconvex and nonsmooth composite optimization problem:
$$\min\ q(x):= f(x) + h(F(x)),$$ 
where $f$ is a continuously differentiable function, $F$ is a continuously differentiable mapping and $h$ is a proper and l.s.c. function,
Bolte et al. \cite{bolte2018nonconvex} considered global convergence guarantees under a broad scheme of augmented Lagrangian approach, relying on the nonsmooth Kurdyka-\L{}ojasiewicz (KL) property. When $F$ is a linear mapping, $h$ is a proper closed function and $f$ is twice
continuously differentiable with a bounded Hessian, Li and Pong \cite{li2015global} showed that if the penalty parameter is sufficiently increased at each iteration, then the entire sequence of iterates converges to a stationary point for problems satisfying the constraint qualification and some regularity conditions.

\subsection{Local convergence}
In the convex case, local convergence analysis is typically conducted using the proximal point algorithm applied to the dual form of the augmented Lagrangian problem, as shown in Section \ref{sec:convex:local}. This analysis ensures Q-linear convergence of the dual sequence and R-linear convergence of the primal sequence within the ALM. However, this approach is not applicable to nonconvex settings due to its reliance on duality, which is not available in nonconvex scenarios. Therefore, there exists
many different approaches to establish the local convergence of this method. 

Let us first consider the general NLP in the form \eqref{prob}. 
Recall that the problem is given as follows:
\begin{equation}\label{sec5:NLP-prob}
    \begin{split}
    \min_{x\in\mathbb{R}^n}\ & f(x),\\
    \st\ & c_i(x)=0, i\in\mathcal{E},\\
    &c_i(x)\leq0, i\in\mathcal{I}.
    \end{split}
\end{equation}
We denote by $\mathcal{M}(\bar{x})$ the set of Lagrange multipliers of problem \eqref{sec5:NLP-prob} associated with a given stationary point $\bar{x}$, i.e., $(\lambda,\mu) \in \mathcal{M}(\bar{x})$ if and only if $(\bar{x}, \lambda,\mu)$ satisfies the KKT system \eqref{kkt}. Let
$$
\mathcal{A}=\mathcal{A}(\bar{x})=\left\{i \in \mathcal{I} \cup \mathcal{E} : c_i(\bar{x})=0\right\}, \quad \mathcal{B}=\mathcal{B}(\bar{x})=\mathcal{I} \cup \mathcal{E}\backslash \mathcal{A}
$$
be the sets of indices of active and inactive constraints at $\bar{x}$, respectively. For each $(\lambda,\mu) \in \mathcal{M}(\bar{x})$, we introduce the following standard partition of the set $\mathcal{A}$ :
$$
\mathcal{A}_1(\lambda,\mu)=\mathcal{E} \cup\left\{i \in \mathcal{I} : \mu_i>0\right\}, \quad \mathcal{A}_0(\lambda,\mu)=\mathcal{A} \backslash \mathcal{A}_1(\lambda, \mu).
$$
The critical cone of problem \eqref{sec5:NLP-prob} at its stationary point $\bar{x}$ is given by
$$
\begin{aligned}
\mathcal{C}(\bar{x}) & =\left\{u \in \mathbb{R}^n : \, \begin{array}{l}
\left\langle \nabla f(\bar{x}), u\right\rangle=0,\left\langle \nabla c_i(\bar{x}), u\right\rangle=0 \text { for } i \in\mathcal{E} \\
\left\langle \nabla  c_i(\bar{x}), u\right\rangle \leq 0 \,\text { for }\, i \in \mathcal{I} \,\text { with }\, c_i(\bar{x})=0
\end{array}\right\} \\
& =\left\{u \in \mathbb{R}^n : \nabla c_{\mathcal{A}_1(\hat{\mu})}(\bar{x}) u=0, \nabla c_{\mathcal{A}_0(\hat{\mu})}(\bar{x}) u \leq 0\right\},
\end{aligned}
$$
where the second equality is independent of the choice of $(\hat{\lambda},\hat{\mu}) \in \mathcal{M}(\bar{x})$.
We say that the second-order sufficient optimality condition (SOSC) is satisfied at $(\bar{x}, \bar{\lambda}, \bar{\mu})$ with $(\bar{\lambda},\bar{\mu}) \in \mathcal{M}(\bar{x})$ if
\begin{equation}\label{sec5:sosc}
    \left\langle \nabla^2 \mathbb{L}_{\rho}(\bar{x}, \bar{\lambda}, \bar{\mu}) u, u\right\rangle>0 \quad \forall u \in \mathcal{C}(\bar{x}) \backslash\{0\}.
\end{equation}

The usual convergence rate statements for the ALM assume the linear independence constraint qualification, i.e., 
$$\left\{\nabla c_i(\bar{x}): i \in \mathcal{A}\right\}$$ are linearly independent, and hence the multiplier $\bar{\mu} \in \mathcal{M}(\bar{x})$ is unique. Other conditions include SOSC \eqref{sec5:sosc} and strict complementarity, i.e., $\bar{\mu}_i>0$ for all $i \in \mathcal{A}$.  
Bertsekas \cite{bertsekas2014constrained} established that the generated dual sequence converges Q-linearly and the corresponding primal sequence converges $R$-linearly under SOSC, LICQ and the strict complementarity for NLP.  Next, we give a specific example to show the local convergence analysis of ALM under the SOSC condition \eqref{sec5:sosc}. In particular, we focus on the result in \cite{fernandez2012local}. The ALM for solving \eqref{sec5:NLP-prob} is given as follows: 
\begin{equation}
\left\{
    \begin{aligned}
        x^{k+1} & \approx \arg\min_{x} \mathbb{L}_{\rho_k}(x,\lambda^k,\mu^k), \\
        \lambda_i^{k+1} & = \lambda_i^k + \rho_kc_i(x^{k+1}),~i\in \mathcal{E},\\
        \mu_i^{k+1} &= \max\{0, \mu_i^k + \rho_kc_i(x^{k+1})  \},~i \in \mathcal{I}.
    \end{aligned}
    \right.
\end{equation}
Before presenting the main result, we give the following stopping criteria for $x$-subproblem:
\begin{equation}\label{sec5:nonconvex:local:stop}
    \begin{aligned}
        \| \nabla_x \mathbb{L}_{\rho_k}(x^{k+1}, \lambda^k,\mu^k)  \| & \leq \epsilon_k,\\ 
        \left\|  \left[  \begin{array}{c}
      x^{k+1} - x^k \\
    \rho_k  c_{\mathcal{E}}(x^{k+1}) \\
    (\mu^k + \rho_k c_{\mathcal{I}}(x^{k+1}) )_+ - \mu^k
      \end{array}\right] \right\| & \leq \hat{\delta} \sigma(x^k,\lambda^k,\mu^k),
    \end{aligned}
\end{equation}
where $\hat{\delta}\in (0,+\infty)$ is a given constant and
\begin{equation}
    \sigma(x,\lambda,\mu) = \left\|  \left[ \begin{array}{c}
    \nabla f(x) + \nabla c_{\mathcal{E}}(x)^\top \lambda + \nabla c_{\mathcal{I}}(x)^\top \mu \\
         c_{\mathcal{E}}(x)\\
         (\mu +  c_{\mathcal{I}}(x) )_+ - \mu
    \end{array}  \right] \right\|.
\end{equation}
\begin{theorem}[\cite{fernandez2012local}, Theorem 3.4] Let $(\bar{x}, \bar{\lambda}, \bar{\mu})$ satisfy the SOSC condition \eqref{sec5:sosc} and $B$ be a closed unit ball. 
Then there exist $\bar{\epsilon}, \bar{\rho}>0$ such that if $\left(x^0, \lambda^0,\mu^0\right) \in((\bar{x},\bar{\lambda}, \bar{\mu})+\bar{\epsilon} B)$ and the sequence $\left\{\left(x^k, \lambda^k, \mu^k\right)\right\}$ is generated according to \eqref{sec5:nonconvex:local:stop} with $\rho_k \geq \bar{\rho}$ for all $k$ and $\epsilon_k=o\left(\sigma\left(x^k, \lambda^k, \mu^k\right)\right)$, the following assertions hold: 
\begin{itemize}
    \item The sequence $\left\{\left(x^k,\lambda^k, \mu^k\right)\right\}$ is well-defined and converges to $(\bar{x},\hat{\lambda}, \hat{\mu})$, with some $(\hat{\lambda},\hat{\mu}) \in \mathcal{M}(\bar{x})$. %
    \item For any $q \in(0,1)$ there exists $\bar{\rho}_q$ such that if $\rho_k \geq \bar{\rho}_q$ for all $k$, then the convergence rate of $\left\{\left(x^k, \lambda^k, \mu^k\right)\right\}$ to $(\bar{x}, \hat{\lambda},  \hat{\mu})$ is $Q$-linear with quotient $q$. 
    \item If $\rho_k \rightarrow+\infty$, the convergence rate is $Q$-superlinear.
\item If $\epsilon_k \leq t_1 \sigma\left(x^k, \lambda^k,  \mu^k\right)^s$ and $\rho_k \geq t_2 / \sigma\left(x^k, \lambda^k, \mu^k\right)^{s-1}$, where $t_1, t_2>0$ and $s>1$, then the convergence rate is Q-superlinear with quotient $s$. %
\end{itemize} 
\end{theorem}

 Conn et al. \cite{conn1991globally}, Contesse-Becker \cite{contesse1993extended}, and Ito and Kunisch \cite{ito1990augmented} derived linear convergence rate for the ALM of
general NLP without the strict complementarity condition. 
Under SOSC condition \eqref{sec5:sosc},  the Q-linear convergence of the primal-dual
sequence constructed by the ALM was presented in \cite{fernandez2012local}. Notably, this analysis does not require a constraint qualification or uniqueness of Lagrange multipliers. Instead of looking for duality in the augmented Lagrangian problem, they show that the primal-dual iterates of the ALM are indeed a solution to a particular perturbation of the KKT system of the original problem.  Recently Hang and Sarabi \cite{hang2021local} established
the local convergence for piecewise linear quadratic composite optimization problems under merely
SOSC, which relies on the validity of upper Lipschitz continuous of KKT solution mapping when SOSC is satisfied. 
 A new second-order variational property, called the semi-stability of second subderivatives was presented  in \cite{hang2023convergence}. Under the condition and a certain SOSC, they established Q-linear convergence of the primal-dual sequence for an inexact version of the ALM for composite programs.  In \cite{wang2024local},  the linear rate of ALM was provided for nonlinear semidefinite programming under SOSC and semi-isolated calmness of the KKT pair without requiring the multiplier to be unique.

Another possible approach to conduct local convergence analysis of the ALM for nonconvex optimization problems was recently developed by \cite{rockafellar2023convergence}, where he extended his
original idea in \cite{rockafellar1976augmented} of using duality and the proximal point algorithm to obtain convergence of the dual sequence in the ALM. The principal idea therein was to assume the strong variational
convexity of the AL function. Following this idea, \cite{wang2023strong} proved the equivalence between the strong variational sufficiency and the strong SOSC for nonlinear semidefinite programming, without requiring the uniqueness of multiplier or any other constraint qualifications.

\subsection{Iteration complexity analysis} Given the limited research on complexity analysis of the ALM algorithm for fully nonconvex composite problems \eqref{nonconvex-composite}, we instead review existing results on nonconvex problems under additional convexity assumptions.

To incorporate recent advances in this area, we focus on the following cone convex constrained nonconvex composite optimization problem from \cite{kong2023iteration2}: 
\begin{equation}\label{sec5:kong2023}
\min \ \phi(x):=f(x)+h(x), \quad \st \ g(x) \preceq_\mathcal{K} 0,
\end{equation}
where $\mathcal{K}$ is a closed convex cone, $g$ is a differentiable convex function with a Lipschitz continuous gradient; $h$ is a proper closed convex function with compact domain; $f$ is a nonconvex differentiable function on the domain of $h$ with a Lipschitz continuous gradient. This model covers many previously studied models, making it a comprehensive framework for our discussion. 

Before proceeding with the summary of recent complexity analysis, we provide some definitions and assumptions relevant to this section. For a given tolerance pair $(\eta_1, \eta_2)$, if there exist $(x, \lambda, w, q)$ satisfying
\begin{equation}\label{sec5:complexity:sta}
    \begin{aligned}
 & w \in \nabla f(x)+\partial h(x)+\nabla g(x)\lambda,\quad \langle g(x)+q,\lambda\rangle=0,\\
 &g(x)+q\succeq_\mathcal{K} 0,\quad p\preceq_\mathcal{K^*} 0,\quad
  \|w\| \leq \eta_1,\quad \|q\| \leq \eta_2, 
\end{aligned}
\end{equation}
then $x$ is called an approximate stationary point of \eqref{sec5:kong2023}.
Note
 that a necessary condition for a point $x$ to be a local minimum of \eqref{sec5:kong2023} is that there exists a multiplier $\lambda \in \mathcal{K}^*$ such that $(x, \lambda, w, q) =(x, \lambda, 0, 0) $ satisfies \eqref{sec5:complexity:sta} under a certain regularity condition provided as follows.

\begin{assumption}\label{sec4:inexact:rc} 
To
ensure near feasibility of an approximate stationary point to the
AL function, we show the following conditions \cite{kong2023iteration2}.
\begin{itemize}
\item[(\uppercase\expandafter{\romannumeral 1})]\textbf{Boundedness of the Domain:}
 The quantity $\sup_{x\in \mathrm{dom} h} |\phi(x)|$ is finite, and/or $\mathrm{dom} h$ is bounded, and/or the feasible set is bounded.
\item[(\uppercase\expandafter{\romannumeral 2})]\textbf{Full Rank Jacobian:}  The constraint Jacobian is full rank on a certain bounded level set. For example, if the constraint is in the form of $Ax = b$, then $A$ has full row rank.
\item[(\uppercase\expandafter{\romannumeral 3})]\textbf{Regularity Condition:} There exists some $\nu >0$ such that $$\nu \|g(x^k)\| \leq \mathrm{dist}(0, \nabla g(x^k) g(x^k)+\rho_k \partial h(x^k))$$ for generated sequences $\{x^k\}$ and $\{\rho_k\}$. 
\item[(\uppercase\expandafter{\romannumeral 4})]\textbf{Lipschitz Continuity:} The function $h$ restricted to its domain is $r$-Lipschitz continuous. 
\item[(\uppercase\expandafter{\romannumeral 5})] \textbf{Slater's Condition:} If $g(x) \preceq_{\mathcal{K}} 0$ can be divided into $g_e(x) = 0$ and $g_l(x) \preceq_{\mathcal{J}} 0$ for some closed convex cone $\mathcal{J}$, then there exists $\bar{x} \in \mathrm{int}(\mathrm{dom} h)$ such that $g_e(\bar{x}) = 0$ and $g_l(\bar{x}) \prec_{\mathcal{J}} 0$. 
\end{itemize}
\end{assumption}
\begin{remark}
    Note that without any conditions in Assumption \ref{sec4:inexact:rc} on the
nonconvex constraints, one cannot even achieve feasibility. We focus on explaining the third regularity condition, as the other conditions are more commonly known.
It remains unclear whether Assumption  \ref{sec4:inexact:rc} (\uppercase\expandafter{\romannumeral 3}) is stronger or weaker than other common regularity conditions, such as Slater’s condition and the MFCQ condition \cite{li2021rate}.
However, this condition is particularly useful because it simplifies to more straightforward forms in specific scenarios, such as when $h(x)=0$
, making it closely related to well-known regularity conditions like the Polyak-{\L}ojasiewicz condition \cite{karimi2016linear}. It implies a faster-than-quadratic growth of the gradient as we move away from the optimal solution, thus ensuring quicker convergence. Furthermore, when 
$h$ is an indicator function for a convex set, condition (\uppercase\expandafter{\romannumeral 3}) reflects the basic constraint qualification in \cite{rockafellar1993lagrange}, which is a generalization of the Mangasararian-Fromovitz condition. This condition can be seen as a local regularity requirement, applicable near the constraint set, providing more flexibility compared to global conditions. For further details, see \cite{sahin2019inexact}.
\end{remark}

To review the recent work on the complexity results, we summarize them in Table \ref{table:comparison:complexity}. We can see that several algorithmic strategies have been developed to improve the iteration complexity of ALM for nonconvex problems:
\begin{itemize}
    \item[(i)] Adaptive Penalty Update: Introducing adaptive penalty parameters that adjust according to the progress of the optimization can help in achieving better convergence rates. This involves dynamically updating the penalty parameter $\rho_k$
  based on the violation of the constraints and the reduction in the objective function.
  \item[(ii)] Proximal Terms: Adding proximal terms to the ALM can help in stabilizing the iterates and improving the convergence rates. This approach modifies the AL function to include proximal regularization, which helps in handling nonconvexity and ensuring that the iterates stay within a reasonable region.
  \item[(iii)] Second-Order Methods: Utilizing second-order information, such as Hessians or their approximations, can significantly enhance the convergence of ALM. Methods like the Newton-conjugate-gradient algorithm, when applied to solve the AL subproblems, can lead to faster convergence rates for certain classes of problems. 
\end{itemize}

\begin{table}[h]
\centering
\footnotesize
\caption{Comparison of the complexity results of several methods to produce a stationary point. $\mathbb{L}_\rho$ gives the formulations of the AL functions. For the subproblem solvers, accelerated composite gradient method is shortened as ACG, accelerated proximal gradient method is shortened as APG, inexact proximal point method is shortened as iPPM, Newton-conjugate-gradient algorithm is shortened as NCG. The complexities are based on the number of gradient evaluations. A(\uppercase\expandafter{\romannumeral1},\uppercase\expandafter{\romannumeral 2}), A(\uppercase\expandafter{\romannumeral1},\uppercase\expandafter{\romannumeral 3}), etc. refer to specific conditions in Assumption \ref{sec4:inexact:rc}, where A(\uppercase\expandafter{\romannumeral1}) is Assumption \ref{sec4:inexact:rc} (\uppercase\expandafter{\romannumeral1}), A(\uppercase\expandafter{\romannumeral2}) is Assumption \ref{sec4:inexact:rc} (\uppercase\expandafter{\romannumeral2}), and so forth. Let  $\epsilon = \min\{\eta_1, \eta_2\}$. The complexity column lists the iteration complexity under different conditions. An asterisk $(*)$ indicates that the listed complexity corresponds to specific conditions as described in the table.
}
  \setlength{\tabcolsep}{0.8mm}{
  \begin{tabular}
  { >{\centering\arraybackslash}m{3.0cm}| >{\centering\arraybackslash}m{2.5cm}| > {\centering\arraybackslash}m{3.0cm}| >{\centering\arraybackslash}m{2.0cm}| >{\centering\arraybackslash\hsize=.6\hsize\linewidth=\hsize}m{2.1cm}| >{\centering\arraybackslash\hsize=.8\hsize\linewidth=\hsize}m{2.5cm}}
\hline
Problem & Property  & $\mathbb{L}_\rho$ &\makecell[c]{Subproblem\\ solver }  & Regularity & Complexity \\
\hline
 $\begin{array}{ll}
\min  f(x)\\
\st c(x)=0
\end{array}$   & \makecell[c]{$f$ is nonconvex\\ $c$ is nonlinear} &$\begin{array}{ll}
f(x)
+\\e_\rho\delta_{\{0\}}(c(x)+\frac \lambda \rho)
\end{array}$ & NCG  &\makecell[c]{A(\uppercase\expandafter{\romannumeral1},\uppercase\expandafter{\romannumeral 2})} & \makecell[c]{$O(\epsilon^{-2})$ \\\cite{xie2021complexity}} \\\hline
  $\begin{array}{ll}
\min  f(x)+h(x)\\
\st  Ax=b
\end{array}$  & \makecell[c]{$f$ is nonconvex\\ $h$ is convex} &$\begin{array}{ll}
f(x)+h(x)+\\
e_\rho\delta_{\{b\}}(Ax+\frac \lambda\rho)
\end{array}$  &ACG  &\makecell[c]{A(\uppercase\expandafter{\romannumeral1},\uppercase\expandafter{\romannumeral4},\uppercase\expandafter{\romannumeral 5})} & \makecell[c]{$O(\epsilon^{-3})$ \\ \cite{kong2023iteration}} \\\hline
 $\begin{array}{ll}
\min  f(x)+h(x)\\
\st  Ax=b, x\!\in\! \mathcal{K}\!
\end{array}$  & \makecell[c]{$f$ is nonconvex\\ $h,\mathcal{K}$ is convex} 
 &$\begin{array}{ll}
f(x)+h(x)+\\
e_\rho\delta_{\{b\}}(Ax+\frac \lambda\rho)
\end{array}$  &PPM  & A(\uppercase\expandafter{\romannumeral1},\uppercase\expandafter{\romannumeral3})  & \makecell[c]{$O(\epsilon^{-3})$ \\ \cite{hajinezhad2019perturbed}} \\\hline
 $\begin{array}{ll}
\min  f(x)+h(x),\\
\st  Ax=b
\end{array}$  & \makecell[c]{$f$ is nonconvex\\ $h$ is convex} 
&$\begin{array}{ll}
f(x)+\theta\langle\lambda, Ax-b\rangle\\
+h(x)+\!\frac{\rho}{2}\|Ax-b\|^2\!
\end{array}$\!   &ACG   & A(\uppercase\expandafter{\romannumeral4},\uppercase\expandafter{\romannumeral5})  & \makecell[c]{$O(\epsilon^{-2.5})$ \\ \cite{melo2023proximal} }\\\hline
$\begin{array}{ll}
\min  f(x)+h(x),\\
\st \mathcal{A}(x)=0
\end{array}$& \makecell[c]{$f$ is nonconvex\\ $h$ is convex, \\$\mathcal{A}$ is nonlinear} & $\begin{array}{ll}
f(x)+h(x)+\\
e_\rho\delta_{\{0\}}(\mathcal{A}(x)+\frac \lambda\rho)
\end{array}$&\makecell[c]{APGN \\ Trust Region$^*$ }
& A(\uppercase\expandafter{\romannumeral1},\uppercase\expandafter{\romannumeral3}) & $\makecell[c]{O(\epsilon^{-3})\\O(\epsilon^{-5})^*\\ \cite{sahin2019inexact}}$ \\\hline
$\begin{array}{ll}
\min  f(x)+h(x),\\
\st c(x)=0
\end{array}$ & \makecell[c]{$f$ is nonconvex\\ $h$ is convex, $c$ is \\linear/ nonlinear$^*$} &$\begin{array}{ll}
f(x)+h(x)+\\
e_\rho\delta_{\{0\}}(c(x)+\frac \lambda\rho)
\end{array}$ & IPPM & A(\uppercase\expandafter{\romannumeral1},\uppercase\expandafter{\romannumeral3}) & \makecell[c]{${O}(\epsilon^{-\frac{5}{2}})$\\$O(\epsilon^{-3})^*$\\ \cite{li2021rate}} \\\hline
  $\begin{array}{ll}
\min  f(x)+h(x)\\
\st  g(x) \preceq_\mathcal{Q} 0
\end{array}$  & \makecell[c]{$f$ is nonconvex\\ $h,g,\mathcal{Q}$ is convex} &$\begin{array}{ll}
f(x)+h(x)+\\
e_\rho\delta_{\mathcal{Q}}(g(x)+\frac \lambda \rho)
\end{array}$  &ACG  &\makecell[c]{A(\uppercase\expandafter{\romannumeral1},\uppercase\expandafter{\romannumeral4},\uppercase\expandafter{\romannumeral 5})} & \makecell[c]{$O(\epsilon^{-3})$ \\ \cite{kong2023iteration2}} \\\hline
\end{tabular}}
\label{table:comparison:complexity}
\end{table}

\section{Other variants of ALM}\label{sec:variants_ALM}
Several important variations of the ALM have been developed to address specific computational challenges or improve performance.  This section explores some of the key variations of ALM that have emerged in recent years. We begin by examining the linearized ALM, which simplifies the process of generating  subproblem solutions. Next, we discuss the proximal ALM, which combines ideas from proximal point methods with ALM to enhance convergence properties.  The section also covers the popular alternating direction method of multipliers (ADMM), a variant particularly well-suited for problems with separable objective functions. Finally, we examine fully primal-dual methods that treat primal and dual updates more symmetrically. Each of these variants provides unique advantages for specific problem, expanding the versatility and applicability of the core ALM framework.

\subsection{Linearized ALM}\label{sec:linearized-ALM}
The classic ALM is a double-loop algorithm that includes an inner loop that minimizes the AL function with respect to variable $x$, which may require large computational costs and make the implementation of ALM challenging. This inner minimization can be computationally expensive, posing challenges for practical implementation. To address this issue, the linearized ALM was introduced in \cite{yang2013linearized}, which simplifies the $x$-subproblem by linearizing the objective function.  Taking the convex composite problem \eqref{sec2:pro:composite1} as an example, the linearization strategy replaces the subproblem \eqref{composite-x-subprob} with its linear approximation:
\begin{equation*}
x^{k+1}=\argmin_x\ \langle\nabla_x \mathbb{L}_\rho(x^k,\nu^k,\lambda^k,\mu^k), x\rangle+\frac{1}{2\eta_k}\|x-x^k\|^2,
\end{equation*}
where $\eta_k$ is the stepsize. This update corresponds to a single gradient descent step on the AL subproblem. As a result, the method reduces to a single-loop algorithm, which is both simpler in structure and easier to implement in practice.

Since the variant of the AL function defined in \eqref{eqn:AL-h} is generally nondifferentiable, a full linearization of the AL function is not feasible in the linearized ALM framework. To overcome this, one can apply the linearization only to the smooth part of the objective, while retaining the nonsmooth component.  Specifically, the $x$-update is given by
\begin{equation}\label{sec3:linearALM-h}
\begin{aligned}
x^{k+1}=\argmin_x\ &h(x) +\frac{1}{2\eta_k}\|x-x^k\|^2\\
&+\left\langle \nabla f(x^k) + \rho \mathcal{A}^*(\hat{\Pi}_{\mathcal{Q}}(\mathcal{A}(x^k) + \lambda^k/\rho) ) + \rho \hat{\Pi}_{\mathcal{K}}(x^k + \mu^k/\rho), x\right\rangle.
\end{aligned}
\end{equation}
 This corresponds to applying one step of the proximal gradient method to the AL subproblem. Importantly, when $h$ is proximal-friendly—that is, when the proximal operator of $h$ can be evaluated efficiently—the subproblem \eqref{sec3:linearALM-h} admits a closed-form update.

{Many recent works have focused on the study and development of linearized ALM. In the convex case, Xu \cite{xu2017accelerated} considered affinely constrained composite convex programs and proposes a method that performs a single proximal gradient update to the primal variable at each iteration. The algorithm achieves an iteration complexity of $\mathcal{O}(\epsilon^{-1})$
 for obtaining a primal $\epsilon$-optimal solution. This framework was further extended in \cite{xu2021first} to handle more general problems involving both affine and smooth nonlinear constraints. In the nonconvex setting, several works \cite{lu2022single,bourkhissi2023complexity,bourkhissi2025complexity} proposed linearized and perturbed variants of ALM for addressing nonsmooth and nonconvex composite optimization problems with nonlinear constraints. In addition, recent extensions to nonconvex and stochastic optimization were developed, including momentum-based \cite{shi2025momentum} and Bregman-based \cite{shi2025bregman} linearized ALMs.}

\subsection{Proximal ALM}\label{sec3:PPA:dual}

The proximal ALM is a combination of the PPA and the classical ALM, which improves the convergence result and the numerical performance. For problem \eqref{prob}, Rockafellar \cite{rockafellar1976augmented} studied an early version of proximal ALM for the case where $f$ and $c_i (i \in \mathcal{I})$ are convex,
and $c_i (i \in \mathcal{E})$ are linear functions. The proximal ALM updates the primal and dual variables by:
\begin{equation*}
\left\{
\begin{aligned}
        x^{k+1}  & \in  \argmin_x \left\{ 
\mathbb{L}_{\rho_k}(x,\lambda^k,\mu^k) + \frac{\rho_k}{2}\|x  - x^k\|_\mathcal{M}^2\right\}, \\
    \lambda_i^{k+1}&=\lambda_i^k+\rho_k c_i(x^{k+1}),~i\in\mathcal{E}, \nonumber\\
     \mu_i^{k+1}&=\max\left\{\mu_i^k+\rho_k c_i(x^{k+1}),0\right\},~i\in\mathcal{I},\nonumber
\end{aligned}
\right.
\end{equation*}
where $\mathcal{M}$ is a self-adjoint (not necessarily positive definite) linear operator and $\|x\|_\mathcal{M}^2:= \langle x, \mathcal{M} x\rangle$. 
The motivation behind proximal ALM stems from the fact that it is more practical to relax the subproblem when an accurate solution to the subproblem is computationally expensive, we can select a proximal parameter $\rho_k$ which results in a strongly convex objective function for the $x$-subproblem of proximal ALM, and many methods can be employed to solve the subproblem efficiently. Moreover, the proximal term ensures that the primal subproblem is decomposable over the variables for certain applications \cite{hajinezhad2019perturbed}.

With the increasing difficulty of the problem, the proximal ALM is also widely used in composite optimization problems. For the convex composite optimization problem, 
Li and Qu \cite{li2021inexact} proposed an inexact proximal augmented Lagrangian framework for unconstrained composite convex optimization problems. 
Ding et al. \cite{ding2023proximal} designed inexact proximal augmented Lagrangian-based decomposition methods for dual block-angular convex composite programming problems. 

For the nonconvex composite optimization problem, Dhingra et al. \cite{dhingra2018proximal} utilized the Moreau envelope of the regularization term to derive the proximal AL function.
Hermans et al. \cite{hermans2022qpalm} proposed a nonconvex quadratic programming (QP) solver based on the proximal ALM, which was introduced in a previous paper for convex QPs \cite{hermans2022qpalm}. 
Arnesh and Renato \cite{sujanani2023adaptive} presented an adaptive superfast proximal ALM
for solving linearly-constrained smooth nonconvex composite optimization problems. Kong et al. \cite{kong2023iteration2} utilized a proximal ALM for solving smooth nonconvex composite optimization problems with nonlinear $\mathcal{K}$-convex constraints, i.e., the constraints are convex with respect to the order given by a closed convex cone $\mathcal{K}$. Melo et al. \cite{melo2023proximal} proposed an inexact proximal accelerated ALM for solving linearly constrained smooth nonconvex composite optimization problems with subproblems using an accelerated composite gradient method.

\subsection{Accelerated ALM}
The accelerated ALM (AALM) is a combination of the accelerated method and the ALM. It can be primarily classified into two categories. The first category stems from a dual perspective and focuses on accelerating the dual variables. The second category involves utilizing the accelerated method to solve the subproblems of the ALM.

Let us first review the first category. In convex problems, the ALM can be viewed as a gradient ascent method on the dual problem. Based on this perspective, we can apply the accelerated gradient algorithm to the dual problem. Specifically, consider the following convex problem with linear constraints:
\begin{equation}\label{prob:aalm}
    \min_{x\in\mathcal{X}}\ f(x)+h(x)\quad \mathrm{s. t. } \quad Ax = b, 
\end{equation}
 where $f,h$ are convex functions. The AL function is given by
 \begin{equation*}
     \widetilde{\mathbb{L}}_{\rho}(x,\lambda):=  f(x) + h(x) + \lambda^\top (Ax-b) + \frac{\rho}{2}\|Ax - b\|^2.
 \end{equation*}
Therefore, the augmented Lagrangian dual problem is
$$
\max_\lambda \ \widetilde{\Phi}(\lambda),
$$
where $\widetilde{\Phi}(\lambda)=\min_x \widetilde{\mathbb{L}}_{\rho}(x,\lambda)$. Applying the accelerated gradient ascent yields the scheme
$$
\left\{
\begin{aligned}
y^{k+1}&=\lambda^k+\eta\nabla \widetilde{\Phi}(\lambda^k), \\
t_{k+1}&=\frac{1+\sqrt{1+4 t_k^2}}{2}, \\
\lambda^{k+1}&=y^{k+1}+\frac{t_k-1}{t_{k+1}}\left(y^{k+1}-y^k\right),
\end{aligned}
\right.
$$
where $\eta>0$ is a stepsize. Similar with the derivations of supergradient in Section \ref{sec:supergrad}, we can also evaluate $\nabla \widetilde{\Phi}(\lambda^k)$ by two steps. Then the update rule becomes
$$
\left\{\begin{array}{l}
x^{k+1}=\argmin _{x \in \mathcal{X}}\widetilde{\mathbb{L}}_\rho(x,\lambda^k), \\
y^{k+1}=\lambda^k+\eta\left(A x^{k+1}-b\right), \\
t_{k+1}=\frac{1+\sqrt{1+4 t_k^2}}{2}, \\
\lambda^{k+1}=y^{k+1}+\frac{t_k-1}{t_{k+1}}\left(y^{k+1}-y^k\right),
\end{array}\right.
$$
where $t_0 \in (0,1]$. Note that the scheme can also be viewed as applying the accelerated proximal point method to the Lagrangian dual problem \cite{guler1992new}. Besides the accelerated scheme mentioned above, there are also some other variants, where the difference lies only in the final step of the iteration, which is the update of $\lambda$. For example, the authors of  \cite{kang2013accelerated,he2010acceleration,ke2017accelerated} propose the following update rule
$$
\lambda^{k+1}=y^{k+1}+\frac{t_k-1}{t_{k+1}}\left(y^{k+1}-y^{k}\right)+\frac{t_k}{t_{k+1}}\left(y^{k+1}-\lambda^k\right).
$$
Another accelerated ALM \cite{nedelcu2014computational} is to use Nesterov's accelerated dual average method \cite{devolder2014first} for the augmented Lagrangian dual problem:
$$
\lambda^{k+1}=\left(1-\frac{1}{t_{k+1}}\right) y^{k+1}+\frac{1}{t_{k+1}}\left(y^0+\eta \sum_{j=0}^k t_j\left(A x^{j+1}-b\right)\right).
$$

The second category of the accelerated ALM focuses on the primal variables and employs an acceleration algorithm to solve the \(x\)-subproblem, resulting in enhanced computational complexity. Specifically, we consider problem \eqref{prob:aalm}, where we assume that \(h\) is a convex and Lipschitz continuous function with a compact domain, and \(f\) is a (potentially) nonconvex differentiable function defined on the domain of \(h\) with a Lipschitz continuous gradient. An inner accelerated inexact proximal augmented Lagrangian (IAIPAL) method was proposed in \cite{kong2023iteration} for solving linearly-constrained smooth nonconvex composite optimization problems, based on the classical AL function. In the IAIPAL method, each iteration involves the approximate solution of a proximal AL subproblem using an Accelerated Composite Gradient (ACG) method, followed by a classical Lagrange multiplier update. Under certain assumptions, it has been shown that IAIPAL produces an approximate stationary solution within \(\mathcal{O}(\epsilon^{-5/2})\) ACG iterations. Recently, an accelerated inexact dampened ALM was proposed and analyzed in \cite{kong2023accelerated}. Furthermore, an extension to smooth nonconvex composite optimization problems with nonlinear \(\mathcal{K}\)-convex constraints was also studied in \cite{kong2023iteration2}.

\subsection{Stochastic ALM}
{Many modern applications—ranging from large‑scale empirical risk minimization to data‑driven resource allocation—give rise to constrained optimization problems in which the objective or the constraints (or both) take the form of a finite sum or an expectation.  This motivates
us to design stochastic ALM to efficiently solve those problems. In particular, we first introduce the constrained stochastic optimization problem:
\begin{equation}\label{stochastic-prob}
    \begin{split}
    \min_{x\in\mathbb{R}^n}\ & \mathbb{E}_{\xi }[f(x,\xi)],\\
    \st\ & c_i(x)=0, i\in\mathcal{E},\\
    &c_i(x)\leq0, i\in\mathcal{I},
    \end{split}
\end{equation}
where $\xi$ is a random variable in the probability space $\mathcal{D}$. We often assume the stochastic function $f(x,\xi)$ is continuously differentiable.  To apply the stochastic ALM, one needs construct an AL function as follows:
\begin{equation}\label{stochastic-AL}
\mathbb{L}_\rho(x,\lambda,\mu) = \mathbb{E}_{\xi }[f(x,\xi)] + \sum_{i\in\mathcal{E}} \lambda_i c_i(x) + \frac{\rho}{2} \sum_{i\in\mathcal{E}} c_i^2(x) + \frac{\rho}{2} \sum_{i\in\mathcal{I}} \left(\left[\frac{\mu_i}{\rho} + c_i(x)\right]_+^2 - \frac{\mu_i^2}{\rho^2}\right).
\end{equation}
The only distinction between the AL function and that in \eqref{AL} is that the former employs the stochastic objective $f(x,\xi)$. Furthermore, one may draw multiple independent samples $\{\xi_i\}_{i=1}^m$ and replace $f(x,\xi)$ by the mini‑batch approximation 
$
\frac{1}{m}\sum_{i=1}^m f(x,\xi_i)\,
$. Therefore, the stochastic ALM gives the following update
\begin{equation*}
\left\{
\begin{aligned}
        x^{k+1}  &  =\argmin_x \left\{ 
\mathbb{L}_{\rho_k}(x,\lambda^k,\mu^k)\right\}, \\
    \lambda_i^{k+1}&=\lambda_i^k+\rho_k c_i(x^{k+1}),~i\in\mathcal{E}, \nonumber\\
     \mu_i^{k+1}&=\max\left\{\mu_i^k+\rho_k c_i(x^{k+1}),0\right\},~i\in\mathcal{I},\nonumber
\end{aligned}
\right.
\end{equation*}
 Since the primal subproblem involves an expectation over a random variable, stochastic gradient methods are employed to approximately solve it. In the convergence analysis of the ALM, it is typically required that each subproblem be solved to a certain level of accuracy. In the stochastic setting, the inexactness condition for subproblem solutions is often formulated as  
\begin{equation} \label{eq:stoch_inexact}
\mathbb{E}[\| \nabla_x \mathbb{L}_{\rho_k}(x^{k+1},\lambda^k,\mu^k) \|] \leq \epsilon_k,
\end{equation}
where the tolerance parameter $\epsilon_k$ vanishes asymptotically. However, this condition is primarily of theoretical interest, as it involves expectations that are generally intractable to evaluate or verify in practice. The central difficulty in designing stochastic ALM algorithms lies in defining and enforcing practical inexactness criteria for the stochastic solvers employed to handle the primal subproblems \cite{li2024stochastic}. One approach is to draw a mini-batch of i.i.d. samples $\{\xi_t\}_{t=1}^m$ and solve the following subproblem:
$$
x^{k+1}   =  \argmin_x \left\{ 
\tilde{\mathbb{L}}_{\rho_k}(x,\lambda^k,\mu^k)\right\},
$$
where $$\tilde{\mathbb{L}}_\rho(x,\lambda,\mu) = \frac{1}{m} \sum_{t=1}^m f(x,\xi_t) + \sum_{i\in\mathcal{E}} \lambda_i c_i(x) + \frac{\rho}{2} \sum_{i\in\mathcal{E}} c_i^2(x) + \frac{\rho}{2} \sum_{i\in\mathcal{I}} \left(\left[\frac{\mu_i}{\rho} + c_i(x)\right]_+^2 - \frac{\mu_i^2}{\rho^2}\right).$$ 
To further enhance computational efficiency, stochastic linearized augmented Lagrangian methods \cite{shi2025momentum,shi2025bregman} and primal-dual type algorithms \cite{xu2020primal,zhang2022solving,jin2022stochastic,singh2025stochastic} have been proposed. These methods fall into the category of single-loop algorithms, where the $x$-subproblem can be solved without invoking inner subroutines. Recently, stochastic augmented Lagrangian methods have also been applied to nonsmooth manifold optimization problems \cite{geiersbach2024stochastic,deng2024oracle} and bilevel optimization problems \cite{lu2022stochastic}. Finally, stochastic ALM has also been proposed to  training neural networks with constraints \cite{dener2020training,wang2025augmented,franke2024improving}. 

}

\subsection{Manifold 
 ALM}\label{sec7:manifold-optimization}
Manifold constrained problem pertains to a distinctive category of constrained optimization problems, where the constraints typically define a manifold:
\begin{equation}\label{prob:manifold}
    \min_{x\in \mathcal{M}} \ f(x),
\end{equation}
where $\mathcal{M}$ is denoted as a Riemannian manifold. 
Manifold optimization methods \cite{absil2009optimization,boumal2023introduction} represent an main approach for tackling this problem. 
{ Nevertheless, when an optimization problem involves both manifold constraints and additional general constraints, or when the objective admits a composite form with nonsmooth terms, pure Riemannian approaches may not suffice.  In this section, we describe the manifold ALM. Its central idea is to retain the manifold constraint explicitly while embedding all other constraints and the difficult portions of the objective into an augmented Lagrangian function. Consequently, each ALM subproblem reduces to a “simple” optimization problem constrained solely by the manifold, which can then be solved efficiently by gradient‐based Riemannian algorithms.}   In particular, we first consider the optimization problem on manifold with extra constraints, which has the form:  
  \begin{equation} \label{liupro}
\begin{array}{ll}
  \min\limits_{x\in \mathcal{M}}   & f(x)   \\
  \st 
          &   c_i(x)\le 0,\quad i\in \mathcal{I} =\{1, \cdots, n\},\\
          &  c_j(x) =0,\quad j\in \mathcal{E}=\{n+1, \cdots, n+m\}.
\end{array}
\end{equation}
 The   AL function associated with problem \eqref{liupro} is given by
$$\mathbb{L}_\rho(x, \lambda, \mu) =  f(x) +\frac{\rho}{2}\left(\sum_{j\in \mathcal{E}}\Big(c_j(x)+\frac{\lambda_j}{\rho}\Big)^2+\sum_{i\in \mathcal{I}} \max\left\{0,  \frac{\mu_i}{\rho}+c_i(x)\right\}^2\right).$$
The corresponding manifold ALM \cite{liu2019simple} is given as follows:
\begin{equation}
    \left\{
\begin{aligned}
    x^{k+1} & = \arg\min_{x\in \mathcal{M}} \mathbb{L}_{\rho_k}(x, \lambda^k, \mu^k), \\
   \mu_i^{k+1} &= \max\{0, \mu_i^{k} + \rho_k c_i(x^{k+1})\},~ i\in \mathcal{I} \\
  \lambda_i^{k+1} &= \lambda_i^{k} + \rho_{k} c_i(x^{k+1}),~i\in \mathcal{E}.
\end{aligned}
    \right.
\end{equation}
The $x$-subproblem is a smooth manifold optimization problem and is solved inexactly.  In particular, denoted $\varphi(x): = \mathbb{L}_{\rho}(x, \lambda, \mu)$, given $x_0\in \mathcal{M}$,  a Riemannian gradient descent method applied to the $x$-subproblem at the $t$-th inner iteration  is
\begin{equation}
    x_{t+1} = \mathcal{R}_{x_t}(-\alpha_t \text{grad} \varphi(x_t)), 
\end{equation}
where $\mathcal{R}_x$ is a retraction operator at $x\in \mathcal{M}$ and $\text{grad}\varphi(x)$ is the Riemannian gradient. For more details, please refer to \cite{liu2019simple}.

We can also consider the following nonsmooth problem on the Riemannian manifold:
\begin{equation*}
    \min_{x\in \mathcal{M}}\ f(x) + h(\mathcal{A}x),
\end{equation*}
where $\mathcal{M}$ is a Riemannian submanifold embedding in a Euclidean space $\mathbb{E}$, $f: \mathcal{M}\rightarrow \mathbb{R}$ is a smooth but possibly nonconvex function and $h:\mathbb{R}^{m} \rightarrow \mathbb{R}$ is convex   but nonsmooth in usual Euclidean space, $\mathcal{A}:\mathbb{E} \rightarrow \mathbb{R}^{m}$ is a linear operator. By retaining the manifold constraint $x\in\mathcal{M}$, they construct the following AL function:
\begin{align*}
    \mathbb{L}_{\rho}(x,\nu)= & f(x) + h(\mathrm{prox}_{h/\rho}(\mathcal{A}x - \frac\nu\rho)) + \frac{\rho}{2}\left\| \mathcal{A}x - \frac\nu\rho - \mathrm{prox}_{h/\rho}(\mathcal{A}x - \frac\nu\rho) \right\|^2 - \frac{1}{2\rho}\|\nu\|^2\\
    = & f(x) + e_{\rho}h(\mathcal{A}x - \frac\nu\rho) - \frac{1}{2\rho}\|\nu\|^2. 
\end{align*}
Then the $k$-th iteration of the ALM \cite{deng2022manifold,deng2024oracle} can be written as 
$$
\left\{
\begin{aligned}
    x^{k+1} &= \argmin_{x\in \mathcal{M}} \mathbb{L}_{\rho}(x,\nu^{k}), \\
    \nu^{k+1} &= \nu^k - \rho(x^{k+1} - \mathrm{prox}_{h/\rho}(\mathcal{A}x^{k+1} - \nu^k/\rho)).
\end{aligned}
\right.
$$
Note that the $x$-subproblem is a smooth problem on manifold, and can be solved via the Riemannian optimization method.  Zhou et al. \cite{zhou2022semismooth} considered a more general nonsmooth and nonconvex manifold optimization problem:
\begin{equation}\label{sec5:manifold-nonsmooth}
\min_{x}\ f(x) + \psi(h_1(x)),\; \text{s.t. } \; x\in \mathcal{M},\ h_2(x)\leq 0,
\end{equation}
where $\mathcal{M}$ is a Riemannian manifold, $f:\mathcal{M}\rightarrow \mathbb{R},h_1:\mathcal{M}\rightarrow \mathbb{R}^m,h_2:\mathcal{M}\rightarrow \mathbb{R}^q$ are continuously differentiable, and $\psi:\mathbb{R}^m \rightarrow \mathbb{R}$ is convex. The corresponding AL function is given as follows:
\begin{equation*}
    \mathbb{L}_{\rho}(x;\nu,\mu): = f(x) + e_{\rho}\psi\left(h_1(x) + \frac{\nu}{\rho}\right) + \frac{\rho}{2} \| \max\{0,h_2(x) + \frac{\mu}{\rho}\} \|^2.
\end{equation*}
 Then the ALM can be applied. Moreover, they proposed a globalized semismooth Newton method to solve the augmented Lagrangian subproblem on manifolds efficiently.

\subsection{Alternating direction method of multipliers}\label{sec6:admm}

Known as ADMM, the alternating direction method of multipliers was originally proposed by Glowinski and Marrocco \cite{glowinski1975approximation}, and Gabay and Mercier \cite{gabay1976dual}. Recently, due to its great success in solving the optimization problems arising from machine learning, statistics, artificial intelligence, ADMM gains more and more attention, and there have several survey papers from different points of view \cite{boyd2011distributed,glowinski2014alternating,eckstein2015understanding}. 
When an optimization problem exhibits two blocks or multi-block structure, ADMM is considered more advantageous than ALM for solving such problems. It is particularly well-suited for optimization problems with separable objectives and linear constraints, where the decision variables can be naturally partitioned into blocks. One main reason for the renaissance of ADMM is that when applying to the modern application models, the subproblems are easy to solve, and in fact, in many cases, they possess closed-form solutions.

The classical ADMM is designed for solving the linearly constrained convex optimization problems with two blocks of variables and objective functions:
\begin{equation}\label{sec3:serap}
    \min_{x_1,x_2}\ f(x_1) + h(x_2), \quad \mathrm{s.t. } \ A_1 x_1 + A_2 x_2 = b,
\end{equation}
where $f$ and $h$ are closed proper convex functions. Let $\lambda$ be the Lagrange multiplier, and let $\rho>0$ be a penalty parameter. The AL function of \eqref{sec3:serap} is
$$
\mathcal{L}_\rho\left(x_1, x_2, \lambda\right)=f\left(x_1\right)+h\left(x_2\right)+\left\langle\lambda, A_1 x_1+A_2 x_2-b\right\rangle+\frac{\rho}{2}\left\|A_1 x_1+A_2 x_2-b\right\|^2.
$$

In ALM, the first step of iteration, which involves optimizing both $x_1$ and $x_2$ simultaneously, can sometimes be difficult. It might be simpler to fix one variable and solve for the minimal value with respect to the other variable. Therefore, we can consider alternating the minimization with respect to $x_1$ and $x_2$, which is the basic idea behind the ADMM and the iterative scheme of ADMM is
 
 \vspace{1.5em
 }
 
$$
\left\{\begin{array}{l}
x_1^{k+1}=\argmin _{x_1} \mathcal{L}_\rho\left(x_1, x_2^k, \lambda^k\right), \\
x_2^{k+1}=\argmin _{x_2} \mathcal{L}_\rho\left(x_1^{k+1}, x_2, \lambda^k\right), \\
\lambda^{k+1}=\lambda^k+\tau \rho\left(A_1 x_1^{k+1}+A_2 x_2^{k+1}-b\right),
\end{array}\right. 
$$
where $\tau \in\left(0, \frac{1+\sqrt{5}}{2}\right)$. Ideally, the subproblems in the ADMM update admit closed-form solutions.  When the problem is nonconvex, the implementation of the ADMM remains valid under properly selected penalty parameters, see e.g. \cite{hong2016convergence,zhang2020proximal,chartrand2013nonconvex}.

For the problem formulations we discussed in Section \ref{sec2:AL:sd}, the ADMM can still be applicable. Take the nonconvex composite problem \eqref{nonconvex-composite} as an example. As mentioned in \eqref{prob:nc-slack}, problem \eqref{nonconvex-composite} can be rewritten as an equality constrained problem by introducing additional slack variables. For convenience, we write it again in the following:
\begin{align*}
    \min_{x,u,v,w} &~ f(x) + h(u),\\
    \mathrm{s.t.}&~ x = u, \ x = v, \ c(x) = w, \ v \in \mathcal{K}, \ w\in \mathcal{Q}.
\end{align*}
In this problem, the primal variables are partitioned into multiple blocks and the multi-block ADMM minimizes the AL function $\mathcal{L}_{\rho}(x,u,v,w,\nu,\lambda,\mu)$ of \eqref{eqn:NC-AL} in an alternating block-coordinate manner, i.e.,
\begin{equation*}
\left\{
 \begin{aligned}
     x^{k+1} & = \argmin_x \mathcal{L}_{\rho}(x,u^k,v^k,w^k,\nu^k,\lambda^k,\mu^{k}),\\
     u^{k+1} & = \argmin_u \mathcal{L}_{\rho}(x^{k+1},u,v^k,w^k,\nu^k,\lambda^k,\mu^{k}),\\
     v^{k+1} & = \argmin_v \mathcal{L}_{\rho}(x^{k+1},u^{k+1},v,w^k,\nu^k,\lambda^k,\mu^{k}),\\
     w^{k+1} & = \argmin_w \mathcal{L}_{\rho}(x^{k+1},u^{k+1},v^{k+1},w,\nu^k,\lambda^k,\mu^{k}),\\
             \nu^{k+1} &= \nu^k + \rho(x^{k+1} - u^{k+1}), \\ 
             \lambda^{k+1} & = \lambda^k + \rho(c(x^{k+1}) - w^{k+1}),\\
             \mu^{k+1} & = \mu^k + \rho(x^{k+1} - v^{k+1}).
 \end{aligned}
 \right.
\end{equation*}
We have to mention that the function $c(x)$ is often required to be linear in literature of ADMM. Note that the $u$-subproblem, $v$-subproblem and the $w$-subproblem have the closed-form solutions. The $x$-subproblem may have no closed-form solutions but it is differentiable with respect to $x$ and we can apply the gradient-type method to solve it iteratively. To compare, the AL function after elimination \eqref{eqn:NC-comp-AL} may be nondifferentiable and $x$-subproblem of ALM can be difficult to solve.

The ADMM replaces the exact primal minimization of ALM with one cycle of block coordinate minimization, which can be viewed as an approximate dual gradient ascent. However, the approximation in the primal minimization step makes the convergence of ADMM subtle. When the number of blocks equals two, the convergence of ADMM was established in the context of operator splitting \cite{eckstein1992douglas}. When the number of blocks is greater than two, a counterexample was constructed in \cite{chen2016direct}, indicating that the direct extension of ADMM may not necessarily converge for general multi-block problems. To ensure convergence, one must make further modification to the algorithm \cite{deng2017parallel,gao2019randomized}, or assuming certain extra conditions on the problem \cite{cai2017convergence,chen2013convergence,li2015convergent,lin2015global,lin2015sublinear}. 

Compared with ADMM, the ALM is usually more efficient (less CPU time in finding an approximate solution) and robust (performance that does not heavily depend on parameters such as the initial point, the step size). However, the ADMM method has its advantage for structured problems where the AL function admits closed-form solutions for each variable block while being hard to optimize jointly. It also leads to a decentralized algorithm in the primal subproblems, e.g., the objective function is the sum of some component functions.

\subsection{Fully primal-dual method}\label{sec:fully-primal-dual}

The primal-dual algorithm is a major category of popular optimization algorithms. In a broad sense, primal-dual algorithms refer to those that simultaneously update primal and dual variables, in this sense, ALM and many of its variants fall into this category. In this section, we discuss the primal-dual algorithm in a narrow sense. Specifically, ALM solves a subproblem for the primal variables in order to take a gradient step for the dual variables, creating a noticeable difference in the primal and dual updates. In contrast, the fully primal-dual algorithm refers to the updates of primal and dual variables are treated equally and symmetrically.

To rigorously introduce the method, we consider the convex composite problem \eqref{sec2:pro:composite1}. The primal saddle point problems, and they can be applied to problem \eqref{sec2:prop:strong-duality} directly.
The update rule of primal-dual hybrid gradient (PDHG) method \cite{zhu2008efficient} is
$$
\left\{
\begin{aligned}
x^{k+1}&=x^k-\tau \nabla_x\mathbb{L}_\rho(x^k,\Lambda^k),\\
\Lambda^{k+1}&=\Lambda^k+\sigma \nabla_\Lambda\mathbb{L}_\rho(x^{k+1},\Lambda^k),
\end{aligned}
\right.
$$
where $\tau,\sigma>0$ are the primal and dual stepsizes, respectively. 

Extensive research has been conducted on such methods for saddle point problems, and they can be applied to problem \eqref{sec2:prop:strong-duality} directly. For example, the Chambolle-Pock method conducts an additional dual extrapolation:
$$
\left\{
\begin{aligned}
		x^{k+1} &= x^k-\tau \nabla_x\mathbb{L}_\rho(x^k,\Lambda^k),\\
		\Lambda^{k+1} &= \Lambda^k+\sigma\left(2 \nabla_\Lambda\mathbb{L}_\rho(x^{k+1},\Lambda^k)- \nabla_\Lambda\mathbb{L}_\rho(x^k,\Lambda^k)\right).
\end{aligned}
\right.
$$

The above two schemes are Gauss-Seidel iteration (utilizing the latest $x$ when updating $\Lambda$). 
We can also consider the algorithms of Jacobian iteration (utilizing $x$ of the last iteration when updating $\Lambda$). For example, the gradient descent ascent (GDA) method:
$$
\left\{
\begin{aligned}
x^{k+1} &= x^k-\tau \nabla_x\mathbb{L}_\rho(x^k,\Lambda^k),\\
\Lambda^{k+1} &=\Lambda^k+\sigma \nabla_\Lambda\mathbb{L}_\rho(x^k,\Lambda^k),
\end{aligned}
\right.
$$
and the optimistic gradient descent ascent (OGDA) method:
$$
\left\{
\begin{aligned}
x^{k+1} &= x^k-\tau \left(2 \nabla_x\mathbb{L}_\rho(x^k,\Lambda^k)-\nabla_x\mathbb{L}_\rho(x^{k-1},\Lambda^{k-1})\right),\\
\Lambda^{k+1} &= \Lambda^k+\sigma\left(2 \nabla_\Lambda\mathbb{L}_\rho(x^k,\Lambda^k)-\nabla_\Lambda\mathbb{L}_\rho(x^{k-1},\Lambda^{k-1})\right).
\end{aligned}
\right.
$$
Recently, all the algorithms mentioned above, including linearized ALM, PDHG, CP, OGDA and GDA, are unified into one algorithm framework in \cite{zhu2022unified}. Specifically, the proposed unified primal-dual framework is
\begin{equation}\label{generalalgo}
\left\{
    \begin{aligned}
    x^{k+1} &= x^k-\tau \left((1+\alpha)\nabla_x\mathbb{L}_\rho(x^{k},\Lambda^{k}) -\alpha \nabla_x\mathbb{L}_\rho(x^{k-1},\Lambda^{k-1})\right),\\
    \Lambda^{k+1} &= \Lambda^k+\sigma\kappa\left((1+\beta) \nabla_\Lambda\mathbb{L}_\rho(x^k,\Lambda^k)-\beta \nabla_\Lambda\mathbb{L}_\rho(x^{k-1},\Lambda^{k-1})\right)\\
    &\quad+\sigma(1-\kappa)\left((1+\beta) \nabla_\Lambda\mathbb{L}_\rho(x^{k+1},\Lambda^k)-\beta \nabla_\Lambda\mathbb{L}_\rho(x^k,\Lambda^k)\right),
    \end{aligned}
    \right.
\end{equation}
where $\tau, \sigma>0$ are primal and dual step sizes, $\alpha\in[0,1]$ and $\beta$ are gradient extrapolation coefficients, $\kappa\in[0,1]$ is the ratio of Gauss-Seidel iteration versus Jacobian iteration. When these parameters are set appropriately, the algorithm framework \eqref{generalalgo} can restore the corresponding existing algorithms. Furthermore, it also covers various unknown algorithms and hence provides a potential method to design new algorithms. A unified analysis for a special problem setting is established by analyzing the constraint violation and function value gap, and the obtained ergodic convergence rate is $\mathcal{O}(1/N)$. Furthermore, the non-ergodic convergence of the algorithm framework is also established and it proves to converge linearly under suitable conditions.

\section{Examples and applications}\label{sec:examples} 
In this section, we illustrate the key concepts of constructing ALM that have been discussed, using simple examples and showcasing success stories from various fields of application. As demonstrated in Section \ref{sec2:AL:sd}, we have explored the construction of the corresponding AL function for a given optimization problem. Notably, the form of the AL function is not unique. This section emphasizes the principles of constructing a correct and efficient AL function and focuses on solving the subproblem within the ALM framework.

\subsection{Nonconvex constrained composite optimization}
In this subsection, we discuss the application of the ALM to a class of nonconvex constrained composite
programs as \eqref{nonconvex-composite}:
\begin{equation}\label{sec7:nonconvex-composite}
\min_x\ f(x)+h(x)\quad \text { s.t. }\ c(x) \in \mathcal{Q},
\end{equation}
where $f$ and $c$ are smooth functions, $h$ is proper and
lower semicontinuous, and $\mathcal{Q}$ is a nonempty closed set. In particular, we consider the situation that the constraints are nonconvex
and possibly complicated, but allow for a fast computation of projections onto this nonconvex set. Typical problem classes which satisfy this requirement are optimization
problems with disjunctive constraints. 
Differ from the safeguarded augmented Lagrangian scheme proposed in \cite{de2023constrained}, we use the Algorithm \ref{alg: practical-ALM} to solve the problem. 
 
Since the projection onto the nonconvex set $\mathcal{Q}$ is relatively easy to obtain, we adopt the variant of the  AL function as shown in \eqref{eqn:nonconvex-AL2} that retains the slack variable $w$ for the constraint $c(x) \in \mathcal{Q}$:
\begin{equation}\label{sec7:constrained_pro}
     \mathbb{L}_{\rho}(x,w;\lambda)  
   =  f(x)+h(x) - \frac{1}{2\rho} \|\lambda\|^2 
    + \frac{\rho}{2}\left\|c(x) - w + \frac \lambda\rho \right\|^2+\delta_\mathcal{Q}(w),
\end{equation}
where $\lambda$ is the corresponding multiplier. 
Then the $k$-th iteration of the ALM framework is 
$$
\left\{
\begin{aligned}
(x^{k+1},w^{k+1}) &\in \argmin_{x\in \R^n}\mathbb{L}_{\rho_k}(x,w;\lambda^k),\\
  \lambda^{k+1} &= \lambda^k + \rho_k(c(x^{k+1}) - w^{k+1}).
 \end{aligned}
 \right.
 $$
The reason for this definition is to ensure that minimizing the AL function $ \mathbb{L}_{\rho}(x,w;\lambda)$ is straightforward. Keeping  the slack variable $w$ in this formulation is due to the potential complexity of the nonconvex set $Q$. It may be difficult to design an efficient high order method for the subproblem with respect to the $x$-variable after removing $w$ particularly when  $Q$ has a complicated structure.

Inspired by \cite{jia2023augmented,de2023constrained}, we can use various proximal gradient-type methods for tackling the above subproblem.
It is beneficial to use the structure of the nonsmooth term $h$ since $h$ is separable with regard to $x$ and $w$, making it easy to calculate the corresponding proximal mapping. To be precise, the computation of the proximal mapping is straightforward due to the separability of $h$.  Since the ALM for solving problem \eqref{sec7:nonconvex-composite} consists of outer and inner iterations, we let the subscript $k$ denote the outer iteration number, and the subscript $t$ denote the inner iteration number. Specifically, given $\lambda^k$ and $\rho^k$, the $t$-th iteration of the inner solver in the $k$-th outer iteration can be written as:
    \begin{align}
    x^{k,t+1} &\!= \mathrm{prox}_{\gamma_t h}\left(x^{k,t}-\gamma_t \left(\nabla f(x^{k,t})+\rho\nabla c(x^{k,t})^\top\left[c(x^{k,t})-w^{k,t}+\frac{\lambda^k}{\rho_k}\right]\right)\right),\nonumber\\
    w^{k,t+1} &\!= \Pi_\mathcal{Q}(c(x^{k,t})+\lambda^k/\rho_k),\label{sec7:nonsmmoth:wup}
    \end{align} 
where $\gamma_t > 0$ is the stepsize. 
Notice that the above update rule simultaneously updates 
$w$ and $x$. Alternatively, a Gauss-Seidel approach can be employed, where $w$ and $x$ are updated alternately. In this case, we  replace \eqref{sec7:nonsmmoth:wup}  with
    \begin{align*}
    w^{k,t+1} = \Pi_\mathcal{Q}(c(x^{k,t+1})+\lambda^k/\rho_k).
    \end{align*} 
 Both methods can be used as update strategies for solving subproblems. Typically, we can find a stationary point of the subproblem. The convergence analysis of the proximal gradient methods is presented in \cite{bolte2014proximal}. Therefore, we can derive convergence results of the ALM for this type of problem.

The benefit of accelerated proximal-gradient methods for solving the subproblems is demonstrated by \cite{de2023constrained} on some nonconvex constrained composite examples such as sparse switching time optimization and matrix completion problems. 
Certainly, we can also employ the ADMM method discussed in Section \ref{sec6:admm} to alternately solve for $x$ and $w$. To be specific, the $k$-th iteration is
\begin{equation}
\left\{
    \begin{aligned}
    x^{k+1} &\in \argmin_{x\in \R^n}\mathbb{L}_{\rho_k}(x,w^k;\lambda^k), \label{sec7:prox:subproblem}\\
    w^{k+1} &= \Pi_\mathcal{Q}(c(x^{k+1})+\lambda^k/\rho_k),\nonumber\\
 \lambda^{k+1}&=\lambda^k+\rho_k(c(x^{k+1})-w^{k+1}).\nonumber
    \end{aligned}
    \right.
    \end{equation}
The $x$-subproblem  can also be solved by using the proximal gradient methods. Since problem \eqref{sec7:nonconvex-composite} is nonconvex, the theoretical convergence guarantee of the ADMM method is
still an open problem.

\subsection{Nonlinear Programming with box constraints} If the composite term $h(x)$ is not included in the nonconvex optimization, then the proximal gradient methods will not be suitable for solving the subproblem. To demonstrate how to solve the subproblem of the ALM in such cases, we take general nonlinear programming with box constraints as an example of nonconvex constrained optimization problems without the composite term:
\begin{equation}\label{primal}
	\min_{x\in \mathbb{R}^n}\ f(x) \quad \text{ s.t. } \quad  c(x)\in \mathcal{Q}, \ x\in\mathcal{K}, \tag{P}
\end{equation}
where $f:\mathbb{R}^n\to \mathbb{R}$, $c(x)=[c_1(x),c_2(x),\dots,c_m(x)]^{\top}:\mathbb{R}^n\to \mathbb{R}^m$ are twice continuous and differentiable, and at least one of them is nonlinear. The set of constraints is
\begin{equation*}
	\mathcal{Q}:=\{w \in \mathbb{R}^m\ | \ \ell^c\leq w\leq u^c\},\ \mathcal{K}:=\{v \in \mathbb{R}^n\ | \ \ell^x\leq v\leq u^x\},
\end{equation*}
where $\ell^c,u^c\in\mathbb{R}^m$ and $\ell^x,u^x\in\mathbb{R}^n$ are constant vectors. The corresponding AL function is 
\begin{align}\label{sec7:lxys}
	\mathbb{L}_\rho(x;\lambda,\mu)=f(x)+\frac{1}{2\rho}\left\|\hat{\Pi}_{\rho \mathcal{Q}}\left(\rho c(x)+\lambda\right)\right\|^2+\frac{1}{2\rho}\left\|\hat{\Pi}_{\rho \mathcal{K}}\left(\rho x+\mu\right)\right\|^2.
\end{align}
Then the subproblem of the $k$-iteration for the ALM is 
\begin{equation}\label{sec7:nonlinear:sub}
x^{k+1} \in \argmin_x \mathbb{L}_\rho(x;\lambda^k,\mu^k).
\end{equation}
Since the AL function defined in \eqref{sec7:lxys} is continuous and differentiable, we can derive its gradient:
\begin{equation}\label{sec7:lx_gradient}
	\nabla_x \mathbb{L}_\rho(x;\lambda,\mu)= \nabla f(x)
	+ \nabla c(x) \hat{\Pi}_{\rho \mathcal{Q}}\left(\rho c(x)+\lambda\right)+\hat{\Pi}_{\rho \mathcal{K}}\left(\rho x+\mu\right).
\end{equation}
 Therefore, the above subproblem can be solved by gradient-type methods.

To accelerate the convergence of the ALM, we proceed to introduce a second-order algorithm, known as the semi-smooth Newton method (SSN, see Appendix \ref{app:seminewton:sub}), to solve \eqref{sec7:nonlinear:sub}, which is to solve the following nonlinear equation:
\begin{equation}\label{sec7:grad-zeros-psi}
    F(x):= \nabla_x \mathbb{L}_\rho(x;\lambda,\mu) = 0.
\end{equation}
One needs to utilize the second-order information of the AL function. Let $
    \psi_{\rho}(x;\lambda)=\hat{\Pi}_{\rho \mathcal{Q}}\left(\rho c(x)+\lambda\right)$.
If $f$ and $c$ are twice differentiable, then the following operator is well defined:
$$	\partial F(x):=\nabla^2 f(x)
	+ \nabla^2 c(x) [\psi_{\rho}(x;\lambda)]+ \nabla c(x) \partial \psi_{\rho}(x;\lambda)\nabla c(x)^{\top}+\partial\hat{\Pi}_{\rho \mathcal{K}}\left(\rho x+\mu\right),$$
where $\partial \hat{\Pi}_{\rho \mathcal{K}}\left(\rho x+\mu\right)$ is the Clarke subdifferential of the Lipschitz continuous mapping $\hat{\Pi}_{\rho \mathcal{K}}(\cdot)$ at the point $(\rho x+\mu)$ (see the definition of Clarke subdifferential in \cite{clarke1975generalized}). 
Subsequently, we can write
 the generalized Hessian matrix of the AL function in \eqref{sec7:lxys}:
\begin{equation}\label{lx_hessian}
	W(x) :=\nabla^2 f(x)
	+ \nabla^2 c(x) [\psi_{\rho}(x;\lambda)]+\rho \nabla c(x)D_c \nabla c(x)^{\top}+\rho D_x,
\end{equation}
where $\rho D_c \in \partial\psi_{\rho}(x)$ and $\rho D_x \in \partial\hat{\Pi}_{\rho \mathcal{K}}\left(\rho x+\mu\right)$. By the definition of the projection, we can take
\begin{equation}\label{UV}
	D_c = \text{Diag}(d^c)\ \text{and} \ D_x =   \text{Diag}(d^x),
\end{equation}
where the diagonal elements $d_i^c~(i=1,2,\cdots,m)$ and $d^x_j~(j=1,2,\cdots,n)$ satisfy
\begin{equation*}
	d^c_i=\left\{\begin{array}{ll}
		0, &\text{if}~ c_i(x)+\lambda_i/\rho \in [\ell^c_i,u^c_i],\\
		1, &\text{otherwise},
	\end{array}\right.
	\quad
	d^x_j=\left\{\begin{array}{ll}
		0, &\text{if}~ x_j+\mu_j/\rho \in [\ell^x_j,u^x_j],\\
		1, &\text{otherwise}.
	\end{array}\right.
\end{equation*}

Therefore, the Hessian matrix of the augmented Lagrangian function can be efficiently approximated, allowing the use of an inexact semismooth Newton method. If the approximated matrix is symmetric positive definite, it follows from \cite{qi1993nonsmooth} that the SSN method applied to the subproblem is Q-superlinearly convergent in a neighborhood of $x^*$, where $F(x^*) = 0$. Similarly, we can adopt the quasi-Newton method to solve the subproblem.

\subsection{Sparse optimization}\label{sec:nonsmooth}

The composite programming problem shown below has been widely studied in recent literature:
\be \label{pro:almssn_primal}
\min_{x\in \mathbb{R}^n} \ f(\mathcal{A}x) + h(x),
\ee
where $f$ is a convex, differentiable function with a Lipschitz continuous gradient, $h$ is a closed proper convex but possibly nonsmooth function, and $\mathcal{A}:\mathbb{R}^m \rightarrow \mathbb{R}^n$ is a linear operator. In many cases, the nonsmooth term is typically a sparse regularization function.  The dual of \eqref{pro:almssn_primal} can be written as:
\be \label{pro:almssn_dual}
\min_{\nu\in \mathbb{R}^m}\ f^*(\nu) + h^*(-\mathcal{A}^*\nu),
\ee
where $f^*$ is strongly convex and essentially smooth. The gradient of $f^*$ is locally Lipschitz continuous on the interior of its domain.
To guarantee the differentiability of $f^*$, we assume that $f$ is essentially locally strongly convex. This means that for any compact convex set $K\subset \operatorname{dom} \partial f$, there exists a positive constant $\beta_K$ such that for all $y,y^\prime \in K$ and $s\in [0,1]$, the following inequality holds:
\begin{equation*}
(1-s)f(y^\prime) + s f(y) \geq f((1-s)y^\prime + s y)+ \beta_Ks (1-s)\|y-y^\prime\|^2/2.
\end{equation*}
These assumptions are satisfied by many commonly used loss functions in machine learning literature. For instance, $f$ can be the loss function in linear, logistic, or Poisson regression models. The related applications include: Lasso \cite{li2018highly}; fused lasso \cite{li2018efficiently}; OSCAR and SLOPE models \cite{luo2019solving}; clustered Lasso \cite{lin2019efficient}; group graphical Lasso \cite{zhang2020proximal}; sparse group Lasso \cite{zhang2020efficient}; fused multiple graphical Lasso \cite{zhang2019efficient}, etc.

\subsubsection{ALM for primal and dual problem}
We first give the AL function of the primal problem \eqref{pro:almssn_primal}:
\begin{equation}
    \mathbb{L}_{\rho}^p(x,\nu): = f(\mathcal{A}x) + e_{\rho}h(x+\nu/\rho) - \frac{1}{2\rho}\|\nu\|^2. 
\end{equation}
Here, we use the notation $\mathbb{L}_{\rho}^p$ to denote the AL function for the primal problem, distinguishing it from the AL function for the dual problem, denoted by $\mathbb{L}_{\rho}^d$. 
A direct application of ALM onto \eqref{pro:almssn_primal} can be given as follows:
\begin{equation}
    \left\{
\begin{aligned}
    x^{k+1} & = \arg\min_{x\in \mathbb{R}^n}  \mathbb{L}^p_{\rho_k}(x,\nu^k), \\
    \nu^{k+1} &= \nu^k - \rho_k(x^{k+1} - \text{prox}_{h/\rho_k}(x^{k+1} - \nu^k/\rho_k)).
\end{aligned}
    \right.
\end{equation}
In the above algorithm, the main computational cost is from the update of $x^{k+1}$.  Given that the AL function is continuously differentiable due to the convexity of $h$, we can employ gradient-based methods to solve it. However, when the dimension of $x$ is very large, solving the 
$x$-subproblem becomes challenging.

In a recent study, the authors in \cite{li2018highly} discovered that applying the ALM to the dual problem \eqref{pro:almssn_dual} and employing a semismooth Newton algorithm to address the ALM subproblem can be  effective. The AL function of the dual problem \eqref{pro:almssn_dual} is defined as 
\begin{equation}
    \mathbb{L}^d_{\rho}(\nu,x) = f^*(\nu) + e_{\rho}h^*(-\mathcal{A}^*\nu + x/\rho) - \frac{1}{2\rho}\|x\|^2. 
\end{equation}
Therefore, the main iterative process of ALM is given as follows:
\begin{equation}
    \left\{
\begin{aligned}
    \nu^{k+1} & = \arg\min_{\nu\in \mathbb{R}^m}  \mathbb{L}^d_{\rho_k}(\nu,x^k), \\
    x^{k+1} &= x^k - \rho_k(-\mathcal{A}^*\nu^{k+1} + \text{prox}_{h^*/\rho_k}(-\mathcal{A}^*\nu^{k+1} + x^k/\rho_k)).
\end{aligned}
    \right.
\end{equation}

Let us focus on the $\nu$-subproblem.  For simplicity, we define $\varphi(\nu): = \mathbb{L}^d_{\rho_k}(\nu,\tilde{x})$ for any fixed $\rho$ and $\tilde{x}$.  It is important to note that $\varphi(\cdot)$ is strongly convex and continuously differentiable within $\operatorname{int}(\operatorname{dom} \left.f^*\right)$. The gradient is given by
$$
\nabla \varphi(\nu)=\nabla f^*(\nu)-\mathcal{A} \operatorname{prox}_{\rho h}(\tilde{x}-\rho\mathcal{A}^* \nu), \quad \forall \nu \in \operatorname{int}\left(\operatorname{dom} f^*\right).
$$
Consequently, gradient-based algorithms can be employed to tackle the $\nu$-subproblem. Nonetheless, such algorithms tend to have slow convergence rates and may not attain high precision.  An alternative is the application of second-order algorithms, which, despite their higher per-iteration costs, can substantially lower the overall computational burden and expedite convergence, especially when problem \eqref{pro:almssn_primal} has specific structures, such as when $h$ is sparse regularization. In particular, consider  solving the nonlinear equation:
\begin{equation}\label{grad-zeros-psi}
     \nabla \varphi(\nu) = 0.
\end{equation}
The SSN method is given as follows:
\begin{equation}\label{sec7:ssniter}
    \nu^{k+1} = \nu^k + d^k,
\end{equation}
where $d^k$ is determined by solving the linear equation:
\begin{equation}\label{eq:ssnal:Newton-equation}
    V_k d = - \nabla \psi(\nu^k).
\end{equation}
Here $V_k \in \hat{\partial} (\nabla \psi(\nu^k))$. 
The reason for using the semi-smooth Newton method rather than the classical Newton method is due to the semi-smooth nature of the gradient $\nabla \varphi(\nu)$, which will be explained in detail in the Appendix \ref{app:seminewton:sub}.

\subsubsection{Solving SSN subproblem efficiently}

In this subsection, we focus on an efficient implementation of \eqref{eq:ssnal:Newton-equation} in the case where the nonsmooth regularizer $h$ takes the form of the $\ell_1$ penalty, namely $h(x)=\mu\|x\|_1$.  
Given $\nu \in \mathbb{R}^m$, $\tilde{x}\in \mathbb{R}^n$, and a parameter $\rho>0$, we consider the following Newton system (with superscripts suppressed for clarity):
\begin{equation}\label{sec7:newton-system}
    (H+\rho A U A^\top)d = -\nabla \varphi(\nu),
\end{equation}
where $H \in \partial (\nabla f^*)(\nu)$, $U\in \partial \operatorname{prox}_{\rho\mu\|\cdot\|_1}(\tilde{x}-\rho\mathcal{A}^* \nu)$, and $A$ denotes the matrix representation of $\mathcal{A}$.

In many practical models, the matrix $H$ is sparse. For example, if $f$ corresponds to either squared loss or logistic loss, then $H$ is diagonal. This observation allows us to simplify \eqref{eq:ssnal:Newton-equation} to the following form without loss of generality:
\begin{equation}\label{lasso-newton-equation-2}
    \left(I_m+\rho A U A^\top\right) d = -\nabla \varphi(\nu).
\end{equation}

In sparse optimization, exploiting the sparsity pattern of $U$ can significantly reduce computational overhead, making it negligible compared with other costs.  
For $x=\tilde{x}-\rho A^\top \nu$, we employ the diagonal matrix representation $U=\operatorname{Diag}(u)$, where each $u_i$ is specified as
\[
u_i=\begin{cases}
0, & \text{if } |x_i|\le \rho \mu,\\[0.3em]
1, & \text{otherwise},
\end{cases}
\quad i=1,\ldots,n.
\]
Since the proximal operator satisfies
\[
\operatorname{prox}_{\rho\mu\|\cdot\|_1}(x)=\operatorname{sign}(x)\circ \max\{|x|-\rho\mu,0\},
\]
it follows that $U\in \partial \operatorname{prox}_{\rho\mu\|\cdot\|_1}(x)$. Let $\mathcal{J}=\{j:\,|x_j|>\rho \mu,\ j=1,\ldots,n\}$ and denote its cardinality by $r=|\mathcal{J}|$. Making use of the sparsity of $U$, we obtain
\begin{equation}\label{lasso-newton-equation-1}
    A U A^\top=(AU)(AU)^\top = A_{\mathcal{J}}A_{\mathcal{J}}^\top,
\end{equation}
where $A_{\mathcal{J}}\in \mathbb{R}^{m\times r}$ is the submatrix of $A$ formed by selecting the columns indexed by $\mathcal{J}$.  
This representation reduces the complexity of evaluating $AU A^\top$ and $AU A^\top d$ for any $d$ to $\mathcal{O}(m^2 r)$ and $\mathcal{O}(mr)$, respectively. Because the $\ell_1$ regularizer promotes sparsity, the value of $r$ is usually much smaller than $n$.

Moreover, when $r\ll m$---a scenario frequently encountered in large-scale problems with sparse solutions---the Sherman--Morrison--Woodbury formula can be applied to invert $I_m+\rho AU A^\top$ by instead inverting an $r\times r$ matrix:
\[
\left(I_m+\rho A U A^\top\right)^{-1}=\left(I_m+\rho A_{\mathcal{J}} A_{\mathcal{J}}^\top\right)^{-1}
=I_m - A_{\mathcal{J}}\left(\rho^{-1} I_r + A_{\mathcal{J}}^\top A_{\mathcal{J}}\right)^{-1}A_{\mathcal{J}}^\top .
\]
Consequently, the total computational complexity of solving the Newton system \eqref{lasso-newton-equation-2} can be reduced from $\mathcal{O}(m^2(m+r))$ to $\mathcal{O}(r^2(m+r))$. Further implementation details can be found in \cite{li2018highly}.

\subsection{Semidefinite programming}
SDP is one of the exciting development in mathematical
programming in the last thirty years. It has applications in such diverse fields as traditional convex constrained optimization, control theory, and combinatorial optimization. Consider the following standard SDP problem:
\begin{equation}\label{sec4:stard-sdp-problem}
\min_{X\in \mathbb{S}^n}\ \left<C,X\right>\quad\mathrm{ s.t. } \quad \mathcal{A}(X) = b, X\succeq 0,
\end{equation}
where $\mathcal{A}:\mathbb{S}^n\rightarrow \mathbb{R}^m$ is a linear map. The dual problem is given as follows
\begin{equation}\label{sec4:stard-sdp-dual}
    \max_{\lambda\in \mathbb{R}^m} \ \left<b,\lambda\right>\quad \mathrm{ s.t. }\quad \mathcal{A}^*(\lambda) \preceq C. 
\end{equation}
In this section, we will discuss how to use ALM to solve the SDP problem and its dual problem. 
\subsubsection{Primal problem}
We first consider applying the ALM to the primal problem \eqref{sec4:stard-sdp-problem}. We can construct the AL function by penalizing the linear constraint and retaining the positive definite constraint:
\begin{equation*}
    \mathbb{L}^p_{\rho}(X,\lambda): =\left<C,X\right> + \frac{\rho}{2}\|\mathcal{A}(X) - b + \lambda/\sigma\|^2 - \frac{1}{2\rho}\|\lambda\|^2.  
\end{equation*}
Then in the $k$-th iteration, the ALM has the following update:
\begin{equation*}
\left\{
    \begin{aligned}
        X^{k+1} &= \argmin_{X\succeq 0} \mathbb{L}^p_{\rho_k}(X,\lambda^k), \\
        \lambda^{k+1} & = \lambda^k + \rho_k (\mathcal{A}(X^{k+1}) - b). 
    \end{aligned}
    \right.
\end{equation*}
Note that the $X$-subproblem is a convex optimization problem on the positive semidefinite constraint. In \cite{burer2006solving,burer2010optimizing}, a coordinate descent method and an eigenvalue decomposition are used to solve the $X$-subproblem. In \cite{wen2009row}, by fixing any $(n-1)$-dimensional principal submatrix of $X$ and using its Schur complement, the positive semidefinite constraint is reduced to a simple second-order cone constraint and then a sequence of second-order cone programming problems constructed from the primal
AL function is minimized. 

 However, the constraint $X\succeq 0$ is the most challenging aspect of solving \eqref{sec4:stard-sdp-problem} since the objective function and constraints are only linear in $X$. To circumvent this difficult constraint, the Burer-Monteiro method considers the case that \eqref{sec4:stard-sdp-problem}  admits a low-rank solution $X^*$. This method introduces the change of variables $X = VV^\top$ where $V$ is a real, $n\times r$ matrix, and $r<n$. In terms of the new variable $V$, the SDP problem \eqref{sec4:stard-sdp-problem} is reduced to the resulting nonlinear program:
\begin{equation}\label{sec4:stard-sdp-problem-dual}
    \min_{V\in \mathbb{R}^{n\times r}}  \ \left<C,VV^\top\right>\quad \mathrm{ s.t. }\quad \mathcal{A}(VV^\top) = b.
\end{equation}
With the elimination of the positive semidefiniteness constraint, \eqref{sec4:stard-sdp-problem-dual} gains a notable advantage over \eqref{sec4:stard-sdp-problem}. However, this advantage comes with a trade-off: the objective function and constraints become quadratic and generally nonconvex, rather than linear. Despite this, the absence of the positive semidefiniteness constraint and the potential for lower dimensionality compared to $X$ prompted Burer and Monteiro \cite{burer2003nonlinear,burer2005local} to apply the ALM to solve \eqref{sec4:stard-sdp-problem-dual}, achieving unexpectedly positive results. Specifically, the AL function for \eqref{sec4:stard-sdp-problem-dual} is defined as follows:
\begin{equation}
    \mathbb{L}^p_{\rho}(V, \lambda): = \left<C,VV^\top\right> + \frac{\rho}{2}\|\mathcal{A}(VV^\top) - b + \lambda/\rho\|^2 - \frac{1}{2\rho}\|\lambda\|^2. 
\end{equation}
Then in the $k$-th iteration, the ALM has the following update:
\begin{equation*}
\left\{
    \begin{aligned}
        V^{k+1} &= \argmin_{V} \mathbb{L}^p_{\rho_k}(V,\lambda^k), \\
        \lambda^{k+1} & = \lambda^k + \rho_k (\mathcal{A}(V^{k+1} (V^{k+1})^\top ) - b). 
    \end{aligned}
    \right.
\end{equation*}

It is important to note that due to the nonconvex nature of \eqref{sec4:stard-sdp-problem-dual}, local optimization methods typically yield a stationary point rather than a global solution. A critical area of study is how to recover the optimal solution $X^*$ in \eqref{sec4:stard-sdp-problem} via the optimal solution $V^*$ in \eqref{sec4:stard-sdp-problem-dual}. Recently, Boumal et al. \cite{boumal2020deterministic} demonstrated that if the set of constraints on $V$ regularly defines a smooth manifold, then, despite nonconvexity, the first- and second-order necessary optimality conditions are also sufficient, provided $r$ is sufficiently large. Under these conditions, a global optimum $V$ corresponds to a global optimum $X = VV^\top$ of the SDP problem \eqref{sec4:stard-sdp-problem}.

\subsubsection{Dual problem}
The ALM can also be applied to the dual problem \eqref{sec4:stard-sdp-dual}. In particular,  the AL function associated with \eqref{sec4:stard-sdp-dual} is given as 
\begin{equation*}
    \mathbb{L}^d_{\rho}(\lambda,X): = \left<b,\lambda\right> + \frac{\rho}{2}\|\Pi_{\mathbb{S}_+^n} (C - \mathcal{A}^*(\lambda)  + X/ \rho)  \|^2 - \frac{1}{2\rho}\|X\|^2. 
\end{equation*}
Then the ALM can be written as 
\begin{equation*}
\left\{
    \begin{aligned}
        \lambda^{k+1} &= \argmin_{\lambda}  \mathbb{L}^d_{\rho}(\lambda,X^k),\\
       X^{k+1} & =\Pi_{\mathbb{S}^n_+} (X^k + \rho (C - \mathcal{A}^*(\lambda^{k+1}))).
    \end{aligned}
    \right.
\end{equation*}
For some fixed $X \in\mathbb{S}^n$ and $\rho>0$, we need to consider the following form of inner problems:
\begin{equation}\label{sec4:sdpna-y-subproblem}
    \min \left\{\varphi(\lambda):=\mathbb{L}^d_\rho(\lambda, X) \mid \lambda \in \mathbb{R}^m\right\}.
\end{equation}
Note that $\varphi$ is convex and continuously differentiable with 
$$
\nabla \varphi(\lambda)=b-\mathcal{A} \Pi_{\mathbb{S}^n_+}\left(X-\rho\left(\mathcal{A}^* \lambda-C\right)\right).
$$
In applying Newton's algorithm to solve the inner subproblem \eqref{sec4:sdpna-y-subproblem}, we find a direction $d$ by solving the following linear system:
\begin{equation}\label{sec4:sdpna-linear-system}
(V_\lambda^0+\epsilon I) d=-\nabla \varphi(\lambda),
\end{equation}
where $V_\lambda^0 \in \hat{\partial}^2 \varphi(\lambda)$, where $\hat{\partial}^2 \varphi(\lambda)$ is defined in \eqref{def:hat-jacobian-sdp}. Moreover,  a practical conjugate gradient (CG) method is then designed to solve the linear system \eqref{sec4:sdpna-linear-system}. For more details, please refer to \cite{zhao2010newton}.
A majorized semismooth Newton-CG augmented Lagrangian method is further developed in \cite{yang2015sdpnal+} for the general SDP problems with additional nonnegative constraints.  

\vspace{0.3cm}
\textbf{Computing the Jacobian matrix.}
 We next explain the explicit expression of $V_\lambda^0$ or $V_\lambda^0 d$ for any $d$. Let us first define:
\begin{equation}\label{def:hat-jacobian-sdp}
\hat{\partial}^2 \varphi(\lambda):=\rho \mathcal{A} \partial \Pi_{\mathbb{S}^n}\left(X-\rho\left(\mathcal{A}^* \lambda-C\right)\right) \mathcal{A}^* .
\end{equation}
 Since $X-\rho\left(\mathcal{A}^* \lambda-C\right)$ is a symmetric matrix in $\mathbb{R}^{n \times n}$, there exists an orthogonal matrix $Q \in \mathbb{R}^{n \times n}$ such that
$$
X-\rho\left(\mathcal{A}^* \lambda-C\right)=Q \Gamma_\lambda Q^{\mathrm{T}},
$$
where $\Gamma_\lambda$ is the diagonal matrix with diagonal entries consisting of the eigenvalues $\lambda_1 \geq \lambda_2 \geq \cdots \geq \lambda_n$ of $X-\rho\left(\mathcal{A}^* \lambda-C\right)$ being arranged in nonincreasing order. Define three index sets
$
\alpha:=\left\{i : \lambda_i>0\right\}, \ \beta:=\left\{i : \lambda_i=0\right\}, \ \text { and } \ \gamma:=\left\{i : \lambda_i<0\right\}. 
$
Define the operator $W_\lambda^0: \mathbb{S}^n\rightarrow \mathbb{S}^n$ by
$
W_\lambda^0(H):=Q\left(\Omega \circ\left(Q^{\mathrm{T}} H Q\right)\right) Q^{\mathrm{T}}$ with $ H \in \mathbb{S}^n,
$
and
$$
\Omega=\left[\begin{array}{cc}
E_{\alpha \alpha} & \nu_{\alpha \bar{\alpha}} \\
\nu_{\alpha \bar{\alpha}}^{\mathrm{T}} & 0
\end{array}\right], \quad \nu_{i j}:=\frac{\lambda_i}{\lambda_i-\lambda_j}, i \in \alpha, j \in \bar{\alpha}:=\{1, \ldots, n\} \backslash \alpha,
$$
where $E_{\alpha \alpha} \in \mathbb{S}^{|\alpha|}$ is the matrix of ones. Then $V_y^0: \mathbb{R}^m \rightarrow \mathbb{S}^n$ is obtained by
$$
V_\lambda^0 d:=\rho \mathcal{A}\left[Q\left(\Omega \circ\left(Q^{\mathrm{T}}\left(\mathcal{A}^* d\right) Q\right)\right) Q^{\mathrm{T}}\right], \quad d \in \mathbb{R}^m. 
$$
It follows from \cite[Lemma 11]{pang2003semismooth} that
$$
W_\lambda^0 \in \partial \Pi_{\mathbb{S}^n_+}\left(X-\rho\left(\mathcal{A}^* \lambda-C\right)\right),
$$
which further gives $V_\lambda^0=\rho \mathcal{A} W_\lambda^0 \mathcal{A}^* \in \hat{\partial}^2 \varphi(\lambda)$.  More details are referred to \cite{zhao2010newton,yang2015sdpnal+}.

\subsubsection{Nonsmooth SDP}\label{sec7:nonsmooth-sdp}
More recently, a class of more general nonsmooth and nonlinear semidefinite programming problems is considered in \cite{wang2023decomposition}:
\begin{equation}\label{sec4:sdp:general}
\min_{X\in \mathcal{D}} \  f(X)+h(X)\quad \text{ s.t. }  \ \mathcal{A}(X)=b,
\end{equation}
where 
$ \mathcal{A}(X): \mathbb{S}^n \rightarrow \mathbb{R}^m$.
The function $f(X)$ is smooth while $h(X)$ is possibly nonsmooth. The domain $\mathcal{D}=\{X\in \mathbb{S}^n: X\succeq 0, \mathcal{B}(X)=b_0\}$ with $ \mathcal{B}(X)\in \mathbb{S}^n \rightarrow \mathbb{R}^{m_0}$ defines a certain Riemannian structure. 
A decomposition method is proposed based on the augmented Lagrangian framework (SDPDAL). Denote the AL function associated with \eqref{sec4:sdp:general} by
\begin{equation}\label{aug_L:1}
\begin{aligned}
\mathbb{L}_\rho (X,\nu,\lambda) = &f(X)+e_{\rho} h\left(X - \frac{\nu}{\rho}\right) + \frac{\rho}{2}\left\|\mathcal{A}(X)-b - \frac{\lambda}{\rho}\right\|_2^2 - \frac{1}{2\rho} ( \|\nu\|^2 + \|\lambda\|^2 ).
\end{aligned}
\end{equation}
where $\nu\in \mathbb{S}^n,\lambda\in \mathbb{R}^m$ are corresponding Lagrange multipliers.  The domain of the primal variables $X$ is $ \mathcal{D}$. The $k$-th iteration of the ALM is given as follows:
\begin{align}
  X^{k+1} &= \mathop{\argmin}_{X\in \mathcal{D}} \mathbb{L}^p_{\rho_k}(X,\nu^{k},\lambda^{k}),\label{sub:xw} \\
 \nu^{k+1} &=\nu^k-\rho_k(X^{k+1} -\prox_{h/\rho_k}(X^{k+1}-\nu^k/{\rho_k})),\label{sdpdal-nu}\\
 \lambda^{k+1} &=\lambda^k-\rho_k(\mathcal{A} (X^{k+1}) -b).\label{sdpdal-lambda}
\end{align}
The detailed method can be seen in Algorithm \ref{alg:alm_ssn}. 

\begin{algorithm}[h]
\caption{The SDPDAL method}\label{alg:alm_ssn}
\begin{algorithmic}[1]
\REQUIRE Initial trial point $R^0\in \mathcal{M}$ %
, ALM step size $\alpha_k$, parameters $\rho_k>0$.
\STATE Set $k=0, \nu^k=0,\lambda^k=0$. 
\WHILE {not converge}
\STATE Obtain $R^{k+1}$ by solving \eqref{sub:rk} inexactly. Formulate $X^{k+1}=(R^{k+1})^\top R^{k+1}$ either explicitly or implicitly.
\STATE Update Lagrange multipliers $\nu^{k+1}$, $\lambda^{k+1}$ via \eqref{sdpdal-nu} and \eqref{sdpdal-lambda}.
\STATE Set $k=k+1$.
\ENDWHILE
\end{algorithmic}
\end{algorithm}

\textbf{Solving the subproblem.}
To solve the subproblem $X$-subproblem, one can factorize $X=R^\top R$, and consider the following Riemannian optimization problem:
\begin{equation}\label{sub:rk}
\min_{R\in \mathcal{M}} \Psi_k(R):=\Phi_k(R^\top R),
\end{equation}
where $\mathcal{M}: = \{R\in \mathbb{R}^{p\times n}:R^\top R \in \mathcal{D}\}$. As shown in 
\cite{wang2023decomposition},  if $\{B_iX\}_{i=1}^{m_0}$ are linear independent, $\mathcal{M}$ is a smooth manifold.   
 We apply  an adaptive regularized Riemannian semismooth Newton method in 
 \cite{hu2018adaptive} to solve \eqref{sub:rk}. At a point $R_l$, the following subproblem is considered:
  \begin{equation}\label{sub:mk}
\min_{U\in T_{R_l}\mathcal{M}}  m_l(U)=\langle\mathrm{grad} \Psi_k(R_l),U\rangle+\frac{1}{2} \langle \mathbf{H}_l[U],U \rangle+\frac{\nu_l}{2} \|U\|_F^2,
\end{equation}
where $T_{R_l}\mathcal{M}$ denote the tangent space of $\mathcal{M}$ at $R_l\in \mathcal{M}$, $\mathbf{H}_l\in \mathrm{Hess} \Psi_k(R_l)$ is a generalized Riemannian Hessian operator and $\nu_l>0$ is a regularization parameter. 
In each step, we inexactly solve the following linear equation \[\mathrm{grad} \Psi_k(R_l)+\mathbf{H}_l[U]+\nu_l U=0.
\]
For more details on manifold optimizations, please refer to the book \cite{absil2009optimization}.

\subsection{Multi-block convex composite optimization}\label{sec7:multi-block}

  This subsection focuses on a class of linearly constrained multi-block convex composite optimization problems of the following form. We begin with the equivalent saddle point problem and directly design a second-order algorithm for solving the saddle point problem, i.e., both the primal variable and dual variable are updated by the second-order step. The main idea follows from \cite{deng2023augmented}.  In particular, we consider the following problem:
\begin{equation}\label{sec7:prob:multiblock}
\min_{\bm{x}_1,\dots,\bm{x}_n} \sum_{i \in \mathcal{I}_1} h_i(\bm{x}_i) + \sum_{i \in \mathcal{I}_2} h_i(\bm{x}_i), \quad \mathrm{s.t.} \quad \sum_{i=1}^{n-1} \mathcal{A}_i \bm{x}_i + \bm{x}_n = \bm{b},
\end{equation}
where $\mathcal{I}_1 = [1,\dots,n_1]$, $\mathcal{I}_2 = [n_1+1,\dots,n]$, and each $h_i$ is either a convex differentiable function or a proper closed convex function. The linear operators $\mathcal{A}_i$ couple the variables via the equality constraint. This problem formulation allows for both smooth and non-smooth components, making it applicable to a wide range of optimization problems, such as those involving $\ell_1$ regularization or nuclear norm constraints. It also covers the general SDPs whose variable is both positive semidefinite and nonnegative. 

One commonly used model in image processing is the following three-term composite optimization problem:
\begin{equation}\label{pro:three}
    \min_{x} \;\; f(\mathcal{A}x) + g(\mathcal{B}x) + h(x),
\end{equation}
where $f$ is a continuously differentiable function,  $g$ and $h$ are two regularization terms, and $\mathcal{A}$ and $\mathcal{B}$ are linear maps. 
The dual formulation of \eqref{pro:three} can be written as:
\begin{equation}\label{pro:three:dual}
    \min_{\lambda,\mu,w} \;\; f^*(\lambda) + g^*(\mu) + h^*(w)\quad \mathrm{s.t.} \quad \mathcal{A}^*\lambda + \mathcal{B}^*\mu + w = 0,
\end{equation}
where $\mathcal{A}^*$ and $\mathcal{B}^*$ denote the adjoint operators of $\mathcal{A}$ and $\mathcal{B}$, respectively, and $f^*$, $g^*$, and $h^*$ are the Fenchel conjugate functions of $f$, $g$, and $h$. This dual formulation represents a special case of problem \eqref{sec7:prob:multiblock}.

\subsubsection{The AL function} 
Following the approach for composite optimization outlined in Section \ref{sec:AL-composite}, we derive the AL function for problem \eqref{sec7:prob:multiblock}:
\begin{equation}\label{eq:alfunc}
\begin{aligned}
&\mathbb{L}_\rho(\bm{x};\bm{z}) = h_n\left(\prox_{h_n/\rho}\left(\bm{b} - \frac{\bm{z}_n}{\rho} - \sum_{i=1}^{n-1} \mathcal{A}_i \bm{x}_i\right)\right) - \underbrace{\frac{1}{2\rho} \sum_{i \in \mathcal{I}_2} \|\bm{z}_i\|^2}_{\rm multipliers}\\
& + \sum_{i \in \mathcal{I}_2 \backslash n} \underbrace{\left(h_i\left(\prox_{h_i/\rho}\left(\bm{x}_i - \frac{\bm{z}_i}{\rho}\right)\right) + \frac{\rho}{2} \left\|\bm{x}_i - \frac{\bm{z}_i}{\rho} - \prox_{h_i/\rho}\left(\bm{x}_i - \frac{\bm{z}_i}{\rho}\right)\right\|^2\right)}_{\text{Moreau envelope of } h_i}\\
& + \sum_{i \in \mathcal{I}_1} h_i(\bm{x}_i) + \underbrace{\frac{\rho}{2} \left\|\sum_{i=1}^{n-1} \mathcal{A}_i \bm{x}_i + \prox_{h_n/\rho}\left(\bm{b} - \frac{\bm{z}_n}{\rho} - \sum_{i=1}^{n-1} \mathcal{A}_i \bm{x}_i\right) - \bm{b} + \frac{\bm{z}_n}{\rho}\right\|^2}_{\sum_{i=1}^{n-1} \mathcal{A}_i \bm{x}_i + \bm{x}_n = \bm{b}}.
\end{aligned}
\end{equation}
For simplicity, we write $\bm{w} = (\bm{x}, \bm{z})$. Thus, the AL function $\mathbb{L}_\rho$ can be expressed as:
\begin{equation*}\begin{aligned}
\mathbb{L}_\rho(\bm{w}) &= \sum_{i \in \mathcal{I}_1} h_i(\bm{x}_i) + \sum_{i \in \mathcal{I}_2 \backslash n} e_{\rho} h_i\left(\bm{x}_i - \frac{\bm{z}_i}{\rho}\right) \\
& + e_{\rho} h_n\left(\bm{b} - \frac{\bm{z}_n}{\rho} - \sum_{i=1}^{n-1} \mathcal{A}_i \bm{x}_i\right) - \frac{1}{2\rho} \sum_{i \in \mathcal{I}_2} \|\bm{z}_i\|^2.
\end{aligned}
\end{equation*}
The associated saddle point problem is then formulated as:
\begin{equation}\label{prob:minimax}
\min_{\bm{x}} \max_{\bm{z}} \mathbb{L}_\rho(\bm{x}; \bm{z}).
\end{equation}

\subsubsection{The primal-dual nonlinear system}

We now outline a semi-smooth Newton system of nonlinear equations designed to characterize the optimality conditions of the original multi-block problem. This system is defined as:
\begin{equation} \label{eq:F}
F(\bm{w}) = \begin{pmatrix} \nabla_{\bm{x}} \mathbb{L}_\rho(\bm{w}) \\ - \nabla_{\bm{z}} \mathbb{L}_\rho(\bm{w}) \end{pmatrix}.
\end{equation}
While this reformulation is relatively straightforward, methods discussed in \cite{xiao2018regularized,li2018semismooth} rely heavily on first-order techniques that leverage specific structural properties of the problem.

Next, we compute the generalized Jacobian of $F(\bm{w})$. Given that for a convex function $h$, its proximal operator $\prox_{th}$ is Lipschitz continuous, we define the following sets:
\begin{equation}
\begin{aligned}
 \hat{\mathcal{D}}_{h_i}  &:= \nabla^2 h_i(\bm{x}_i), \quad i \in \mathcal{I}_1, \\
 \mathcal{D}_{h_i}  &:= \partial  \prox_{\rho h_i^*}(\rho \bm{x}_i - \bm{z}_i), \quad i \in \mathcal{I}_2 \backslash n, \\
 \mathcal{D}_{h_n} &:= \partial \prox_{\rho h_n^*} \left(\bm{b} + \frac{\bm{z}_n}{\rho} - \sum_{i=1}^{n-1}\mathcal{A}_i\bm{x}_i\right),
\end{aligned}
\end{equation}
where $\mathcal{D}_{h_i}$ and $\mathcal{D}_{h_n}$ represent Clarke’s generalized Jacobian for their respective proximal operators.

We now define the following matrix operators:
\begin{equation} \label{def:H123}
\mathcal{H}_{\bm{xx}} :=
    \left( \begin{array}{ccc}
         \rho \hat{\mathcal{D}}_{h_1}  + \rho \mathcal{A}_1^* \hat{\mathcal{D}}_{h_n} \mathcal{A}_1   
         & \cdots &  \rho \mathcal{A}_1^* \hat{\mathcal{D}}_{h_n} \mathcal{A}_{n-1} \\
        \vdots &  \ddots & \vdots \\
         \rho \mathcal{A}_{n-1}^* \hat{\mathcal{D}}_{h_n} \mathcal{A}_1 &  \cdots &   \rho \hat{\mathcal{D}}_{h_{n-1}} + \rho \mathcal{A}_{n-1}^* \hat{\mathcal{D}}_{h_n} \mathcal{A}_{n-1}
    \end{array} \right),
\end{equation}
\[
\mathcal{H}_{\bm{xz}} := - \left[ \mathrm{blkdiag} \left( \left\{ \hat{\mathcal{D}}_{h_i} \right\}_{i=1}^{n-1} \right), \left( \mathcal{A}_1^*; \dots; \mathcal{A}_{n-1}^* \right) \hat{\mathcal{D}}_{h_n}^{\top} \right], 
\]
\[
\mathcal{H}_{\bm{zz}} := \mathrm{blkdiag}\left(\left\{\frac{1}{\rho} I - \frac{1}{\rho} \hat{\mathcal{D}}_{h_i}\right\}_{i=1}^{n}\right),
\]
where $\hat{\mathcal{D}}_{h_i} \in \mathcal{D}_{h_i}$ for $i \in \mathcal{I}_2$, and $\mathrm{blkdiag}$ denotes the block diagonal matrix operator.

For any $\bm{w}$, the generalized Jacobian of $F(\bm{w})$ is given by:
\begin{equation}\label{equ:jaco}
\hat{\partial} F(\bm{w}) := \left\{ \begin{pmatrix}
\mathcal{H}_{\bm{xx}} & \mathcal{H}_{\bm{xz}} \\
-\mathcal{H}_{\bm{xz}}^{\top} & \mathcal{H}_{\bm{zz}}
\end{pmatrix} : \hat{\mathcal{D}}_{h_i} \in \mathcal{D}_{h_i}, \forall i \in \mathcal{I}_2 \right\}.
\end{equation}
Since the operator $\mathcal{A}_i$, for $i \in \mathcal{I}_1 \cup \mathcal{I}_2 \backslash n$, is linear and differentiable, it holds that $\hat{\partial} F(\bm{w})[\bm{d}] = \partial F(\bm{w})[\bm{d}], \forall \bm{d}$. Moreover, any element of $\hat{\partial} F(\bm{w})$ is positive semidefinite, which is crucial for the design of the semi-smooth Newton method.

\subsubsection{Solving the semi-smooth Newton system}
Since $F(\bm{w})$ is locally Lipschitz continuous and monotone, we employ a semi-smooth Newton method to solve the nonlinear system $F(\bm{w}) = 0$. The key advantage of this approach lies in utilizing Clarke’s generalized Jacobian $\hat{\partial} F(\bm{w})$ to compute update directions. Each iteration solves a linear system of the form:
\begin{equation}\label{eq:ssn-pd}
(J^k + \tau_{k} \mathcal{I}) \bm{d}^{k} = -F(\bm{w}^k),
\end{equation}
where $\tau_{k}$ is a regularization parameter, $J^k$ is an element of the generalized Jacobian, and $\bm{d}^{k}$ is the semi-smooth Newton direction. The corresponding trial step is then defined as:
\[
  \bar{\bm{w}}^{k} = \bm{w}^k + \bm{d}^{k}.
\]
A critical step is computing the search direction $\bm{d}^{k}$ from \eqref{eq:ssn-pd}. For simplicity, we omit the superscripts or subscripts $k$, and denote the primal and dual components of the search direction as $\bm{d}_{\bm{x}} = (\bm{d}_{\bm{x}_1}, \dots, \bm{d}_{\bm{x}_{n-1}})^{\top}$ and $\bm{d}_{\bm{z}} = (\bm{d}_{\bm{z}_{n_1+1}}, \dots, \bm{d}_{\bm{z}_n})^{\top}$. The linear system for the semi-smooth Newton direction is given by:
\begin{equation}\label{eq:ssn-xz}
\left(
    \begin{array}{cc}
    \mathcal{H}_{\bm{xx}} + \tau \mathcal{I} & \mathcal{H}_{\bm{xz}} \\
    -\mathcal{H}_{\bm{xz}}^T & \mathcal{H}_{\bm{zz}} + \tau \mathcal{I}
    \end{array}
    \right)
    \left(
    \begin{array}{c}
    \bm{d}_{\bm{x}} \\
    \bm{d}_{\bm{z}}
    \end{array}
    \right) 
    = \left(
    \begin{array}{c}
    -F_{\bm{x}} \\
    -F_{\bm{z}}
    \end{array}
    \right),
\end{equation}
where $F_{\bm{x}} = (\nabla_{\bm{x}_1} \Phi(\bm{w}), \dots, \nabla_{\bm{x}_{n-1}} \Phi(\bm{w}))^{\top}$, $F_{\bm{z}} = (-\nabla_{\bm{z}_{n_1+1}} \Phi(\bm{w}), \dots, -\nabla_{\bm{z}_n} \Phi(\bm{w}))^{\top}$.

To simplify this system, we apply Gaussian elimination. First, for a given $\bm{d}_{\bm{x}}$, the direction $\bm{d}_{\bm{z}}$ can be computed as:
\begin{equation}\label{cor:z}
\bm{d}_{\bm{z}} = (\mathcal{H}_{\bm{zz}} + \tau \mathcal{I})^{-1} (\mathcal{H}_{\bm{xz}}^{\top} \bm{d}_{\bm{x}} - F_{\bm{z}}).
\end{equation}
Substituting this into \eqref{eq:ssn-xz} reduces the system to a linear equation in terms of $\bm{d}_{\bm{x}}$:
\begin{equation}\label{eqn:simp}
\mathcal{H}_{\bm{r}} \bm{d}_{\bm{x}} = F_{\bm{r}},
\end{equation}
where $F_{\bm{r}} := \mathcal{H}_{\bm{xz}} (\mathcal{H}_{\bm{zz}} + \tau \mathcal{I})^{-1} F_{\bm{z}} - F_{\bm{x}}$ and
\[
\mathcal{H}_{\bm{r}} := \mathcal{H}_{\bm{xx}} + \mathcal{H}_{\bm{xz}} (\mathcal{H}_{\bm{zz}} + \tau \mathcal{I})^{-1} \mathcal{H}_{\bm{xz}}^{\top} + \tau \mathcal{I}.
\]
The structure of $\mathcal{H}_{\bm{xz}}$, as defined in \eqref{def:H123}, allows us to write:
\begin{equation}\label{eqn:equation}
\mathcal{H}_{\bm{r}} =
\left(
\begin{array}{ccc}
\overline{\mathcal{D}}_{h_1} + \mathcal{A}_1^* \overline{\mathcal{D}}_{h_n} \mathcal{A}_1 & \cdots & \mathcal{A}_1^* \overline{\mathcal{D}}_{h_n} \mathcal{A}_{n-1} \\
\vdots & \ddots & \vdots \\
\mathcal{A}_{n-1}^* \overline{\mathcal{D}}_{h_n} \mathcal{A}_1 & \cdots & \overline{\mathcal{D}}_{h_{n-1}} + \mathcal{A}_{n-1}^* \overline{\mathcal{D}}_{h_n} \mathcal{A}_{n-1}
\end{array}
\right),
\end{equation}
where $\overline{\mathcal{D}}_{h_i} = \rho \hat{\mathcal{D}}_{h_i} + \tilde{\mathcal{D}}_{h_i}$, and $\tilde{\mathcal{D}}_{h_i} = \hat{\mathcal{D}}_{h_i} (\frac{1}{\rho} \mathcal{I} - \frac{1}{\rho} \hat{\mathcal{D}}_{h_i} + \tau \mathcal{I})^{-1} \hat{\mathcal{D}}_{h_i}$.
Thus, we only need to solve \eqref{eqn:simp} with respect to $n-1$ variables for problem \eqref{prob}. Depending on the specific structure of $\mathcal{H}_{\bm{r}}$, the linear system \eqref{eqn:simp} can be solved using various direct methods such as sparse Cholesky factorization or the Sherman-Morrison-Woodbury formula. Alternatively, iterative methods such as preconditioned conjugate gradient or QMR may be employed when appropriate.

This semi-smooth Newton method provides a robust framework for solving linearly constrained multi-block convex composite problems, achieving superlinear convergence by efficiently handling non-smooth terms through the use of second-order information from the augmented Lagrangian. Unlike the approaches in \cite{xiao2018regularized,li2018semismooth,milzarek2019stochastic}, which rely on additional first-order steps to guarantee convergence, this method updates both primal and dual variables simultaneously in a single semi-smooth Newton step. It contrasts with ALM-based methods, where separate outer and inner loops are employed \cite{li2018highly}.

\subsection{Integer programming}
Integer programming is studied extensively in computer vision, machine learning, and theoretical computer science. Here, we introduce a specific and widely applicable example: linear integer programming with a block structure. It is formulated as follows:
\begin{equation}\label{sec7:mip}
\min \ \mathbf{c}^{\top} \mathbf{x} \quad \mathrm{ s.t. } \ \mathbf{Ax} \leq \mathbf{b},~ \mathbf{x}_j \in \cX_j,\ j=1,2,\cdots,p,
\end{equation}
where $\mathbf{x}_j \in \R^{n_j}$ is the $j$-th block variable of $\mathbf{x} \in \R^{n}$, i.e., $\mathbf{x} = (\mathbf{x}_1;\cdots;\mathbf{x}_p)$ for $p \geq 1$ with $n=\sum_{j=1}^p n_j$.  $\bA \in \R^{m\times n}, \mathbf{b} \in  \R^{m}, \mathbf{c} \in  \R^{n}$ and the constraint $\mathcal{X}_j$ is the set of $0,1$ vectors in a polyhedron, i.e.,
$$\cX_j:=\{\mathbf{x}_j  \in \{0,1\}^{n_j}: \Bb_j\mathbf{x}_j \leq \mathbf{d}_j\},~j=1,2,\cdots,p.$$
The constraints can be reformulated as $\mathbf{x} \in \mathcal{X}:=\{\mathbf{x} \in \{0,1\}^n:~\Bb\mathbf{x} \leq \mathbf{d}\}$ where the block diagonal matrix $\Bb \in \R^{q\times n}$ is formed by the small submatrices $\Bb_j$ as the main diagonal and $\mathbf{d}=(\mathbf{d}_1;\cdots;\mathbf{d}_p)$. Correspondingly, $\mathbf{c}$ and $\bA$ can be rewritten as $\mathbf{c}=(\mathbf{c}_1;\mathbf{c}_2;\cdots;\mathbf{c}_p)$ with $\mathbf{c}_j \in \R^{n_j}$ and $\bA=(\bA_1~ \bA_2~\cdots~\bA_p)$ with $\bA_j \in \R^{m \times n_j}$.

We define an AL function by
  \begin{equation}\label{sec7:alf}
    L_{\rho}(\mathbf{x},\lambda)  =\sum\limits_{j=1}^p~ \mathbf{c}_j^{\top}\mathbf{x}_j+\lambda^{\top}\left(\sum\limits_{j=1}^p\bA_j\mathbf{x}_j - \mathbf{b}\right) +\frac{\rho}{2}\left\|\left(\sum\limits_{j=1}^p\bA_j\mathbf{x}_j - \mathbf{b}\right)_+\right\|^2,
  \end{equation}
The strong duality presented in Section \ref{IP-strong-duality} allows us to obtain a globally optimal solution to problem \eqref{sec7:mip} by solving the augmented Lagrangian dual problem. Therefore, it is reasonable to utilize ALM to solve this problem. Since the problem is block-structured, then a BCD method can be employed to solve the subproblem in the ALM framework by updating individual blocks iteratively \cite{rui2024}.  

 The BCD method minimizes the function $L$ by iterating cyclically in order $\mathbf{x}_1,\cdots,\mathbf{x}_p$, fixing the previous iteration during each iteration. Denote $\mathbf{x}^t = \left(\mathbf{x}_1^t; \mathbf{x}_2^t;\cdots;\mathbf{x}_p^t\right)$ where $\mathbf{x}^t_j$ is the value of $\mathbf{x}_j$ at its $t$-th update. Let
$$ L^t_{\rho}(\mathbf{x}_j,\lambda)= L(\mathbf{x}^{t+1}_1,\cdots,\mathbf{x}^{t+1}_{j-1},\mathbf{x}_j,\mathbf{x}^t_{j+1},\cdots,\mathbf{x}^t_{p},\lambda),$$
and
$$\mathbf{x}^t(j) = \left(\mathbf{x}_1^{t+1},\mathbf{x}_2^{t+1},\cdots,\mathbf{x}_{j-1}^{t+1},\mathbf{x}_{j}^{t},\mathbf{x}_{j+1}^{t},\cdots,\mathbf{x}_{p}^{t}\right).$$
Therefore $\mathbf{x}^t(1) = \left(\mathbf{x}_1^t; \mathbf{x}_2^t;\cdots;\mathbf{x}_p^t\right)=\mathbf{x}^t$ and $\mathbf{x}^t(p+1) = \left(\mathbf{x}_1^{t+1}; \mathbf{x}_2^{t+1};\cdots;\mathbf{x}_p^{t+1}\right)=\mathbf{x}^{t+1}.$ Given fixed parameters $\lambda\geq 0$ and $\rho>0$,
we can also calculate the gradient of $L_\rho(\mathbf{x}^t(j),\lambda)$ at $\mathbf{x}_j$ by
$$
  g_j(\mathbf{x}^t):=\nabla_{\mathbf{x}_j} L_\rho(\mathbf{x}^t(j),\lambda)=\mathbf{c}_j+\bA_j^{\top}\lambda+\rho\bA_j^{\top}\left(\bA\mathbf{x}^t(j) - \mathbf{b}\right)_+.
$$
At each step,  we consider two types of updates for every $\mathbf{x}_j \in \cX_j$:
\begin{subequations}
\begin{align}
  & \text{Classical:}~  \mathbf{x}^{t+1}_{j} \in \mathop{\argmin}_{\mathbf{x}_j \in \mathcal{X}_j}~ L^t_{\rho}(\mathbf{x}_j,\lambda),  \label{xjko}\\
  &\text{Proximal linear:}~
   \mathbf{x}_j^{ t + 1}  \in \mathop{\argmin}_{\mathbf{x}_j \in  \mathcal{X}_j} \left\{\langle \mathbf{x}_j-\mathbf{x}_j^t, g_j(\mathbf{x}^t)\rangle + \frac{1}{2\tau}\|\mathbf{x}_j-\mathbf{x}_j^t\|^2\right\},\label{xjkl}
  \end{align}
\end{subequations}
where $\tau>0$ is a step size.
We can obtain the following equivalent form of \eqref{xjkl}:
$$ \mathbf{x}_j^{ t + 1}=\Pi_{\cX_j}\left(\mathbf{x}_j^t-\tau g_j(\mathbf{x}^t)\right).$$

In general, the classical subproblem \eqref{xjko} is fundamentally harder to solve because of the quadratic term in the objective function. However,  we derive a simplified form of this subproblem under certain conditions. By contrast, the prox-linear subproblem \eqref{xjkl} is relatively easy to solve because the objective function is linear with respect to $\mathbf{x}_j$.
In addition, Jacobian iterative methods can be adapted to solve the subproblems in parallel. Consequently, combining the BCD method with the ALM yields the advantage of decomposing large-scale problems into lower-dimensional subproblems based on variable structures, resulting in reduced computation time.

\subsection{Reinforcement learning}\label{sec:RL}
With the recent development and empirical successes of neural networks, deep reinforcement learning has been resoundingly applied in many industrial fields. The focus of reinforcement learning (RL) is solving a Markov decision process (MDP) which models the framework for interactions between the agent and the environment \cite{sutton2018reinforcement}. Most of the work focuses on the Bellman equation, while a recent study \cite{li2023rl} directly solves the corresponding linear programming problem using ALM.

\subsubsection{ALM for LP}
The optimal Bellman equation can be interpreted as solving the LP problem \cite{bertsekas1995dynamic,puterman2014markov}:
\begin{equation}\label{eq:rllp-primal-eq}
\begin{split}
\min_{h\geq0,V}&\  \sum_{s} \rho_0(s)V(s), \\
\st &\ V(s)=  r(s,a)+  \gamma  \mathbb{E}_{s'|s,a}[ V(s')]+h(s,a), \forall s, a.
\end{split}
\end{equation}
If $a^*$ is the optimal action at state $s$, the relationship between the optimal policy and the optimal slack variable can be expressed as follows:
\be\label{handpi}
\begin{cases}
\pi^*(a|s)=1, h^*(s,a)=0, & a = a^*, \\
\pi^*(a|s)=0, h^*(s,a)>0, & a \neq a^*.
\end{cases}
\ee
Denote $x(s,a)$ as the multiplier associated with the constraints in \eqref{eq:rllp-primal-eq}. For simplicity, we introduce a $Z$-function as
\bee
Z_\mu(V,h,x,s,a) =  x(s,a)+  \mu \left(h(s,a)+r(s,a)+\gamma \mathbb{E}_{s'|s,a}[ V(s')] - V(s)\right),\label{zfunc}
\eee
and further define
\bee
z_\mu(V,h,x,s,a,s') =  x(s,a)+  \mu \left(h(s,a)+r(s,a)+\gamma V(s') - V(s)\right),\label{zfunc1}
\eee
where $\mu>0$ is the penalty parameter. Since 
$\sum_{s'} P(s'|s,a) = 1$,
we have
\bea\label{Z-z}
Z_\mu(V,h,x,s,a) = \mathbb{E}_{s'|s,a} [z_\mu(V,h,x,s,a,s')],
\eea
the $Z$-function can be interpreted as the conditional expectation over $s'$. For the LP problem \eqref{eq:rllp-primal-eq}, the weighted AL function is defined as 
\begin{equation}\label{randALM}
\mathbb{L}_{\mu}(V,h,x) = \sum_{s} \rho_0(s)V(s) +\frac{1}{2\mu}\sum_{s,a}w(s,a)[Z_\mu(V,h,x,s,a)]^2,
\end{equation}
where the weight function $w(s,a)$ is a distribution over the state and action space. At the $k$-th iteration, the update scheme of ALM can be written as
\begin{align}
(h^{k+1},V^{k+1}) &= \argmin_{h\geq0,V} \mathbb{L}_{\mu}(V,h,x^k), \label{eq:Vk}\\ 
x^{k+1}(s,a) &= Z_\mu(V^{k+1},h^{k+1},x^k,s,a) ,\forall s,a. \label{eq:xk}
\end{align}
However, when the state and action spaces are exceedingly large or even continuous, computing the elementwise update in \eqref{eq:xk} poses significant challenges. This difficulty arises from the limitations in storing and representing the Lagrange multipliers effectively. As a result, the general ALM becomes impractical due to these computational constraints.

\subsubsection{Parameterization of multipliers}
Let $ V $, $ h $, and $ x $ be parameterized by $ V_\phi $, $ h_\psi $, and $ x_\theta $, respectively. The non-negativity of $ h_\psi $ can be  ensured by the structure of the network. For simplicity, in the following discussion, we will refer to $ V_\phi $, $ h_\psi $, and $ x_\theta $ as $ \phi $, $ \psi $, and $ \theta $ without further clarification. Denote the network parameters at the $ k $-th iteration as $ \phi^k $, $ \psi^k $, and $ \theta^k $. The update rule for \eqref{eq:Vk} can then be expressed as

\begin{equation}\label{eqn:param-primal}
\min_{\phi, \psi} \hat{\mathbb{L}}_{\mu}(\phi, \psi, \theta^k) = \sum_{s} \rho_0(s) V_{\phi}(s) + \frac{1}{2\mu} \sum_{s, a} w(s, a) \left[Z_{\mu}(\phi, \psi,\theta^k, s, a) \right]^2.
\end{equation}
To address the difficulties associated with the element-wise update in \eqref{eq:xk}, we propose approximating the intractable function $Z_\mu(\phi^{k+1}, \psi^{k+1}, \theta^k, s, a)$ with a function $x_\theta(s, a)$ using a weighted least square regression:
\begin{equation}\label{eqn:param-dual}
\theta^{k+1} = \argmin_{\theta} \sum_{s, a} w(s, a) \left( Z_{\mu} (\phi^{k+1}, \psi^{k+1}, \theta^k, s, a) - x_{\theta}(s, a) \right)^2.
\end{equation}

We next present a  method to combine the two steps \eqref{eqn:param-dual} and \eqref{eqn:param-primal} into a single optimization model. To meet the requirement in \eqref{eqn:param-dual}, the term $x_\theta(s, a)$ should ideally approximate the $Z$-function well. Hence, one can replace $\left[ Z_{\mu} (\phi, \psi, \theta, s, a) \right]^2/2$ 
in \eqref{eqn:param-primal} by $x_{\theta}(s, a) Z_{\mu} (\phi, \psi, \theta, s, a)$, and enforce the matching between $x_{\theta}(s, a)$ and $Z_{\mu} (\phi, \psi,\theta, s, a)$ by adding a proximal term. This leads to the construction of a quadratic penalty function as follows:
\begin{equation}\label{eqn:rl-L}
\begin{split}
\hat{\mathbb{L}}_{\beta, \mu} (\phi, \psi, \theta) := &\sum_{s} \rho_0(s) V_{\phi}(s) + \frac{1}{\mu} \sum_{s, a} w(s, a) x_{\theta}(s, a) Z_{\mu}(\phi, \psi,\theta, s, a) \\
&+ \frac{\beta}{2} \sum_{s, a} w(s, a) \left[ x_{\theta}(s, a) - Z_{\mu}(\phi, \psi,\theta, s, a) \right]^2.
\end{split}
\end{equation}
However,  estimating the gradient with respect to $\phi$ is impractical, which is the well-known double-sampling obstacle.
In order to overcome the obstacle, we introduce a target value network $V_{\phi^\text{targ}}$ as the usage in DQN. Additionally, we also replace $x_{\theta^k}$ with a target multiplier network $x_{\theta^\text{targ}}$ which shares the same structure as that of $x_\theta$. We denote
\begin{align*}
z_{\mu}^{\text{targ}}(\phi, \psi, s, a, s') &= x_{\theta^{\text{targ}}}(s, a) + \mu \left( h_{\psi}(s, a) + r(s, a) + \gamma V_{\phi^{\text{targ}}}(s') - V_{\phi}(s) \right), \\
Z_{\mu}^{\text{targ}}(\phi, \psi, s, a) &= \mathbb{E}_{s'|s, a} \left[ z_{\mu}^{\text{targ}}(\phi, \psi, s, a, s') \right].
\end{align*}
By adding a term unrelated to $\phi$, $\theta$ and $\psi$, we can rearrange \eqref{eqn:rl-L} as
\begin{align*}
\hat{\mathbb{L}}_{\beta, \mu} (\phi, \psi, \theta) = \sum_{s} \rho_0(s) V_{\phi}(s) + \frac{1}{\mu} \sum_{s, a, s'} p_w(s, a, s') x_{\theta}(s, a) z_{\mu}^{\text{targ}}(\phi, \psi, s, a, s') \\
+ \frac{\beta}{2} \sum_{s, a, s'} p_w(s, a, s') \left( x_{\theta}(s, a) - z_{\mu}^{\text{targ}}(\phi, \psi, s, a, s') \right)^2,
\end{align*}
where $p_w(s,a,s')=w(s,a)P(s'|s,a)$. Then the gradient can be denoted as
\begin{align*}
e_i &= \beta \left( z_\mu^{\text{targ}} \left( \phi^l, \psi^l, s_i, a_i, s_i' \right) - x_{\theta^l}(s_i, a_i) \right), \\
g_i &= \nabla_\phi V_{\phi^l} (s_{0, i}) - \left( x_{\theta^l}(s_i, a_i) + \mu e_i \right) \nabla_\phi V_{\phi^l} (s_i), \\
q_i &= \left( x_{\theta^l}(s_i, a_i) + \mu e_i \right) \nabla_\psi h_{\psi^l}(s_i, a_i), \\
m_i &= \frac{1}{\mu} \left(z_\mu^{\text{targ}} \left(\phi^l, \psi^l, s_i, a_i, s_i' \right) - \mu e_i \right) \nabla_\theta x_{\theta^l}(s_i, a_i).
\end{align*}
The deep parameterized scheme for the ALM \eqref{eq:Vk}-\eqref{eq:xk} is to solve the optimization problem
\begin{equation}\label{eqn:rl-subprob}
\min_{\phi,\psi,\theta} \hat{\mathbb{L}}_{\beta, \mu} (\phi, \psi, \theta).
\end{equation}
Therefore, the stochastic gradient descent method can be applied to update the parameters as follows:
\begin{equation}\label{eqn:rl-alm-update}
\begin{split}
\phi^{l+1} &= \phi^l - \frac{\alpha_{\phi}}{b} \sum\nolimits_{i=1}^{b} g_i, \\
\psi^{l+1} &= \psi^l - \frac{\alpha_{\psi}}{b} \sum\nolimits_{i=1}^{b} q_i, \\
\theta^{l+1} &= \theta^l - \frac{\alpha_{\theta}}{b} \sum\nolimits_{i=1}^{b} m_i,
\end{split}
\end{equation}
where $\alpha_\phi , \alpha_\theta , \alpha_\psi > 0$ are step sizes.

\subsection{Distributed optimization}\label{sec:dis-opt}
Consider a network consisting of $N$ agents who collectively optimize the following problem
\begin{equation}\label{prob:disributed}
\min_{x \in \mathbb{R}^n}\ f(x):=\sum_{i=1}^N f_i(x),
\end{equation}
where $f_i(x): \mathbb{R}^n \rightarrow \mathbb{R}$ is a function local to agent $i$.  Suppose the agents are connected by a network defined by an undirected graph $\mathcal{G}=\{\mathcal{V}, \mathcal{E}\}$, with $|\mathcal{V}|=N$ vertices and $|\mathcal{E}| = E$ edges. The graph is associated with a $N\times N$ symmetric, doubly stochastic weight matrix $W$.  Each agent can only communicate with its immediate neighbors, and it is responsible for optimizing one component function $f_i$. This problem has found applications in various domains such as distributed consensus, distributed communication networking, distributed and parallel machine learning and distributed signal processing. The distributed version of problem \eqref{prob:disributed} is given as below
\begin{equation}\label{prob:disributed-1}
    \min_{x_1,\cdots,x_N}\ \sum_{i=1}^N f_i(x_i)\quad \mathrm{s.t.} \quad x_i = x_j, \forall (i,j) \in \mathcal{E}. 
\end{equation}
The standard distributed gradient method \cite{tsitsiklis1986distributed,nedic2009distributed} has the following update rule:
\begin{equation}\label{update:dgd}
    x_i^{k+1} = \sum_{j = 1}^N W_{i,j} x_{j}^k - \alpha \nabla f_i(x_{i}),~ i\in [N],
\end{equation}
where $\alpha$ is the stepsize. If we let $\mathbf{x} = [x_1;\cdots;x_N] \in \mathbb{R}^{nN \times 1}$, $f(\mathbf{x}): = \sum_{i=1}^N f_i(x)$ and $\mathbf{W} = W \otimes I_n$, \eqref{update:dgd} can be rewritten as  
\begin{equation}
    \mathbf{x}^{k+1} = \mathbf{W} \mathbf{x}^k - \alpha \nabla f(\mathbf{x}_k).
\end{equation}

In fact, \eqref{prob:disributed-1} is a constrained optimization problem, and can be solved by ALM. In particular, we rewrite \eqref{prob:disributed-1} as the following compact form:
\begin{equation}\label{pro:decentra}
    \min_{\mathbf{x} \in \mathbb{R}^{nN \times 1}} \ \sum_{i=1}^N f_i(x_i)\quad\mathrm{s.t.}\quad A\mathbf{x} = 0. 
\end{equation}
The matrix $A\in \mathbb{R}^{nN \times nN}$, and satisfy that the linear constraint  $A\mathbf{x} = 0$ is  equivalent to $x_1 = x_2= \cdots = x_N$. For example, $A = (I - \mathbf{W})$. 
A direct application of ALM for \eqref{pro:decentra} then has the following form:
\begin{equation}\label{eq:distributed-ALM}
    \left\{
\begin{aligned}
    \mathbf{x}^{k+1} & = \arg\min_{\mathbf{x}} \mathbb{L}_{\rho_k}(\mathbf{x},\lambda^k),\\
   \lambda^{k+1} &=\lambda^{k} + \rho_k A\mathbf{x}^{k+1}.
\end{aligned}
    \right.
\end{equation}
In the above algorithm, the quadratic term $\|A\mathbf{x}\|^2$ in the primal variable update couples all the agents, spoils the otherwise nice separable structure of the primal update and blocks a directed implementation of this method in a distributed setting \cite{jakovetic2020primal}. 

However, we can make some modifications for ALM to obtain distributed-friendly algorithms. The key is to decouple the quadratic term. In fact, there is a very high degree of flexibility of how one can perform the primal update and the dual update \cite{jakovetic2020primal}. 
An even more aggressive change to \eqref{eq:distributed-ALM} is to carry out but one gradient step from the current iterate, i.e., carrying out an Arrow-Hurwitz-Uzawa (AHU) algorithm. This also gives an advantage of avoiding ``double-looped'' methods, where primal and dual updates are carried out at different time scales. The AHU-type algorithm blueprint leaves us near the structure of efficient distributed first-order primal-dual methods such as EXTRA \cite{shi2015extra}, DIGing \cite{nedic2017achieving}, and their subsequent generalization in \cite{jakovetic2018unification}. More precisely, while EXTRA and DIGing have been developed from a perspective different than primal-dual, they are in fact instances of primal-dual methods as first shown in \cite{nedic2017achieving}. Moreover, \cite{hong2017prox} can also consider the proximal ALM framework to ensure that the primal problem is decomposable over different network nodes, hence distributionally implementable. 

To illustrate, we consider the case that the $\mathbf{x}$-subproblem is only updated by a single gradient step, %
i.e.,
$$
\left\{
\begin{aligned}
& \mathbf{x}^{k+1}=\mathbf{x}^{k}-\alpha \nabla_\mathbf{x} \mathbb{L}_{\rho_k}\left(\mathbf{x}^{k}, \lambda^k\right), \\
& \lambda^{k+1}=\lambda^k+\alpha \nabla_{\lambda} \mathbb{L}_{\rho_k}\left(\mathbf{x}^{k+1}, \lambda^k\right).
\end{aligned}
\right.
$$
where $\alpha$ is the stepsize. When $\rho_k = \rho = 1/\alpha$ and the matrix $A$ is chosen as $A = \rho\sqrt{I-\mathbf{W}}$,  the above iterative process can be rewritten as 
$$
\left\{
\begin{aligned}
\mathbf{x}^{0}= & W \mathbf{x}^{0}-\alpha \nabla f\left(\mathbf{x}^{0}\right), \\
\mathbf{x}^{k+1}= & 2 W \mathbf{x}^{k}-\alpha \nabla f\left(\mathbf{x}^{k}\right)-W \mathbf{x}^{k-1}  +\alpha \nabla f\left(\mathbf{x}^{k-1}\right), \quad k=1,2, \cdots.
\end{aligned}
\right.
$$
It corresponds to the EXTRA method proposed by \cite{shi2015extra}; see also \cite{jakovetic2018unification}, this is a method developed from the standpoint of correcting the fixed point equation of the standard distributed gradient method \cite{nedic2009distributed}. 
On the other hand, the methods in \cite{qu2017harnessing,nedic2017geometrically}, see also \cite{xu2015augmented,nedic2017geometrically}, utilize a gradient tracking principle and work as follows:
\begin{equation}\label{eq:distri:trcking}
\left\{
\begin{aligned}
\mathbf{x}^{k+1} & =W \mathbf{x}^{k}-\alpha \mathbf{s}^{k}, \\
\mathbf{s}^{k+1} & =W \mathbf{s}^{k}+\nabla f\left(\mathbf{x}^{k+1}\right)-\nabla f\left(\mathbf{x}^{k}\right),
\end{aligned}
\right.
\end{equation}
where $\mathbf{s}^{k}$ is an auxiliary variable that tracks the global average of the $f_i$ 's individual gradients along iterations. The method also accounts for a primal-dual interpretation, as shown in \cite{nedic2017achieving,nedic2017geometrically,jakovetic2018unification};  For example, following \cite{jakovetic2018unification}, the update rule \eqref{eq:distri:trcking} can be rewritten as follows (here, $\rho_k = 1/\alpha$):
$$
\left\{
\begin{aligned}
\mathbf{x}^{k+1} & =\mathbf{x}^{k}-\alpha \nabla_\mathbf{x} \mathbb{L}_{1/\alpha}\left(\mathbf{x}^{k}, \lambda^{k}\right), \\
\nu^{k+1} & =\nu^{k}+\alpha \nabla_\lambda \mathbb{L}_{1/\alpha}\left(\mathbf{x}^{k+1}, \lambda^{k}\right)  -\alpha W \nabla_\lambda \mathbb{L}_{1/\alpha}\left(\mathbf{x}^{k}, \lambda^{k}\right).
\end{aligned}
\right.
$$
More details are referred to \cite{jakovetic2018unification}.

\section{Numerical Experiments}
\label{sec-numerical-experiments} 
{We now illustrate the practical performance and numerical behavior of several representative methods, drawing on our own computational experience. This section serves only as a high-level overview; specific algorithmic settings, implementation details, and comprehensive results are provided in the corresponding references.}
\subsection{ALM for nonlinear programming}
{
We evaluate the performance of the ALM on a class of quadratic fractional programming problems given by:
\begin{equation}\label{sec8:eq:nlp}
\begin{aligned}
\min_{x \in \mathbb{R}^n} \quad &\sum_{i=1}^m \quad \left( \frac{x^\top a_i}{x^\top b_i} - \frac{x^\top c_i}{x^\top d_i} \right)^2, \\
\text{s.t.} \quad &  x^\top h_1 = 1,\ x^\top h_2 = 0, \
\ell \leq x \leq u,
\end{aligned}
\end{equation}
where $a_i, b_i, c_i, d_i \in \mathbb{R}^n$ for $i = 1, 2, \ldots, m$, and $\ell, u, h_1, h_2 \in \mathbb{R}^n$ are given constant vectors. This formulation can be equivalently cast into the general framework by setting $c(x) = [x^\top h_1,\ x^\top h_2]^\top$, with constraint set $\mathcal{K}_c = \{ [1, 0]^\top \}$ and variable bounds $\mathcal{K}_x = [\ell, u]$.

In the experiments, we generate 20 problem instances with randomly sampled data. The problem dimension is fixed at $n = 20$ and the number of fractional terms $m = 10000$.
For each instance, we apply the ALM-based methods where the AL function is minimized using the gradient method with the Barzilai–Borwein step sizes (ALM(GBB)), the limited-memory BFGS (ALM(L-BFGS)), the full BFGS (ALM(BFGS)), and the Newton-type updates (ALM(Newton)). These variants are designed to efficiently solve the constrained optimization problems. To benchmark the effectiveness of these ALM-based methods, we also include results from MATLAB’s built-in solver fmincon, using both the interior-point (IP) method and the sequential quadratic programming (SQP) method as baselines. These solvers are widely adopted in practice and serve as strong reference points in terms of solution quality and computational time. The performance results of the total running time and the number of function evaluations are summarized in Figure \ref{fig:ALM_nonlinear_perf}, where logarithmic scaling is applied to better visualize the differences across problem instances. The detailed numerical results can be found in \cite{zhou2024augmented}.

\begin{figure*}[!htb]
    \centering
      \subfigure[Running Time]{\includegraphics[scale = 0.42]{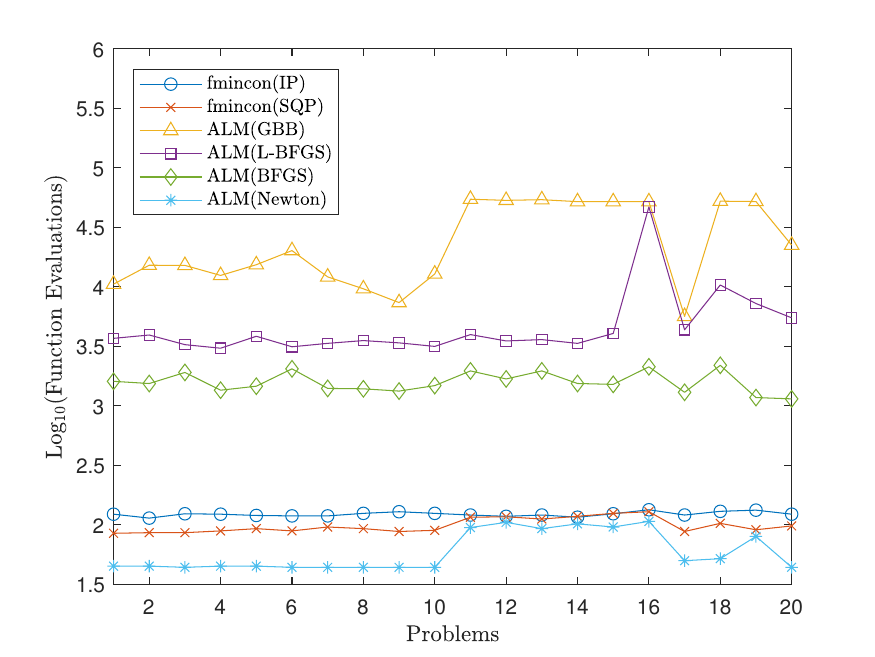}}
      \subfigure[Number of Function Evaluations]{\includegraphics[scale = 0.42]{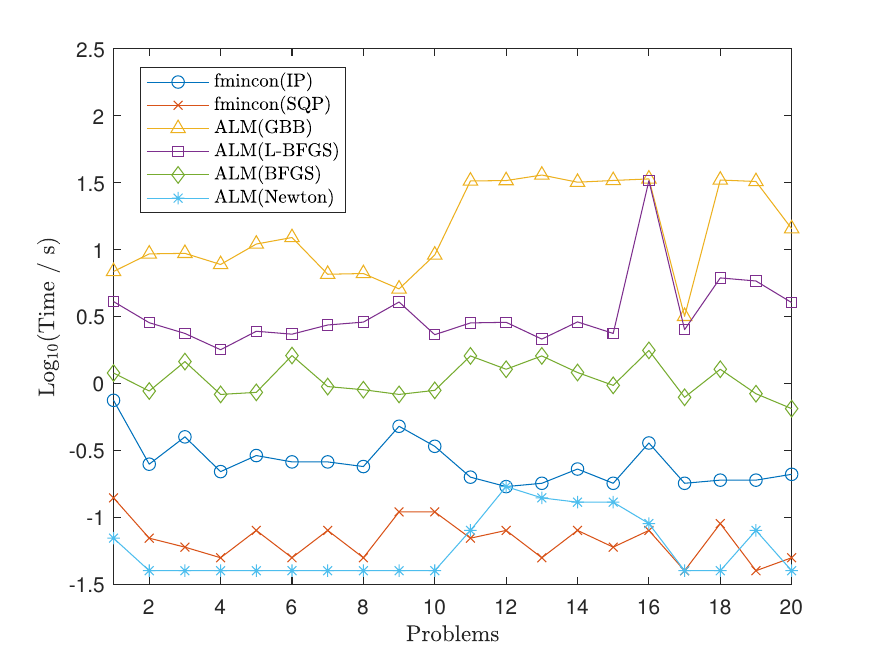}}
    \caption{Comparison of the performance of different algorithms on 20 problems in the form of \eqref{sec8:eq:nlp}}
      \label{fig:ALM_nonlinear_perf}
\end{figure*}

It is observed from these figures that the performance of ALM-based methods is competitive with that of fmincon. Among them, ALM(Newton) achieves the lowest overall running time across most problems, benefiting from the utilization of second-order information. ALM(L-BFGS) and ALM(BFGS) also show stable behavior and offer a good trade-off between computational cost and convergence behavior. In contrast, the ALM(GBB) requires more function evaluations in some cases, though it has limited memory requirements.  %
}

\subsection{An AL primal-dual first-order method}
To evaluate the performance of the primal dual methods discussed in Section \ref{sec:fully-primal-dual}, we consider the basis pursuit problem: 
\begin{equation}\label{bp}
        \min_x\ \|x\|_1,\quad \st\ Ax=b,
\end{equation}
where $A\in\mathbb{R}^{m\times n}$ is of full row rank and $b\in\mathbb{R}^m$.
By following the update rule given in \eqref{generalalgo}, we can easily get the primal-dual methods for solving \eqref{bp}.
Define the following metrics to describe the relative error between the iterate and the optimum:
\[
    \text{RelErr} = \frac{\|x-x^*\|_2}{\max(\|x^*\|_2,1)}.
\]

Our test problems \cite{milzarek2014semismooth} are constructed as follows.
Firstly, we create a sparse solution $x^* \in \mathbb{R}^{n}$ with $k$ nonzero entries, where $n = 512^2 = 262144$ and $k = [n/40] = 5553$. 
The $k$ different indices are sampled uniformly from $\{1,2,\cdots,n\}$.
The magnitude of each nonzero element is determined by $x^*_{i}=\eta_{1}(i) 10^{d \eta_{2}(i) / 20}$, 
where $d$ is a dynamic range, $\eta_{1}(i)$ and $\eta_{2}(i)$ are uniformly randomly chosen from $\{-1, 1\}$ and $[0, 1]$, respectively. 
The linear operator $A$ is defined as $m = n/8 = 32768$ random cosine measurements, i.e., $A x=(\operatorname{dct}(x))_{J}$, 
where $\operatorname{dct}$ is the discrete cosine transform and the set $J$ is a subset of $\{1, 2, \cdots, n\}$ with $m$ elements.

The relative error with respect to the total number of iterations for different problems are shown in Figure \ref{fig:bp_05_25}. 
To ensure a fair performance comparison among different algorithms, we select the step sizes $\tau$ and $\sigma$, along with the penalty factor 
$\rho$, from specific ranges and then choose the optimal parameter set for comparison. To investigate the impact of the penalty term in AL, we set the parameter $\rho=0$ in both CP and OGDA to contrast CP-AL and OGDA-AL. Furthermore, SOGDA-AL is a newly proposed method \cite{zhu2022unified} by setting $\mu=1, \alpha=0, \beta=1$ and $\rho>0$.
ADMM is implemented in the software YALL1 \cite{yang2011alternating}.
It is worth noticing that the subproblem of ADMM for the dual problem can be solved exactly since $AA^T=I$.

The penalty term potentially accelerates the primal-dual method, especially for OGDA and high accuracy solution.
SOGDA-AL outperforms other algorithms in most of the problems.
By contrast, although ADMM solves the subproblem exactly and adopts the strategy of updating the penalty parameters, it has no obvious advantage over the other algorithms.
In addition, the linear convergence rate of the primal-dual methods can be observed, as we have proved.

\begin{figure*}[!htb]
    \centering
      \subfigure[20dB]{\includegraphics[scale = 0.42]{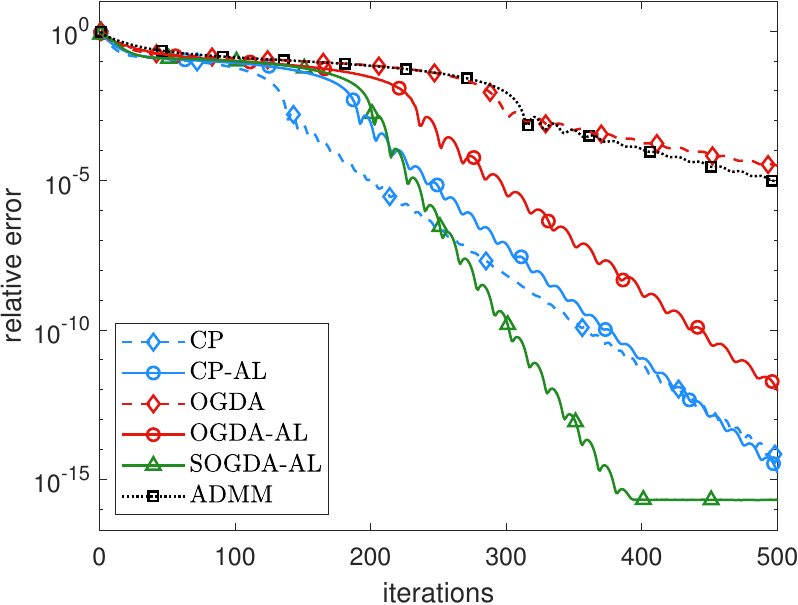}}
      \subfigure[40dB]{\includegraphics[scale = 0.42]{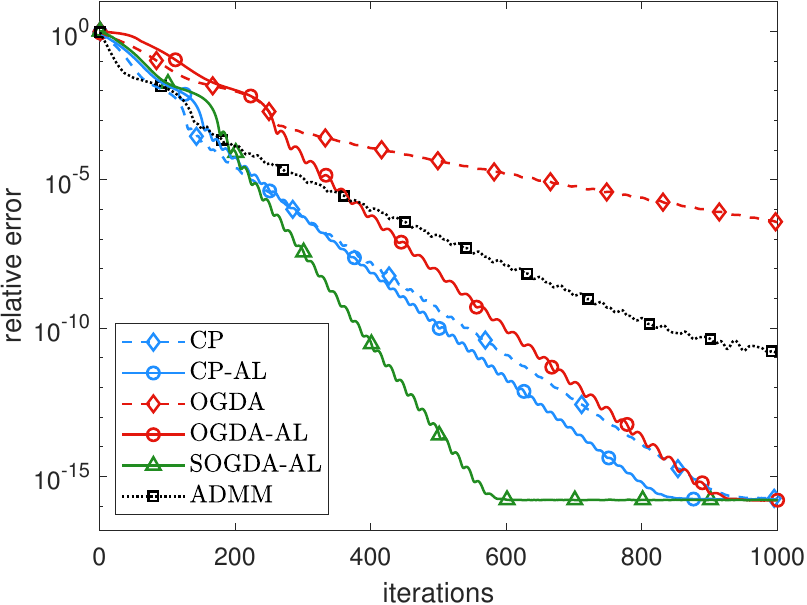}}
    \caption{Comparisons among fully primal-dual methods on basis pursuit problems.}
      \label{fig:bp_05_25}
\end{figure*}

\subsection{An AL decomposition method for SDP}
{ In this subsection, we evaluate the performance of ALM for solving  SDP, as discussed in Section \ref{sec7:nonsmooth-sdp}. Specifically, we focus on the semidefinite relaxation of the Max-Cut problem augmented with a cutting-plane procedure.
Given an undirected graph with $n$ nodes, the SDP relaxation of the Max-Cut problem can be formulated as
\begin{equation}\label{eqn:mc-cut}
    \min  -\frac{1}{4}\operatorname{Tr}(CX), \;
        \operatorname{ s.t. }  \; \diag(X) = e,\; \mathcal{A}(X) \geq -e, \; X\succeq 0,
\end{equation}
where $C$ is the graph Laplacian matrix, and $\mathcal{A}(X) \geq -e$ stands for
a subset of the following cutting planes: $\forall~ 1 \leq i < j < k \leq n,$
\[
\begin{aligned}
    X_{ij} + X_{ik} + X_{jk} &\geq -1,
    X_{ij} - X_{ik} - X_{jk} &\geq -1, \\
    -X_{ij} + X_{ik} - X_{jk} &\geq -1,
    -X_{ij} - X_{ik} + X_{jk} &\geq -1.
\end{aligned}
\]
In order to compute $\mathcal{A}(X)$, we only need to collect the related components $X_{ij}$ that appear in $\mathcal{A}$ instead of
forming $X = R^TR$ explicitly.
A complete \texttt{Gset} %
dataset is tested. We compare the accuracy and efficiency of SDPDAL with that of SDPNAL+ using the performance profiling method proposed in \cite{dolan2002benchmarking}.
Let $t_{p,s}$ be some performance quantity (e.g. time or accuracy, lower is better) associated with the $s$-th solver on problem $p$.
Then one computes the ratio $r_{p,s}$ between $t_{p,s}$ over the smallest value obtained by $n_s$ solvers on problem $p$, i.e., $r_{p,s} :=\frac{t_{p,s}}{\min\{t_{p,s}: 1\leq s \leq n_s\}}$. For $\tau >0$, the value
$
\pi_{s}(\tau) : = \frac{\text{number of problems where } \log_2(r_{p,s}) \leq \tau}{\text{total number of problems}}
$
indicates that solver $s$ is within a factor $2^\tau \geq 1$ of the performance obtained by the best solver. Then the performance plot is a curve $\pi_s(\tau)$ for each solver $s$ as a function of $\tau$. In Figure \ref{fig:mc-perf}, we show the performance profiles of two criteria: error induced by the KKT condition and CPU time.
In particular, the intercept point of the axis ``ratio of problems'' and the curve in each subfigure is the percentage of the slower/faster one between the two solvers. These figures show that the accuracy and the CPU time of SDPDAL are better than SDPNAL+ on most problems.

\begin{figure}[!htb]
\centering
\subfigure[error]{
\includegraphics[width=0.45\textwidth]{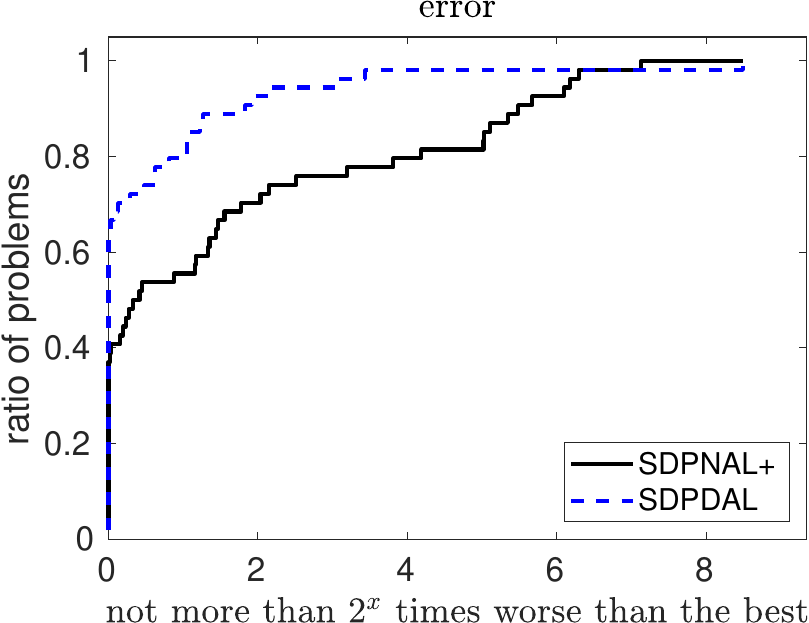}}
\subfigure[CPU]{
\includegraphics[width=0.45\textwidth]{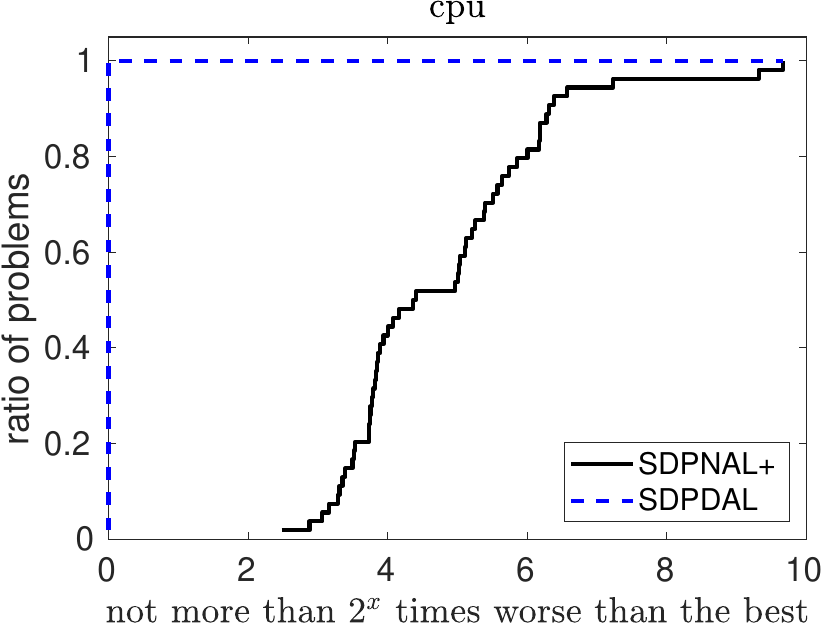}}
\caption{The performance profiles of SDPDAL and SDPNAL+ for Max-Cut problems \emph{(\texttt{G01--G54})}}\label{fig:mc-perf}
\end{figure}
}

\subsection{An AL primal-dual semismooth Newton method}
{In this subsection, we evaluate the efficiency of the AL primal-dual semismooth Newton method for solving the multi-block convex composite optimization problem in Section \ref{sec7:multi-block}. In particular, we consider a specific example of problem \eqref{sec7:prob:multiblock}:
\begin{equation} \label{image}
\min_{\bm{u}\in \mathbb{R}^n} \quad \frac{1}{2}\|\bm{A}\bm{u}-\bm{b}\|^2 + \delta_{\mathcal{Q}}(\bm{u}) + \lambda \|\bm{D}\bm{u}\|_1,    
\end{equation}
where $\bm{b}\in \mathbb{R}^m$ is the observed term, $\mathcal{Q} = \{ \bm{u}\in \mathbb{R}^n:0 \le \bm{u} \le 255\}   $ is the pixel range,  $\bm{A} \in \mathbb{R}^{m \times n}$ is a given linear operator, $\bm{D}$ is a given image transform operator, such as the TV operator or Wavelet transform, and $\| \cdot \|_1$ represents the $\ell_1$ norm. When $\bm{A}$ is the identity operator, then the corresponding problem is the image denoising problem.   When $\bm{A}$ is the blurring operator, it leads to the image deblurring problem.  When $\bm{A}$ is the Radon or Fourier transform, the corresponding problem is usually the CT/MRI reconstruction problem. As a consequence, problem \eqref{image} contains various kinds of image restoration problems. The dual problem of \eqref{image} is
\begin{equation} \label{dual:image2}
\begin{aligned}
\min_{\bm{y},\bm{\mu},\bm{\nu}} \quad & \frac{1}{2}\|\bm{y}\|^2+\iprod{\bm{y}}{\bm{b}} + \delta^*_{\mathcal{Q} }(\bm{\nu}) + \delta_{\|\cdot\|_{\infty}<\lambda}(\bm{\mu}),\\
\st \quad & \bm{A}^{\top}\bm{y}+\bm{D}^{\top}\bm{s}+\bm{\nu} =0, \quad \bm{s}=\bm{\mu},
\end{aligned}
\end{equation}
where $\|\cdot \|_{\infty}$ denotes the $\ell_{\infty}$ norm.

In this example, we test the CT image restoration problem. We  generate the data from 180 angles at
$1^{\circ}$ increments from  $1^{\circ}$  to  $180^{\circ}$ with equidistant parallel X-ray beams and then choose 50 angles randomly.  
The sizes of tested classic images Phantom\footnote{\href{https://www.kaggle.com/datasets/kmader/siim-medical-images?resource=download}{ https://www.kaggle.com/datasets/kmader/siim-medical-images}} are 128 $\times$ 128. 
Then the corresponding operator matrix   
$\bm{A} \in \mathbb{R}^{9250 \times 16384}$ represents a sampled radon transform, while $\bm{D}: \mathbb{R}^{128 \times 128} \rightarrow \mathbb{R}^{256 \times 128}$ corresponds to the anisotropic total variation regularization. The sampling process and the images recovered by the filter-backprojection (FBP) method are visualized in Figure \ref{fig:FBP1}, which demonstrate the effectiveness of our model.  In numerical experiments, we choose $\lambda = 0.1,0.05$.   We
define the relative function error, $\mbox{ferror}:= \frac{f-f^*}{f}$  between the current objective function value $f$ and the optimal function value $f^*$, where $f^*$ is obtained by stopping PD3O for 10000 iterations.  We compare our algorithm with {first-order methods}, PD3O \cite{yan2018new},   PDFP \cite{chen2016primal},  and Condat-Vu  \cite{condat2013primal}. 
The numerical results
are shown in Table \ref{tab:image1}, where  ``iter'' denotes the number of iterations, ``num'' represents the number of matrix-vector multiplications of $\bm{Au}$, and ``time'' is the total wall-clock time. We present the iterative trajectories of all algorithms in
Figure \ref{fig:image1}. %
We observe that ALPDSN requires significantly less time than the other three algorithms to achieve a satisfactory {peak signal-to-noise ratio (PSNR)}. Additionally, ALPDSN converges the fastest in terms of error.
} %
\begin{figure}[htbp]   
\centering
\begin{minipage}{0.48\textwidth}
    \centering
\subfigure[Radon transforms of Phantom]{\includegraphics[width=0.45\textwidth]{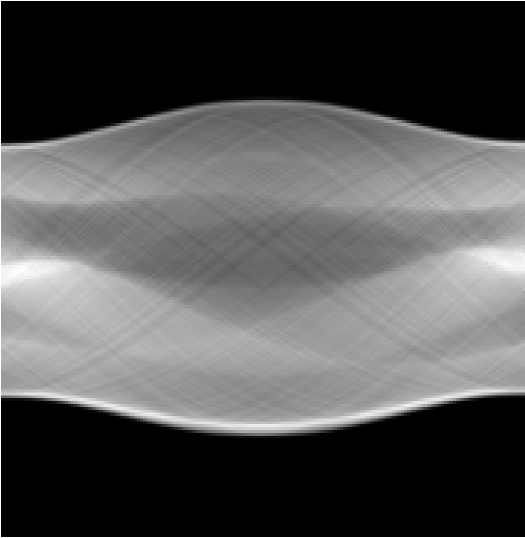}}
    \hspace{0.1in} %
\subfigure[ Sampled
Radon transforms]{\includegraphics[width=0.45\textwidth]{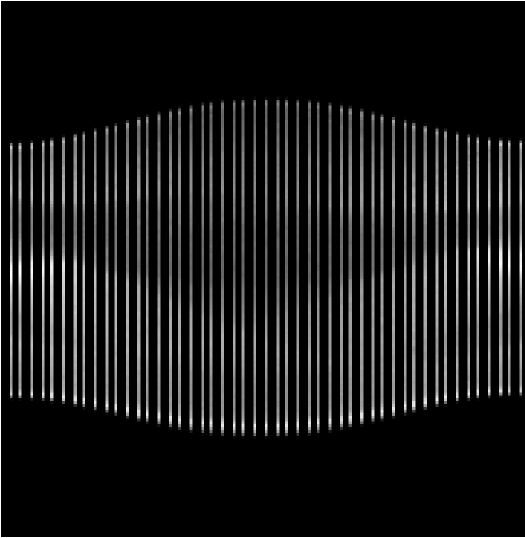}}
\end{minipage}
\hfill
\begin{minipage}{0.48\textwidth}
    \centering
\subfigure[Restored image by FBP ]{\includegraphics[width=0.45\textwidth]{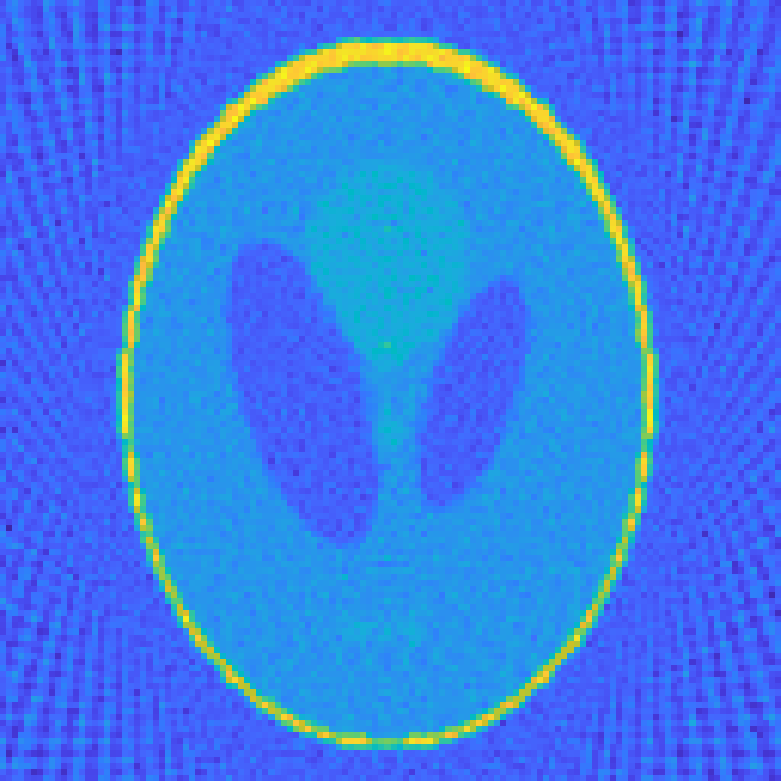}}
    \hspace{0.1in} %
\subfigure[Restored image by ALPDSN]{\includegraphics[width=0.45\textwidth]{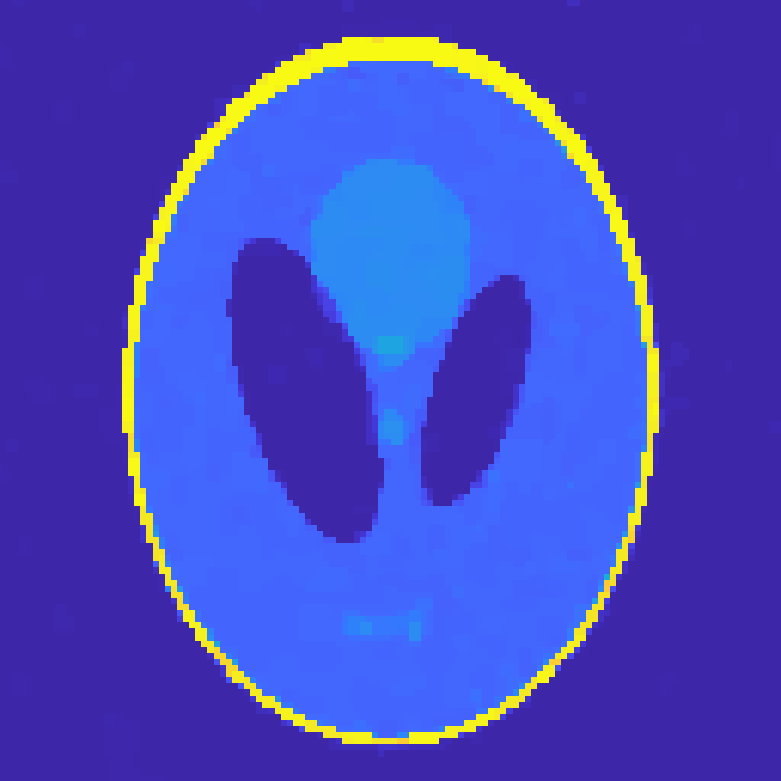}}
\end{minipage}
\caption{Phantom: the PSNR of FBP and  ALPDSN are $19.26$, $35.37$, respectively.} 
\label{fig:FBP1}
\end{figure}

{\renewcommand{\arraystretch}{1}{
\begin{table}[t]
\caption{
Summary on the CT image restoration problem}\label{tab:image1}
  \centering
\begin{tabular}{|c|c|c|c|c|c|}
\cline{1-6}
Image Name &  &ALPDSN & PD3O & PDFP & Condat-Vu \\ \cline{1-6}
\multirow{4}{*}{Phantom($\lambda = $0.1)}& iter & 346 &2689&3196 &4678 \\ \cline{2-6}
 & num &2506 & 2689 & 3185 & 4686 \\\cline{2-6}
 & ferror & 9.45e-8 & 9.98e-8 & 9.97e-8 & 9.99e-07  \\ \cline{2-6}
&time &\textbf{14.42}&19.80 & 22.85 & 35.19\\        \cline{1-6}

\multirow{4}{*}{Phantom($\lambda = $0.05)}& iter & 404 &2992&3571 &5312 \\ \cline{2-6}
 & num &2125 & 2992 & 3571 & 5312 \\\cline{2-6}
 & ferror & 7.39e-8 & 9.98e-8 & 9.97e-8 & 9.98e-08  \\ \cline{2-6}
&time &\textbf{14.05}&22.23 & 26.35 & 37.58\\        \cline{1-6}

\end{tabular}
\end{table}
}

\begin{figure}[htbp]
    \centering
      \subfigure[ferror vs time]{\includegraphics[scale = 0.32]{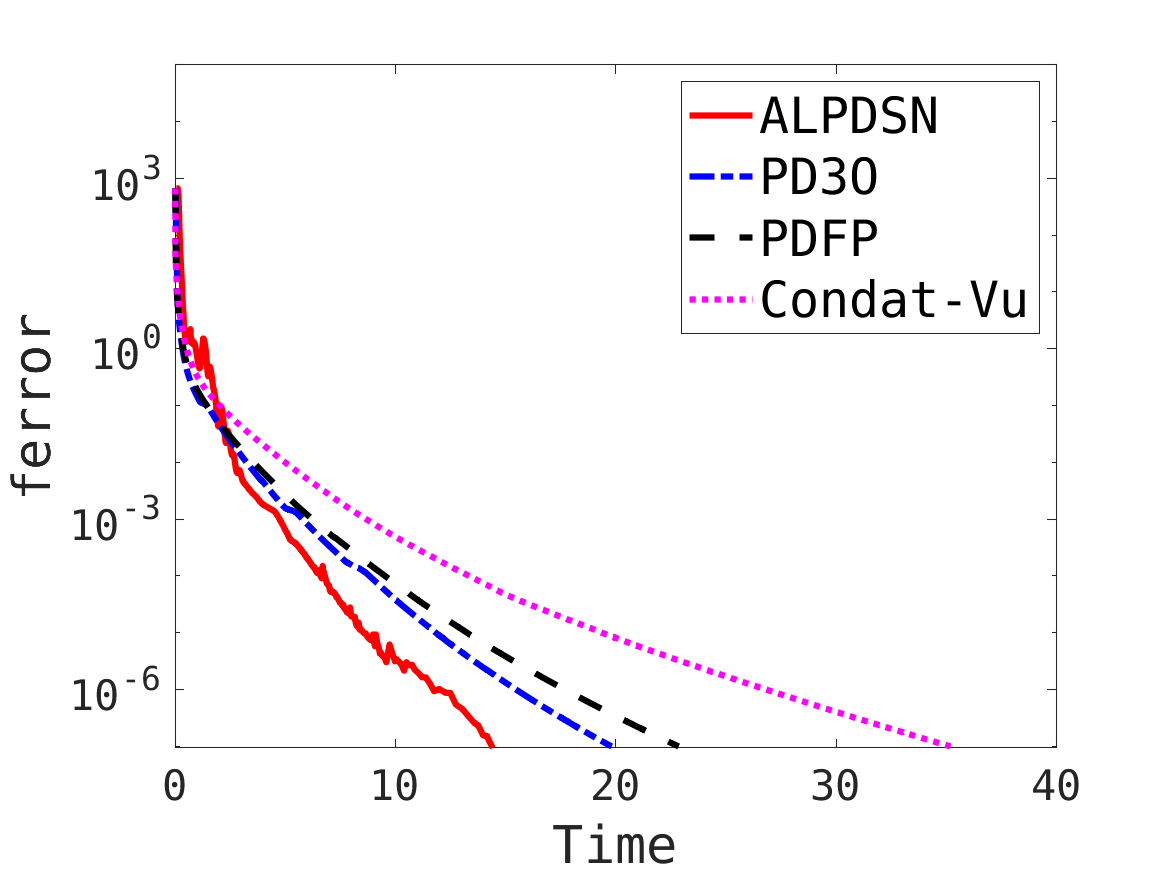}} 
      \subfigure[PSNR vs time]{\includegraphics[scale = 0.32]{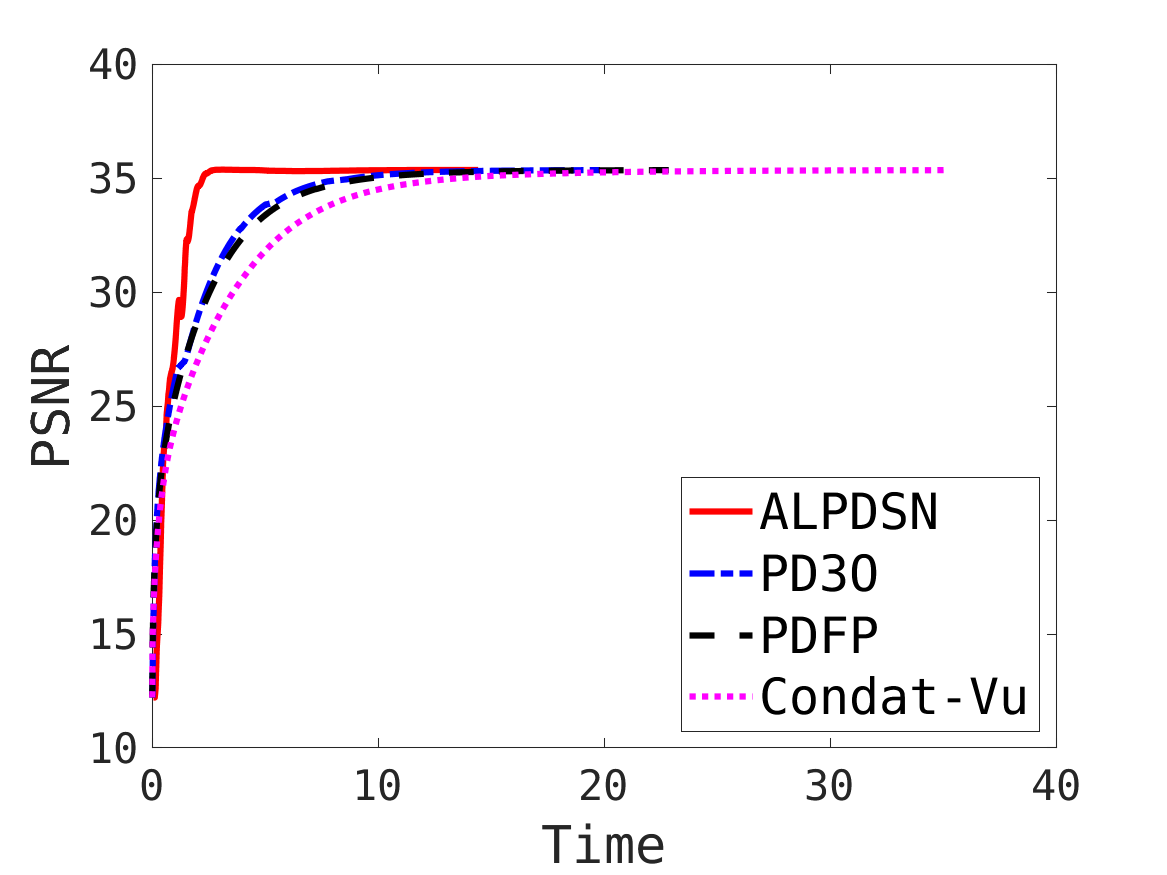}}
    \caption{Comparison between ALPDSN and other first-order methods on Phantom. %
    }
      \label{fig:image1}      
\end{figure}

\subsection{ALM for reinforcement learning}
We assess the stability and efficiency of the composite augmented Lagrangian algorithm (SCAL), introduced in Section~\ref{sec:RL}, on a suite of continuous control tasks from the OpenAI Gym benchmark~\cite{brockman2016openai}, simulated in MuJoCo~\cite{todorov2012mujoco}. These environments cover a range of dynamic behaviors and state–action dimensionalities, thereby serving as a rigorous testbed for reinforcement learning methods. As DQN is not applicable in continuous settings, we compare SCAL against two state-of-the-art baselines: trust region policy optimization (TRPO) and proximal policy optimization (PPO).

For the SCAL method, the update rule follows \eqref{eqn:rl-alm-update}. We parameterize the value network \( V_\phi \), the slack network \( h_\psi \), and the multiplier network \( x_\theta \) using fully connected neural networks. To enhance computational efficiency and stability, we employ several well-established techniques, including proportional sampling, Adam-based learning rate annealing, and local gradient clipping. All algorithms are implemented using the OpenAI Baselines toolkit~\cite{baselines}, which provides standardized environment wrappers to ensure fair comparisons. As noted in~\cite{henderson2017deep}, such implementation choices significantly influence performance. Additional experimental details are provided in \cite{li2023rl}.  

The results are summarized in Figure~\ref{fig:comparisons}, where solid lines denote mean performance and shaded regions represent the standard deviation across five random seeds. The horizontal axis tracks the number of environment interactions. SCAL consistently outperforms TRPO and PPO, achieving higher average rewards for the same number of interactions.

\begin{figure}[tbhp]
    \centering
\includegraphics[scale = 0.38]{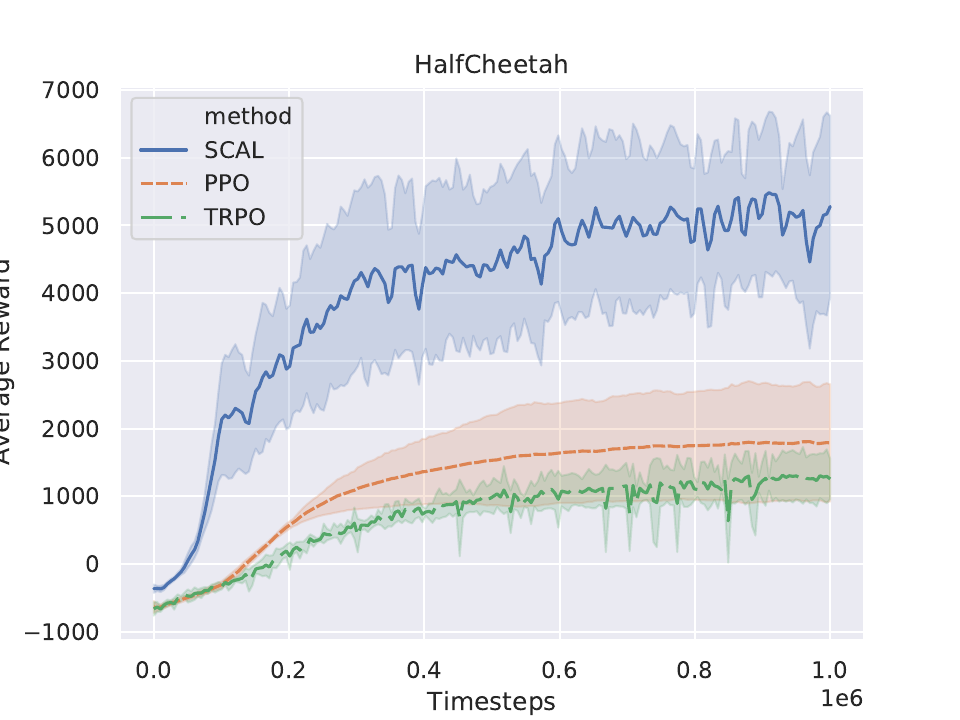}
\includegraphics[scale = 0.38]{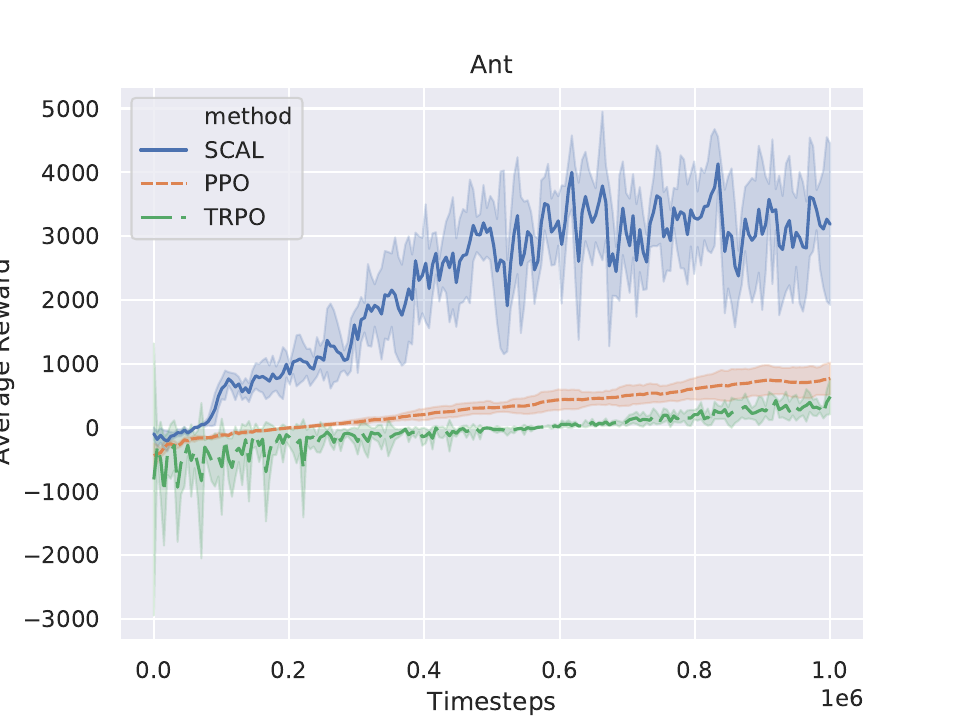}
\caption{Comparisons among several RL algorithms.}
\label{fig:comparisons}
\end{figure}

\subsection{ALM for block-structured integer programming}
{To show the practical performance of the ALM in discrete optimization, we present  numerical experiments on the following space-time network model for the train timetabling problem (TTP) \cite{caprara2002modeling}: 
\begin{subequations}\label{sec8:STMN}
  \begin{align}
    \label{eq.g.obj}    \max   \        & \sum\nolimits_{j \in N_p} \sum\nolimits_{e \in E^j} p_e x_e         &                                                                       \\
    \label{eq.g.start}     \text{ s.t. }\ & \sum\nolimits_{e \in \delta_{j}^{+}(\sigma)} x_e \le 1, & j \in N_p                                                          \\
    \label{eq.g.flow}                  &  \sum\nolimits_{e \in \delta_{j}^{-}(v)} x_e=\sum\nolimits_{e \in \delta_{j}^{+}(v)} x_e, & j \in N_p,   v \in V \backslash\{\sigma, \tau\} \\
    \label{eq.g.end}                   &  \sum\nolimits_{e \in \delta_{j}^{-}(\tau)} x_e \le 1, & j \in N_p                                                             \\
    \label{eq.g.interval}              &  \sum\nolimits_{v' \in \mathcal N(v)} \sum\nolimits_{j \in \mathcal{T}(v')} \sum\nolimits_{e \in \delta_{j}^{-}(v')} x_e \le 1, & v \in V         \\
    \label{eq.g.clique}                &  \sum\nolimits_{e \in C} x_e \le 1, & C \in \mathcal{C}                                                                         \\
    \label{eq.g.binary}                &  x_e \in \{0, 1\}, & e \in E.
  \end{align}
\end{subequations}
In this model,
$V$ and $E$  denote the set of all nodes and the set of all arcs. For each train $j\in N_p$, the sets or parameters indexed by $j$ correspond specifically to the route or variables associated with train $j$.  
A binary decision variable  $x_e$ is introduced for each arc $e \in E_j$, indicating whether arc $e$ is included in the routing plan for train $j$. 
 For each node $v \in V$, let $\delta^+_j(v)$  and  $\delta^-_j(v)$ be the sets of arcs in $E_j$ leaving and entering node $v$, respectively.
 $p_e$ is the ``profit'' of using a certain arc $e$. $\sigma$ and $\tau$ denote artificial origin and destination nodes, respectively.   $\mathcal{T}(v)$ and $\mathcal{N}(v)$ denote the set of trains may passing through node $v$ and the set of nodes conflicted with node $v$, respectively. $\mathcal{C}$ denotes the (exponentially large) family of maximal subsets $C$ of pairwise incompatible arcs. 

We can rewrite this space-time network model in the general form as (\ref{ip_inequality}).  
We aim to demonstrate that the ALM can generate feasible and high-quality timetables for real-world railway systems. Our numerical experiments focus on the Jinghu (Beijing–Shanghai) high-speed railway, one of the busiest high-speed lines globally, which served over 210 million passengers in 2019\footnote{Source: https://en.wikipedia.org/wiki/Beijing-Shanghai\_high-speed\_railway.}. The case study involves 292 trains in both directions over 29 stations, with two primary speed levels (300 km/h and 350 km/h). To evaluate performance, we compare the ALM-based methods \cite{rui2024} with the commercial solver Gurobi and the ADMM  \cite{ZHANG2020102823} on real-world instances, as presented in Table \ref{mpnsize}. The problem parameters include the number of trains ($|F|$), stations ($|S|$), and time windows ($|T|$). For large instances where Gurobi fails to return optimal solutions within two hours, we set this as the maximum computation time. We let ALM-GC and ALM-GP denote the ALM's subproblem using the
classical update \eqref{xjko} and proximal linear update \eqref{xjkl}.

 The model \eqref{sec8:STMN} contains tens of millions of variables and constraints, which highlights the scalability challenges of this application. 
Table \ref{tab:performance_BCD_LU} shows the performance results of our proposed methods and the Gurobi solver. In this table, UB corresponds to the objective values of feasible solutions generated by these compared methods, OV denotes the optimal value.  gap$=|\text{UB-OV}|/|\text{OV}|$. The detailed numerical results can be found in \cite{rui2024}.   

\begin{table}[h]
  \caption{Five examples of the large practical railway network}
  \label{mpnsize}
  \centering
  \begin{tabular}{ccccccc}
    \toprule
    \multirow{1}{*}{No.} & \multirow{1}{*}{$|F|$} & \multirow{1}{*}{$|S|$} & \multirow{1}{*}{$|T|$} & \multirow{1}{*}{$m$} & \multirow{1}{*}{$q$} & \multirow{1}{*}{$n$}\cr
    \midrule
    1                    & 50                     & 29                     & 300                    & 49,837               & 231,722              & 308,177 \cr
    2                    & 50                     & 29                     & 600                    & 113,737              & 1,418,182            & 1,396,572 \cr
    3                    & 100                    & 29                     & 720                    & 145,135              & 1,788,088            & 4,393,230 \cr
    4                    & 150                    & 29                     & 960                    & 199,211              & 3,655,151            & 9,753,590 \cr
    5                    & 292                    & 29                     & 1080                   & 228,584              & 13,625,558           & 26,891,567 \cr
    \bottomrule
  \end{tabular}
\end{table}

\renewcommand{\arraystretch}{1} %
\begin{table*}[h]
  \caption{Performance comparison between the Gurobi solver, ADMM, ALM-GP and ALM-GC on the large networks for maximizing the number of trains in the timetable in Table \ref{mpnsize}}
  \setlength{\tabcolsep}{3pt}
  \label{tab:performance_BCD_LU}
    \centering
  \begin{threeparttable}
    \centering
    \begin{tabular}{ccccccccccccccc} 
      \toprule
      \multirow{2}{*}{No.} & \multirow{2}{*}{OV} & \multicolumn{3}{c}{ Gurobi} & \multicolumn{3}{c}{ADMM} & \multicolumn{3}{c}{ ALM-GP} & \multicolumn{3}{c}{ ALM-GC}\cr
      \cmidrule(lr){3-5} \cmidrule(lr){6-8} \cmidrule(lr){9-11} \cmidrule(lr){12-14}
                           &                     & UB                          & gap                      & Time                        & UB                             & gap  & Time    & UB            & gap & Time          & UB            & gap & Time \cr
      \midrule
      1                    & --                  & \textbf{27}                & --                       & 3836.0                     & \textbf{27}                   & --   & 3.4    & \textbf{27}  & --  & \textbf{3.4} & \textbf{27}  & --  & 3.5 \cr
      2                    & 50                 & \textbf{50}                & \textbf{0}                        & 2965.8                     & \textbf{50}                   & \textbf{0}    & 8.2    & \textbf{50}  & \textbf{0}   & \textbf{8.1} & \textbf{50}  & \textbf{0}   & 8.2   \cr
      3                    & 100                & 11                         & 89.0\%                    & 7200.0                     & 89                            & 11.0\%   & 3001.4 & \textbf{100} & \textbf{0}   & 127.4        & \textbf{100} & \textbf{0}   & \textbf{96.4} \cr
      4                    & 150                & 14                         & 90.7\%                    & 7200.0                     & 148                           & 1.3\% & 3005.3 & \textbf{150} & \textbf{0}   & 99.5         & \textbf{150} & \textbf{0}   & \textbf{96.4} \cr
      5                    & 292                & 30                         & 89.7\%                    & 7200.0                     & 286                           & 2.1\% & 3003.8 & \textbf{292} & \textbf{0}   & 203.1        & \textbf{292} & \textbf{0}   & \textbf{172.1} \cr
      \bottomrule
    \end{tabular}
  \end{threeparttable}
\end{table*}

The results show that the proposed methods achieve comparable or better objective values than Gurobi, while significantly reducing computation time. These gaps further indicate the solution quality and convergence behavior. Notably, when using the simplified objective function, the ALM achieves near-optimal timetables with much less computational effort, suggesting its practicality for real-time or large-scale deployment. 
}

\subsection{ALM for nonsmooth manifold optimization}
{In this section, we demonstrate the performance of the ALM in nonsmooth manifold optimization, as discussed in Section \ref{sec7:manifold-optimization}.  In particular, we consider the sparse principal component analysis problem.   Given a data set $\{b_1,\ldots, b_m\}$ where $b_i\in\mathbb{R}^{n\times 1}$, the sparse principal component analysis (SPCA) problem is
\begin{equation}\label{spca}
\begin{aligned}
  \min_{X\in\mathbb{R}^{n\times r}}   \sum_{i=1}^{m}\|b_i - XX^Tb_i\|_2^2 + \mu \|X\|_1, ~ ~ \mbox{s.t. }~~  X^TX = I_r,
  \end{aligned}
\end{equation}
where $\mu > 0$ is a regularization parameter. Let $B = [b_1,\cdots, b_m ]^{T}\in\mathbb{R}^{m\times n}$, problem \eqref{spca} can be rewritten as: %
\begin{equation}\label{spcaM}
  \begin{aligned}
     \min_{X\in\mathbb{R}^{n\times r}}   -\mathrm{tr}(X^TB^TBX) + \mu \|X\|_1,     ~~  \mbox{s.t. }~~  X^TX = I_r.
  \end{aligned}
\end{equation}
Here, the constraint consists of the Stiefel manifold $\texttt{St}(n,r):=\{X\in\mathbb{R}^{n\times r}~:~X^\top X = I_r\}$, which is compact. Specifically, the closedness is deduced from the continuity of the map \( X^\top X - I \) and the fact that \( \{O\} \) is closed in \( \mathbb{R}^{r \times r} \). The boundedness directly follows from the norm constraint \( \|X\| \leq \sqrt{r} \).
The tangent space of $\texttt{St}(n,r)$ is defined by $T_{X}\texttt{St}(n,r) = \{\eta\in \mathbb{R}^{n\times r}~:~X^\top \eta + \eta^\top X = 0\}$. Given any $U\in\mathbb{R}^{n\times r}$, the projection of $U$ onto $T_{X}\texttt{St}(n,r)$ is $\mathcal{P}_{T_{X}\texttt{St}(n,r)}(U) = U - X \frac{U^\top X + X^\top U}{2}$ \cite{absil2009optimization}. In our experiment, the data matrix $B\in\mathbb{R}^{m\times n}$ is produced by MATLAB function $\texttt{randn}(m, n)$, in which all entries of $B$ follow the standard Gaussian distribution. We shift the columns of $B$ such that they have zero mean, and finally the column vectors are normalized. We use the polar decomposition as the retraction mapping. 

We compare the ALM proposed in \cite{deng2024oracle}, including the deterministic variants {ManIAL-I} and {ManIAL-II}, as well as the stochastic version {stoManIAL},  with  the subgradient method (we call it Rsub) in \cite{li2021weakly} that achieves the state-of-the-art oracle complexity for solving \eqref{spcaM} and the proximal gradient algorithm ManPG \cite{chen2020proximal}.  We compare the performance of those algorithms on the SPCA problem with random data, i.e.,  we set $m=5000$. For the {StoManIAL}, we partition the 50,000 samples into 100 subsets, and in each iteration, we sample one subset. The tolerance is set $\mathrm{tol} = 10^{-8}\times n\times r$.   Figure \ref{fig:2} presents the results of the four algorithms for fixed $n=1000$ and varying $r=10,20$ and $\mu=0.4,0.6,0.8$. The horizontal axis represents CPU time, while the vertical axis represents the objective function value gap: $F(x^k) - F_M$, where $F_M$ is given by the {ManIAL-I}. The results indicate that {StoManIAL} outperforms the deterministic version. Moreover, among the deterministic versions, {ManIAL-I} and {ManIAL-II} have similar performance. In conclusion, compared with the other algorithms, our algorithms achieve better performance. %

\begin{figure}[htpb]
\centering
\setlength{\abovecaptionskip}{0.cm}
\subfigure[$r = 10, \mu = 0.4$]{
\includegraphics[width=0.45\textwidth]{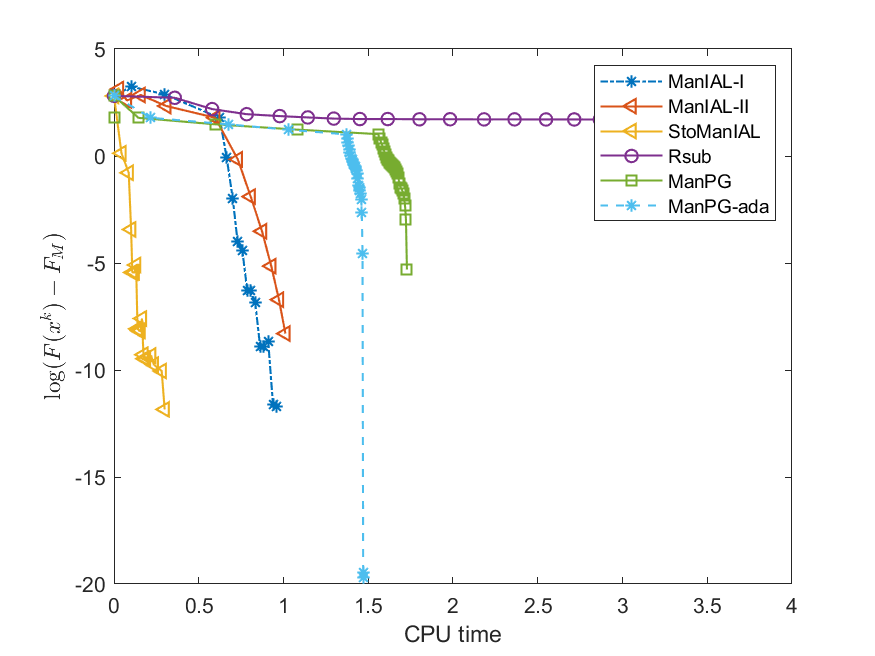}}
\subfigure[$r = 20, \mu = 0.4$]{
\includegraphics[width=0.45\textwidth]{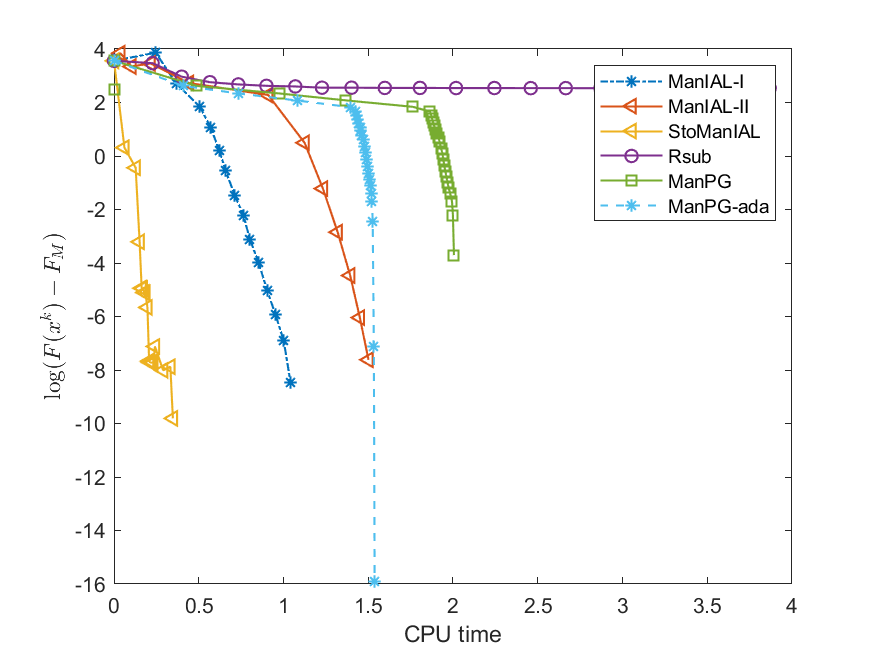}} %
\caption{The performance on SPCA with random data $(n = 1000)$}
\label{fig:2}
\end{figure}

}

\section{Summary and Perspectives}\label{sec:summary}
The augmented Lagrangian function forms the core foundation of many constrained optimization algorithms in the field of continuous optimization. In this review, we provide a detailed examination and summary of its applications and importance. First, we utilized the Moreau envelope to provide a unified and clear construction method for the different variants of augmented Lagrangian functions. We also extended the strong duality principle to this function, allowing us to reintroduce and understand the augmented Lagrangian function method and other related algorithms from the perspective of the dual ascent method. Then we summarize the latest convergence analyses for both convex and non-convex problems separately, facilitating comparisons for the readers. Finally, we present several practical examples and demonstrate how the ALM can be applied to solve these examples,  and how algorithms for the corresponding subproblem are designed. 

Overall, the ALM continues to be a robust and adaptable tool in the optimization toolkit, with substantial theoretical and practical advancements. It is still an active topic for many interesting applications, such as, bilevel optimization \cite{xu2014smoothing, tsaknakis2022implicit, lu2022stochastic}, minimax optimization \cite{dai2022rate, lu2023first} and some specific applications like training neural networks under constraints \cite{dener2020training,kotary2024learning,wang2025augmented,basir2023adaptive,son2023enhanced,CiCP-31-3,zhang2025physics,franke2024improving}. Moreover, in the context of large-scale data applications, single-loop ALM \cite{kim2024fast,liao2025bundle,lu2022single,shi2025momentum,alacaoglu2024complexity} and stochastic ALM \cite{li2024stochastic, zhang2022solving} become prominent directions of future research.  We believe that there are many unexplored and valuable research directions concerning the ALM, including but not limited to the following points:
 convergence analysis with constant penalty parameter in nonconvex case \cite{alacaoglu2024complexity,xie2021complexity,birgin2025global}, 
   trade off between  operation complexity and iteration complexity (reasonable approximation of Hessian matrix), 
  better combination between second-order and structural information / faster subproblem solver.

\section*{Acknowledgments}
 The authors are grateful to the associate editor Prof. Francesco Tudisco and the anonymous referees for their detailed and valuable comments and suggestions. K. Deng was supported by the National Natural Science Foundation of China under the grant number 12401419. R. Wang was supported by the National Natural Science Foundation of China under the grant number 12401414. 
 Z. Wen was supported in part by the National Natural Science Foundation of China under the grant numbers 12331010 and 12288101, and National Key Research and Development Program of China under the grant number 2024YFA1012903.

\appendix
\section{Proofs}\label{appen:proof}

\subsection{Proofs in Section \ref{sec2:AL:sd}}\label{appen:proof-sec2}
\begin{proof}[Proof of Theorem \ref{theorem:convex-al-duality}]
First of all, the min-max inequality { %
\cite[lemma 36.1]{rockafellar1997convex}} guarantees that $$\max_{\lambda,\mu}\min_x\ \mathbb{L}_\rho(x,\lambda,\mu)\leq \min_x\max_{\lambda,\mu}\ \mathbb{L}_\rho(x,\lambda,\mu).$$ 
One only needs to prove the opposite of inequality. Consider an equivalent formulation of \eqref{prob2}:
    \begin{equation}\label{prob:NLP*}
        \begin{split}
        \min_{x,s}\ & f(x)+\frac\rho 2 \sum_{i\in\mathcal{E}}c_i^2(x)+\frac\rho2\sum_{i\in\mathcal{I}}\left(c_i(x)+s_i\right)^2,\\
        \st\ & c_i(x)=0,\hspace{8mm} i\in\mathcal{E},\\
        &c_i(x)+s_i=0, \ i\in\mathcal{I},\\
        &s_i\geq 0, \hspace{1.3cm} i\in\mathcal{I},
        \end{split}
    \end{equation}
    whose Lagrangian function is equal to $\mathcal{L}_\rho(x,s,\lambda,\mu)$ defined in \eqref{AL0}. If problem \eqref{prob2} is convex and Slater's condition holds, then \eqref{prob:NLP*} is also a convex problem and satisfies Slater's condition, hence Theorem  \ref{sec2:general:sd} implies the strong duality for problem \eqref{prob:NLP*}:
    \begin{equation}\label{strong-duality0}
    \max_{\lambda,\mu}\min_{s\geq0,x}\ \mathcal{L}_\rho(x,s,\lambda,\mu)=\min_{s\geq0,x}\max_{\lambda,\mu}\ \mathcal{L}_\rho(x,s,\lambda,\mu).
    \end{equation}
    By the definition of $\mathbb{L}_\rho$ in \eqref{AL}, we have the left-hand side (LHS) of \eqref{strong-duality0} equals the LHS of \eqref{strong-duality}. It remains to show that the right-hand side (RHS) of \eqref{strong-duality0} equals the RHS of \eqref{strong-duality}:
    \begin{equation}\label{remains}
    \min_{s \geq0,x}\max_{\lambda,\mu}\ \mathcal{L}_\rho(x,s,\lambda,\mu)= \min_x\max_{\lambda,\mu}\ \mathbb{L}_\rho(x,\lambda,\mu).
    \end{equation}
    Suppose $\tilde{x}$ is feasible to \eqref{prob}, that is, $c_i(\tilde{x})=0,i\in\mathcal{E}$ and $c_i(\tilde{x})\leq 0,i\in\mathcal{I}$. Note that for any fixed $\tilde{x}$, setting $s_i=-c_i(\tilde{x})$ and take any $\lambda,\mu$ (e.g., $\lambda=0,\mu=0$) gives a closed form solution to the LHS of \eqref{remains}, and hence
    $\min_{s}\max_{\lambda,\mu}\ \mathcal{L}_\rho(\tilde{x},s,\lambda,\mu)=f(\tilde{x})$.
    For the RHS of \eqref{remains}, $\mu=0$ and any $\lambda$ (e.g., $\lambda=0$) are the optimal solutions, and hence $\max_{\lambda,\mu}\ \mathbb{L}_\rho(\tilde{x},\lambda,\mu)=f(\tilde{x})$.
    On the other hand, if $\tilde{x}$ is not a feasible point, that is, $\exists \ i\in\mathcal{E}$ such that $c_i(\tilde{x})\neq0$ or $\exists \ i\in\mathcal{I}$ such that $c_i(\tilde{x})>0$, it is straightforward to see that
    \begin{align*}
        \max_{\lambda,\mu}\mathbb{L}_\rho(\tilde{x},\lambda,\mu)=&\max_{\lambda,\mu}\min_{s\geq0}\mathcal{L}_\rho(\tilde{x},s,\lambda,\mu)\\
        \geq &\max_{\lambda,\mu}\min_{s\geq0} f(\tilde{x})+\sum_{i\in\mathcal{E}}\lambda_i c_i(\tilde{x})+\sum_{i\in\mathcal{I}}\mu_i (c_i(\tilde{x})+s_i)\\
        \geq & \max_{\mu\geq 0,\lambda}\min_{s\geq0} f(\tilde{x})+\sum_{i\in\mathcal{E}}\lambda_i c_i(\tilde{x})+\sum_{i\in\mathcal{I}}\mu_i (c_i(\tilde{x})+s_i)\\
        = & \max_{\mu\geq 0,\lambda} f(\tilde{x})+\sum_{i\in\mathcal{E}}\lambda_i c_i(\tilde{x})+\sum_{i\in\mathcal{I}}\mu_i c_i(\tilde{x})\\
        =& +\infty = \max_{\lambda,\mu}\mathcal{L}_\rho(\tilde{x},s,\lambda,\mu).
    \end{align*}
    Combining the above two cases, we prove the equality \eqref{remains}. Hence we complete the proof.
    \end{proof}

\begin{proof}[Proof of Theorem \ref{sec1:nonconvex:strong_duality}]
 First, due to min-max inequality, it is straightforward to see that  
 \begin{equation}
    \label{sec1:prop:right0}
    \mathop{\max}_{ \rho>0,\lambda,\mu}\min_x\ \mathbb{L}_\rho(x,\lambda,\mu) \leq \min_x\max_{ \rho>0,\lambda,\mu}\ \mathbb{L}_\rho(x,\lambda,\mu).
 \end{equation} 
 Now it remains to show the reverse inequality. For the right-hand side of \eqref{npnconvex_strong_duality}, we have
\begin{equation}\label{sec1:prop:right} 
    \min_{x}\max_{ \rho>0,\lambda,\mu}\mathbb{L}_\rho(x,\lambda,\mu) = f^* := \min_x \bar{f}(x)>-\infty,  
\end{equation}
{where $f^*$ exists due to the lower bounded assumption on $f$, and the function $\bar{f}$ is defined by 
$$\bar{f}(x) = f(x) + \sum_{i\in\mathcal{E}}\delta_{\{0\}}(c_i(x)) + \sum_{i\in\mathcal{I}}\delta_{-\R_+}(c_i(x)).$$
Also note that by taking $\lambda=0$ and $\mu=0$, the AL function reduces to the quadratically penalized objective function $\Phi_\rho(x)$ defined in \eqref{penalty_inequality}. 
It can be observed that $\mathbb{L}_\rho(x,0,0)=\Phi_\rho(x)$ increases pointwise as $\rho \rightarrow +\infty$ to the function $\bar{f}$. Note that $\Phi_\rho(x)$ is l.s.c. and is level-bounded for sufficiently large $\rho$, applying Theorem 7.33 in \cite{rockafellar2009variational} yields that
\begin{equation}\label{sec1:prop:left} 
    \mathop{\max}_{\rho>0,\lambda,\mu}\min_x\ \mathbb{L}_\rho(x,\lambda,\mu)\geq \lim_{\rho\to+\infty}\min_x\ \mathbb{L}_\rho(x,0,0) = \min_x\bar{f}(x) = f^*.  
\end{equation}
Combining \eqref{sec1:prop:right0}, \eqref{sec1:prop:right}, and \eqref{sec1:prop:left} proves the theorem. }
\end{proof}

\begin{proof}[Proof of Proposition \ref{sec2:composite:propo:duality}]
It follows from the well-known min-max inequality that
\begin{equation*}
    \max_{\nu,\lambda,\mu}\min_{x}\mathbb{L}_{\rho}(x,\nu,\lambda,\mu) \leq \min_{x}\max_{\nu,\lambda,\mu}\mathbb{L}_{\rho}(x,\nu,\lambda,\mu).
\end{equation*}
Thus, it remains to prove that
\begin{equation*}
     \max_{\nu,\lambda,\mu}\min_{x}\mathbb{L}_{\rho}(x,\nu,\lambda,\mu) \geq \min_{x}\max_{\nu,\lambda,\mu}\mathbb{L}_{\rho}(x,\nu,\lambda,\mu).
\end{equation*}
    The original problem \eqref{sec2:pro:composite} is equivalent to 
\begin{align*}
      \min_x &~ f(x) + h(u)  + \delta_{\mathcal{Q}}(w) + \delta_{\mathcal{K}}(v) + \frac{\rho}{2}\left(   \|x  -u\|^2 + \|  \mathcal{A}(x) - w\|^2 + \|x - v\|^2\right)\\
\mathrm{s.t. }&~ x = u, \mathcal{A}(x) = w, x = v.
\end{align*}
    Let $\psi^*$ be the  optimal value of \eqref{sec2:pro:composite}. Then we have that
\begin{equation*}
    \min_{x}\max_{\nu,\lambda,\mu}\mathbb{L}_{\rho}(x,\nu,\lambda,\mu) = \psi^*.
\end{equation*}
Let $x^*$ be a solution of \eqref{sec2:pro:composite}. As the objective function of the original composite problem has non-empty relative interior, Slater's condition holds for problem \eqref{sec2:pro:composite:auxilary-new}. Then due to the assumption of a non-empty solution set, the KKT condition of problem \eqref{sec2:pro:composite:auxilary-new} can then be simplified to the existence of $(\nu^*,\lambda^*,\mu^*)$ such that
\begin{equation}\label{sec:composite:auxilary:kktstar}
\left\{
    \begin{aligned}
    0 &= \nabla f(x^*) + \nu^* + \mathcal{A}^*(\lambda^*) + \mu^*,\\
0 & \in \partial h(x^*) - \nu^*, ~ 0 \in \partial \delta_{\mathcal{K}}(x^*) - \mu^*, 0 \in \partial \delta_{\mathcal{Q}}(\mathcal{A}(x^*)) - \lambda^*.    \end{aligned}
    \right.
\end{equation}
Based on the definition of the proximal operator, we have for any $\rho>0$ that
\begin{equation*}
    x^* = \mathrm{prox}_{h/\rho}\left(x^* + \frac{\nu^*}{\rho}\right), \;  x^* = \Pi_{\mathcal{K}}\left(x^* + \frac{\mu^*}{\rho}\right), \; \mathcal{A}(x^*) = \Pi_{\mathcal{Q}}\left( \mathcal{A}(x^*) + \frac{\lambda^*}{\rho} \right).
\end{equation*}
Plugging the above equalities into \eqref{sec2:al:eliminate} yields
\begin{align}
    \mathbb{L}_{\rho}(x^*,\nu^*,\lambda^*,\mu^*) &= f(x^*) + h(x^*) + \delta_{\mathcal{Q}}(\mathcal{A}(x^*)) + \delta_{\mathcal{K}}(x^*)\nonumber\\
    &= \psi^* = \min_{x}\max_{\nu,\lambda,\mu}\mathbb{L}_{\rho}(x,\nu,\lambda,\mu).
\end{align}
It follows from \eqref{sec:composite:auxilary:kktstar} that
\begin{align*}
    0 = &\nabla f(x^*) + \nu^* + \mathcal{A}^*(\lambda^*) + \mu^* \\
     =& \nabla f(x^*) + \rho \left(x^* + \frac{\nu^*}{\rho} -  \mathrm{prox}_{h/\rho}\left(x^* + \frac{\nu^*}{\rho}\right) \right) + \rho \left(x^* + \frac{\mu^*}{\rho} -  \Pi_{\mathcal{K}}\left(x^* + \frac{\mu^*}{\rho}\right) \right)\\
    & + \rho \mathcal{A}^* \left( \mathcal{A}(x^*) + \frac{\lambda^*}{\rho} - \Pi_{\mathcal{Q}}\left( \mathcal{A}(x^*) + \frac{\lambda^*}{\rho} \right)  \right) \\
     = & \nabla_x  \mathbb{L}_{\rho}(x^*,\nu^*,\lambda^*,\mu^*),
\end{align*}
which implies that
\begin{equation}
\min_x \mathbb{L}_{\rho}(x,\nu^*,\lambda^*,\mu^*) =  \mathbb{L}_{\rho}(x^*,\nu^*,\lambda^*,\mu^*).\nonumber
\end{equation}
Combining the above results yields that
\begin{align*}
    \max_{\nu,\lambda,\mu}\min_{x}\mathbb{L}_{\rho}(x,\nu,\lambda,\mu) & \geq   \min_x \mathbb{L}_{\rho}(x,\nu^*,\lambda^*,\mu^*) \\
     & =  \mathbb{L}_{\rho}(x^*,\nu^*,\lambda^*,\mu^*) \\
     & = \min_{x}\max_{\nu,\lambda,\mu}\mathbb{L}_{\rho}(x,\nu,\lambda,\mu).
\end{align*}
This proof is completed. 
\end{proof}

 \begin{proof}[Proof of Lemma \ref{IP_strongd}]
 For any $\mu\geq0$ and $\rho > 0$, since $\{x \in \bar{\mathcal{K}} : Ax - b \leq 0\} \subseteq \bar{\mathcal{K}}$, we have
\begin{equation}\label{fip}
z_\rho^{\mathrm{LR}+}(\mu) \leq  \mathop{\min}_{x \in \bar{\mathcal{K}} \atop Ax - b\leq 0} ~L_{\rho}(x,\mu)
 \leq \mathop{\min}_{x \in \bar{\mathcal{K}}\atop Ax - b \leq 0} c^{\top}x =f^{\mathrm{IP}},
\end{equation}
and hence $z_\rho^{\mathrm{LD}+} \leq f^{\mathrm{IP}}$.  Therefore, it suffices to find a finite $\rho^*>0$ such that $z_{\rho^*}^{\mathrm{LD}+} \geq f^{\mathrm{IP}}$.

Let $x^0$ be an arbitrary feasible solution to problem \eqref{ip_inequality} and denote by $f^{\mathrm{LP}}$ the optimal value of the LP relaxation of problem \eqref{ip_inequality}. 
By \cite{blair1977value}, the LP relaxation of \cite{blair1977value} has a bounded optimal value, namely,
$$-\infty < f^{\mathrm{LP}}\leq c^{\top}x^0 < +\infty.$$ Let us set $\rho^*=2( c^{\top} x^0-f^{\mathrm{LP}})/\delta$, then  $0 < \rho^* < +\infty$ and it holds that 
\begin{equation}\label{fld}
z_{\rho^*}^{\mathrm{LD}+} = \max_{\mu\in \R_+^m} z_{\rho^*}^{\mathrm{LR}+}(\mu)\geq z_{\rho^*}^{\mathrm{LR}+}(\mu^*)\ =\min_{x \in\bar{\mathcal{K}}} \mathcal{L}_{\rho^*}(x, \mu^*),
\end{equation}
where $\mu^* \in \R_+^m$ is a given parameter.
Denote $\mathbb{I}:=\{i \in [m]: a_ix-b_i > 0, x \in \bar{\mathcal{K}}\}$, we consider following two cases:\vspace{0.2cm}

\noindent Case 1: $\mathbb{I}= \emptyset$. In this case, we have $Ax \leq b$ for all $x \in \bar{\mathcal{K}}$. Setting $\bar{\mu}=0$, we  obtain
\begin{align}\label{fip1}
L_{\rho^*}(x, \bar{\mu})&=c^{\top}x+\bar{\mu}^{\top}(Ax - b)+\frac{\rho^*}{2}\|\left(Ax - b\right)_+\|^2 = c^{\top}x \geq f^{\mathrm{IP}}.
\end{align}

\noindent Case 2: $\mathbb{I}\neq \emptyset$. Denote by $\mu^{\mathrm{LP}}$  a positive optimal vector of dual variables for $Ax \leq b$ in the LP relaxation of (\ref{ip_inequality}). In this case, we get
\begin{align*}
L_{\rho^*}(x, \mu^{\mathrm{LP}})
&=c^{\top}x+(\mu^{\mathrm{LP}})^{\top}(Ax - b)+\frac{\rho^*}{2}\sum_{i \in \mathbb{I}}(a_{i}x - b_{i})^2+\frac{\rho^*}{2}\sum_{i \notin \mathbb{I}}((a_{i}x - b_{i})_+)^2\nonumber\\
&=c^{\top}x+(\mu^{\mathrm{LP}})^{\top}(Ax - b)+\frac{\rho^*}{2}\sum_{i \in \mathbb{I}}(a_{i}x - b_{i})^2. 
\end{align*}
Since 
$$\frac{\rho^*}{2}\sum_{i \in \mathbb{I}}(a_{i}x - b_{i})^2 \geq \frac{\rho^*}{2}\min_{i \in \mathbb{I}}(a_{i}x - b_{i})^2 \overset{\eqref{delta}}{\geq} \frac{\rho^*}{2} \delta = c^{\top}x^0-f^{\mathrm{LP}},$$
it yields
\begin{align}\label{fip2}
L_{\rho^*}(x, \mu^{\mathrm{LP}})
&\geq c^{\top}x+(\mu^{\mathrm{LP}})^{\top}(Ax - b)+( c^{\top}x^0-f^{\mathrm{LP}}) \nonumber\\
&\geq f^{\mathrm{LP}}+( c^{\top}x^0-f^{\mathrm{LP}})= c^{\top}x^0 \geq f^{\mathrm{IP}}, 
\end{align}
where the second inequality holds due to the definition of $\mu^{\mathrm{LP}}$.

By letting $\mu^*$ be $\bar{\mu}$ and $\mu^{\mathrm{LP}}$ in the inequalities \eqref{fip1} and \eqref{fip2}, we arrive at
$$z_{\rho^*}^{\mathrm{LR}+}(\mu^*) = \min_{x \in \bar{\mathcal{K}}} ~L_{\rho^*}(x,\mu^*) \geq f^{\mathrm{IP}}. $$
Together with \eqref{fld} and \eqref{fip}, we prove
$z_{\rho^*}^{\mathrm{LD}+}   = z_{\rho^*}^{\mathrm{LR}+}(\mu^*) = f^{\mathrm{IP}}.$ 
\end{proof}

\subsection{Proofs in Section \ref{sec:nonconvex_case}}\label{appen:sec5}

\begin{proof}[Proof of Proposition \ref{prop:to0}]
If $\rho_k\to\infty$,  it follows that line 10 of Algorithm  \ref{alg: practical-ALM} is executed an infinite number of times. We select the sequence $\{a_N\}_{N=0}^{\infty}$ such that for $k=a_N$,  we have 
$$\rho_{a_N}=\kappa^N\rho_0, \ \rho_{a_{N}+1}=\kappa^{N+1}\rho_0, \ \epsilon_{a_N+1}=1/\rho_{a_N+1}^{\alpha}.$$ 
For $a_N+1\leq k\leq a_{N+1}-1$, the Algorithm  \ref{alg: practical-ALM} executes a series of steps in line 8, then we have
\begin{equation*}
\rho_k =\rho_{a_N+1}, \  \epsilon_k=\rho_{a_N+1}^{-\beta(k-a_N-1)}\epsilon_{a_N+1},\
 \Lambda^{k+1} \in \Lambda^k+\rho_k \partial_\Lambda \mathbb{L}_{\rho_k}(x^{k+1},\Lambda^k).
\end{equation*}
Since $\|\partial_\Lambda \mathbb{L}_{\rho_k}(x^{k+1},\Lambda^k)\| \leq \vartheta(x^{k+1},\Lambda^k,\rho_k)\leq \epsilon_k$, then $\|\Lambda^{k+1}\|\leq \|\Lambda^k\|+\epsilon_k\rho_k$. Therefore, we utilize the increasing property of $\rho_k$ to obtain
\begin{align}
	\|\Lambda^{k+1}\|&\leq \|\Lambda^k\|+\rho_{a_N+1}^{1-\alpha}\rho_{a_N+1}^{-\beta(k-a_N-1)}\leq\cdots \notag\\
	&\leq\|\Lambda^{a_N+1}\|+\rho_{a_N+1}^{1-\alpha}\sum\nolimits_{i=a_N+1}^{k}\rho_{a_N+1}^{-\beta(i-a_N-1)}\notag \\
	&\leq \|\Lambda^{a_N+1}\|+\rho_{a_N+1}^{1-\alpha}\sum\nolimits_{i=0}^{\infty}\rho_{a_N+1}^{-\beta i} \notag \\
	&=\|\Lambda^{a_N+1}\|+\rho_{a_N+1}^{1-\alpha}\frac{1}{1-\rho_{a_N+1}^{-\beta}}\notag\\
	&\leq \|\Lambda^{a_N+1}\|+\rho_{a_N+1}^{1-\alpha}\frac{1}{1-\rho_{0}^{-\beta}},\quad a_N+1\leq k\leq a_{N+1}-1. \label{lambda_change}
\end{align}
Taking $k=a_{N+1}-1$ and using $\Lambda^{a_N+1}=\lambda^{a_N}$, we have
\begin{equation}
	\|\Lambda^{a_{N+1}}\|\leq \|\Lambda^{a_N}\|+\frac{\rho_{a_N+1}^{1-\alpha}}{1-\rho_{0}^{-\beta}}.\label{lambda_K}
\end{equation}
Dividing both sides of equation \eqref{lambda_K} by $\rho_{a_N}^{\alpha_0}$ with $\alpha_0 \in (1-\alpha,1)$, then we have
\begin{equation*}
	\frac{\|\Lambda^{a_{N+1}}\|}{\rho_{a_{N+1}}^{\alpha_0}}\leq
	\frac{\|\Lambda^{a_{N+1}}\|}{\rho_{a_{N}}^{\alpha_0}}\leq
	\frac{\|\Lambda^{a_N}\|}{\rho_{a_{N}}^{\alpha_0}}+
	\frac{\kappa^{N(1-\alpha-\alpha_0)}\kappa^{1-\alpha_2}\rho_0^{1-\alpha-\alpha_0}}{1-\rho_{0}^{-\beta}}.
\end{equation*}
Thus,
\begin{equation}
	\frac{\|\Lambda^{a_{N+1}}\|}{\rho_{a_{N+1}}^{\alpha_0}}\leq
	\frac{\|\Lambda^{a_0}\|}{\rho_0^{\alpha_0}}+\frac{\kappa^{1-\alpha}\rho_0^{1-\alpha-\alpha_0}}{1-\rho_{0}^{-\beta}}\sum_{i=0}^N
	\kappa^{i(1-\alpha-\alpha_0)},
\end{equation}
Since $1-\alpha-\alpha_0<0$, then $$\sum\limits_{i=0}^N
\kappa^{i(1-\alpha-\alpha_0)}<\sum\limits_{i=0}^{\infty}
\kappa^{i(1-\alpha-\alpha_0)}<\infty,$$ which yields that $\left\{	\|\Lambda^{a_{N}}\|/\rho_{a_{N}}^{\alpha_0}\right\}_{N=0}^{\infty}$ is bounded, i.e., there exists $M>0$ such that $$\|\Lambda^{a_{N}}\|/\rho_{a_{N}}^{\alpha_0} \leq M, \quad \text{for~ all}~ N.$$
If $k=a_N+1$, then $\frac{\|\Lambda^{a_{K}+1}\|}{\rho_{a_{N}+1}^{\alpha_0}}\leq\frac{\|\Lambda^{a_{N}}\|}{\rho_{a_{N}}^{\alpha_0}}\leq M$ is bounded; If $a_N+2\leq k\leq a_{N+1}$, by \eqref{lambda_change} and $\rho_k=\rho_{a_N+1}$, $$\frac{\|\Lambda^{k}\|}{\rho_{k}^{\alpha_0}}\leq\frac{\|\Lambda^{a_{N}}\|}{\rho_{k}^{\alpha_0}}+\frac{\rho_{a_N+1}^{1-\alpha-\alpha_0}}{1-\rho_{0}^{-\beta}}\leq \frac{\|\lambda^{a_{N}}\|}{\rho_{a_N}^{\alpha_0}}+\frac{\rho_{0}^{1-\alpha-\alpha_0}}{1-\rho_{0}^{-\beta}}\leq M+\frac{\rho_{0}^{1-\alpha-\alpha_0}}{1-\rho_{0}^{-\beta}}$$ is also bounded. This implies that for all $k$, $\left\{\Lambda^{k}/\rho_{k}^{\alpha_0}\right\}$ is bounded.

    Since $\frac{\|\Lambda^{k}\|}{\rho_{k}}=\frac{\|\Lambda^{k}\|}{\rho_{k}^{\alpha_0}}\rho_k^{\alpha_0-1}$ and   $\rho_k^{-1}\to 0$ as $k\to\infty$, then we can obtain that $\frac{\Lambda^{k}}{\rho_{k}}\to 0$. 
\end{proof}

\begin{proof}[Proof of Theorem \ref{nonconvex_convergence}]
Let $x^*$ be an arbitrary accumulation point of the sequence 
$\{x_k\}$, with $\{x_k\}_K$ as a subsequence converging to 
$x^*$. To prove the feasibility of $x^*$, we consider the following two cases:

(i) If $\{\rho_k \}$ is bounded, it follows that line 10 of Algorithm \ref{alg: practical-ALM} is executed a finite number of times. Consequently, according to the instruction in line 8 of Algorithm \ref{alg: practical-ALM}, $\epsilon_k \rightarrow 0$ as 
$k \rightarrow \infty$. Given the stopping criterion \eqref{nc-subopt2} and the supergradient expression \eqref{eq:nonconvex-diff} of dual variables, taking the limit $k \rightarrow_K \infty$, we have
 $$\left\|c(x^{k+1})-\Pi_{\mathcal{Q}}\left(c(x^{k+1})+\frac{\lambda^k}{\rho_k}\right) \right\|\to 0, \ \left\| x^{k+1}-\Pi_{\mathcal{K}}\left(x^{k+1}+\frac{\mu^k}{\rho_k}\right)\right\|\to 0,$$ then
 \begin{equation*}
     \begin{aligned}
	\|c(x^{k+1})-\Pi_{\mathcal{Q}}(c(x^{k+1}))\|&\leq \left\|c(x^{k+1})-\Pi_{\mathcal{Q}}\left(c(x^{k+1})+\frac{\lambda^k}{\rho_k}\right) \right\|\to 0, \\
		\|x^{k+1}-\Pi_{\mathcal{K}}(x^{k+1})\|&\leq \left\| x^{k+1}-\Pi_{\mathcal{K}}\left(x^{k+1}+\frac{\mu^k}{\rho_k}\right)\right\|\to 0. 
	\end{aligned}
 \end{equation*}
	By the continuity, we let $x^{k}\to_K x^*$, then we can find that  $c(x^*)=\Pi_{\mathcal{Q}}(c(x^*))$, $x^*=\Pi_{\mathcal{K}}(x^*)$, which imply that $c(x^*)\in\mathcal{Q}$, $x^*\in\mathcal{K}$, $x^*$ is a feasible point.

 (ii) If $\{\rho_k\}$ is unbounded, we let $\rho_k\to\infty$. By assumption, we have
\begin{equation*}
 \begin{aligned}
  C  \geq& \mathbb{L}_{\rho_k}(x^{k+1};\Lambda^k) \\
   =&    f(x^{k+1}) - \frac{1}{2\rho_k}(\|\nu^k\|^2 + \|\lambda^k\|^2  + \|\mu^k\|^2)+   e_{\rho_k}h\left(x^{k+1} + \frac{\nu^k}{\rho_k}\right)\\
    &+\frac{\rho_k}{2}\left\|\hat\Pi_{\mathcal{Q}}\left(c(x^{k+1})+\frac {\lambda^k}{\rho_k}\right)\right\|^2+ \frac{\rho_k}{2}\left\|\hat\Pi_\mathcal{K}\left(x^{k+1}+\frac {\mu^k}{\rho_k}\right)\right\|^2.  
 \end{aligned}
\end{equation*}
	Rearranging terms yields the inequality
	\begin{equation*} \left\|\hat\Pi_{\mathcal{Q}}\left(c(x^{k+1})+\frac {\lambda^k}{\rho_k}\right)\right\|^2+ \left\|\hat\Pi_\mathcal{K}\left(x^{k+1}+\frac {\mu^k}{\rho_k}\right)\right\|^2 \leq \frac{2(C-F(x^{k+1},\Lambda^k))}{\rho_k},\quad \forall k\in\mathbb{N},
	\end{equation*}
where $F(x^{k+1},\Lambda^k)=f(x^{k+1}) - \frac{1}{2\rho_k}(\|\nu^k\|^2 + \|\lambda^k\|^2  + \|\mu^k\|^2)+   e_{\rho_k}h\left(x^{k+1} + \frac{\nu^k}{\rho_k}\right)$.	Taking $k\to_K\infty$, by the continuity $f$ and Proposition \ref{prop:to0}, $F(x^{k+1},\Lambda^k)\to_K f(x^*)+h(x^*)$ is bounded, which together with $\rho_k\to_K\infty$ yields that 
	\begin{align*}
		\hat\Pi_{\mathcal{Q}}(c(x^*))&=\lim_{k\to_K\infty}\hat\Pi_{\mathcal{Q}}\left(c(x^{k+1})+\frac{\lambda^k}{\rho_k}\right)=0,\\
		\hat\Pi_\mathcal{K}(x^*)&=\lim_{k\to_K\infty}\hat\Pi_\mathcal{K}\left(x^{k+1}+\frac{\mu^k}{\rho_k}\right)=0.
	\end{align*}
	Then $x^*$ is a feasible point.  

 To prove the final part of the theorem, we can slightly modify the proof of Theorem 3.3  in \cite{de2023constrained} to derive our result. A detailed explanation is omitted here for brevity.
\end{proof}

\begin{proof}[Proof of Proposition \ref{prop:infeas}]
    We consider the following two cases.
   
   (i) If $\{\rho_k\}$ is bounded, we have
   \begin{equation*} 
	\|\hat{\Pi}_{\mathcal{Q}}(c(x^*))\|^2=\|\hat{\Pi}_{\mathcal{K}}(x^*)\|^2=0
\end{equation*}
by Proposition \ref{nonconvex_convergence}. It is obvious to see that $x^*$ is a stationary point of \eqref{feas}.

(ii) If $\{\rho_k\}$ is unbounded, i.e., $\rho_k\to\infty$. Let $\{x^{k+1}\}_K$ be the subsequence satisfying $x^{k+1}\to_K x^*$. We have
\begin{equation*}
	\left\|\nabla_x \mathbb{L}_{\rho_k}(x^{k+1};\Lambda^k) \right\|\leq \eta_k.
\end{equation*}
Divide both sides by $\rho_k$ to obtain
\begin{align*}
\left\|\frac{\nabla\! f(x^{k+1})}{\rho_k}
\!+\!\nabla\! e_{\!\rho_k}h\!\!\left(\!x^{k+1} \!+\! \frac{\nu^k}{\rho_k}\right) \!\!+\! \nabla\! c(x^{k+1})\hat{\Pi}_{\mathcal{Q}}\!\!\left(\! c(x^{k+1})\!+\!\frac{\lambda^k}{\rho_k}\right)\!\!+\!\hat{\Pi}_{\mathcal{K}}\!\!\left(\! x^{k+1}\!+\!\frac{\mu^k}{\rho_k}\right)\!\right\| \!\leq\! \frac{\eta_k}{\rho_k}.
\end{align*}

Let $k\to_K \infty$, since $\rho_k\to \infty$, $x^{k+1}\to_K x^*$, $\eta_k\to 0$, $\frac{\Lambda^k}{\rho_k} \to 0$, and $f$ is continuous, $\nabla f(x^*)$ is bounded, then we have
\begin{equation*}
	\nabla c(x^*) \hat{\Pi}_{\mathcal{Q}}(c(x^*)) + \hat{\Pi}_{\mathcal{K}}(x^*) = 0.
\end{equation*}
 which implies that $x^*$ is a stationary point of \eqref{feas}.
\end{proof}

\section{More examples for constructing AL function}

\subsection{Examples in convex case}\label{sec2:convex:motivation}

The AL function presented in  \eqref{sec2:al:eliminate} is noteworthy for its ability to encompass various significant examples when considering different forms of $h$ and $\mathcal{A}$. To illustrate this point, we provide a few examples below. 
\begin{example}[Linear programming]\label{sec2:motivate:LP}
Consider a linear programming (LP) problem:
\begin{equation}\label{eq:lp}
    \min_{x\in\mathbb{R}^n}\ c^\top x,\quad\mathrm{ s.t. }\ Ax \in \mathcal{Q},~ x\in \mathcal{K},
\end{equation}
 where $c\in\mathbb{R}^n$,  $A\in\mathbb{R}^{m\times n}$, and $\mathcal{Q} = \{y\in\mathbb{R}^m : l^q\leq y \leq  u^q\}, \mathcal{K} = \{x\in\mathbb{R}^n : l^\kappa \leq  x\leq u^\kappa\}$ are two convex sets defined by known constant vectors $l^q,u^q\in \mathbb{R}^m$ and $l^\kappa,u^\kappa\in \mathbb{R}^n$.  %
 \end{example}
In this example, we have $$ \Pi_{\mathcal{Q}}(x) = \min\{u^q,\max\{l^q,x\}\},\quad \Pi_{\mathcal{K}}(x)=\min\{u^\kappa,\max\{l^\kappa,x\}\}.$$ The corresponding AL function is
\begin{equation*}
    \mathbb{L}_{\rho}(x,\lambda,\mu) = c^\top x+ \frac{\rho}{2}\|\hat{\Pi}_{\mathcal{Q}}( Ax + \lambda/\rho)\|^2 +\frac{\rho}{2}\| \hat{\Pi}_{\mathcal{K}}( x + \mu/\rho)\|^2- \frac{1}{2\rho}( \|\lambda\|^2 + \|\mu\|^2 ),
\end{equation*}
where $\rho>0$ is a penalty parameter, and $\lambda,\mu$ are Lagrange multipliers.

We can also consider the dual problem of \eqref{eq:lp}:
\begin{equation}\label{eq:lp:dual}
\begin{aligned}\max &\ \lambda_1^\top l^q - \lambda_2^\top u^q + \mu_1^\top l^\kappa - \mu_2^\top u^\kappa  \\
 \text{s.t.} & \  c - A^\top(\lambda_1 - \lambda_2) - \mu_1 + \mu_2 =0, \\
  & \ \lambda_1,\lambda_2,\mu_1,\mu_2 \geq 0.
  \end{aligned}
\end{equation}
The AL function of \eqref{eq:lp:dual} can be written as
\begin{align*}
     & \mathbb{L}_{\rho}(\lambda_1,\lambda_2,\mu_1,\mu_2,x,s_1,s_2,w_1.w_2) \\
     = & ~\lambda_1^\top l^q-\lambda_2^\top  u^q + \mu_1^\top l^\kappa - \mu_2^\top u^\kappa+ \frac{\rho}{2}\left\| c-A^\top (\lambda_1- \lambda_2)-\mu_1 + \mu_2 + x/\rho\right\|^2 \\
     & +\frac{\rho}{2}\left(\| [s_1/\rho-\lambda_1]_+ \|^2 +\|[s_2/\rho-\lambda_2]_+\|^2  +\|[w_1/\rho-\mu_1]_+\|^2 +\|[w_2/\rho-\mu_2]_+\|^2 \right) \\
     & - \frac{1}{2\rho}( \|x\|^2 + \|s_1\|^2 +\|s_2\|^2 + \|w_1\|^2 +\|w_2\|^2 ).
\end{align*}
When $l^q = u^q = b$ in $\mathcal{Q}$ with a constant vector $b\in \mathbb{R}^m$ and $l^\kappa = 0, u^\kappa = +\infty$ in $\mathcal{K}$, \eqref{eq:lp} reduced to the classical LP:
\begin{equation*}
    \min_{x\in \mathbb{R}^n} c^\top x, \;\; \mathrm{ s.t. }\;\; Ax = b, x\geq 0. 
\end{equation*}
The corresponding AL function can be constructed as follows:
\begin{equation*}
    \mathbb{L}_{\rho}(x,\lambda,\mu) = c^\top x + \frac{\rho}{2}\| Ax -b + \lambda/\rho \|^2 + \frac{\rho}{2} \| [-x + \mu/\rho]_+ \|^2 - \frac{1}{2\rho}( \|\lambda\|^2 + \|\mu\|^2 ).
\end{equation*}

 \begin{example}[Semidefinite programming]\label{sec2:motivate:SP}
Consider the box-constrained semidefinite programming (SDP) of the form 
\begin{equation}\label{sdp:dnn}
    \min_{X\in\mathbb{S}^{n}} \mathrm{Tr}(C^\top X)+\delta_{\mathcal{U}}(X),\quad \mathrm{s.t.}\quad\mathcal{A}(X)\in \mathcal{Q}, X\in \mathcal{K},
\end{equation}
where $C\in\mathbb{S}^{n} = \{X\in\mathbb{R}^{n\times n}:X^\top =X\}$, $\mathcal{A}(\cdot): \mathbb{S}^{n} \rightarrow \mathbb{R}^m$ is a linear mapping, and $\mathcal{U}=\{X\in\mathbb{S}^n : L\leq X\leq U\},\ \mathcal{Q} = \{y\in\mathbb{R}^m : l \leq y \leq u\}$, $\mathcal{K} = \{X\in \mathbb{S}^n : X\succeq 0\}$. 
\end{example}
For any matrix $X\in \mathbb{S}^n$,  let $X = Q\Lambda Q^\top $ be its eigenvalue decomposition, where $\Lambda$ is a diagonal matrix and $Q$ is an orthogonal matrix. The the projection operator of $X$ onto $\mathcal{K}$ can be explicitly computed as:
$$
\Pi_{\mathcal{K}}(X) = Q 
[\Lambda]_+ Q^\top .
$$
 We also have $\Pi_{\mathcal{Q}}(y) = \max\{l,\min\{u,y\}\}$ and $\Pi_{\mathcal{U}}(Y) = \max\{L,\min\{U,Y\}\}$. 
The corresponding AL function of \eqref{sdp:dnn} is given as  
\begin{align*}
       \mathbb{L}_{\rho}(X,y,Z,S) =& ~ \mathrm{Tr}(C^\top X)+\frac{\rho}{2} \| \hat{\Pi}_{\mathcal{Q}} (\mathcal{A}(X) + y/\rho)\|^2  +\frac{\rho}{2}\| \hat{\Pi}_{\mathcal{K}}( X+Z/\rho)  \|^2
       \\&+\frac{\rho}{2}\|  \hat{\Pi}_{\mathcal{U}}(X+S/\rho)\|^2  - \frac{1}{2\rho}( \|y\|^2 + \|Z\|^2 + \|S\|^2  ).
\end{align*}
Moreover, we also consider the dual problem of \eqref{sdp:dnn}:
\begin{equation}\label{eqn:dual-SDP}
    \begin{aligned}
        \min_{y,Z,S} &\ \delta_{\mathcal{Q}}^*(-y) + \delta_{\mathcal{K}}^*(-Z) +\delta_{\mathcal{U}}^*(-S)\\
        \mathrm{s.t.} &\ \mathcal{A}^*(y)+S+Z = C.
    \end{aligned}
\end{equation}
where $\delta^*_{\mathcal{Q}}(-y)=\max_{x\in \mathcal{Q}}\left<x,-y\right>$  and $\mathcal{A}^*$ is the dual linear map  of $\mathcal{A}$ that satisfies $$\langle\mathcal{A}^*(w),v\rangle=\langle w,\mathcal{A}v\rangle, \quad \text{for\ all} \  w,v.$$
The AL function of \eqref{eqn:dual-SDP} is 
\begin{align*}
       \mathbb{L}_{\rho}(y,Z,S,X,z,W,V) = &\frac{\rho}{2} \|C-\mathcal{A}^*(y) - S -Z + X/\rho\|^2 +\frac{\rho}{2} \| \Pi_{\mathcal{Q}}(-y+z/\rho)  \|^2\\& + \frac{\rho}{2} \| \Pi_{\mathcal{K}}(-Z+W/\rho)  \|^2 + \frac{\rho}{2} \| \Pi_{\mathcal{U}}(-S+V/\rho)  \|^2 \\
       & - \frac{1}{2\rho}(\|X\|^2 + \|z\|^2 + \|W\|^2 + \|V\|^2 ).
\end{align*}
Note that \eqref{eqn:dual-SDP} can be rewritten as 
\begin{equation}\label{eqn:dual-SDP1}
    \begin{aligned}
        \min_{y,S} & \ \delta_{\mathcal{Q}}^*(-y) + \delta_{\mathcal{K}}^*(\mathcal{A}^*(y)+S - C) +\delta_{\mathcal{U}}^*(-S).
    \end{aligned}
\end{equation}
The corresponding AL function of \eqref{eqn:dual-SDP1} is given as  
\begin{align*}
    \mathbb{L}_{\rho}(y,S,z,W,X) =& ~ \frac{\rho}{2} \| \Pi_{\mathcal{Q}}(-y+z/\rho)  \|^2 + \frac{\rho}{2} \| \Pi_{\mathcal{U}}(-S+W/\rho)  \|^2\\
    & + \frac{\rho}{2} \|\Pi_{\mathcal{K}}(\mathcal{A}^*(y) + S -C+ X/\rho) \|^2  - \frac{1}{2\rho} ( \|z\|^2 + \|W\|^2 + \|X\|^2 ).
\end{align*}

\begin{example}[SDP relaxation of sparse PCA] The sparse principal component analysis (PCA) problem for a single component has an SDP relaxation:
\begin{equation*}
\min_{X\in \mathbb{R}^{n\times n}} - \mathrm{Tr}(C^\top X)+\gamma \|X\|_1,\quad \mathrm{ s.t. }\ \operatorname{Tr}(X)=1,X\in \mathcal{K},
\end{equation*}
where $C\in \mathbb{S}^n$ is a given matrix, $\|X\|_1=\sum_{i,j}|X_{i,j}|$ and $\mathcal{K} = \{X\in \mathbb{S}^{n\times n} : X\succeq 0\}$ is a positive semidefinite cone.
\end{example}
Define $\|X\|_{\infty} = \max_{i,j}|X_{i,j}|$, and let $\mathcal{B}:=\{X\in \mathbb{R}^{n\times n}:\|X\|_{\infty} \leq \gamma\}$ be a matrix  infinity norm ball with radius $\gamma$. Denote $h(\cdot) = \gamma \|\cdot\|_1$, then the  conjugate function of $h$ is  $h^*(\cdot)=\delta_{\mathcal{B}} (\cdot)$. The corresponding AL function is given by
\begin{align*}
    \mathbb{L}_{\rho}(X,y,Z,S) =  
    & - \mathrm{Tr}(C^\top X)+\gamma\|\mathrm{prox}_{\gamma\rho \|\cdot\|_1}(X+Z/\rho) \|_1 + \frac{\rho}{2}\| \Pi_{\mathcal{B}} (X+Z/\rho)  \|^2 \\
     &+\frac{\rho}{2} \| \mathrm{Tr}(X) - 1 + y/\rho\|^2
        +\frac{\rho}{2}\| \hat{\Pi}_{\mathcal{K}}(X+S/\rho)\|^2 \\
        & - \frac{1}{2\rho}( \|y\|^2 + \|Z\|^2 + \|S\|^2),
\end{align*}
where $\Pi_{\mathcal{B}}(X) = \mathrm{sign}(X)\min\{|X|,\gamma\}$  and  $\mathrm{prox}_{\gamma\rho \|\cdot\|_1}(X): = \mathrm{sign}(X)[|X| - \gamma \rho]_+$. Here, $\min\{|X|,\gamma\}$ and $[|X| - \gamma \rho]_+$ are component-wise operators.

\begin{example}[LASSO problem]
    Given $A\in \mathbb{R}^{m\times n}$ and $b\in \mathbb{R}^m$, the LASSO problem is formulated as an $\ell_1$ regularized optimization problem:
    \begin{equation}\label{eq:lasso}
        \min_x \frac{1}{2}\|Ax - b\|^2 + \gamma \|x\|_1,
    \end{equation}
    \end{example} 
    
 where $\gamma$ is the regularization parameter. Its AL function can be written as 
    \begin{align*}
         \mathbb{L}_{\rho}^p(x,\nu) =& \frac{1}{2}\|Ax - b\|^2+ \frac{\rho}{2} \left\| \Pi_{\|x\|_{\infty} \leq \gamma} (x+\nu/\rho) \right\|^2 - \frac{1}{2\rho}\|\nu\|^2.
    \end{align*}
    Here, we use the notation $\mathbb{L}_{\rho}^p$ to denote the AL function for the primal problem, distinguishing it from the AL function for the dual problem, denoted by $\mathbb{L}_{\rho}^d$. 
     In particular, we consider the dual problem of \eqref{eq:lasso}:
    \[
        \min_{\nu} \frac{1}{2}\|\nu - b\|^2, ~\mathrm{ s.t. }~ \|A^\top \nu\|_{\infty} \leq \gamma.
    \]
    The corresponding AL function is 
    \begin{align*}
         \mathbb{L}_{\rho}^d(\nu,x) =& \frac{1}{2}\|\nu - b\|^2+ \frac{\rho}{2}\|  \mathrm{prox}_{\gamma /\rho\|x\|_1}(A^\top \nu + x/\rho) \|^2 - \frac{1}{2\rho}\|x\|^2.
    \end{align*}

\begin{example}[$\ell_1$ least square problem] Given $A\in \mathbb{R}^{m\times n}$ and $b\in \mathbb{R}^m$, the $\ell_1$ regularization model is given as follows: 
\begin{equation}\label{eq:l1lasso}
    \min_x \|Ax - b\|_1 + \gamma \|x\|_1.
\end{equation}
\end{example}
This problem falls into \eqref{sec2:pro:composite-multiblock} by letting $h_1(y) = \|y-b\|_1, h_2(x) = \gamma\|x\|_1$ and $\mathcal{B}_1(x) = Ax, \mathcal{B}_2(x) = x$.  In this case, the corresponding AL function is given as 
\begin{align*}
     &\mathbb{L}_{\rho}^p(x,\nu_1,\nu_2) \\
     = &\, \frac{1}{2}\|\mathrm{prox}_{\rho \|\cdot\|_1}\!(Ax - b - \nu_1/\rho)\|_1 \!+\frac{\rho}{2}\|Ax - b + \nu_1/\rho -\mathrm{prox}_{\rho \|\cdot\|_1}\!(Ax - b - \nu_1/\rho) \|^2 \\
     &+\gamma\left\|\mathrm{prox}_{\gamma\rho \|\cdot\|_1}(x+\nu_2/\rho)\right\|_1+ \frac{\rho}{2} \left\| \Pi_{B} (x+\nu_2/\rho) \right\|^2 - \frac{1}{2\rho}(\|\nu_1\| + \|\nu_2\|^2),
\end{align*}
where $B=\{x: \|x\|_{\infty} \leq \gamma \}$. The dual problem of \eqref{eq:l1lasso} has the form:
\begin{equation*}
    \min_{\nu}\ b^\top \nu, \quad \mathrm{ s.t. } \ \|A^\top \nu\|_{\infty} \leq \gamma, \|\nu\|_{\infty} \leq 1.
\end{equation*}
The corresponding AL function is  
\begin{align*}
     \mathbb{L}_{\rho}(\nu,x_1, x_2) =& b^\top \nu+ \frac{\rho}{2}\|  \mathrm{prox}_{\gamma /\rho\|\cdot\|_1}(A^\top \nu + x_1/\rho) \|^2 \\
     &+\frac{\rho}{2}\|  \mathrm{prox}_{\gamma /\rho\|\cdot\|_1}(\nu + x_2/\rho) \|^2- \frac{1}{2\rho}(\|x_1\| + \|x_2\|^2).
\end{align*}

\subsection{Examples in nonconvex case}\label{sec2:nonconvex:motivation}

By adapting the forms of the functions 
$h$ and $c$ as well as the constrained sets $\mathcal{Q}$ and $\mathcal{K}$, ALM can be applied to a broad range of significant nonconvex applications. Below, we present several examples to demonstrate the applicability and effectiveness of ALM in nonconvex optimization scenarios.

\begin{example}[Nonlinear programming]\label{sec2:nonconvex:NP}     
Consider the following general constrained nonlinear programming (NLP) problem:
\begin{equation}\label{eq:nlp}
\min_{x\in\mathbb{R}^n}\ f(x) \quad \mathrm{s.t.}\quad c(x) \in \mathcal{Q},~ x\in \mathcal{K},
\end{equation}
where $f$ is nonlinear and/or the feasible region is characterized by nonlinear constraints. 
\end{example}
Let $\mathcal{Q} = \{\lambda\in \R^m: l^c\leq \lambda \leq  u^c\}$ and $\mathcal{K} = \{x\in \R^n: l^x\leq x \leq  u^x\}$, with $l^c,u^c,l^x,u^x$ being some given data.  In this case, we have 
$$ \Pi_{\mathcal{Q}}(\lambda) = \min\{u^c,\max\{l^c,\lambda\}\}, \quad \Pi_{\mathcal{K}}(x)=\min\{u^x,\max\{l^x,x\}\}.$$ 
The corresponding AL function can be obtained as:
\begin{equation*}
    \mathbb{L}_{\rho}(x,\lambda,\mu) = f(x)+ \frac{\rho}{2}\left\|\hat{\Pi}_{\mathcal{Q}}\left( c(x) + \frac\lambda\rho\right)\right\|^2 +\frac{\rho}{2}\left\| \hat{\Pi}_{\mathcal{K}}\left(x + \frac\mu\rho\right)\right\|^2 -\frac{1}{2\rho} ( \|\lambda\|^2 + \|\mu\|^2 ),
\end{equation*}
where $\rho>0$ is a penalty parameter, and $\lambda, \mu$ are Lagrange multipliers.

\begin{example}[Clustering problem \cite{sahin2019inexact}] \label{sec2:nonconvex:CP}       
Given data points $\{z_i\}_{i=1}^n$, the entries of the corresponding Euclidean distance matrix $D \in \mathbb{R}^{n \times n}$ are defined as $D_{i,j} = \|\mathbf{z}_i - \mathbf{z}_j\|^2$.
The clustering problem can be formulated as 
$$
\min_{X \in \mathbb{R}^{n \times r} }\  \operatorname{Tr}(DXX^\top) 
\quad \st\  XX^\top \mathbf{1}_n = \mathbf{1}_n, \ \|X\|_F^2 \leq s, \ X \geq  0.  
$$
\end{example}
In order to cast the above problem as an instance of problem \eqref{nonconvex-composite}, we introduce a new variable $x \in \mathbb{R}^{nr}$ by vectorizing  $X$, defined as 
$x:= [x_1^\top; \ldots; x_n^\top]^\top \in \mathbb{R}^{nr}$,
where $x_i \in \mathbb{R}^r$ represents the $i$-th row of $X$. 
Next, we specify the objective function $f(x)+h(x)$ and the constraint function $c(x)$ as 
\[
f(x) =\sum_{i,j = 1}^n D_{i,j} \langle x_i, x_j \rangle, \
h(x) = 0, \
c(x) = \Big[x_1^\top\sum_{j = 1}^n x_j-1, \ldots, x_n^\top \sum_{j = 1}^n x_j-1\Big]^\top,
\]
and $\mathcal{Q}=\{\mathbf{0}_n\}$,  $\mathcal{K}=\mathbb{R}^{nr}_+\cap B(0,\sqrt{s})$, where $B(0,\sqrt{s})$ denotes a Euclidean ball with radius $\sqrt{s}$ . The corresponding AL function can be written as:
\begin{align*}
    \mathbb{L}_{\rho}(x,\lambda,\mu) = f(x)+\frac{\rho}{2} \left\| c(x) + \frac\lambda\rho\right\|^2
    +\frac{\rho}{2}\left\| \hat{\Pi}_{\mathcal{K}}\left(x+\frac\mu \rho\right)\right\|^2-\frac{1}{2\rho} ( \|\lambda\|^2 + \|\mu\|^2 ).
\end{align*}

\begin{example}[Low-rank SDP \cite{burer2003nonlinear}]
Consider the nonlinear reformulation of an SDP obtained by replacing the positive semidefinite matrix with a factorization:
$$
\min_{X \in \mathbb{R}^{n \times r} }\  \operatorname{Tr}(CXX^\top) 
\quad \st\  \operatorname{Tr}(A_i XX^\top)=b_i, \ i = 1, \ldots, m,$$ 
where the data matrices $ C $ and $ \{A_i\}_{i=1}^m $ are $ n \times n $ real symmetric matrices, and $ b $ is an $ m $-dimensional vector, and $r \leq n$.
\end{example}
The corresponding AL function is 
$$ \mathbb{L}_{\rho}(X, \lambda) = \operatorname{Tr}(CXX^\top) + \sum_{i=1}^{m} \lambda_i (\operatorname{Tr}(A_i XX^\top) - b_i) + \frac{\rho}{2} \sum_{i=1}^{m} (\operatorname{Tr}(A_i XX^\top) - b_i)^2-\frac{1}{2\rho}  \|\lambda\|^2. $$

\begin{example}[Sparse portfolio optimization problem \cite{de2023constrained}] 
\begin{equation}\label{eqn:portfolio}
\min_{x \in \R^n}\ \frac{1}{2} x^{\top} Q x+ h(x) \quad \st\ \gamma^\top x \geq \varrho,\ \mathbf{1}_n^{\top} x \leq 1,\ 0 \leq x \leq u.
\end{equation}
where $Q$ and $\gamma$ denote the covariance matrix and the mean of $n$ possible assets, respectively. $h(x)$ is a sparse regularization term including $\ell_0$ norm, $\ell_p (0 < p < 1)$ norm, smoothly clipped absolute deviation (SCAD), log-sum penalty (LSP), minimax concave penalty (MCP), and capped-$\ell_1$ penalty.  $\varrho$ is some lower bound for the expected return and $u$ provides an upper bound for the individual assets within the portfolio. \end{example}
Denote $ \mathcal{K}=\{z \in \R^n: 0 \leq z \leq u\}$, $\mathcal{Q}=\{y \in \R^2: y \leq 0\}$ and 
$$f(x)=\frac{1}{2} x^{\top} Q x,\qquad c(x)=\left[\begin{array}{ll}
 \varrho -\gamma^\top x \\
\mathbf{1}_n^{\top} x - 1
\end{array}\right].$$
Then the problem is in the form of \eqref{nonconvex-composite} and the corresponding AL function is:
\begin{align}\label{eqn:portfolio:AL}
    \mathbb{L}_{\rho}(x,\lambda,\mu) =& \frac{1}{2} x^{\top} Q x + h(x)+\frac{\rho}{2}\left\|\max\left( c(x) + \frac{\lambda}{\rho},0\right)\right\|^2\\\nonumber
    &
    + \frac{\rho}{2}\left\| \min\left(\max\left(0, x + \frac{\mu}{\rho} -u\right),x\right)\right\|^2-\frac{1}{2\rho} ( \|\lambda\|^2 + \|\mu\|^2 ).
\end{align}
In case $h(x)=\alpha\|x\|_0$, the proximal operator of $h(x)$ has a closed-form expression:
\begin{equation}\label{nco:l0prox}
(\prox_{h}(x))_i=\begin{cases}x_i, & \text { if }\ |x_i|>\sqrt{2\alpha}, \\ 0, & \text { if }\ |x_i|\leq \sqrt{2\alpha},\end{cases} \quad \forall i=1,2,\cdots,n.
\end{equation}
The remaining sparse regularization terms mentioned earlier can also be associated with their respective proximal operators, we will not go into details here. Therefore, we can utilize proximal methods to solve the subproblems of ALM effectively.

\begin{example}[Nonnegative sparse PCA \cite{asteris2014nonnegative}]
The nonnegative sparse principal component analysis (PCA) problem can be modeled as follows:
\begin{eqnarray*}
   & \min\limits_{x \in \R^p}& \sum_{i=1}^{N} \big(-x_i^{\top} \Sigma_i x_i + h_i(x_i)\big) \\
   &\st& \ \|x_i\|_2 \leq 1, \quad i = 1, \ldots, N, \\&& Ax =0, \quad x \geq 0, 
\end{eqnarray*}
where $x_i \in \mathbb{R}^d$ for each $i$, $x:=[x_1^\top; \ldots; x_N^\top]^\top$ stacks all $x_i$'s, $\Sigma_i \in \mathbb{R}^{d \times d}$ is the covariance matrix for the mini-batch data in node $i$, $h_i(\cdot)$ is a regularizer that promotes  sparsity.
\end{example}
Define the constraint sets $\mathcal{Q}=\{\mathbf{0}_m\}$,  $\mathcal{K}=\{z\in \R^{p}: \|z_i\|_2 \leq 1, \ z_i \geq 0,  i = 1, \ldots,N\}$ and 
\[
f(x)=-\sum_{i=1}^{N} x_i^{\top} \Sigma_i x_i,\quad h(x) = \sum_{i=1}^{N} h_i(x_i),\quad c(x)= Ax,
\]
then the problem reduces to an instance of problem \eqref{nonconvex-composite}. The projection to the set $\mathcal{K}$ can be explicitly written as
$$\Pi_{\mathcal{K}}(x)=\max\left\{0, \frac{x}{\max\{1,\|x\|\}}\right\}.$$
 The corresponding AL function can be obtained as:
    \begin{align*}\label{sec2:SPCA:AL}
        \mathbb{L}_{\rho}(x,\nu,\lambda,\mu) =& -\sum_{i=1}^{N} x_i^{\top} \Sigma_i x_i +\frac{\rho}{2} \| Ax + \lambda/\rho\|^2
        +\frac{\rho}{2}\| \hat{\Pi}_{\mathcal{K}}(x+\mu/\rho)\|^2   
    \\\nonumber
        &+h\left(\prox_{ h/\rho}\left(x+\frac{\nu}{\rho}\right)\right) +\frac{\rho}{2} \left\|x + \frac \nu \rho-  \prox_{h/\rho}\left(x + \frac \nu \rho\right)  \right\|^2 \\
        & -\frac{1}{2\rho} ( \|\nu\|^2 + \|\lambda\|^2 + \|\mu\|^2  ) 
    \end{align*}

\begin{example}[Matrix completion problem \cite{sujanani2023adaptive}]       
Given a dimension pair $(p, q) \in \mathbb{N}^2$, a positive scalar triple $(\nu, \tau \theta) \in \mathbb{R}^3_{++}$, a scalar pair $(l, u) \in \mathbb{R}^2$, a matrix $Q \in \mathbb{R}^{p \times q}$ and $\Omega \in \{0, 1\}^{p \times q}$ contains indices of the observed entries in $Q$. The bounded matrix completion problem can be formulated as
\begin{eqnarray*}
&\min\limits_{X \in \R^{p \times q}}&~  \frac{1}{2}\|\Pi_\Omega(X - Q)\|^2 + \tau \sum_{i=1}^{\min\{p,q\}} \Big( g(\sigma_i(X)) - \kappa_0 \sigma_i(X) \Big) + \tau \kappa_0 \|X\|_*, \\
&\st &~ l \leq X_{ij} \leq u,\quad \forall~ (i, j) \in [p] \times [q],
\end{eqnarray*}
where $[n]:=\{1,2,\cdots,n\}$ given any $n\in\mathbb{Z}_+$, the function $\sigma_i(X)$ represents the $i$-th largest singular value of the matrix $X$,  and $\|X\|_*$ is the nuclear norm of $X$ defined as the sum of its singular values. The coefficient $\kappa_0$ and the function $g$ are defined as  $\kappa_0 = \nu/\theta$, $g(t) = \nu \log(1 + |t|/\theta)$, respectively.
\end{example}
Define $\mathcal{K} = \{Z \in \mathbb{R}^{p \times q}: l \leq Z_{ij} \leq u, (i,j) \in [p] \times [q]\}$, $\mathcal{Q}=\{0\}$, $c(X)= 0$, and 
\[
f(X)=\frac{1}{2}\|\Pi_\Omega(X - Q)\|^2 + \tau \sum_{i=1}^{\min\{p,q\}} \left[ \kappa(\sigma_i(X)) - \kappa_0 \sigma_i(X) \right],\ \ h(X)=\tau \kappa_0 \|X\|_*.
\]
Then the problem can be expressed in the form of  \eqref{nonconvex-composite}. The corresponding AL function can be obtained as:
    \begin{align}\label{sec2:completion:AL}
        \mathbb{L}_{\rho}(X,Y,S) = f(X)+h(X)
        +\frac{\rho}{2}\| \hat{\Pi}_{\mathcal{K}}(X+S/\rho)\|^2 - \frac{1}{2\rho}\|S\|^2.
    \end{align}
Note that $h(X)$ is retained in the AL function due to the availability of the following closed-form proximal operator for the nuclear norm:
$$
\text{prox}_{\tau \kappa_0 \|\cdot\|_*}(X) = U  \text{diag}(\mathcal{S}_{\tau \kappa_0}(\sigma(X))  V^\top,
$$
where $X = U \text{diag}(\sigma(X))V^\top$ is the singular decomposition of $X$ and $\text{diag}(\mathcal{S}_{\tau \kappa_0}(\sigma(X))$ is a diagonal matrix obtained by thresholding the singular values of $X$:
$$
\mathcal{S}_{\tau \kappa_0}(\sigma_i(X)) =\mathrm{sign}(\sigma_i(X))[\sigma_i(X) - \lambda]_+
$$
where $\sigma_i$ denotes the $i$-th singular value of $X$.

\begin{example}[Maxcut SDP \cite{jia2023augmented}]
The SDP relaxation of the maxcut problem is
$$
\max_{X\in \mathbb{S}^n}\ \operatorname{Tr}(LX) \quad \st\ \operatorname{diag} (X)=\mathbf{1}_n,\ X \succeq 0, \ \operatorname{rank} (X) \leq 1.
$$
\end{example}
Let us set $\mathcal{Q}=\{\mathbf{1}_n\}, \ \mathcal{K} =\{Z \in \mathbb{S}^n \mid Z \succeq 0, \text { rank }(Z) \leq 1\}$ and 
\[
f(X)=\operatorname{Tr}(LX), \quad h(X)=0, \quad c(X)= \operatorname{diag} (X),
\]
then the problem can fit into \eqref{nonconvex-composite}. 
Let $X\in\mathbb{S}^n$ denote an arbitrary symmetric matrix with maximum eigenvalue $\lambda_{\max}$ and corresponding normalized eigenvector $v$ (note that $\lambda_{\max}$ and $v$ are not necessarily unique), then $\max (\lambda_{\max}, 0) v v^\top$ is a
projection of $X$ to $\mathcal{K}$. The corresponding AL function can be written as:
\begin{equation*}\label{sec2:maxcut:AL}
     \begin{aligned}
        \mathbb{L}_{\rho}(X,y,S) = &\operatorname{Tr}(LX)+\frac{\rho}{2} \| \mathrm{diag}(X) - \mathbf{1}_n + y/\rho\|^2\\
       & +\frac{\rho}{2}\| \hat{\Pi}_{\mathcal{K}}(X+S/\rho)\|^2 - \frac{1}{2\rho}(\|y\|^2 + \|S\|^2).
    \end{aligned}   
 \end{equation*}

\section{Semismooth Newton method}\label{app:seminewton:sub} 
In many practical applications, one has to find a zero of a system of
nonlinear equations:
\begin{equation}\label{appen:eq:nonlineareq}
    F(x) = 0,
\end{equation}
where $F:\mathbb{R}^n \rightarrow \mathbb{R}^n$ is locally Lipschitz but not necessarily continuously differentiable. Since \eqref{appen:eq:nonlineareq} is nonsmooth,  the classical Newton method can not be applied directly. Some extensions of the Newton and quasi-Newton methods have been developed for nonsmooth equations.  In this section,  we focus on the case when $F$ is semismooth.

Let us first recall several objects from nonsmooth analysis which
provide generalizations of the classical differentiability concept. Since $F$ is locally Lipschitz, it follows from Rademacher’s Theorem \cite{rademacher1919partielle} that $F$ is almost everywhere differentiable. Denote 
$$
D_F: = \{  x \mid F ~\text{is differentiable at} ~x \}.
$$
The $B$-subdifferential (here ``$B$'' stands for ``Bouligand'') of $F$ at $x$ is defined as 
$$
\partial_B F: = \{G\in \mathbb{R}^{n\times n}: \exists 
 \{x_k\} \subset D_F ~\text{with}~ x_k \rightarrow x, \nabla F(x_k) \rightarrow G\}. 
$$
Then the Clarke's generalized Jacobian of $F$ at $x$ can be defined by
$$
\partial F(x) = co(\partial_B F(x)),
$$
where $co$ denotes the convex hull. Now we give the formal definition of the semismooth property as follows.
\begin{definition}
   Let $\Omega \subset \mathbb{R}^n$ be nonempty and open.  For function $F:\Omega\rightarrow\mathbb{R}^n$ which is locally Lipschitz continuous,  we call $F$ semismooth at $x$ if it satisfies
    \begin{itemize}
        \item[(i)] $F$ is directional differentiable at $x
        $;
        \item[(ii)] For any $d$ and $J\in\partial F(x+d)$,  the following relationship holds:
        $$
            \|F(x+d)-F(x)-Jd\|=o(\|d\|),\quad \mathrm{as}\ d\rightarrow 0.
        $$
    \end{itemize}
\end{definition} 
Moreover, $F$ is called strongly semismooth at $x$ if  $F$ is directional differentiable and 
         $$
            \|F(x+d)-F(x)-Jd\|=\mathcal{O}(\|d\|^2),\quad \mathrm{as}\ d\rightarrow 0.
        $$

Since $\partial F$ is a generalization of $\nabla F$ in the case where $F$ is not (continuously) differentiable, the semismooth Newton can be designed similar to the classical Newton method. 
The main iterative process of the semismooth Newton method is given as follows:
\begin{equation}\label{appen:ssn-iter}
    x^{k+1} = x^k + d^k,
\end{equation}
where $d^k$ is the Newton direction from the following linear equation:
\begin{equation} \label{appen:ssn-iter2}
    J(x^k) d^k = -F(x^k),~~ J(x^k) \in \partial F(x^k).
\end{equation}
  Note that  \eqref{appen:ssn-iter} reduces to the classical Newton method for a system of equations if $F$ is continuously differentiable. Using semismoothness, Qi and Sun \cite{qi1993nonsmooth} give the following convergence theorem for  the scheme \eqref{appen:ssn-iter}-\eqref{appen:ssn-iter2}.
\begin{theorem}
    Suppose that $F$ is semismooth at $x^*$, where $x^*$ satisfies $F\left(x^*\right)=0$.  If all $J \in \partial F\left(x^*\right)$ are nonsingular, then the scheme \eqref{appen:ssn-iter}-\eqref{appen:ssn-iter2} 
 is $Q$-superlinearly convergent in a neighborhood of $x^*$. Furthermore, if $F$ is strongly semismooth at $x^*$, it
 is quadratically convergent.
\end{theorem}

For many convex composite optimization problems, the system $F(x)$ is often constructed from first-order type methods \cite{xiao2018generalized,li2018semismooth,milzarek2019stochastic,hu2022local,deng2023augmented}. The corresponding generalized Jacobian $J$ is only positive semi-definite and singular. Then a regularized Newton step $d^k$ is obtained by solving
\begin{equation}\label{appen:eq:new}
    (J(x^k) + \mu_k I_n) d^k = - F(x^k),
\end{equation}
where $\mu_k>0$ is a regularized parameter.  A prototype method is outlined in Algorithm \ref{alg:ssn}. For more detailed  on the choice of the regularized parameter $\mu_k$ and the  convergence analysis, the readers are referred to \cite{xiao2018generalized,li2018semismooth,milzarek2019stochastic,hu2022local,deng2023augmented}. 

\begin{algorithm}[htbp]
\caption{The semismooth Newton method for nonsmooth equations \eqref{appen:eq:nonlineareq}}\label{alg:ssn}
\begin{algorithmic}[1]
\REQUIRE Initial point $x^0\in \mathcal{M}, k=0$. 
\WHILE {not converge}
\STATE Obtain the Newton direction $d^k$ by solving equation \eqref{appen:eq:new}. 

\STATE Set $x^{k+1} = x^k + d^k$.
\STATE Set $k=k+1$.
\ENDWHILE
\end{algorithmic}
\end{algorithm}

In fact, the semismooth Newton method can also be applied to solve the convex optimization problem:
\begin{equation}\label{appen:prob:gene}
    \min_x \ f(x),
\end{equation}
where $f:\mathbb{R}^n \rightarrow \mathbb{R}$ is a convex and continuously differentiable with gradient $\nabla f$. Here, we only assume $\nabla f$ is locally Lipschitz. Solving problem \eqref{appen:prob:gene} is equivalent to finding a root of the following nonlinear system:
\begin{equation}\label{appen:prob:gene-newton}
    F(x):=\nabla f(x) = 0.
\end{equation}
 Therefore, we can directly apply the semismooth Newton method to solve problem \eqref{appen:prob:gene-newton}.  Moreover, one can design a damped semismooth Newton method by utilizing the function value information. In particular,  we first obtain a Newton direction $d^k$ by solving the following linear equation:
\begin{equation*}
   ( J(x^k) + \mu_k I_n) d = - \nabla f(x^k),~~~ J(x^k) \in \partial (\nabla f(x^k)).
\end{equation*}
Then the next iterate $x^{k+1}$ is given by 
\begin{equation*}
    x^{k+1} =x^k +  \alpha_k d^k,
\end{equation*}
where $\alpha_k$ is a stepsize  satisfies the Armijo's line search condition with a constant $\sigma>0$ as
\begin{equation*}
    f(x^k + \alpha_k d^k) \leq f(x^k) + \sigma \alpha_k \langle \nabla f(x^k), d^k \rangle.  
\end{equation*}
{The related convergence analysis} can be found in \cite{zhao2010newton}.

\newpage

\small

\normalem
\bibliographystyle{plain}

\bibliography{reference}

\end{document}